\pdfoutput=1  %for arXiv rendering of xymatrix to come out right

\documentclass[11pt]{article}
\usepackage[margin=.8in,left=.8in]{geometry}
\usepackage{amsthm}
\usepackage{mathtools}
\usepackage{xcolor}
\usepackage{stmaryrd}

\usepackage{enumitem}
\setlist{nosep}
\usepackage{adjustbox}

\usepackage{hyperref}
\hypersetup{
  colorlinks = true,
  allcolors=.  % switches off coloring of hyperlinks
}

\usepackage[all]{xy}

\usepackage{tikz}
\usetikzlibrary{cd}

\DeclareMathAlphabet{\mathpzc}{OT1}{pzc}{m}{it} %for math script

\newcommand\mathscr[1]{\scalebox{1.1}{$\mathpzc{#1}$}}

\newcommand{\proofstep}[1]{\scalebox{.85}{ #1 }}

\usepackage[titletoc]{appendix}

\usepackage[safe]{tipa} % for \esh

\definecolor{darkblue}{rgb}{0.05,0.25,0.65}
\definecolor{greenii}{RGB}{20,140,10}
\definecolor{lightgray}{rgb}{0.9,0.9,0.9}
\definecolor{orangeii}{RGB}{200,100,5}

\usepackage{multirow}

\usepackage{stmaryrd}    % for \sslash

\usepackage{enumerate} %Roman enumerate

\usepackage{amssymb,amsmath}
\newcounter{sqindex}

\usepackage{mathptmx}
\usepackage{amsmath}
\usepackage{graphicx}
\DeclareRobustCommand{\coprod}{\mathop{\text{\fakecoprod}}}
\newcommand{\fakecoprod}{%
  \sbox0{$\prod$}%
  \smash{\raisebox{\dimexpr.9625\depth-\dp0}{\scalebox{1}[-1]{$\prod$}}}%
  \vphantom{$\prod$}%
}

\newcommand{\underlying}{\mathrm{undrlng}}

\newcommand{\Sets}{
  \mathrm{Set}
}

\newcommand{\simplicial}{
  \Delta
}

\newcommand{\Groups}{
  \mathrm{Grp}
}

\newcommand{\Abelian}{
  \mathrm{Ab}
}

\newcommand{\AbelianGroups}{
  \Abelian\Groups
}

\newcommand{\NormalizedChains}{N_\bullet}

\newcommand{\ConnectiveChainComplexes}{\mathrm{Ch}^{\geq 0}}

\newcommand{\Actions}[1]{
  {#1}\,\mathrm{Act}
}

\newcommand{\HomotopyQuotient}[2]{
  #1 \!\sslash\! #2
}

\newcommand{\Groupoids}{
  \mathrm{Grpd}
}

\newcommand{\InfinityGroupoids}{
  \Groupoids_\infty
}

\newcommand{\Toposes}{
  \mathrm{Topos}
}

\newcommand{\Topos}{
  \mathbf{H}
}

\newcommand{\Maps}[3]{
  \mathrm{Map}
  #1(
    #2
    ,\,
    #3
  #1)
}

\newcommand{\Homs}[3]{
  \mathrm{Hom}
  #1(
    #2
    ,\,
    #3
  #1)
}

\newcommand{\infinityToposes}{
  \Toposes_\infty
}

\newcommand{\InfinityToposes}{
  \infinityToposes
}

\newcommand{\Categories}{
  \mathrm{Cat}
}

\newcommand{\Functors}{
  \mathrm{Fnctr}
}

\newcommand{\Enriched}{
  \mathrm{Enr}
}

\newcommand{\EnrichedCategories}[1]{
  {#1}\Enriched\Categories
}

\newcommand{\SimplicialCategories}{
  \Categories_{\simplicial}
}

\newcommand{\SimplicialFunctors}{
  \Functors_{\simplicial}
}

\newcommand{\Cofibrations}{\mathrm{Cof}}
\newcommand{\Fibrations}{\mathrm{Fib}}
\newcommand{\WeakEquivalences}{\mathrm{W}}

\newcommand{\ModelCategories}{
  \mathrm{Mdl}\Categories
}

\newcommand{\Localization}[1]{
  \mathrm{Loc}_{{}_{#1}}
}

\newcommand{\HomotopyCategory}{
  \mathrm{Ho}
}

\newcommand{\CombinatorialModelCategories}{
  \mathrm{Comb}\ModelCategories
}

\newcommand{\ProperCombinatorialSimplicialModelCategories}{
  \CombinatorialModelCategories_{\simplicial, \mathrm{prop} }
}

\newcommand{\SimplicialSets}{
  \simplicial\Sets
}

\newcommand{\SimplicialGroups}{
  \Groups(\SimplicialSets)
}

\newcommand{\Integers}{\mathbb{Z}}

\newcommand{\Cartesian}{
  \mathrm{Cart}
}

\newcommand{\CartesianSpaces}{
  \Cartesian\Spaces
}

\newcommand{\topological}{
  \mathrm{Top}
}

\newcommand{\Spaces}{
  \mathrm{Spc}
}

\newcommand{\TopologicalSpaces}{
  \topological\Spaces
}

\newcommand{\Discrete}{
  \mathrm{Disc}
}

\newcommand{\Shape}{
  \mathrm{Shp}
}

\newcommand{\shape}{
  \raisebox{1pt}{\rm\textesh}
}

\newcommand{\CohesiveCircle}{S^1_{\mathrm{coh}}}
\newcommand{\ExtendedInertia}{\Lambda_{\CohesiveCircle}}

\newcommand{\Smooth}{\mathrm{Smth}}

\newcommand{\Manifolds}{\mathrm{Mnfld}}

\newcommand{\SmoothManifolds}{
  \Smooth\Manifolds
}

\newcommand{\Diffeological}{
  \mathrm{Dfflg}
}

\newcommand{\DiffeologicalSpaces}{
  \Diffeological\Spaces
}

\newcommand{\ContinuousDiffeology}{
  \mathrm{Cdfflg}
}

\newcommand{\Sheaves}{
  \mathrm{Sh}
}

\newcommand{\Presheaves}{
  \mathrm{P}\Sheaves
}

\newcommand{\SimplicialPresheaves}{
  \simplicial\Presheaves
}

\newcommand{\InfinitySheaves}{
  \Sheaves_{\infty}
}

\newcommand{\infinitySheaves}{
  \InfinitySheaves
}

\newcommand{\SmoothInfinityGroupoids}{
  \Smooth\InfinityGroupoids
}
\newcommand{\SmoothInfinityStacks}{\SmoothInfinityGroupoids}

\newcommand{\VectorSpaces}{
    \mathrm{VectorSpaces}
      _{\scalebox{.5}{$\mathbb{R}$}}
}

\newcommand{\mapsup}{\rotatebox[origin=c]{90}{$\mapsto$}}
\newcommand{\mapsdown}{\rotatebox[origin=c]{-90}{$\mapsto$}}

\usepackage[new]{old-arrows}   %For \longhookrightarrow

\newdir{> }{{}*!/10pt/@{>}}

\usepackage{amssymb}

\def\acts{\raisebox{1.4pt}{\;\rotatebox[origin=c]{90}{$\curvearrowright$}}\hspace{.5pt}}

\DeclareRobustCommand{\rchi}{{\mathpalette\irchi\relax}}
\newcommand{\irchi}[2]{\raisebox{\depth}{$#1\chi$}} % inner command, used by \rchi

\makeatletter
\newif\if@sup
\newtoks\@sups
\def\append@sup#1{\edef\act{\noexpand\@sups={\the\@sups #1}}\act}%
\def\reset@sup{\@supfalse\@sups={}}%
\def\mk@scripts#1#2{\if #2/ \if@sup ^{\the\@sups}\fi \else%
  \ifx #1_ \if@sup ^{\the\@sups}\reset@sup \fi {}_{#2}%
  \else \append@sup#2 \@suptrue \fi%
  \expandafter\mk@scripts\fi}
\def\tensor#1#2{\reset@sup#1\mk@scripts#2_/}
\def\multiscripts#1#2#3{\reset@sup{}\mk@scripts#1_/#2%
  \reset@sup\mk@scripts#3_/}
\makeatother

\makeatletter
\newbox\slashbox \setbox\slashbox=\hbox{$/$}
\def\itex@pslash#1{\setbox\@tempboxa=\hbox{$#1$}
  \@tempdima=0.5\wd\slashbox \advance\@tempdima 0.5\wd\@tempboxa
  \copy\slashbox \kern-\@tempdima \box\@tempboxa}
\def\slash{\protect\itex@pslash}
\makeatother

\def\clap#1{\hbox to 0pt{\hss#1\hss}}
\def\mathllap{\mathpalette\mathllapinternal}
\def\mathrlap{\mathpalette\mathrlapinternal}
\def\mathclap{\mathpalette\mathclapinternal}
\def\mathllapinternal#1#2{\llap{$\mathsurround=0pt#1{#2}$}}
\def\mathrlapinternal#1#2{\rlap{$\mathsurround=0pt#1{#2}$}}
\def\mathclapinternal#1#2{\clap{$\mathsurround=0pt#1{#2}$}}

\let\oldroot\root
\def\root#1#2{\oldroot #1 \of{#2}}
\renewcommand{\sqrt}[2][]{\oldroot #1 \of{#2}}

\DeclareSymbolFont{symbolsC}{U}{txsyc}{m}{n}
\SetSymbolFont{symbolsC}{bold}{U}{txsyc}{bx}{n}
\DeclareFontSubstitution{U}{txsyc}{m}{n}

\DeclareSymbolFont{stmry}{U}{stmry}{m}{n}
\SetSymbolFont{stmry}{bold}{U}{stmry}{b}{n}

\DeclareFontFamily{OMX}{MnSymbolE}{}
\DeclareSymbolFont{mnomx}{OMX}{MnSymbolE}{m}{n}
\SetSymbolFont{mnomx}{bold}{OMX}{MnSymbolE}{b}{n}
\DeclareFontShape{OMX}{MnSymbolE}{m}{n}{
    <-6>  MnSymbolE5
   <6-7>  MnSymbolE6
   <7-8>  MnSymbolE7
   <8-9>  MnSymbolE8
   <9-10> MnSymbolE9
  <10-12> MnSymbolE10
  <12->   MnSymbolE12}{}

\usepackage{cleveref}

\crefformat{section}{\S#2#1#3} % see manual of cleveref, section 8.2.1
\crefformat{subsection}{\S#2#1#3}
\crefformat{subsubsection}{\S#2#1#3}

\theoremstyle{italics}
\newtheorem{theorem}{Theorem}[section]
\newtheorem{lemma}[theorem]{Lemma}
\newtheorem{proposition}[theorem]{Proposition}
\newtheorem{fact}[theorem]{Fact}

\theoremstyle{definition}
\newtheorem{definition}[theorem]{Definition}
\newtheorem{notation}[theorem]{Notation}
\newtheorem{example}[theorem]{Example}

\newtheorem{remark}[theorem]{Remark}

\usepackage{amsfonts}
\usepackage{colortbl}

\renewcommand{\emph}{\textit}

\def\smooth{\rotatebox[origin=c]{70}{$\subset$}}
\def\orbisingular{\rotatebox[origin=c]{70}{$\prec$}}

\begin{document}

%%%%%%%%%%%%%%%%%%%%%%%%%%%%%%%%%%%%%%%%%%%%%%%%%
\title{Cyclification of Orbifolds}
%%%%%%%%%%%%%%%%%%%%%%%%%%%%%%%%%%%%%%%%%%%%%%%%%
\author{
  Hisham Sati${}^1$
  ,\quad
  Urs Schreiber${}^1$
}

\maketitle

\begin{abstract}
   Inertia orbifolds homotopy-quotiented by rotation of
  {\it geometric} loops play a fundamental role
  not only in ordinary cyclic cohomology,
  but more recently
  in constructions of
  equivariant Tate-elliptic cohomology
  and
  generally of transchromatic characters on generalized cohomology theories.
  Nevertheless, existing discussion of such {\it cyclified stacks}
  has been relying on ad-hoc component presentations
  with intransparent and unverified stacky homotopy type.

  \smallskip
  Following our previous formulation of
  transgression of cohomological charges (``double dimensional reduction''),
  we explain how cyclification of $\infty$-stacks
  is a fundamental and elementary base-change construction
  over moduli stacks in
  cohesive higher topos theory (cohesive homotopy type theory).
  We prove that Ganter/Huan's extended inertia groupoid
  used to define equivariant quasi-elliptic cohomology is indeed a model
  for this intrinsically defined cyclification of orbifolds,
  and we show that cyclification
  implements transgression in group cohomology in general, and hence in particular the transgression
  of degree-4 twists of
  equivariant Tate-elliptic cohomology to degree-3 twists of
  orbifold K-theory on the cyclified orbifold.

  \smallskip
  As an application, we show that the universal shifted integral 4-class of
  equivariant 4-Cohomotopy theory on ADE-orbifolds induces the Platonic
  4-twist of ADE-equivariant Tate-elliptic cohomology; and we close by explaining how this should relate to elliptic M5-brane genera, under our previously formulated {\it Hypothesis H}.
\end{abstract}

\medskip

\tableofcontents

\vfill

\hrule
\vspace{5pt}

{
\footnotesize
\noindent
\def\arraystretch{1}
\tabcolsep=0pt
\begin{tabular}{ll}
${}^1$
&
Mathematics, Division of Science; and
\\
&
Center for Quantum and Topological Systems (CQTS),
\\
&
NYUAD Research Institute,
\\
&
New York University Abu Dhabi, UAE.
\end{tabular}

\vspace{.2cm}

\noindent
The authors acknowledge the support by {\it Tamkeen} under the NYU Abu Dhabi Research Institute grant {\tt CG008}.
}

\newpage

%%%%%%%%%%%%%%%%%%%%%%%%%%%%%%%%%%%%
\section{Introduction and Overview}
\label{IntroductionAndOverview}
%%%%%%%%%%%%%%%%%%%%%%%%%%%%%%%%%%%

\noindent
{\bf Classical cyclic loop spaces.}
Topological spaces of free loops (e.g. \cite{ChataurOancea15}) in a given topological space $\mathcal{X}$, but homotopy-quotiented  by the rigid rotation action of the topological circle group $\CohesiveCircle$ on itself

\vspace{.1cm}
$$
  \overset{
    \mathclap{
      \raisebox{4pt}{
        \scalebox{.7}{
        \color{purple}
        \bf
        \def\arraystretch{.8}
        \begin{tabular}{c}
          cyclic
          \\
          loop space
        \end{tabular}
        }
      }
    }
  }{
    \mathrm{Cyc}(\mathcal{X})
  }
  \;:=\;
  \HomotopyQuotient
   {
    \underset{
      \mathclap{
        \scalebox{.7}{
          \color{darkblue}
          \bf
          \def\arraystretch{.8}
          \begin{tabular}{c}
            free
            loop space
          \end{tabular}
        }
      }
    }{
    \underbrace{
    \Maps{\big}
      {
        \overset{
          \mathclap{
            \hspace{7pt}
            \rotatebox{49}{
              \rlap{
              \hspace{-19pt}
              \scalebox{.65}{
                \color{orangeii}
                \bf
                \def\arraystretch{.75}
                \begin{tabular}{c}
                  circle
                \end{tabular}
              }
              }
            }
          }
        }{
          \CohesiveCircle
        }
      }
      {
        \overset{
          \mathclap{
            \hspace{8pt}
            \rotatebox{49}{
              \rlap{
              \hspace{-24pt}
              \scalebox{.65}{
                \color{orangeii}
                \bf
                \def\arraystretch{.75}
                \begin{tabular}{c}
                  topological
                  \\
                  space
                \end{tabular}
              }
              }
            }
          }
        }{
          \mathcal{X}
        }
      }
      }
      }
   }
   \;
    {
        \underset{
          \mathclap{
            \rotatebox{31}{
              \llap{
              \hspace{-24pt}
              \scalebox{.65}{
                \color{greenii}
                \bf
                \def\arraystretch{.75}
                \begin{tabular}{c}
                  homotopy quotient
                  \\
                  by loop rotations
                \end{tabular}
              }
              \hspace{-14pt}
              }
            }
          }
        }{
         \CohesiveCircle
        }
   }
$$
\vspace{.6cm}

\noindent
have a long tradition in the study of the elliptic cohomology (see eg. \cite{Rezk15}) of $\mathcal{X}$, at least in the extreme but still surprisingly
rich limit of restricting to those loops which are effectively constant, and traditionally perceived through their $\CohesiveCircle$-equivariant K-cohomology
(\cite[\S 5]{KitchlooMorava04}, review in \cite[\S 6.2]{Dove19}).

\smallskip
In the generality where $\mathcal{X} = \HomotopyQuotient{X}{G}$ is itself a homotopy quotient space, this goes back to \cite{Witten1988}, where a topological
point-set model for $\Maps{\big}{\CohesiveCircle}{\HomotopyQuotient{X}{G}}$ is called a ``twisted loop space'' of $X$. There this is thought of as a model
for the configuration space of closed strings propagating on the orbifold $\HomotopyQuotient{X}{G}$ (whose non-trivial orbifold transition functions are
traditionally called their ``twisted sectors'' \cite[p. 3]{DixonHarveyVafaWitten85}, whence Witten's terminology, see also e.g.  \cite{Stapleton13}).
Accordingly, some authors call $\mathrm{Cyc}(\mathcal{X})$ the {\it string space of $\mathcal{X}$} \cite[\S 4.8.1]{Chataur05}\cite[p. 1]{BoekstedtOttosen05}.

\smallskip
However, this terminology is neither widely adopted nor quite appropriate, since not all free loop spaces arise as configuration spaces of strings -- not even
in string theory; and even when they do one is going to be interested in their twisted cohomology in addition to, and in a sense different from, the string's
``twisted sectors''. But {\it Jones' theorem} (\cite[Thm. A]{Jones87},
review in \cite[Cor. 7.3.14]{Loday92}\cite[\S 3,4]{Loday15}) shows that the ordinary cohomology of $\mathrm{Cyc}(X)$ for a simply-connected topological space $X$ is its {\it cyclic cohomology} --- which is meaningful and standard mathematical terminology. Therefore we call $\mathrm{Cyc}(\mathcal{X})$ the {\it cyclification}
of $\mathcal{X}$ (following \cite[\S 3]{FSS16TDuality}\cite[\S 2.2]{BSS18}), also in line with the modern terminology of {\it cyclotomic spectra} \cite{BlumbergMandell15}\cite{NikolausScholze18}, obtained from approximating the free loop $\CohesiveCircle$-space $\Maps{}{\CohesiveCircle}{X}$ by
its suspension spectrum.

\medskip

\noindent
{\bf Cyclic inertia orbifolds?}
In any case, when $G \acts \, X$ carries the geometric structure of a $G$-manifold (as it certainly does already in the motivating examples from string theory),
one should expect a more fine-grained incarnation of  $\mathrm{Cyc}(\HomotopyQuotient{X}{G})$, lifting it from plain homotopy theory to the geometric homotopy
theory (homotopy topos theory \cite{Lurie09HTT}\cite{Rezk10}) of {\it orbifolds} regarded (cf. \cite{Lerman10}\cite{SS20OrbifoldCohomology}) as {\it stacks}
(useful background references for our purposes are \cite{Hollander08}\cite{Jardine15}),  specifically topological stacks or differentiable stacks
(cf., e.g., \cite{Carchedi11}).
Concretely, one should expect to make sense of the cyclification of $\HomotopyQuotient{X}{G}$ regarded as a smooth orbifold, so that the orbifold K-theory of
the {\it cyclic loop stack} $\mathrm{Cyc}(\HomotopyQuotient{X}{G})$ would reflect the properly $G$-equivariant elliptic cohomology of $X$.

\smallskip
Finally, one should expect that such a cyclified orbifold may canonically be restricted to its ``essentially constant loops'', which in themselves ought to
constitute the familiar ``inertia stack'' (recalled in \cref{TheGRHInertiaOrbifold} below)
$\Lambda (\HomotopyQuotient{X}{G}) \,=\, \Maps{}{\mathbf{B}\mathbb{Z}}{\HomotopyQuotient{X}{G}}$ of the orbifold. In conclusion then, one should expect
that the ``essentially constant''-cyclification of an orbifold should be a homotopy $\CohesiveCircle$-quotient of the inertia orbifold embedded inside
the cyclified orbifold:
\vspace{-2mm}
\begin{equation}
  \label{CyclicStackInIntroduction}
  \begin{tikzcd}
  \HomotopyQuotient
  {
  \overset{
    \mathclap{
      \raisebox{3pt}{
        \scalebox{.7}{
          \color{darkblue}
          \bf
          inertia orbifold
        }
      }
    }
  }{
  \overbrace{
  \Maps{\big}
    {
      \mathbf{B}\mathbb{Z}
    }
    {
      \HomotopyQuotient
        { X }
        { G }
    }
  }
  }
  }{
    \underset{
      \mathclap{
        \raisebox{-7pt}{
          \scalebox{.7}{
            \color{orangeii}
            \bf
            \def\arraystretch{.7}
            \begin{tabular}{c}
              stacky
              \\
              $\CohesiveCircle$-quotient
            \end{tabular}
          }
        }
      }
    }{
      \CohesiveCircle
    }
  }
\quad   \ar[
    rrr,
    hook,
    "{ \scalebox{.55}{
        \color{greenii}
        \bf
        include as essentially
      } }",
    "{ \scalebox{.55}{
        \color{greenii}
        \bf
         constant loops
    } }"{swap}
  ]
  &&\phantom{AA}&
  \quad
  \HomotopyQuotient
  {
  \overset{
    \mathclap{
      \raisebox{+4pt}{
        \scalebox{.7}{
          \color{darkblue}
          \bf
          smooth loop stack
        }
      }
    }
  }{
  \overbrace{
  \Maps{\big}
    {
      \CohesiveCircle
    }
    {
      \HomotopyQuotient
        { X }
        { G }
    }
    }
  }
  }{
    \underset{
      \mathclap{
        \raisebox{-7pt}{
          \scalebox{.7}{
            \color{orangeii}
            \bf
            \def\arraystretch{.7}
            \begin{tabular}{c}
              stacky
              \\
              $\CohesiveCircle$-quotient
            \end{tabular}
          }
        }
      }
    }{
      \CohesiveCircle
    }
  }
  \quad =:\;
  \overset{
    \mathclap{
      \raisebox{4pt}{
        \scalebox{.7}{
          \color{darkblue}
          \bf
          cyclified orbifold
        }
      }
    }
  }{
    \mathrm{Cyc}(\HomotopyQuotient{X}{G})
  }
  \,.
  \end{tikzcd}
\end{equation}
\vspace{-.3cm}

\vspace{-2mm}
For a long time it had been unclear how to bear out these expectations, nor had they been approached with tools from geometric homotopy theory. After crucial proposals in  \cite[Def. 2.3 \& 3.1]{Ganter07StringyPower}\cite[Def. 2.6]{Ganter13}, and following advice by C. Rezk, a candidate component model for the expected stack was finally given in \cite[Def. 2.5, 2.9]{Huan18QuasiEllipticI} -- see \eqref{SkeletonOfHuanInertiaStackOfGoodOrbifold} below -- and justified by demonstrating that it does support a satisfactory notion of ``quasi-elliptic'' cohomology (further discussed in \cite{Huan18QuasiTheories}\cite{HuanSpong20TwistedQuasiEllipic}\cite{HuanYoung22}).

\medskip

What has been left open is a proof that the  component presentation \eqref{SkeletonOfHuanInertiaStackOfGoodOrbifold} from \cite{Huan18QuasiEllipticI} does present the abstract stacky construction \eqref{CyclicStackInIntroduction}, hence does satisfy expected abstract properties. This is what we prove here -- in Thm. \ref{AbstractCharacterizationOfHuanInertiaOrbifold}.

\medskip

\noindent
{\bf Embedding into cohesive homotopy theory.}\label{EmbeddingIntoCohesiveHomotopyTheory}
We approach this issue by embedding the situation into the ``extremely convenient'' (in the technical sense going back to \cite{Steenrod67}) higher topos of smooth $\infty$-groupoids ($\infty$-stacks over a site of smooth manifolds), laid out\footnote{The articles \cite{SS20OrbifoldCohomology}\cite{SS21EPB} go further to the ``singular-cohesive'' homotopy theory of equivariant smooth stacks, ultimately needed for their properly equivariant cohomology theory. For brevity, here we do not dwell on this further step, but our results immediately make cyclified orbifolds available in this context of proper equivariant homotopy theory.} in \cite{SS20OrbifoldCohomology}\cite{SS21EPB} (a reference list for our notations is given below in Ntn. \ref{BasicNotation}).

\vspace{.1cm}
\noindent
\begin{equation}
\label{SequenceOfEmbeddingsOfAmbientCategories}
\hspace{-20pt}
\mathclap{
\raisebox{-1.7cm}{
\includegraphics[width=\textwidth]{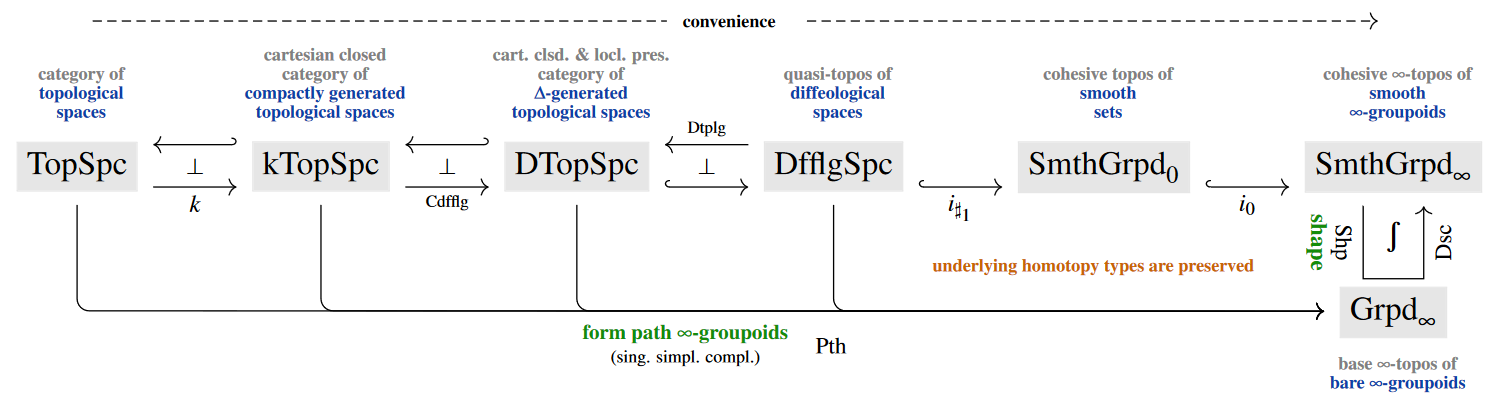}}
}
\end{equation}

As indicated by this homotopy-commutative diagram of $\infty$-categories (reproduced from \cite{SS21EPB}, where the relevant
proofs are given in \S3.3.1), each stage in this sequence of ever more ``convenient'' categories of spaces comes with its own
notion of underlying homotopy types. All these notions are compatible and culminate in the operation
of sending a smooth $\infty$-groupoid $\mathcal{X}$ to its {\it pure shape} $\shape \mathcal{X}$ (cf. \cite[fn 1]{SS21EPB}),
which generalizes (\cite{BBP19}\cite[pp. 144]{SS21EPB}) the traditional construction of {\it singular simplicial complexes} (thought of as higher path $\infty$-groupoids) from topological spaces to smooth higher stacks:

\vspace{1cm}
\begin{equation}
  \label{ShapeAsPathInfinityGroupoid}
    \underset{
    \mathclap{ \phantom{aaa}
      \raisebox{-4pt}{
        \def\arraystretch{.8}
        \scalebox{.6}{
          \color{greenii}
          \bf
          \begin{tabular}{c}
            shape unit
            \\
            transformation
          \end{tabular}
          \hspace{25pt}
          \phantom{AA}
          \rlap{
          \color{greenii}
          \bf
          \begin{tabular}{c}
            includes
            \\
            constant paths
            \rlap{
              \;\;
              \color{black}
              \normalfont
              into
            }
          \end{tabular}
          }
        }
      }
    }
  }{
    \mbox{
      \color{purple}
      $\eta^{\scalebox{.8}{$\shape$}}_{\mathcal{X}}$
    }
  }
  \qquad :\;\;
  \mathcal{X}
  \xrightarrow{\phantom{--------}}
  \quad
  \overset{
    \mathclap{
      \raisebox{2pt}{
        \rotatebox{50}{
          \rlap{
            \hspace{-9pt}
            \scalebox{.7}{
            \color{orangeii}
            \bf
            pure shape of
            }
          }
        }
      }
    }
  }{
    \shape
  }
  \;
  \overset{
    \mathclap{
      \raisebox{2pt}{
        \rotatebox{50}{
          \rlap{
            \hspace{-9pt}
            \scalebox{.7}{
            \color{darkblue}
            \bf
            smooth type
            }
          }
        }
      }
    }
  }{
    \mathcal{X}
  }
  \overset{
    \mathclap{
      \raisebox{2pt}{
        \rotatebox{50}{
          \rlap{
            \hspace{-9pt}
            \scalebox{.7}{
            \color{black}
            is equivalently
            }
          }
        }
      }
    }
  }{
    \qquad \simeq \qquad
  }
  \overset{
    \mathclap{
      \raisebox{2pt}{
        \rotatebox{50}{
          \rlap{
            \hspace{-9pt}
            \scalebox{.7}{
            amalgamation of
            }
          }
        }
      }
    }
  }{
  \underset{\underset{n : \mathbb{N}}{\longrightarrow}}{\lim}
  }
  \;
  \underset{
    \mathclap{
      \scalebox{.7}{
        \color{darkblue}
        \bf
        $n$-dim. paths in $\mathcal{X}$
      }
    }
  }{
  \underbrace{
  \overset{
    \mathclap{
      \raisebox{2pt}{
        \rotatebox{50}{
          \rlap{
            \hspace{-9pt}
            \scalebox{.7}{
            \color{darkblue}
            \bf
            smooth images of
            }
          }
        }
      }
    }
  }{
    \phantom{\mathrm{Maps}}
  }
  \hspace{-24pt}
  \Maps{\big}
    {
  \overset{
    \mathclap{
      \raisebox{2pt}{
        \rotatebox{50}{
          \rlap{
            \hspace{-9pt}
            \scalebox{.7}{
            \color{darkblue}
            \bf
            smooth simplices in
            }
          }
        }
      }
    }
  }{
      \Delta^n_{\mathrm{smth}}
  }
    }
    {
  \overset{
    \mathclap{
      \raisebox{2pt}{
        \rotatebox{50}{
          \rlap{
            \hspace{-9pt}
            \scalebox{.7}{
            \color{darkblue}
            \bf
            the smooth type
            }
          }
        }
      }
    }
  }{
      \mathcal{X}
  }
    }
  }
  }
\end{equation}
\vspace{.2cm}

One of the simplest non-trivial examples of the shape operation is already the most important one for the purpose of cyclification:
Namely, the shape of the smooth circle is equivalent to the groupoid with a single object $\ast$ (witnessing that the circle is connected) which has $\mathbb{Z}$-worth of automorphisms (witnessing the winding number of paths starting and ending at this point, see Lem. \ref{ShapeUnitOfTheSmoothCircle}):
\begin{equation}
  \label{CircleShapeUnitInIntroduction}
  \begin{tikzcd}[
    column sep=40pt
  ]
  \mathbb{R}^1_{\mathrm{smth}}
  \ar[r]
  \ar[
    d,
    ->>
  ]
  \ar[
    dr,
    phantom,
    "{
      \scalebox{.7}{(pb)}
    }"
  ]
  &
  \ast
  \ar[
    d,
    ->>
  ]
  \\[-7pt]
  \mathllap{
    \scalebox{.7}{
      \color{darkblue}
      \bf
      \def\arraystretch{.8}
      \begin{tabular}{c}
        the geometric
        \\
        circle
      \end{tabular}
    }
  }
  \CohesiveCircle
  \ar[
    r,
    "{
      \color{purple}
       \eta^{\scalebox{.6}{$\shape$}}_{\CohesiveCircle}
    }"{swap}]
  &
  \shape
  \,
  \CohesiveCircle
  \ar[r, phantom, "{ \simeq }"]
  &[-35pt]
  \HomotopyQuotient
    { \ast }
    { \mathbb{Z} }
  \;\simeq\;
  \mathbf{B}\mathbb{Z}
  \mathrlap{
    \scalebox{.7}{
      \color{darkblue}
      \bf
      \def\arraystretch{.8}
      \begin{tabular}{c}
        and its
        \\
        pure shape
      \end{tabular}
    }
  }
  \end{tikzcd}
\end{equation}
On general abstract grounds \eqref{TheComparisonMorphism}, this immediately induces the desired structure \eqref{CyclicStackInIntroduction}, by the following homotopy pullback construction (which we explain below in \cref{CyclicInertiaInfinityGroupoids}):
\begin{equation}
  \label{ConstructingExtendedInertiaInIntroduction}
  \begin{tikzcd}
    \HomotopyQuotient
    {
      \Maps{}
        { \mathbf{B} \mathbb{Z} }
        { \mathcal{X} }
    }
    {
      \mathbf{B}\mathbb{Z}
    }
    \ar[d, ->>]
    \ar[from=r]
    \ar[
      dr,
      phantom,
      "{
        \scalebox{.6}{
          (pb)
        }
      }"
    ]
    &
    \HomotopyQuotient
    {
      \Maps{}
        { \mathbf{B} \mathbb{Z} }
        { \mathcal{X} }
    }
    {
      \CohesiveCircle
    }
    \ar[d, ->>]
    \ar[
      r,
      dashed,
      "{
        \scalebox{.6}{
          \color{greenii}
          \bf
          \def\arraystretch{.75}
          \begin{tabular}{c}
            include as
            \\
            essentially constant loops
          \end{tabular}
        }
      }",
      "{
        \HomotopyQuotient
        {
        \Maps{\big}
        {
          {
          \color{purple}
          \eta
            ^{\scalebox{.6}{$\shape$}}
            _{\CohesiveCircle}
          }
        }
        {\mathcal{X}}
        }
        { \CohesiveCircle }
      }"{swap, yshift=-3pt}
    ]
    &[+80pt]
    \HomotopyQuotient
    {
      \Maps{}
        { \CohesiveCircle }
        { \mathcal{X} }
    }
    {
      \CohesiveCircle
    }
    \ar[d, ->>]
    \\
    \mathbf{B}^2 \mathbb{Z}
    \ar[
      from=r,
      "{
        \mathbf{B}
        \color{purple}
        \eta
          ^{\scalebox{.6}{$\shape$}}
          _{\CohesiveCircle}
      }"
    ]
    &
    \mathbf{B}\CohesiveCircle
    \ar[r, Rightarrow, -]
    &
    \mathbf{B}\CohesiveCircle
  \end{tikzcd}
\end{equation}

% $$
%     ( \eta )^\ast
%     ( \eta )_\ast
%     (\mathrm{pt}_{\mathbf{B}\CohesiveCircle})_\ast
%     \mathcal{X}
%     \longrightarrow
%     (\mathrm{pt}_{\mathbf{B}\CohesiveCircle})_\ast
%     \mathcal{X}
% $$
As the terms colored in purple in \eqref{ConstructingExtendedInertiaInIntroduction} are meant to highlight, the problem \eqref{CyclicStackInIntroduction} of forming $\CohesiveCircle$-cyclic inertia orbifolds is abstractly solved by the shape unit \eqref{CircleShapeUnitInIntroduction}
provided by cohesive homotopy theory: this is what knows about ``restriction to essentially constant loops'' in a way compatible with their $\CohesiveCircle$-action.

\medskip

With this good abstract conception of cyclification of orbifolds in hand, and having recovered from it the existing component constructions via Thm. \ref{AbstractCharacterizationOfHuanInertiaOrbifold}, we discover a wealth of interesting induced phenomena:

\medskip

\noindent
{\bf Transgression as cyclification of cocycles.} One should expect that any kind  of cohomology of orbifolds $\HomotopyQuotient{X}{G}$ ``transgresses to'' (i.e.: functorially induces subject to degree shifts) cohomology of the cyclification -- generalizing the familiar notion of transgression in group cohomology, which is the special case where $X = \ast$. In particular, for the purpose of twisted elliptic cohomology one imagines transgressing integral 4-cocycles on an orbifold to integral 3-cocycles on its cyclification, which there serve as twists for topological K-theory.
Again, some component formulas have been considered (\cite[Def. 4.1]{AdemRuanZhang07}\cite[\S 1.3.3]{Willerton08}, recalled as Def. \ref{TransgressionFormulaForGroupCocycles} below), but a general abstract formulation seems to have been missing.

\medskip

However, we may observe that any decent generalized cohomology theory on orbifolds will be represented by some (equivariant) moduli stack $\mathcal{A}$ (cf. \cite[\S II]{FSS20CharacterMap}\cite[p. 6 \& \S 5]{SS20OrbifoldCohomology}\cite[\S 4.3]{SS21EPB}), in that its cocycles are maps of (equivariant) smooth $\infty$-stacks from the orbifold to a moduli $\infty$-stack:
\begin{equation}
  \label{StackyCocycleInIntroduction}
  \begin{tikzcd}[
    column sep=50pt
  ]
    \overset{
      \mathclap{
        \raisebox{6pt}{
          \scalebox{.7}{
            \color{darkblue}
            \bf
            \def\arraystretch{.8}
            \begin{tabular}{c}
              differentiable  stack
              \\
              / orbifold
              \\
            \end{tabular}
          }
        }
      }
    }{
      \HomotopyQuotient{X}{G}
    }
\phantom{AA}
        \ar[
      rr,
      bend left=25,
      "{
        F
      }",
      "{
        \scalebox{.7}{
          \color{greenii}
          \bf
          \def\arraystretch{.8}
          \begin{tabular}{c}
            cocycle in
            \\
            $G$-equivariant
          \scalebox{1.3}{$\mathcal{A}$}-cohomology
          \end{tabular}
        }
      }"{yshift=7pt},
      "{  }"{swap, name=s}
    ]
    \ar[
      rr,
      bend right=25,
      " F' "{swap},
      "{
        \scalebox{.7}{
          \color{greenii}
          \bf
          cohomologous cocycle
        }
      }"{swap, yshift=-10pt},
      "{\ }"{name=t}
    ]
    \ar[
      from=s,
      to=t,
      shorten=-2pt,
      Rightarrow,
      "{
        \scalebox{.7}{
          \color{orangeii}
          \bf
          \def\arraystretch{.8}
          \begin{tabular}{c}
            homotopy/
            \\
            coboundary
          \end{tabular}
        }
      }"{description}
    ]
      &{\phantom{AAAA}}& \phantom{AA}
    \overset{
      \mathclap{
        \raisebox{6pt}{
          \scalebox{.7}{
            \color{darkblue}
            \bf
            \def\arraystretch{.8}
            \begin{tabular}{c}
              differential
              \\
              moduli
              \\
              $\infty$-stack
            \end{tabular}
          }
        }
      }
    }{
      \mathcal{A}
    }
  \end{tikzcd}
\end{equation}
For example, in the simple case of ordinary cohomology in some degree $n \in \mathbb{N}$ with coefficients in an abelian group $A$, the moduli stack is the $n$-fold delooping stack $\mathbf{B}^n A$ (we assume $A$ to be discrete just for the purpose of exposition):
\begin{equation}
  \label{OrdinaryCohomologyInIntroduction}
  \overset{
    \mathclap{
    \raisebox{3pt}{
      \scalebox{.7}{
        \color{darkblue}
        \bf
        \def\arraystretch{.8}
        \begin{tabular}{c}
          ordinary
          \\
          equivariant cohomology
        \end{tabular}
      }
    }
    }
  }{
  H^n
  (
    \HomotopyQuotient{X}{G}
    ;\,
    A
  )
  }
  \quad
  \simeq
  \;\;
  \pi_0
  \Big\{
    \HomotopyQuotient{X}{G}
    \xrightarrow{\;\; F \;\;}
    \mathbf{B}^n A
  \Big\}
  \,.
\end{equation}

But with cyclification of orbifolds identified as the abstract stacky construction \eqref{CyclicStackInIntroduction}, it immediately extends to an $\infty$-functor on all $\infty$-stacks. This allows to readily define/construct the cyclic transgression of any generalized orbifold cocycles \eqref{StackyCocycleInIntroduction} simply as the image formed under this $\infty$-functor:
\begin{equation}
  \label{CyclificationOfStackyCocycleInIntroduction}
  \begin{tikzcd}[
    column sep=50pt
  ]
    \overset{
      \mathclap{
        \raisebox{6pt}{
          \scalebox{.7}{
            \color{darkblue}
            \bf
            \def\arraystretch{.8}
            \begin{tabular}{c}
              cyclified
              \\
              orbifold
              \\
            \end{tabular}
          }
        }
      }
    }{
      \mathrm{Cyc}\big(
        \HomotopyQuotient{X}{G}
      \big)
    }
    \ar[
      rr,
      bend left=15pt,
      "{
        \mathrm{Cyc}(F)
      }",
      "{
        \scalebox{.7}{
          \color{greenii}
          \bf
          \begin{tabular}{c}
            cyclified cocycle
            \\
            in \scalebox{1.3}{$\mathrm{Cyc}(\mathcal{A})$}-cohomology
          \end{tabular}
        }
      }"{yshift=10pt},
      "{\ }"{swap, name=s}
    ]
    \ar[
      rr,
      bend right=15pt,
      "{ \mathrm{Cyc}(F') }"{swap},
      "{\  }"{name=t}
    ]
    \ar[
      from=s,
      to=t,
      Rightarrow,
      "{\sim}"{sloped}
    ]
    &{\phantom{AAA}}&
    \overset{
      \mathclap{
        \raisebox{6pt}{
          \scalebox{.7}{
            \color{darkblue}
            \bf
            \def\arraystretch{.8}
            \begin{tabular}{c}
              cyclified
              \\
              moduli
              $\infty$-stack
            \end{tabular}
          }
        }
      }
    }{
      \mathrm{Cyc}\big(
        \mathcal{A}
      \big)
    }
  \end{tikzcd}
\end{equation}

Moreover, in the example \eqref{OrdinaryCohomologyInIntroduction} of ordinary cohomology with coefficients
in a torsion-free abelian group $A$ (such as the integers $\mathbb{Z}$),
one finds a retraction of the cyclification of the classifying stack in degree $n+1$ onto that in degree $n$, which has the interpretation of ``integrating cocycles along loops'':
\vspace{-.4cm}
\begin{equation}
  \label{CyclificationOfDeloopingOfAbelian}
  \def\arraystretch{1}
  \begin{array}{l}
    A : \mathrm{AbGrp},
    \\
    n : \mathbb{N}_{\geq 2}
  \end{array}
  \hspace{15pt}
  \vdash
  \hspace{17pt}
  \begin{tikzcd}[
   column sep=45pt
  ]
     \mathbf{B}^{n} A
     \ar[
       r,
       "{
         (\mathrm{id}, 0)
       }"
     ]
     \ar[
       rrr,
       Rightarrow,
       -,
       rounded corners,
       to path={
            ([yshift=-00pt]\tikztostart.south)
         -- ([yshift=-12pt]\tikztostart.south)
         -- ([yshift=-12pt]\tikztotarget.south)
         -- ([yshift=-00pt]\tikztotarget.south)
       }
     ]
     &
     \mathbf{B}^n A
       \,\times\,
     \mathbf{B}^{n+1} A
     \ar[
       r, ->>
     ]
     &[-23pt]
     \overset{
       \mathrm{Cyc}(
         \mathbf{B}^{n+1}A
       )
     }{
     \overbrace{
       \HomotopyQuotient{
         \big(
           \mathbf{B}^{n}A
             \times
           \mathbf{B}^{n+1} A
         \big)
     }{\CohesiveCircle}
     }
     }
     \ar[
       r,
       "{
         \overset{
           \mathclap{
             \raisebox{3pt}{
               \scalebox{.7}{
                 \color{orangeii}
                 \bf
                 \def\arraystretch{.8}
                 \begin{tabular}{c}
                   loop
                   \\
                   integration
                 \end{tabular}
               }
             }
           }
         }{
           \scalebox{1.2}{$\int_{\CohesiveCircle}$}
         }
       }"
     ]
     &
     \mathbf{B}^{n} A
  \end{tikzcd}
  \,.
\end{equation}
\vspace{.1cm}

This may nicely be seen, due to the assumption that $A$ is torsion-free, with tools from rational homotopy theory
(comprehensive review and further pointers are given in \cite[\S 3.2]{FSS20CharacterMap}): Using the general formula
for Sullivan models (see particularly \cite{Menichi15}) of cyclic loop spaces from \cite[Thm. A]{VPB85} (described
in our context in \cite[Prop. 3.2]{FSS16SphereValued}\cite[Rem. 3.1]{FSS16TDuality}, and iterated and extended to the
nilpotent case in \cite{SatiVoronov21}), we have the following evident
corresponding retraction of rational Sullivan dg-algebras (showing the polynomial generators, with degrees in subscript,
and their differential relations):
\vspace{-3mm}
\begin{equation}
  \label{SullivanModelForCycRetraction}
  \hspace{-20pt}
  \begin{tikzcd}[
    row sep=-5pt,
    column sep=14pt
  ]
    \scalebox{.7}{$
      \def\arraystretch{.8}
      \begin{tabular}{c}
        Retraction of
        cyclification of
        \\
        torsion-free
        classifying space
        \\
        onto its based loop space
      \end{tabular}
    $}
    &[-10pt]
    \mathbf{B}^n \mathbb{Z}
    \ar[r]
    &
    \mathbf{B}^n \mathbb{Z}
    \times
    \mathbf{B}^{n+1} \mathbb{Z}
    \ar[r]
    &
    \mathrm{Cyc}\big(
     \mathbf{B}^{n+1} \mathbb{Z}
    \big)
    \ar[
      r,
      "{ \int_{\CohesiveCircle} }"
    ]
    &
    \mathbf{B}^n \mathbb{Z}\;.
    \\
    \scalebox{.7}{$
      \def\arraystretch{.8}
      \begin{tabular}{c}
        witnessed by
        the respective
        \\
        Sullivan model
        dgc-algebras
      \end{tabular}
    $}
    &
    \scalebox{.8}{$
    \left(
    \hspace{-4pt}
    \begin{array}{l}
      \mathrm{d}
      \;
      c_{n} \;=\; 0
    \end{array}
    \hspace{-4pt}
    \right)
    $}
    \ar[from=r]
    &
  \scalebox{.7}{$  \left(
    \hspace{-4pt}
    \begin{array}{l}
      \mathrm{d}
      \;
      c_{n+1} \;=\; 0
      \\
      \mathrm{d}
      \;
      c_{n} \;\;\;\;\;=\; 0
    \end{array}
    \hspace{-4pt}
    \right)
    $}
    \ar[from=r]
    &
    \scalebox{.8}{$
    \left(
    \hspace{-4pt}
    \begin{array}{l}
      \mathrm{d}
      \;
      c_{n+1} \;=\;
      c_n \wedge \omega_2
      \\
      \mathrm{d}
      \;
      c_{n} \;\;\;\;\;=\; 0
      \\
      \mathrm{d}
      \;
      \omega_{2} \;\;\;\;=\; 0
    \end{array}
    \hspace{-4pt}
    \right)
    $}
    \ar[from=r]
    &
  \scalebox{.7}{$  \left(
    \hspace{-4pt}
    \begin{array}{l}
      \mathrm{d}
      \;
      c_{n} \;=\; 0
    \end{array}
    \hspace{-4pt}
    \right)
    $}
  \end{tikzcd}
\end{equation}

\noindent
Now under this retraction, the cyclification operation
\eqref{CyclificationOfStackyCocycleInIntroduction}
reproduces
the traditional transgression formula in discrete group cohomology (recalled as Def. \ref{TransgressionFormulaForGroupCocycles} below)
and shows that it descends from the free loop stack to the cyclification:
\vspace{-4mm}
\begin{equation}
  \label{TransgressionMapViaCyclification}
  \begin{tikzcd}[
    column sep=30pt
  ]
    &[-20pt]
    \mathllap{
      G : \mathrm{Grp}(\Sets),
      \;\;\;\;\;
    }
    \overset{
      \mathclap{
        \raisebox{4pt}{
          \scalebox{.7}{
            \color{darkblue}
            \bf
            \def\arraystretch{.8}
            \begin{tabular}{c}
              $G$-orbi-
              \\
              singularity
            \end{tabular}
          }
        }
      }
    }{
      \HomotopyQuotient{\ast}{G}
    }
    \ar[
      rr,
      "{ F }",
      "{
        \scalebox{.7}{
          \color{greenii}
          \bf
          cocycle in group cohomology
        }
      }"{swap}
    ]
    &{\phantom{AAAAA}}&
    \mathbf{B}^{n+1} A
    \mathrlap{
      \hspace{30pt}
      \vdash
    }
    &[-30pt]
    \\
    \overset{
      \mathrlap{
        \scalebox{.7}{
        \raisebox{4pt}{
          \color{darkblue}
          \bf
          inertia orbifold
        }
        }
      }
    }{
      \Lambda \mathbf{B}G
    }
    \ar[r, phantom, "="]
    &[-20pt]
    \Maps{\big}
      {\CohesiveCircle}
      {\HomotopyQuotient{\ast}{G}}
    \ar[d, - >>]
    \ar[
      rr,
      "{
        \Maps{}
          { \CohesiveCircle }
          { F }
      }",
      "{
        \scalebox{.7}{
          \color{greenii}
          \bf
          looped cocycle
        }
      }"{swap}
      "{
        \scalebox{.7}{
          \color{greenii}
          \bf
        }
      }"{swap}
    ]
    &{\phantom{AAAAA}}&
    \Maps{\big}{\CohesiveCircle}{\mathbf{B}^n A}
    \ar[r, phantom,  "\simeq"]
    \ar[d, - >>]
    &[-20pt]
    \mathbf{B}^{n}A
    \times
    \mathbf{B}^{n+1} A
    \ar[r, ->> , "{ \mathrm{pr}_1 }"]
    \ar[d, - >>]
    &[-5pt]
    \mathbf{B}^{n}A
    \ar[
      from=dl,
      "{
        \scalebox{1.2}{$\int_{\CohesiveCircle}$}
        \mathrlap{
          \hspace{-12pt}
          \scalebox{.7}{
            \begin{tabular}{c}
              \color{orangeii}
              \bf
              loop integration
              \\              \eqref{CyclificationOfDeloopingOfAbelian}
            \end{tabular}
          }
        }
      }"{swap, pos=.52, xshift=-5pt}
    ]
    \ar[
      from=lllll,
      crossing over,
      rounded corners,
      to path={
           ([yshift=-00pt]\tikztostart.south)
        -- ([yshift=-13pt]\tikztostart.south)
        -- node[yshift=-7pt, pos=.375] {
             \scalebox{.7}{
               \color{greenii}
               \bf
               transgressed cocycle
             }
           }
           ([yshift=-13pt]\tikztotarget.south)
        -- ([yshift=-00pt]\tikztotarget.south)
      }
    ]
    \\[+10pt]
    &
    \underset{
      \mathclap{
        \raisebox{-3pt}{
        \scalebox{.7}{
          \color{darkblue}
          \bf
          \def\arraystretch{.8}
          \begin{tabular}{c}
            cyclified
            orbifold
          \end{tabular}
        }
        }
      }
    }{
      \mathrm{Cyc}(\HomotopyQuotient{\ast}{G})
    }
    \ar[
      rr,
      "{
        \mathrm{Cyc}(c)
      }",
      "{
        \scalebox{.7}{
          \color{greenii}
          \bf
          cyclified cocycle
        }
      }"{swap}
    ]
    &&
    \underset{
      \mathrlap{
        \raisebox{-3pt}{
        \scalebox{.7}{
          \color{darkblue}
          \bf
          \def\arraystretch{.8}
          \begin{tabular}{c}
            cyclified
            coefficient
            $\infty$-stack
          \end{tabular}
        }
        }
      }
    }{
      \mathrm{Cyc}\big(
        \mathbf{B}^{n+1} A
      \big)
    }
    \ar[
      r,
      phantom,
      "\simeq"
    ]
    &
    \HomotopyQuotient
    {
    \big(
      \mathbf{B}^{n}A
      \times \mathbf{B}^{n+1} A
    \big)
    }
    { \CohesiveCircle }
  \end{tikzcd}
\end{equation}
 This is Thm. \ref{TransgressionInGroupCohomologyViaLooping} below, proving a suggestion in \cite[\S 1.3.3]{Willerton08}.
 But notice that the construction \eqref{TransgressionMapViaCyclification} works for $\HomotopyQuotient{\ast}{G} \,=\, \mathbf{B}G$ replaced by any $\infty$-stack $\mathcal{X}$ and hence defines transgression in this generality.
 In the special case when  $\mathcal{X}$ is at most an orbifold, transgression is considered in the literature as pullback in cohomology along the evaluation map $\CohesiveCircle \times \Maps{\big}{\CohesiveCircle}{\mathcal{X}} \xrightarrow{\mathrm{ev}} \mathcal{X}$
 \eqref{EvaluationMap}
 followed by suitable fiber integration over the $\CohesiveCircle$-factor (e.g. \cite[p. 2]{Luperciouribe06Holonomy}, following analogous discussion for manifolds \cite[\S 3.5]{Brylinski93}). The proof of Thm. \ref{TransgressionInGroupCohomologyViaLooping} brings out that this is what \eqref{TransgressionMapViaCyclification} reduces to in these cases; see around \eqref{EvaluationPrecomposedWithEilenbergZilber} below.

\medskip

Therefore, cyclification subsumes transgression in broad generality, but it retains more information.
We next see that this extra information is such as to recover the original cocycle:

\medskip

\noindent
{\bf Fiber integration via cyclification.}
In giving higher topos-theoretic meaning to the cyclification-construction of orbifolds, a web of further structure surrounding the construction becomes manifest, related to the topic of higher transformation groups and higher principal bundles (cf. \cite[\S 2.2]{SS20OrbifoldCohomology}\cite[\S 3.2.3]{SS21EPB}):

\medskip

First (this is the content of \cref{CyclicInertiaOrbifolds} below), cyclification immediately generalizes from the circle group $\CohesiveCircle$ to any group $\infty$-stack
$\mathcal{T} \,\in\, \Groups\big( \SmoothInfinityStacks \big)$, since the mapping stacks of the form $\Maps{}{\mathcal{T}}{\mathcal{X}}$ carry a canonical $\mathcal{T}$-action by precomposition with the multiplication action of $\mathcal{T}$ on itself. Moreover, the general theory of principal $\infty$-bundles shows that the resulting homotopy quotient projection is a $\mathcal{T}$-principal bundle which is classified by the homotopy quotient of the terminal map $\Maps{}{\mathcal{T}}{\mathcal{X}}$, in that we have homotopy pullback squares of this form:
\begin{equation}
  \begin{tikzcd}[column sep=large]
    \mathcal{T}
    \ar[r]
    \ar[d]
    \ar[
      dr,
      phantom,
      "{
        \scalebox{.7}{(pb)}
      }"
    ]
    &
    \Maps{}{\mathcal{T}}{\mathcal{X}}
    \ar[d, ->>]
    \ar[r]
    \ar[
      dr,
      phantom,
      "{
        \scalebox{.7}{ (pb) }
      }"
    ]
    &
    \ast
    \ar[d]
    \\
    \ast
    \ar[
      r,
      "{  }"
    ]
    &
    \underset{
      \mathrm{Cyc}_{\mathcal{T}}
    }{
    \underbrace{
    \HomotopyQuotient
      {
        \Maps{}{\mathcal{T}}{\mathcal{X}}
      }
      { \mathcal{T} }
      }
      }
      \ar[
        r,
        "{
          c_1
        }"{swap}
      ]
      &
      \underset{
        \mathbf{B}\mathcal{T}
      }{
      \underbrace{
        \HomotopyQuotient
        {
          \ast
        }
        { \mathcal{T} }
      }
      }
  \end{tikzcd}
\end{equation}
In fact, these $\mathcal{T}$-principal bundles over $\mathcal{T}$-cyclic stacks have a fundamental universal property \eqref{CyclicicationAsRightDerivedBaseChange}: their construction is {\it right adjoint} (in the $\infty$-category theoretic sense of, e.g., \cite[Def. 1.1.2]{RiehlVerity16}) \ of sending a $\mathcal{T}$-principal bundle to its total space (i.e.: total $\infty$-stack):
\begin{equation}
  \label{ExtCycAdjunctionInIntroduction}
  \begin{tikzcd}[column sep=large]
    \SmoothInfinityGroupoids
    \ar[
      from=rr,
      shift right=5pt,
      "{
        \scalebox{.7}{
          \color{greenii}
          \bf
          total space
        }
      }"{swap}
    ]
    \ar[
      rr,
      shift right=6pt ,
      "{
        \mathllap{
        \mathcal{X}
        }
        \;\;\;\mapsto\;\;\;
        \mathrlap{
        \big(
          \Maps{}{\mathcal{T}}{\mathcal{X}}
          \,\to\,
          \HomotopyQuotient
            { \Maps{}{\mathcal{T}}{\mathcal{X}} }
            { \mathcal{T} }
        \big)
        }
      }"{swap}
    ]
    \ar[
      rr,
      phantom,
      "{
        \scalebox{.7}{$\bot$}
      }"
    ]
    \ar[
      rrrr,
      rounded corners,
      to path={
           ([yshift=-00pt]\tikztostart.south)
        -- ([yshift=-30pt]\tikztostart.south)
        -- node[yshift=+7pt, pos=.48] {
             \scalebox{.7}{$
               \mathcal{X} \;\; \mapsto \;\;
                 \overset{
                   \scalebox{.8}{
                     \rotatebox{90}{$\;\,\mathclap{=}$}
                   }
                 }{
                   \mathrm{Cyc}_{\mathcal{T}}(\mathcal{X})
                 }
               $}
           }
           node[yshift=-8pt, pos=.48] {
             \scalebox{.7}{
               \color{greenii}
               \bf
               \scalebox{1.1}{$\mathcal{T}$}-cyclification
             }
           }
           ([yshift=-30pt]\tikztotarget.south)
        -- ([yshift=-00pt]\tikztotarget.south)
      }
    ]
    &&
    \mathcal{T}\mathrm{PrncplBundl}
    \big(
      \SmoothInfinityGroupoids
    \big)
    \ar[
      rr,
      "{
        \scalebox{.7}{
          \color{greenii}
          \bf
          base space
        }
      }"
    ]
    &&
    \SmoothInfinityGroupoids
  \end{tikzcd}
\end{equation}

This adjunction means that if a domain $\mathcal{X}$ is an $\mathcal{T}$-principal bundle over a base $\mathcal{Y} \,\simeq\, \HomotopyQuotient{\mathcal{X}}{\mathcal{T}}$ (as such classified by a cocycle $\mathcal{Y} \xrightarrow{\;} \mathbf{B}\mathcal{T})$, then its $\mathcal{A}$-cohomology \eqref{StackyCocycleInIntroduction} is equivalently the cohomology of $\mathcal{Y}$ with coefficients in $\mathrm{Cyc}(\mathcal{A})$ subject to identification of the underlying $\mathcal{T}$-bundles on both sides:
\begin{equation}
  \label{ReductionAndOxidation}
  \begin{tikzcd}[
    row sep=7pt, column sep=huge
  ]
    \mathcal{T}
    \ar[r]
    &[-13pt]
    \overset{
      \mathclap{
        \raisebox{6pt}{
      \scalebox{.7}{
      \color{darkblue}
      \bf
      \def\arraystretch{.85}
      \begin{tabular}{c}
        $\mathcal{T}$-principal
        \\
        bundle
      \end{tabular}
      }
        }
      }
    }{
      \mathcal{X}
    }
    \ar[
      rr,
      "{
        F
      }"{pos=.46},
      "{
        \scalebox{.65}{
          \color{greenii}
          \bf
          cocycle on total space
        }
      }"{swap}
    ]
    \ar[dd]
    &\phantom{AAAAAAAA}&
    \overset{
      \mathclap{
        \raisebox{3pt}{
        \scalebox{.7}{
          \color{darkblue}
          \bf
          \def\arraystretch{.9}
          \begin{tabular}{c}
            generalized cohomology
            \\
            coefficients
          \end{tabular}
        }
      }
      }
    }{
      \mathcal{A}
    }
    \ar[
      dd,
      <->,
      gray,
      dashed,
      start anchor={[yshift=-2pt]},
      end anchor={[xshift=-9pt, yshift=-3pt]},
      shorten <=-14pt,
      shorten >=-9pt,
      bend left=80,
      shift left=20pt,
      "{
        \phantom{\vert^{\vert}}
      }"{description},
      "{
        \hspace{-10pt}
        \scalebox{.7}{
          \it oxidation
        }
      }"{pos=.3},
      "{
        \hspace{-6pt}
        \scalebox{.7}{
          \it reduction
        }
      }"{pos=.7},
    ]
    \\
    &
    {}
    \ar[
      rr,
      phantom,
      "{ \rotatebox{90}{\scalebox{.9}{$\leftrightarrow$}} }"
    ]
    && {}
    \\[-6pt]
    &
    \mathcal{Y}
    \ar[
      rr,
      "{
        \scalebox{.7}{
          \color{greenii}
          \bf
          \scalebox{1.2}{$\mathrm{Cyc}$}-adjoint cocycle on base
        }
      }"{yshift=2pt},
      "{
        \widetilde{ F }
      }"{description},
      "{}"{swap, name=s}
    ]
    \ar[
      dr,
      "{
        \vdash \mathcal{X}
      }"{sloped, swap, pos=.5},
      "{
      }"{name=t, pos=1}
    ]
    &&
    \mathrm{Cyc}_{\mathcal{T}}(\mathcal{A})
    \ar[
      dl,
      "{
        c_1
      }"{}
    ]
    \\
    &
    &
    \mathbf{B}\mathcal{T}
    \ar[
      from=s,
      to=t,
      Rightarrow,
      "{ \sim }"{sloped, swap, pos=.35}
    ]
  \end{tikzcd}
\end{equation}

In the case \eqref{CyclificationOfDeloopingOfAbelian} of ordinary cohomology of $\CohesiveCircle$-bundles, this yields the fiber integration of cocycles:
\begin{equation}
  \label{DoubleDimReductionOfCharges}
  \begin{tikzcd}
    \CohesiveCircle
    \ar[r]
    &[-5pt]
    \overset{
      \mathclap{
        \raisebox{6pt}{
          \scalebox{.7}{
            \color{darkblue}
            \bf
            \def\arraystretch{.85}
            \begin{tabular}{c}
              $\CohesiveCircle$-principal
              \\
              bundle
            \end{tabular}
          }
        }
      }
    }{
      \mathcal{X}
    }
    \ar[dd]
    \ar[
      rr,
      "{ F }",
      "{
        \scalebox{.7}{
          \color{greenii}
          \bf
          \scalebox{1.2}{$(n+1)$}-cocycle
        }
      }"{swap}
    ]
    &&
    \overset{
      \mathclap{
        \raisebox{4pt}{
          \scalebox{.7}{
          \color{darkblue}
          \bf
          \def\arraystretch{.8}
          \begin{tabular}{c}
            ordinary cohomology
            \\
            coefficients
          \end{tabular}
        }
        }
      }
    }{
      \mathbf{B}^{n+1} A
    }
    &
    {}
    \ar[
      dd,
      gray,
      dashed,
      start anchor={[yshift=-2pt]},
      end anchor={[xshift=-9pt]},
      shorten <=-14pt,
      shorten >=-9pt,
      bend left=80,
      shift left=20pt,
      "{
        \hspace{-14pt}
        \scalebox{.7}{
          \def\arraystretch{.8}
          \it
          \begin{tabular}{c}
            fiber integration /
            \\
            double dim-reduction
          \end{tabular}
        }
      }"{pos=.6},
    ]
    \\[-12pt]
    \\[-12pt]
    &
    \mathcal{Y}
    \ar[
      rrr,
      rounded corners,
      to path={
           ([yshift=+00pt]\tikztostart.south)
        -- ([yshift=-10pt]\tikztostart.south)
        -- node[yshift=-6pt]{
             \scalebox{.7}{
               \color{greenii}
               \bf
               \scalebox{1.2}{$n$}-cocycle
             }
           }
           ([yshift=-10pt]\tikztotarget.south)
        -- ([yshift=-00pt]\tikztotarget.south)
      }
    ]
    \ar[
      rr,
      "{ \widetilde F }"
    ]
    &\phantom{AAAA}&
    \mathrm{Cyc}
    \big(
      \mathbf{B}^{n+1} A
    \big)
    \ar[
      r,
      "{
        \int_{\CohesiveCircle}
      }"
    ]
    &
    \mathbf{B}^n A
  \end{tikzcd}
\end{equation}

\noindent
{\bf Double dimensional reduction.}
If in \eqref{DoubleDimReductionOfCharges} we think of
\begin{itemize}[leftmargin=.5cm]
\item $\mathcal{X}$ as a spacetime-orbifold (e.g. as in \cite{DixonHarveyVafaWitten85}\cite{Acharya99}\cite{SS19Tad})

\item $\mathbf{B}^{n+1} A$ as the coefficients of charges of some solitonic physical objects (``branes'') in higher generalization (cf. \cite[(2)]{FSS20CharacterMap})
of Dirac's classical magnetic charge quantization (reviewed e.g. in \cite[\S 2]{Alvarez85}),
\item
$\mathcal{T} = \CohesiveCircle$ as the ``Kaluza-Klein compactification''-space (see \cite{Duff94} for traditional pointers and \cite{Alfonsi20} for discussion in our higher differential geometric contex)
\end{itemize}
then
 \eqref{ReductionAndOxidation} captures the ``double dimensional reduction'' \cite{DuffHoweInamiStelle87} of these brane charges (\cite[\S ]{FSS16TDuality}\cite[\S 4]{SuperBraneLectures} \cite{BSS18}), namely their descent to charges of lower-dimensional branes in the lower-dimensional base spacetime:
\begin{equation}
  \label{ReductionOfFourCohomotopyCharges}
  \hspace{-3mm}
  \begin{tikzcd}[
    column sep=40pt
  ]
    \CohesiveCircle
    \ar[r]
    &[-25pt] \quad
    \overset{
      \mathclap{
        \raisebox{5pt}{
      \scalebox{.6}{
        \color{darkblue}
        \bf
        \def\arraystretch{.8}
        \begin{tabular}{c}
          1+(d+1)-dimensional
          \\
          $S^1$-fibered
          \\
          spacetime orbifold
        \end{tabular}
      }
        }
      }
    }{
      \mathcal{X}
    }
    \ar[
      dd,
      shorten >=-5pt
    ]
    \quad
    \ar[
      rr,
      "{ F }",
      "{
        \scalebox{.6}{
          \color{greenii}
          \bf
          \def\arraystretch{.8}
          \begin{tabular}{c}
            magnetic
            \scalebox{1.3}{$p$}-brane charge
            \\
            coupling to electric
            \scalebox{1.3}{$d-p-2$}-branes
          \end{tabular}
        }
      }"{swap}
    ]
    &
    &
    \mathbf{B}^{ d - p   } \mathbb{Z}
    &[-10pt]
    &[-38pt]
    \scalebox{.7}{
      e.g.
      \def\arraystretch{.9}
      \tabcolsep=3pt
      \begin{tabular}{rll}
        $1 + (d+1) = 11$ & M-theory spacetime
        \\
        $ p = 5$ & M5-brane
        \\
        $ d - p - 2 = 2$ & M2-brane
      \end{tabular}
    }
    \\[-12pt]
    \\[-12pt]
    &
    \mathllap{
      \scalebox{.6}{
        \color{darkblue}
        \bf
        \def\arraystretch{.8}
        \begin{tabular}{c}
          \color{black}
          \normalfont
          ``KK-compactified''
          \\
          $1+d$-dimensional
          \\
          spacetime orbifold
        \end{tabular}
      }
    }
    \mathcal{Y}
    \ar[
      rr,
      "{
        \mathrm{Cyc}(F)
      }"
    ]
    \ar[
      rrr,
      rounded corners,
      to path={
           ([yshift=+10pt]\tikztostart.south)
        -- ([yshift=-2pt]\tikztostart.south)
        -- node[yshift=-10pt] {
              \scalebox{.6}{
                \color{greenii}
                \bf
                \def\arraystretch{.8}
                \begin{tabular}{c}
                  magnetic $p$-brane charge
                  \\
                  coupling to electric $d-p-3$-branes
                \end{tabular}
              }
           }
           ([yshift=-12pt]\tikztotarget.south)
        -- ([yshift=-00pt]\tikztotarget.south)
      }
    ]
    &&
    \mathrm{Cyc}\big(
      \mathbf{B}^{d-p}
      \mathbb{Z}
    \big)
    \ar[
      r,
      "{ \int_{\CohesiveCircle} }"
    ]
    &
    \mathbf{B}^{d-p-1}
    \mathbb{Z}
    &
    \scalebox{.7}{
      e.g.
      \def\arraystretch{.9}
      \tabcolsep=3pt
      \begin{tabular}{rll}
        $1 + d = 10$ & type IIA spacetime
        \\
        $ p = 5$ & NS5-brane
        \\
        $ d - p - 3 = 1$ & NS1-branes (string)
      \end{tabular}
    }
  \end{tikzcd}
\end{equation}
Here on the right, we are indicating the example of the double dimensional reduction of M2-brane charge to NS1-brane/string charge (see \cite[\S 4]{MathaiSati04}) for which we consider a more refined model further below, around \eqref{MBraneChargeViaHypothesisH}.
In a different but related context, a
dimensional reduction procedure for supersymmetric Euclidean field theories over an orbifold
is proposed  in \cite{Stoffel}.

\medskip

It is expected that if such dimensional reduction is carried out with due care, then it does not lose information and may be reversed (``oxidized'', e.g. \cite{LPSS95}). This is exactly what we find here, formalized by the cyclification adjunction with its hom-isomorphism \eqref{ReductionAndOxidation}. Notice that this means (i.e.: proves) that brane charges on spacetimes which are iterated principal fiber bundles may be reduced all the way down, possibly all the way to the point, by iterated cyclification, without losing information; see \cite{SatiVoronov21}.

\medskip

\noindent
{\bf Generalized T-duality.}
Notice the subtle distinction between the bare cyclification $\mathrm{Cyc}_{\mathcal{T}}(\mathcal{A}) \,\in\, \SmoothInfinityStacks$ and its incarnation as the base of an $\mathcal{T}$-principal bundle \eqref{ShortExactSequenceOfCohomologyOfBorelConstructionOfADEGroupActingOn4Sphere}, $\mathrm{Cyc}(\mathcal{A}) \,\in\,(\SmoothInfinityStacks)_{/\mathbf{B}\mathcal{T}}$. The reduction/oxidation-equivalence \eqref{ReductionAndOxidation} says that with this $\mathcal{T}$-bundle data retained, maps from a base space $\mathcal{Y}$ into the cyclification of $\mathcal{A}$ fully recover the $\mathcal{A}$-valued cocycles on $\mathcal{X}$.
But it may happen that two in-equivalent moduli stacks $\mathcal{A} \;\neq\;  \mathcal{B}$ have equivalent bare $\mathcal{T}$-cyclifications, $\mathrm{Cyc}_{\mathcal{T}}(\mathcal{A}) \,\simeq\,\mathrm{Cyc}_{\mathcal{T}}(\mathcal{B}) \,\in\, \SmoothInfinityGroupoids$.

\smallskip
In this case the two maps $c_1^A, c_1^B$ to $\mathbf{B}\mathcal{T}$ are necessarily in-equivalent, so that $\mathrm{Cyc}(\mathcal{A})$-valued cocycles on $\mathcal{Y}$ have two different interpretations as cocycles $F^A$, $F^B$ on different $\mathcal{T}$-bundles $\mathcal{X}_A, \mathcal{X}_B$. In such a situation, while pairs of cocycles $F^A$, $F^B$ are in general different (not even of the same type), they are essentially equivalent after reduction along the $\mathcal{T}$-fibers. As such it makes sense to call them {\it $\mathcal{T}$-dual} to each other:
\vspace{-.2cm}
\begin{equation}
  \label{GeneralizedTDuality}
  \scalebox{.8}{$
  \def\arraystretch{1.2}
  \begin{array}{l}
    \mathcal{A},
    \mathcal{B}
    : \SmoothInfinityStacks
    \\
    \mathcal{T}
    :
    \Groups(\SmoothInfinityStacks)
    \\
    \mathrm{Cyc}_{\mathcal{T}}(\mathcal{A})
    \simeq
    \mathrm{Cyc}_{\mathcal{T}}(\mathcal{B})
    \\
    \\
    \mathcal{Y} : \SmoothInfinityStacks
    \\
    \widetilde F
    :
    \mathcal{Y} \to
    \mathrm{Cyc}_{\mathcal{T}}
    (\mathcal{A}/\mathcal{B})
    \
  \end{array}
  $}
  \hspace{12pt}
  \vdash
  \hspace{-20pt}
  \begin{tikzcd}[]
    &
    {\color{gray} \ast}
    \ar[ddr, gray]
    \ar[
      dddd,
      phantom,
      "{
        \scalebox{.7}{
          \color{gray}
          (pb)
        }
      }"{pos=.6}
    ]
    \\[-7pt]\
    \fcolorbox{lightgray}{white}{
      \hspace{-2pt}
      $\mathcal{X}_A$
    }
    \ar[dddr]
    \ar[ur, gray]
    \ar[
      rrr,
      crossing over,
      "{
        \raisebox{2pt}{
          \scalebox{1.2}{$F^A$}
        }
      }"{description, pos=.2}
    ]
    &[-20pt]
    &[+13pt]
    &[-13pt]
    \mathcal{A}
    &[-40pt]
    \scalebox{.7}{
      \color{orangeii}
      \bf
      \def\arraystretch{.8}
      \begin{tabular}{c}
        \scalebox{1.2}{$\mathcal{A}$}-cocycle on
        \\
        $\mathcal{T}$-fibration of $c_1^A$
      \end{tabular}
    }
    \ar[
      dddddd,
      <->,
      gray,
      rounded corners,
      to path={
            ([xshift=+0pt]\tikztostart.east)
         -- ([xshift=+42pt]\tikztostart.east)
         -- node[xshift=13pt, rotate=-90] {
              \color{purple}
              \bf
              \def\arraystretch{.8}
              \begin{tabular}{c}
                generalized
                \\
                $\mathcal{T}$-duality
              \end{tabular}
            }
            ([xshift=+42pt]\tikztotarget.east)
         -- ([xshift=+0pt]\tikztotarget.east)
      }
    ]
    \\[-10pt]
    &
    &
    \mathbf{B}\mathcal{T}
    &
    \\[-20pt]
    &
    &&
    \mathrm{Cyc}_{\mathcal{T}}(\mathcal{A})
    \ar[
      ul,
      shorten <=-3pt,
      shorten >=-3pt,
      "{
        c^A_1
      }"{swap}
    ]
    \ar[
      dd,
      phantom,
      shift right=10pt,
      "{
        \simeq
      }"{rotate=-90, pos=.8}
    ]
    \ar[
      dd,
      phantom,
      shift right=18pt,
      "{
        \left\{
          \def\arraystretch{1.32}
          \begin{array}{c}
            \phantom{A}
            \\
            \phantom{A}
          \end{array}
        \right.
      }"{}
    ]
    \\[-30pt]
    &
    \mathcal{Y}
    \ar[
      rr,
      shift right=1pt,
      "{
        \scalebox{1.2}{$\widetilde{ F } $}
      }"{description, pos=.2},
      shorten >=28pt
    ]
    \ar[
      uur,
      dashed
    ]
    \ar[
      ddr,
      dashed
    ]
    &&
    {}
    &[-30pt]
    \scalebox{.7}{
      \begin{tabular}{c}
        \color{darkblue}
        \bf
        \def\arraystretch{.8}
        \begin{tabular}{c}
          single Cyc-cocycle
          \\
          on base space
        \end{tabular}
      \end{tabular}
    }
    \ar[
      uuu,
      <->,
      gray,
      shorten=5pt,
      "{
        \scalebox{.7}{
          \def\arraystretch{.8}
          \begin{tabular}{c}
            reduction/oxidation
            \\
            with respect to \scalebox{1.3}{$c_1^A$}
          \end{tabular}
        }
      }"{swap}
    ]
    \ar[
      ddd,
      <->,
      gray,
      shorten=5pt,
      "{
        \scalebox{.7}{
          \def\arraystretch{.8}
          \begin{tabular}{c}
            reduction/oxidation
            \\
            with respect to \scalebox{1.3}{$c_1^B$}
          \end{tabular}
        }
      }"
    ]
    \\[-30pt]
    &
    &&
    \mathrm{Cyc}_{\mathcal{T}}(\mathcal{B})
    \ar[
      dl,
      shorten <=-5pt,
      shorten >=0pt,
      "{
        c^B_1
      }"{}
    ]
    \\[-20pt]
    &
    &
    \mathbf{B}\mathcal{T}
    \ar[from=ddl, gray]
    &
    &
    \\[-10pt]
    \fcolorbox{lightgray}{white}{
      \hspace{-2pt}$
      \mathcal{X}_B$
    }
    \ar[uuur]
    \ar[dr, gray]
    \ar[
      rrr,
      crossing over,
      "{
        \raisebox{2pt}{
          \scalebox{1.2}{$F^B$}
        }
      }"{pos=.2, description}
    ]
    &
    &&
    \mathcal{B}
    &
    \scalebox{.7}{
      \color{orangeii}
      \bf
      \def\arraystretch{.8}
      \begin{tabular}{c}
        \scalebox{1.2}{$\mathcal{B}$}-cocycle on
        \\
        $\mathcal{T}$-fibration of $c_1^B$
      \end{tabular}
    }
    \\[-7pt]
    &
    {\color{gray} \ast}
    \ar[
      uuuu,
      phantom,
      "{
        \scalebox{.7}{
          \color{gray}
          (pb)
        }
      }"{pos=.6}
    ]
  \end{tikzcd}
\end{equation}

In the approximation of (super-)rational homotopy theory\footnote{Here {\it super homotopy theory} refers to the $\infty$-topos not just over the site of smooth manifolds, as considered in \eqref{SequenceOfEmbeddingsOfAmbientCategories}, but further embedded into that over super-manifolds \cite{SupergeometryLectures}\cite{Giotopoulos23}, where (super-)rational homotopy theory is modeled by super-dgc-algebras \cite[\S 3.2]{HSS18} known in the supergravity literature as ``FDA''s (see \cite{FSS15WZW}). All our discussion here immediately passes to that context, but for brevity we shall not further dwell on this point here.},
this notion of $\mathcal{T}$-duality has been shown (\cite{FSS16TDuality}, reviewed in \cite[\S 9]{FSS19RationalStructure}) to reproduce the expected notion of (topological) T-duality from string theory: Here $\mathcal{A} = \big(\HomotopyQuotient{ \mathrm{KU} }{ \mathbf{B}^2\mathbb{Z} }\big)^{\mathbb{Q}}$ and $\mathcal{B} = \big(\HomotopyQuotient{ \Sigma\mathrm{KU} }{ \mathbf{B}^2\mathbb{Z} }\big)^{\mathbb{Q}}$ are the rationalizations of the twisted complex K-theory spectra in degree 0 and in degree 1, respectively. The full lift of this situation beyond the rational approximation remains to be discussed elsewhere, but we may readily spell out the comparatively simple but crucial sector of the twisting by passing along the projection
\vspace{-4mm}
$$
  \begin{tikzcd}[
    row sep=-1pt
  ]
  \HomotopyQuotient
    { \mathrm{KU}_0 }
    { \mathbf{B}^2 \mathbb{Z} }
  \ar[rr]
  &&
  \HomotopyQuotient
    { \ast }
    { \mathbf{B}^2 \mathbb{Z}  }
   \,\simeq\,
   \mathbf{B}^3 \mathbb{Z}
  \\
  \mathrm{Cyc}
  \big(
    \HomotopyQuotient
      { \mathrm{KU}_0 }
      { \mathbf{B}^2 \mathbb{Z} }
  \big)
  \ar[rr]
  &&
  \mathrm{Cyc}
  (
    \mathbf{B}^3 \mathbb{Z}
  )
  \end{tikzcd}
$$

\vspace{-2mm}
\noindent (i.e., focusing on the the ``B-field'' while ignoring the ``RR-field'' for the moment): From \eqref{SullivanModelForCycRetraction}
and standard facts about homotopy (co-)fibers of maps of Sullivan models (e.g. \cite[Prop. 3.5]{FSS15WZW}) one finds that the
cyclification of $\mathbf{B}^3 \mathbb{Z}$ is the delooping of the shape of the {\it T-duality 2-group} \cite[\S 3.2.1]{FSS12CS}
\cite[Def. 7.1]{FSS16TDuality}, defined to be the homotopy fiber of the cup product on degree-2 cohomology:
\begin{equation}
  \label{TDualityTwoGroup}
  \begin{tikzcd}[
    row sep=-3pt,
    column sep=20pt
  ]
    \scalebox{.7}{
      \begin{tabular}{c}
        cyclification of deg=3 cohomology is
        \\
        classifying space for T-duality pairs
      \end{tabular}
    }
    &[-20pt]
    \mathrm{Cyc}\big(
      \mathbf{B}^3 \mathbb{Z}
    \big)
    \ar[
      rr,
      "{ \mathrm{fib}(\cup) }"
    ]
    &&
    \mathbf{B}^2 \mathbb{Z}
      \times
    \mathbf{B}^2 \mathbb{Z}
    \ar[
      rr,
      "{ \cup }"
    ]
    &&
    \mathbf{B}^4 \mathbb{Z}
    \\
    \scalebox{.7}{
      \begin{tabular}{c}
        witnessed by homotopy (co)fiber
        \\
        of cup product on deg=2
      \end{tabular}
    }
    &
  \scalebox{.7}{$  \left(
    \hspace{-4pt}
    \begin{array}{l}
      \mathrm{d}
      \;
      c_{3} \;\;=\; c_2 \wedge \omega_2
      \\
      \mathrm{d}
      \;
      c_{2} \;\;=\; 0
      \\
      \mathrm{d}
      \;
      \omega_2 \;=\; 0
    \end{array}
    \hspace{-4pt}
    \right)
    $}
    \ar[
      from=rr,
      "{ \mathrm{hocofib} }"{swap}
    ]
    &&
\scalebox{.7}{$    \left(
    \hspace{-4pt}
    \begin{array}{l}
      \mathrm{d}
      \;
      c_{2} \;=\; 0
      \\
      \mathrm{d}
      \;
      \omega_2 \;=\; 0
    \end{array}
    \hspace{-4pt}
    \right)
    $}
    \ar[
      from=rr,
      "{
   \scalebox{.7}{$     c_2 \wedge \omega_2
        \;\mapsfrom\;
        c_4$}
      }"{swap}
    ]
    &&
  \scalebox{.7}{$  \left(
    \hspace{-4pt}
    \begin{array}{l}
      \mathrm{d}
      \;
      c_{4} \;=\; 0
    \end{array}
    \hspace{-4pt}
    \right)\;.
    $}
  \end{tikzcd}
\end{equation}
This homotopy fiber is known to be the classifying space for topological T-duality pairs (\cite[Thm. 2.17]{BunkeSchick05} \cite[\S 6.2]{Rosenberg09}\cite[Rem. 7.2]{FSS16TDuality}); and from \eqref{GeneralizedTDuality}
we transparently see how this comes about:
Given a $\mathrm{Cyc}\big(\mathbf{B}^3 \mathbb{Z}\big)$-valued cocycle $\phi$ on a base orbifold $\mathcal{Y}$, then by iterative application of the pasting law (Fact \ref{PastingLaw}), we may extract the following pasting diagram of homotopy cartesian squares which makes appear, on the left of the diagram, two total spaces $\mathcal{X}_{A/B}$ carrying ``gerbes'' $\mathcal{G}_{A/B}$ (here incarnated as $\mathbf{B}^2 \mathbb{Z}$-principal $\infty$-bundles) subject to the constraint that their pullback to the correspondence orbifold $\mathcal{X}_A \times_{\mathcal{Y}} \mathcal{X_B}$ are both equivalent to the ``Poincar{\'e} gerbe'' $\mathcal{G}$:
$$
  \begin{tikzcd}[
    column sep=15pt
  ]
   &
    &
    \mathcal{G}
    \ar[d]
    \ar[rrr]
    &&&
    \ast
    \ar[d]
    \\
    &
    &[-8pt]
    \mathcal{X}_A
      \times_{\mathcal{Y}}
    \mathcal{X}_B
    \ar[ddddl]
    \ar[rrr]
    \ar[
      ddr,
      gray
    ]
    &[-8pt]
    &
    &[-20pt]
    \mathbf{B}^3 \mathbb{Z}
    \ar[rrr]
    \ar[
      ddr,
      gray,
      "{ (0, \mathrm{id}) }"{description}
    ]
    &[-20pt]
    &
    &[-14pt]
    \ast
    \ar[
      ddr
    ]
    &[-14pt]
    \\
    \\[-20pt]
    &
    &&
    \color{gray}
    \mathcal{X}_B
    \ar[
      rrr,
      gray,
      "{
        (
        \vdash\, \mathcal{G}_B
        ,\,
        \cdots
        )
      }"{pos=.2, description}
    ]
    \ar[
      dddl,
      gray
    ]
    &&&
    \color{gray}
    \mathbf{B}^3 \mathbb{Z}
      \times
    \mathbb{B}^2 \mathbb{Z}
    \ar[
      rrr, gray
    ]
    \ar[
      dddl,
      gray
    ]
    &&&
    \mathbf{B}^2 \mathbb{Z}
     \ar[
      dddl,
      "{ (0, \mathrm{id}) }"{description}
    ]
    \\[-17pt]
    \\[-17pt]
    &
    \mathcal{X}_A
    \ar[dr]
    \ar[
      rrr,
      crossing over,
      "{
        \big(
        \vdash\, \mathcal{G}_A
        ,\,
        \cdots
        \big)
      }"{pos=.3, description}
    ]
    &&&
    \mathbf{B}^3 \mathbb{Z}
    \times
    \mathbf{B}^2\mathbb{Z}
    \ar[
      from=uuuur,
      crossing over,
      "{ (\mathrm{id}, 0) }"{description}
    ]
    \ar[dr]
    \ar[
      rrr,
      crossing over
    ]
    &&&
    \mathbf{B}^2\mathbb{Z}
    \ar[
      dr,
      "{ (\mathrm{id},0) }"{description}
    ]
    \ar[
      from=uuuur,
      crossing over
    ]
    \\
    &
    &
    \color{purple}
    \mathcal{Y}
    \ar[
      rrr,
      purple,
      "{ \phi }"{description}
    ]
    &&&
    \color{purple}
    \mathrm{Cyc}(\mathbf{B}^3\mathbb{Z})
    \ar[d]
    \ar[
      rrr,
      "{ (c_1^A,\, c_1^B) }"{description}
    ]
    &&&
    \mathbf{B}^2\mathbb{Z}
    \times
    \mathbf{B}^2\mathbb{Z}
    \ar[
      d,
      "{ \cup }"{description}
    ]
    \\
    &
    &
    &&&
    \ast
    \ar[
      rrr
    ]
    &&&
    \mathbf{B}^4 \mathbb{Z}
  \end{tikzcd}
$$
the basic axiomatics \cite{BunkeSchick05} of ``topological T-duality'' (going back to \cite{BEM04}, good review in \cite[\S 1]{Waldorf22})
in the familiar situation and generalizes it to orbifolds and higher stacks.

\medskip

By way of further outlook, we close this overview by indicating the following concrete application, which deserves to be discussed in more detail elsewhere:

\medskip

\noindent
{\bf Application to brane physics.}
While much of the interest in elliptic cohomology goes back to the suggestion \cite{Witten1988}
that the partition function of the heterotic string is an elliptic genus of the string's target spacetime $\mathcal{X}$, realized via $\CohesiveCircle$-equivariant K-theory of its free loop space $\Maps{}{\CohesiveCircle}{\mathcal{X}}$, only a couple of authors (\cite{KrizSati04}\cite{KrizSati05}) have investigated the idea that elliptic cohomology might relate to the widely expected hypothesis that the topological K-theory of $\mathcal{X}$ itself measures the brane charges in (not heterotic but) type I/II string theories (a comprehensive list of references for this classical hypothesis may be found in \cite[\S 1]{BSS18}).

\smallskip
One circumstantial hint of what may really be going on comes the fact that the elliptic genus of the heterotic string is expected (at physics level of rigour) to really just be a special case of the general notion of elliptic genera of toroidally compactified super-branes, notably of the M5-brane \cite{GSY07}\cite{GaiottoYin07}\cite{AHHKRW15} \cite{GukovPeiPutrovVafa21}. This leads us to ask whether elliptic cohomology measures aspects of M-brane charges?

\smallskip
One coherent answer to what may be going on is suggested by {\it Hypothesis H} \cite{FSS19TwistedCohomotopy}\cite{SS19Tad}\cite{SS21MF}: This postulates that the charges of branes in M-theory are measured in tangentially $\mathrm{Sp}(2) \,\simeq\, \mathrm{Spin}(5)$-twisted equivariant unstable 4-Cohomotopy (see \cite{FSS20CharacterMap} for details), i.e. in the non-abelian cohomology theory whose plain moduli stack is the 4-sphere homotopy quotiented by its canonical $\mathrm{Spin}(5) \to \mathrm{SO}(5)$-action, with the ordinary cohomology charge shown in \eqref{ReductionOfFourCohomotopyCharges} being only a subtle integral characteristic class $\widetilde{\Gamma}_4$ (\cite[Lem. 3.12]{FSS19TwistedCohomotopy}\cite[(7)]{FSS22TwistorialGreenSchwarz}\cite[(4)]{FSSTwistedStringStruc21})
of this twisted non-abelian cohomology theory:
\vspace{-.4cm}
\begin{equation}
  \label{MBraneChargeViaHypothesisH}
  \begin{tikzcd}[
    column sep=35pt
  ]
    \mathllap{
      \scalebox{.7}{
        \color{darkblue}
        \bf
        \def\arraystretch{.8}
        \begin{tabular}{c}
          11d spacetime
          \\
          orbifold
        \end{tabular}
      }
    }
    \mathcal{X}
    \ar[
      ddr,
      dashed,
      "{ \tau }"{description},
      "{
        \scalebox{.7}{
          \scalebox{1.2}{$\mathrm{Sp}(2)$}-structure
        }
      }"{sloped, swap, yshift=0pt, pos=.43}
    ]
    \ar[
      rr,
      dashed,
      "{
        \scalebox{.7}{
          \def\arraystretch{.8}
          \color{darkblue}
          \bf
          \begin{tabular}{c}
            cocycle in tangentially
            \\
            twisted 4-Cohomotopy
          \end{tabular}
        }
      }",
      "{
        \scalebox{.7}{
          \def\arraystretch{.8}
          \begin{tabular}{c}
            \color{greenii}
            \bf
            {\color{black}=} M-brane charges
            \\
            (by {\it Hypothesis H})
          \end{tabular}
        }
      }"{swap}
    ]
    \ar[
      dddr,
      bend right=30,
      "{
        \vdash
        \,
        \mathrm{Fr}(\mathcal{X})
      }"{description, sloped },
      "{
        \scalebox{.7}{
          \color{greenii}
          \bf
          frame bundle
        }
      }"{swap, yshift=-4pt, sloped}
    ]
    \ar[
      drrrr,
      rounded corners,
      gray,
      to path={
         -- ([yshift=13pt]\tikztostart.north)
         -- node[yshift=5pt, xshift=3.2cm]{
              \scalebox{.7}{
                \color{gray}
                pure M5-brane charge
              }
         }
            ([yshift=31pt]\tikztotarget.north)
         -- ([yshift=00pt]\tikztotarget.north)
      }
    ]
    &
    &
    \HomotopyQuotient
      { S^4 }
      { \mathrm{Sp}(2) }
    \ar[ddl]
    \ar[
      drr,
      "{
        \widetilde \Gamma_4
        \,:=\,
        \tfrac{1}{2}\rchi_4
        +
        \tfrac{1}{4}p_1
      }"{description, sloped},
      "{
        \scalebox{.7}{
          \color{greenii}
          \bf
          \def\arraystretch{.8}
          \begin{tabular}{c}
            universal 4-class
            \\
            extracts \scalebox{1.3}{$G_4$}-flux
          \end{tabular}
        }
      }"{swap, yshift=-6pt, sloped}
    ]
    &
    \\[-25pt]
    &&
    &
    \phantom{AAAAAAA}
    &
    \mathbf{B}^4 \mathbb{Z}
    \\[-5pt]
    &
    \mathbf{B}\mathrm{Sp}(2)
    \ar[d]
    \\
    &
    \mathbf{B}\mathrm{Spin}(8+n)
  \end{tikzcd}
\end{equation}
\vspace{-.4cm}

If $\mathcal{X}$ is an $S^1$-principal bundle over a 10d spacetime orbifold $\mathcal{Y}$, as usually assumed (starting with \cite{DuffHoweInamiStelle87}\cite[p. 10]{Witten95}, see \cite[p. 2]{MathaiSati04}), then the $\mathrm{Cyc}$-adjunction says
\eqref{ExtCycAdjunctionInIntroduction}
that such charges on $\mathcal{X}$ are equivalently twisted equivariant $\mathrm{Cyc}(S^4)$-valued charges on $\mathscr{Y}$. Curiously, these cyclic 4-sphere coefficients are close to the twisted K-theory expected for D-brane charge (\cite[Ex. 3.3]{FSS16SphereValued}\cite[Ex. 2.47]{BSS18}):
\vspace{-2mm}
$$
  \begin{tikzcd}
    \overset{
      \mathclap{
        \raisebox{5pt}{
          \scalebox{.7}{
          \color{darkblue}
          \bf
          \def\arraystretch{.8}
          \begin{tabular}{c}
            rational 6-truncation of
            \\
            cyclified 4-sphere
          \end{tabular}
        }
        }
      }
    }{
    \big[
    \mathrm{Cyc}
    (
      S^4
    )
    \big]_{6}^{\mathbb{Q}}
    }
    \quad
    \ar[
      r,
      phantom,
      "\simeq"
    ]
    &
    \overset{
      \mathclap{
        \raisebox{5pt}{
          \scalebox{.7}{
          \color{darkblue}
          \bf
          \def\arraystretch{.8}
          \begin{tabular}{c}
            rational 6-truncation of
            \\
            twisted K-theory spectrum
          \end{tabular}
        }
        }
      }
    }{
    \big[
    \big(
      \HomotopyQuotient
        { \mathrm{Fred}^{0}_{\mathbb{C}} }
        { \mathrm{PU}(\mathscr{H}) }
    \big)
    \big]_{6}^{\mathbb{Q}}
    }\;.
  \end{tikzcd}
$$

\vspace{.1cm}

This suggests that the pullback of the {\bf twisted ADE-equivariant quasi-elliptic cohomology of the 4-sphere} along twisted Cohomotopy cocycles
produces good observables on M-brane charge in a Tate-elliptic enhancement of D-brane charge in twisted equivariant K-theory (for which our notation follows \cite[Ex. 4.5.4]{SS21EPB}):

\vspace{-.3cm}
\begin{equation}
  \label{QuasiEllitpicCohomologyOfFourSphere}
  \hspace{9mm}
  \begin{tikzcd}[
    column sep=60pt,
    row sep=0pt
  ]
    &
    \HomotopyQuotient
      {
        \mathbb{R}^4_{\mathrm{cpt}}
      }
      { G }
    \ar[
      dddr,
      end anchor={[yshift=+2pt]},
      "{
        \scalebox{.7}{
          platonic 2-gerbe
        }
      }"{sloped, description},
      "{
        \scalebox{.7}{
          (Prop. \ref{UniversalShiftedIntegral4FluxRestrictsToGeneratorOnADESingularity},
          Rem. \ref{BackgroundIntegral4FluxAtFlatADESingularity})
        }
      }"{pos=.5, sloped, yshift=3pt}
    ]
    \ar[
      dd,
      shorten=-3pt
    ]
    \ar[
      rr,
      lightgray,
      "{
        \scalebox{.7}{
          \color{lightgray}
          \def\arraystretch{.8}
          \begin{tabular}{c}
            hypothetical
            cocycle in
          \end{tabular}
        }
      }",
      "{
        \scalebox{.7}{
          \color{lightgray}
          twisted equivariant
          elliptic cohom.
        }
      }"{swap},
    ]
    &&
    \color{lightgray}
    \HomotopyQuotient
      { \mathrm{Ell} }
      { \mathbf{B}^2 \mathrm{U}(1) }
    \ar[
      dddl,
      lightgray
    ]
    \\[-4pt]
    \mathllap{
      \scalebox{.7}{
        \color{darkblue}
        \bf
        \def\arraystretch{.8}
        \begin{tabular}{c}
          M-theory
          \\
          circle bundle
        \end{tabular}
      }
      S^1
      \xrightarrow{\;}
      \;
    }
    \mathcal{X}
    \ar[
      dr,
      "{
        \scalebox{.7}{
          \begin{tabular}{c}
            \color{greenii}
            \bf
            M-brane charges
            \\
            (by \it Hypothesis H)
          \end{tabular}
        }
      }"{swap, sloped, yshift=-2pt}
    ]
    \ar[
      dddd
    ]
    \ar[
      ur,
      dashed,
      "{
        \scalebox{.7}{
          \def\arraystretch{.8}
          \begin{tabular}{c}
            consider: vicinity of
            \\
            ADE-singularity
          \end{tabular}
        }
      }"{sloped, yshift=+2pt}
    ]
    \ar[
      urrr,
      rounded corners,
      lightgray,
        to path={
           -- ([xshift=2pt, yshift=38pt]\tikztostart.east)
           --
           node[yshift=6pt, xshift=1.5cm]{
             \scalebox{.7}{
               M-brane charge
               hypothetically
               seen in
               full elliptic cohomology
             }
           }
           ([yshift=+10pt]\tikztotarget.north)
           -- ([yshift=-00pt]\tikztotarget.north)
        }
    ]
    &
    &
    &
    {}
    \\[-6pt]
    &
    \HomotopyQuotient
      { S^4 }
      { \mathrm{Sp}(2) }
    \ar[
      dr,
      "{
        \widetilde \Gamma_4
      }"{description},
      "{
        \scalebox{.65}{
          \def\arraystretch{.8}
          \color{greenii}
          \bf
          \begin{tabular}{c}
            universal shifted 4-class
            \\
            extracts \scalebox{1.3}{$G_4$}-flux
          \end{tabular}
        }
      }"{swap, yshift=-6pt,sloped}
    ]
    \\[+10pt]
    &&
    \mathbf{B}^4
    \,
    \mathbb{Z}
    \\[+40pt]
    &
    \overset{
      \mathclap{
      \raisebox{2pt}{
        \color{darkblue}
        \scalebox{.7}{
          GRH's ext. inertia
        }
      }
      }
    }{
    \Lambda_{\CohesiveCircle}
    \big(
      \HomotopyQuotient
        { S^4 }
        { \mathrm{G} }
    \big)
    }
    \ar[
      dd,
      shorten <=-3pt,
      shorten >=-5pt,
    ]
    \ar[
      rr,
      shorten <=120pt
    ]
    \ar[
      rr,
      line width=2pt,
      densely dashed,
      -,
      shorten >=2pt,
      gray,
      "{
        \scalebox{.7}{
          \color{black}
          (canonical??)
          \color{orangeii}
          cocycle on 4-sphere in
        }
      }",
      "{
        \scalebox{.7}{
          \color{orangeii}
          twisted equivariant quasi-elliptic
          cohom.
        }
      }"{swap}
    ]
    &&
    \HomotopyQuotient
      { \mathrm{Fred}^0_{\mathbb{C}} }
      { \mathrm{PU}(\mathscr{H}) }
    \ar[
      dddl,
      "{
        \scalebox{.65}{
          \color{gray}
          \def\arraystretch{.8}
          \begin{tabular}{c}
            local coefficient bundle
            \\
            for twisted equivariant K-theory
          \end{tabular}
        }
      }"{sloped, swap}
    ]
    \\[-6pt]
    \mathllap{
      \scalebox{.7}{
        \color{darkblue}
        \bf
        \def\arraystretch{.8}
        \begin{tabular}{c}
          type II
          \\
          spacetime orbifold
        \end{tabular}
      }
    }
    \mathcal{Y}
    \ar[
      dr,
      "{
        \scalebox{.7}{
          \def\arraystretch{.8}
          \begin{tabular}{c}
            \color{greenii}
            \bf
            dim-reduced
            \\
            \color{greenii}
            \bf
            brane charges
            \\
            \;\;\;\;
            (by \scalebox{1.3}{$\mathrm{Cyc}$}-adjunction)
          \end{tabular}
        }
      }"{swap, sloped, yshift=-2pt}
    ]
    \ar[
      ur,
      dashed,
      "{
        \scalebox{.7}{
          \def\arraystretch{.8}
          \begin{tabular}{c}
            assume:
            \\
            small M-circle
          \end{tabular}
        }
      }"{sloped, yshift=2pt}
    ]
    \ar[
      urrr,
      rounded corners,
      gray,
        to path={
           -- ([xshift=10pt, yshift=38pt]\tikztostart.east)
           --
           node[yshift=6pt, xshift=2cm]{
             \scalebox{.7}{
               M-brane charge seen in
               quasi-elliptic cohomology
             }
           }
           ([yshift=+10pt]\tikztotarget.north)
           -- ([yshift=-00pt]\tikztotarget.north)
        }
    ]
    \\[-8pt]
    &
    \mathrm{Cyc}
    \big(
      \HomotopyQuotient
        { S^4 }
        { \mathrm{Sp}(2) }
    \big)^{\phantom{A}}
    \ar[
      dr,
      "{
        \int_{S^1}
        \hspace{-1pt}
        \mathrm{Cyc}(\widetilde \Gamma_4)
      }"{pos=.5, sloped},
      "{
        \scalebox{.7}{
          \color{greenii}
          \bf
          extract \scalebox{1.3}{$H_3$}-flux
        }
      }"{swap, sloped, yshift=-2pt}
    ]
    \\[-0pt]
    &
    &
    \mathbf{B}^3 \mathbb{Z}
  \end{tikzcd}
\end{equation}

We compute the relevant twisting 4-class below in \cref{CharacteristicFourClassesOfEquivariantFourCohomotopy}, cf. Rem. \ref{ConclusionOnMBraneChargeTwist}.
In order to further analyze \eqref{QuasiEllitpicCohomologyOfFourSphere}, one will need to first compute the equivariant quasi-elliptic cohomology of representation 4-spheres of finite subgroups of $\mathrm{SU}(2)$ twisted by the resulting transgressed 3-class $\int_{S^1} \mathrm{Cyc}\big( \widetilde \Gamma_4\big)$.
Of course, this is just one of the ``twisted equivariant homotopy-groups'' (rather: ``-modules'', due to the non-trivial twist) of quasi-elliptic cohomology, which are bound to be of fundamental interest in their own right.
We leave their computation to the quasi-elliptic community.

\medskip

\noindent
{\bf Outline.}

\noindent
In \cref{CyclicInertiaOrbifolds} we prove that the abstract cyclification construction \eqref{ConstructingExtendedInertiaInIntroduction} recovers GRH's component model (Thm. \ref{AbstractCharacterizationOfHuanInertiaOrbifold}).

\noindent
In \cref{TransgressionViaCyclification} we prove that the abstract transgression operation \eqref{TransgressionMapViaCyclification} recovers traditional component formulas (Thm. \ref{TransgressionInGroupCohomologyViaLooping}).

\noindent
In \cref{CharacteristicFourClassesOfEquivariantFourCohomotopy} we compute the integral 4-class (transgressing to a 3-class) to be used in \eqref{QuasiEllitpicCohomologyOfFourSphere} for measuring M5-brane genera.

\noindent
In \cref{TechnicalMaterial}
we compile some technical background in simplicial \& geometric homotopy theory, for reference.

\newpage

%%%%%%%%%%%%%%%%%%%%%%%%%%%%%%%%%%%%%%%%%%%%%
\section{Cyclic inertia orbifolds}
\label{CyclicInertiaOrbifolds}
%%%%%%%%%%%%%%%%%%%%%%%%%%%%%%%%%%%%%%%%%%%%

\noindent
 In \cref{TheGRHInertiaOrbifold}
we recall the ``extended'' inertia orbifolds due to Ganter, Rezk \& Huan (GRH);

\noindent
in \cref{CyclicInertiaInfinityGroupoids} we present a general abstract theory of
 cyclic intertia $\infty$-groupoids;

\noindent
in \cref{ReproducingTheGRHIntertiaOrbifold} we prove that the GRH construction models the abstract definition.

%%%%%%%%%%%%%%%%%%%%%%%%%%%%%%%%%%%%%%%%%%%%%%%%%%%%%
\subsection{GRH's extended intertia orbifold}
\label{TheGRHInertiaOrbifold}
%%%%%%%%%%%%%%%%%%%%%%%%%%%%%%%%%%%%%%%%%%%%%%%%%%%%%

\noindent
{\bf Smooth loop and inertia stacks.}
Given an orbifold or, more generally, any
smooth $\infty$-grouoid $\mathcal{X} \,\in\, \SmoothInfinityGroupoids$
(see Nota. \ref{BasicNotation})
it is well-known (albeit not always stated in the following model-independent stack-theoretic manner\footnote{Regarding $\CohesiveCircle$ as an object of $\SmoothInfinityGroupoids$ ensures that stacky maps out of it are modeled by simplicial maps out of any good open cover (cf. \cite[Ex. 3.3.41]{SS21EPB}), which is what takes care of component constructions such as in \cite[Def. 3.1.1]{LupercioUribe02LoopGroupoids} }) that:

\noindent
{\bf (i)} its {\it smooth loop stack} (\cite[\S 3]{LupercioUribe02LoopGroupoids}\cite[\S 5]{BGNX07StackyStringTopology}) is the mapping stack
out of the
smooth circle $\CohesiveCircle \in \SmoothManifolds$ $\xhookrightarrow{\;y\;} \SmoothInfinityGroupoids$:
\begin{equation}
  \label{FreeLoopStack}
  \mathllap{
    \raisebox{2pt}{
      \tiny
      \color{darkblue}
      \bf
       smooth loop stack
     }
    \;\;
  }
  \mathcal{L}\mathcal{X}
  \;\coloneqq\;
  \Maps{\big}
    { \CohesiveCircle }
    { \mathcal{X} }
  \;\;\;
  \in
  \;
  \SmoothInfinityGroupoids
\end{equation}

\noindent
{\bf (ii)}
its {\it inertia stack} (e.g. \cite[\S 4]{LupercioUribe04})
is the mapping stack out of the delooping groupoid of the integers
$\mathbf{B} \mathbb{Z}
  \,\in\,
 \InfinityGroupoids
  \xrightarrow{\Discrete}
  \SmoothInfinityGroupoids
$:

\vspace{-.5cm}
\begin{equation}
  \label{InertiaStack}
  \mathllap{
    \raisebox{2pt}{
      \tiny
      \color{darkblue}
      \bf
      inertia stack
    }
    \;\;
  }
  \Lambda\mathcal{X}
  \;\coloneqq\;
  \Maps{\big}
    { \mathbf{B}\mathbb{Z} }
    { \mathcal{X} }
  \;\;\;
  \in
  \;
  \SmoothInfinityGroupoids
  \,.
\end{equation}

\begin{remark}[Components of inertia]
In the special case of {\it good orbifolds},
or, more generally, of good diffeological orbispaces \cite[\S 4]{SS20OrbifoldCohomology},
  \vspace{-2mm}
\begin{equation}
  \label{GoodOrbifold}
  \mbox{
    $\mathcal{X}$
    is good
  }
  \;\;\;\;
  \Leftrightarrow
  \;\;\;\;
  \underset{
  \scalebox{.7}{$
    {G \in \Groups(\Sets)}
    \atop
    {G \scalebox{.7}{$\acts$}\, X
    \;\in\; \Actions{G}(\DiffeologicalSpaces)}
  $}
  }{\exists}
  \mathcal{X} \,\simeq\, X \!\sslash\! G
  \,,
\end{equation}
% i.e. equivalent to the
% homotopy quotient (in $\SmoothInfinityGroupoids$)
% of a smooth manifold (generally: a diffeological space) $X$ by the smooth action of a
% discrete group $G$,
\noindent
the inertia stack is readily found (and well-known) to be equivalent to
a disjoint union over the conjugacy classes $[g] \in G/_{{}_{\!\mathrm{ad}}} G$ of
the corresponding fixed loci
$X^g \subset X$ by their residual
action of the centralizer subgroups $C_g \subset G$,
as follows \footnote{If $G \acts \, X$ is a proper action on a smooth manifold, then the fixed loci $X^g$ are themselves
smooth manifolds, so that the inertia stack is again a good orbifold.
But the equivalence \eqref{SkeletonOfInertiaStackOfGoodOrbifold} holds more generally,
as shown, for
$\mathcal{X} \simeq X \!\sslash\! G$ any {\it good diffeological orbi-space},
where $X$ may be any {\it diffeological space} with smooth $G$-action,
faithfully subsuming finite-dimensional smooth manifolds as well as infinite-dimensional
Fr{\'e}chet-manifolds. This is a convenient generalization,
as the smooth loop stack construction \eqref{FreeLoopStack}
restricts to an endo-functor (2-functor) on diffeological orbispaces,
where it may hence be iterated.}

\vspace{-.5cm}
\begin{equation}
  \label{SkeletonOfInertiaStackOfGoodOrbifold}
  G \acts \, X \,\in\, \Actions{G}\big(\DiffeologicalSpaces\big)
  \;\;
  \mbox{and}
  \;\;
  \mathcal{X} \,\simeq\, X \!\sslash\! G
  {\phantom{AAAA}}
  \Rightarrow
  {\phantom{AAAA}}
  \overset{
    \mathclap{
    \raisebox{4pt}{
      \tiny
      \color{darkblue}
      \bf
      inertia orbifold
    }
    }
  }{
    \Lambda \mathcal{X}
  }
  \quad \simeq\;
  \overset{
    \mathclap{
    \raisebox{4pt}{
      \tiny
      \color{darkblue}
      \bf
      skeletal presentation
    }
    }
  }{
  \underset{
    [g] \in G/_{{}_{\!\mathrm{ad}}} G
  }{\coprod}
  X^g \!\sslash\! C_g
  }
  \,.
\end{equation}
\end{remark}

\begin{remark}[Cohesion knows about essentially constant loops]
The difference and the relation between these two
constructions  \eqref{FreeLoopStack} and \eqref{InertiaStack}
is brought out by the shape modality:
The shape \eqref{ShapeModality}
of the cohesive circle is (Lemma \ref{ShapeUnitOfTheSmoothCircle})
equivalently the delooping groupoid of the
integers
  \vspace{-2mm}
$$
  \shape \CohesiveCircle \;\simeq\; \mathbf{B}\mathbb{Z}
  \;\;\;\;
  \in
  \;
  \SmoothInfinityGroupoids
$$

  \vspace{-2mm}
\noindent
so that we may understand the inertia stack \eqref{InertiaStack}
as being a loop stack itself, but for loops that are just the
{\it bare shape}
of a smooth circle:

\vspace{-.5cm}
\begin{equation}
  \label{InertiaStackToSmoothLooStackAsShapeUnitPullback}
  \begin{tikzcd}[column sep=large]
    \overset{
      \mathclap{
      \raisebox{4pt}{
        \tiny
        \color{darkblue}
        \bf
        \begin{tabular}{c}
          inertia stack
        \end{tabular}
      }
      }
    }{
      \Lambda \mathcal{X}
    }
    \ar[r,phantom,"\simeq"]
      \ar[
        rrrr,
        rounded corners,
        to path={
           -- ([yshift=-8pt]\tikztostart.south)
           --node[below]{
               \mbox{
                 \tiny
                 \color{greenii}
                 \bf
                 inclusion of cohesively constant loops
               }
             } ([yshift=-8pt]\tikztotarget.south)
           -- (\tikztotarget.south)}
      ]
    &
    \overset{
     \mathclap{
     \raisebox{4pt}{
       \tiny
       \color{darkblue}
       \bf
       $\shape \CohesiveCircle$-loop stack
     }
     }
    }{
      \Maps{\big}
        { \shape \CohesiveCircle }
        { \mathcal{X} }
    }
    \ar[
      rr,
      "{ \scalebox{0.7}{$
        \Maps{\small}
          { \eta_{\mathcal{X}}^{\scalebox{.5}{\shape}} }
          { \CohesiveCircle }
     $}
      }"{description},
      "{
        \mbox{
          \tiny
          \color{greenii}
          \bf
          shape unit
        }
      }"{above, xshift=-8pt, yshift=6pt}
    ]
    &&
    \overset{
      \mathclap{
      \raisebox{4pt}{
        \tiny
        \color{darkblue}
        \bf
        $\CohesiveCircle$-loop stack
      }
      }
    }{
    \Maps{\big}
      { \CohesiveCircle }
      { \mathcal{X} }
    \big]
    }
    \ar[r,phantom,"\simeq"]
    &
    \overset{
      \mathclap{
      \raisebox{3pt}{
        \tiny
        \color{darkblue}
        \bf
        smooth loop stack
      }
      }
    }{
      \mathcal{L}\mathcal{X}
    }.
  \end{tikzcd}
\end{equation}
It is this cohesive relation that formalizes the traditional notion that the inertia stack is the restriction of the smooth loop stack to the essentially constant loops: The ``bare shape'' of a loop cannot transverse any non-constant path in a manifold, but it can still jump between the ``twisted sectors'' to which the point belongs that it is constantly sitting on.

\end{remark}

%Further in the special case that $\mathcal{X} = X$ is a plain smooth manifold, the
%loop stack \eqref{FreeLoopStack} reduces to the smooth loop space
%Fr{\'e}chet manifold (cite).

\noindent
{\bf Cyclic loop spaces and cyclic homology.}
Closely related to \eqref{InertiaStackToSmoothLooStackAsShapeUnitPullback} is the fact
that for
$\mathcal{X} = \shape(\mathrm{X})$ the shape \eqref{ShapeModality} of a topological space
$\mathrm{X} \in \TopologicalSpaces \xhookrightarrow{\ContinuousDiffeology} \SmoothInfinityGroupoids$
\cite[Ex. 3.18]{SS20OrbifoldCohomology}
its inertia $\infty$-stack \eqref{InertiaStack}
is the shape
(by the {\it smooth Oka principle}, cf. \cite[p. 7]{SS21EPB})
of the topological free loop space of $\mathrm{X}$
(e.g. \cite{ChataurOancea15}):
\vspace{-2mm}
$$
  \Lambda \, \shape \, \mathrm{X}
  \;\simeq\;
  \shape \, \mathrm{Maps}(\CohesiveCircle, \mathrm{X})
  \,.
$$

\vspace{-2mm}
\noindent
In this context it is a familiar idea
(\cite[Thm. A]{Jones87},
review in \cite[Cor. 7.3.14]{Loday92}\cite[\S 3,4]{Loday15})
that associated to the free topological loop
space is what we call its {\it cyclic loop space} \cite[\S 3]{FSS16TDuality}\cite[\S 2.2]{BSS18},
namely the homotopy quotient (Borel construction)
by the circle action which rigidly rotates the loops:
\vspace{-3mm}
\begin{align*}
  \overset{
    \mathclap{
    \raisebox{4pt}{
      \tiny
      \color{darkblue}
      \bf
      cyclic topological loop space
    }
    }
  }{
  \mathrm{Cyc}(\mathrm{X})
  \;\coloneqq\;
  \mathrm{Maps}
  \big(
    \CohesiveCircle,
    \,
    \mathrm{X}
  \big)
  }
  \!\sslash\!
  \CohesiveCircle
&  \;\;
    \simeq
  \;\;
  \Big(
  \mathrm{Maps}
  \big(
    \CohesiveCircle,
    \,
    \mathrm{X}
  \overset{
    \mathclap{
    \raisebox{4pt}{
      \tiny
      \color{darkblue}
      \bf
      Borel construction
    }
    }
  }{
  \big)
  \times
  E \CohesiveCircle
  \Big) / \CohesiveCircle
  }
  \,,
%  {\phantom{AAAAA}}
 \\
  H^\bullet
  \big(
    \mathrm{Cyc}(\mathrm{X})
  \big)
  &\;\simeq\;
  \overset{
   \mathclap{
    \raisebox{4pt}{
      \tiny
      \color{darkblue}
      \bf
      cyclic homology
    }
    }
  }{
    \mathrm{HC}_{\bullet}
    \big(
      C^\bullet(\mathrm{X})
    \big)
  }
  \,,
\end{align*}

\vspace{-3mm}
\noindent
whose ordinary cohomology
(for any commutative coefficient ring $R$)
is the cyclic homology of (the graded $R$-algebra of $R$-chains of) $\mathrm{X}$.

\medskip

\noindent
{\bf Cyclic loop stacks.}
Similarly, it is clear that the loop stack \eqref{FreeLoopStack} carries a canonical
smooth action of the circle group $\CohesiveCircle \simeq \mathrm{U}(1)$ by rotation of loops
(component constructions are given in \cite[\S 3.6]{LupercioUribe02LoopGroupoids}, the general abstract construction follows by \eqref{CyclicicationAsRightDerivedBaseChange} below),
so that, in view of \eqref{InertiaStackToSmoothLooStackAsShapeUnitPullback},
we may consider the further homotopy quotient by this action,
which we will denote as
\vspace{-2mm}
\begin{equation}
  \label{CyclifiedSmoothLoopStack}
  \mathllap{
    \mbox{
      \tiny
      \color{darkblue}
      \bf
      \begin{tabular}{c}
        cyclified stack
        \\
        wrt smooth circle
      \end{tabular}
    }
  }
    \mathrm{Cyc}_{\CohesiveCircle}
    (\mathcal{X})
  \;\coloneqq\;
  \Maps{\big}
    { \CohesiveCircle }
    { \mathcal{X} }
  \!\sslash\!
  \CohesiveCircle
  \;\;\;
  \in
  \;
  \SmoothInfinityGroupoids
  \,.
\end{equation}

\vspace{-2mm}
\noindent
This fundamental construction has not received much attention
(it is alluded to in \cite[\S 2.1]{Ganter07StringyPower}\cite[p. 2]{Stapleton13})
until the recent introduction
%(following observations by C. Rezk)
of  modified orbifold loop groupoids
in \cite{Huan18QuasiEllipticI}
(denoted ``$Loop^{\mathrm{ext}}$'' in Def. 2.5, 2.9 there)
which we may
understand as plausible models for the homotopy quotient in
\eqref{CyclifiedSmoothLoopStack}.
Restricting these ad-hoc models to cohesively constant loops
leads to a definition (\cite[Def. 2.14]{Huan18QuasiEllipticI}
in slight variation of \cite[Def. 2.3]{Ganter07StringyPower},
review in \cite[p. 62]{Dove19}\cite[Def. 2.1]{HuanSpong20TwistedQuasiEllipic})
of a variation of the inertia orbifold \eqref{SkeletonOfInertiaStackOfGoodOrbifold},
as follows:
\vspace{-2mm}
\begin{equation}
  \label{SkeletonOfHuanInertiaStackOfGoodOrbifold}
  \overset{
    \mathclap{
    \raisebox{4pt}{
      \tiny
      \color{darkblue}
      \bf
      GRH inertia orbifold
    }
    }
  }{
    \ExtendedInertia
    \mathcal{X}
  }
  \quad \coloneqq\;
  \underset{
    [g] \in G/_{{}_{\mathrm{ad}}} G
  }{\coprod}
  X^g \!\sslash\! \Lambda_g
  \,,
  \;\;\;\;\;\;\;\;\;\;
  \proofstep{
     where}\footnote{The quotient group $\Lambda_g$ in \eqref{SkeletonOfHuanInertiaStackOfGoodOrbifold}
    may be motivated
    (as explained in \cite[(6.5)]{Dove19})
    as that which implements
    the ``rotation condition''
    proposed in \cite[Def. 2.6]{Ganter13}.
  }
 % }
  \;\;
    \Lambda_g
    \;\coloneqq\;
    \frac{
      C_g \times \mathbb{R}
    }{
      \langle (g^{-1},1) \rangle
    }
    \;
  \,.
\end{equation}

\vspace{-2mm}
\noindent
In \cite{Huan18QuasiEllipticI}
and followups (\cite{Huan18QuasiTheories}\cite{HuanSpong20TwistedQuasiEllipic}),
this definition is justified
{\it a posteriori}
by the  fact that
a completion of the orbifold K-theory of GRH's inertia orbifold \eqref{SkeletonOfHuanInertiaStackOfGoodOrbifold}
yields a good model
of the equivariant elliptic cohomology at the Tate curve of the original
orbifold:
\vspace{-2mm}
\begin{equation}
  \label{QuasiEllipticCohomology}
  \overset{
    \mathclap{
    \raisebox{4pt}{
      \tiny
      \color{darkblue}
      \bf
      \begin{tabular}{c}
        orbifold K-theory
      \end{tabular}
    }
    }
  }{
    \mathrm{KU}_{\mathrm{orb}}
  }
  \big(
    \underset{
      \mathclap{
      \raisebox{-2pt}{
        \color{orangeii}
        \bf \tiny
        \begin{tabular}{c}
          GRH's extended
          \\
          inertia orbifold
        \end{tabular}
      }
      }
    }{
      \Lambda_{\CohesiveCircle} \mathcal{X}
    }
  \big)
    \otimes_{\mathbb{Z}[q,q^{-1}]}
  \mathbb{Z}((q))
  \;\;
  \;\simeq
  \qquad \quad
  \overset{
    \mathclap{
    \mbox{
      \tiny
      \color{darkblue}
      \bf
      \begin{tabular}{c}
        orbifold Tate-elliptic cohomology
      \end{tabular}
    }
    }
  }{
    \mbox{Ell}^{\mathrm{Tate}}_{\mathrm{orb}}
  }
    (
    \mathcal{X}
    ).
\end{equation}

%%%%%%%%%%%%%%%%%%%%%%%%%%%%%%%%%%
\subsection{Cyclic intertia $\infty$-Groupoids}
\label{CyclicInertiaInfinityGroupoids}
%%%%%%%%%%%%%%%%%%%%%%%%%%%%%%%%%

\noindent
{\bf A general theory of stacky transformation $\infty$-groups.}
Our starting point is the observation
\cite{NSS12a}\cite[\S 2.2]{SS20OrbifoldCohomology}
that the general theory of
{\it transformation groups},
i.e. of geometric (topological-, Lie-, ...) groups acting as symmetries of geometric spaces
(\cite{Bredon72}\cite{tomDieck79}\cite{tomDieck87}),
not only has a good generalization to groupal $\infty$-stacks,
but becomes conceptually more transparent in this generalization,
when regarded systematically topic in {\it higher topos theory}
(\cite{ToenVezzosi05}\cite{Lurie09HTT}\cite{Rezk10}).
Namely, \cite[Prop. 0.2.1]{SS21EPB}:

\smallskip
\noindent\
{\bf (i)} for $\mathcal{T} \in \Groups(\Topos)$ a group
object in any $\infty$-topos $\Topos$ (Ntn. \ref{NotationInfinityToposes}),
the $\infty$-category of $\mathcal{T}$-actions on $\infty$-stacks is
equivalent to the {\it slice $\infty$-topos} over the $\mathcal{T}$-moduli $\infty$-stack
$\mathbf{B}\mathcal{T}$ (\cite[Prop. 3.2.76]{SS21EPB}, going back to \cite[\S 4.1]{NSS12a} and \cite{DDK80}, see also Prop. \ref{QuillenEquivalenceBetweenBorelModelStructureAndSliceOverClassifyingComplex} below),

\noindent
{\bf (ii)} for $\mathcal{T}_1 \xrightarrow{\;\phi\;} \mathcal{T}_2$ any homomorphism
of such $\infty$-groups, the constructions of {\it restricted}
and of {\it (co-)induced} $\infty$-actions
along $\phi$ are equivalently given by the base change adjoint triple \eqref{BaseChange}
along $\phi$ (\cite[Ex. 3.2.78]{SS21EPB}):

\vspace{-2mm}
\begin{equation}
\label{TransformationGroupsAsSliceObjects}
\begin{tikzcd}[column sep=75pt]
  \mathllap{
    \mbox{
      \tiny
      \color{darkblue}
      \bf
      \begin{tabular}{c}
        $\infty$-actions of
        \\
        $\mathcal{T}_1$ on $\infty$-stacks
      \end{tabular}
    }
  }
  \Actions{\mathcal{T}_1}(\Topos)
  \ar[
    rr,
    shift left=14pt,
    "
      \mbox{
        \tiny
        \color{greenii}
        \bf
        induced $\infty$-action
      }
    "
  ]
  \ar[
    rr,
    shift right=14pt,
    "
      \mbox{
        \tiny
        \color{greenii}
        \bf
        co-induced $\infty$-action
      }
    "{below}
  ]
    \ar[
      dd,
      "\sim"{sloped, description},
      "{
        \def\arraystretch{.7}
        \begin{array}{c}
          \mathcal{T}_1 \acts \,\mathcal{X}
          \\
          \mapsdown
          \\
          \mathcal{X} \!\sslash\! \mathcal{T}_1
        \end{array}
      }"{left},
      "
        \rotatebox{-90}{
          \tiny
          \color{greenii}
          \bf
          homotopy quotient
        }
      "{right, xshift=2pt}
    ]
  &{\phantom{AAAAA}}&
  \Actions{\mathcal{T}_2}(\Topos)
  \mathrlap{
    \mbox{
      \tiny
      \color{darkblue}
      \bf
      \begin{tabular}{c}
        $\infty$-actions of
        \\
        $\mathcal{T}_1$ on $\infty$-stacks
      \end{tabular}
    }
  }
  \ar[
    ll,
    "
      \mbox{
        \tiny
        \color{greenii}
        \bf
        restricted $\infty$-action
      }
    "{description}
  ]
  \ar[
    ll,
    phantom,
    shift left=9pt,
    "\scalebox{.5}{$\bot$}"
  ]
  \ar[
    ll,
    phantom,
    shift right=9pt,
    "\scalebox{.5}{$\bot$}"
  ]
    \ar[
      dd,
      "\sim"{sloped, description},
      "{
        \def\arraystretch{.7}
        \begin{array}{c}
          \mathcal{T}_2 \acts \,\mathcal{X}
          \\
          \mapsdown
          \\
          \mathcal{X} \!\sslash\! \mathcal{T}_2
        \end{array}
      }"{right},
      "
        \rotatebox{-90}{
          \tiny
          \color{greenii}
          \bf
          homotopy quotient
        }
      "{left, xshift=-2pt}
    ]
  \\
  {\phantom{A}}
  \\
  \mathllap{
    \mbox{
      \tiny
      \color{darkblue}
      \bf
      \begin{tabular}{c}
        $\infty$-stacks over
        \\
        $\mathcal{T}_1$-moduli
      \end{tabular}
    }
  }
  \Topos_{/\mathbf{B}\mathcal{T}_1}
  \ar[
    rr,
    shift left=16pt,
    "{
      (\mathbf{B}\phi)_!
    }"{description},
    "
      \mbox{
        \tiny
        \color{greenii}
        \bf
        left base change
      }
    "{above, yshift=+4pt}
  ]
  \ar[
    rr,
    shift right=16pt,
    "{
      (\mathbf{B}\phi)_\ast
    }"{description},
    "
      \mbox{
        \tiny
        \color{greenii}
        \bf
        right base change
      }
    "{below, yshift=-4pt}
  ]
  &&
  \Topos_{/\mathbf{B}\mathcal{T}_2}
  \mathrlap{
    \mbox{
      \tiny
      \color{darkblue}
      \bf
      \begin{tabular}{c}
        $\infty$-stacks over
        \\
        $\mathcal{T}_2$-moduli
      \end{tabular}
    }
  }
  \ar[
    ll,
    "
      \mbox{
        \tiny
        \color{greenii}
        \bf
        base change/pullback
      }
      (\mathbf{B}\phi)^\ast
    "{description}
  ]
  \ar[
    ll,
    phantom,
    shift left=9pt,
    "{
      \scalebox{.5}{$\bot$}
    }"
  ]
  \ar[
    ll,
    phantom,
    shift right=9pt,
    "{
      \scalebox{.5}{$\bot$}
    }"
  ]
\end{tikzcd}
\end{equation}

\noindent
{\bf Cyclic $\infty$-Stacks.}
Specializing \eqref{TransformationGroupsAsSliceObjects}
to the unique homomorphism from the trivial group
$1 \xrightarrow{e} \mathcal{T}$, where $\mathbf{B}e = \mathrm{pt}_{\mathbf{B}\mathcal{T}}$
is the base point inclusion into the $\mathcal{T}$-moduli $\infty$-stack,
we see that, in full generality, the construction of {\it $\mathcal{T}$-cyclic $\infty$-stacks}
\vspace{-2mm}
\begin{equation}
  \label{SCyclicInfinityStack}
  \mathrm{Cyc}_{\mathcal{T}}
  (
    \mathcal{X}
  )
  \;\coloneqq\;
  \Maps{}
    { \mathcal{T} }
    { \mathcal{X} }
  \!\sslash\!
  \mathcal{T}
  \;\;\;
  \in
  \;
  \Topos_{/\mathbf{B}\mathcal{T}}
  \xrightarrow{\; (p_{{}_{\mathbf{B}\mathcal{T}}})_! \;}
  \Topos
\end{equation}
is equivalently the right base change along the point inclusion
into the $\mathcal{T}$-moduli $\infty$-stack
(derived left base change back to the global context, as desired):
\vspace{-2mm}
\begin{equation}
  \label{CyclicicationAsRightDerivedBaseChange}
  \;\;\;\;\;\;
  \begin{tikzcd}[row sep=large, column sep=57pt]
    \mathllap{
      \raisebox{2pt}{
        \tiny
        \color{darkblue}
        \bf
        \begin{tabular}{c}
          $\infty$-stacks
        \end{tabular}
      }
    }
    \Topos
    \ar[
      dd,-,
      shift left=1pt
    ]
    \ar[
      dd,-,
      shift right=1pt
    ]
    \ar[
      rr,
      shift left=16pt,
      "{
        \mathcal{T} \acts \; \big( \mathcal{T} \times (-) \big)
      }"{description},
      "{
        \mbox{
        \tiny
        \color{greenii}
        \bf
        free $\infty$-action
        }
      }"{above, yshift=6pt}
    ]
    \ar[
      rr,
      shift right=16pt,
      "{
        \mathcal{T} \acts \; [ \mathcal{T},- ]
      }"{description},
      "
        \mbox{
          \tiny
          \color{greenii}
          \bf
          co-free $\infty$-action
        }
      "{below, yshift=-4pt}
    ]
    &{\phantom{AAAAAAAA}}&
    \overset{
      \mathclap{
      \raisebox{4pt}{
        \tiny
        \color{darkblue}
        \bf
        \begin{tabular}{c}
          $\infty$-actions
          \\
          of $\mathcal{T}$
        \end{tabular}
      }
      }
    }{
      \Actions{\mathcal{T}}(\Topos)
    }
    \ar[
      ll,
      "\mbox{
        \tiny
        \color{greenii}
        \bf
        underlying $\infty$-stack
      }"{description}
    ]
    \ar[
      ll,
      phantom,
      shift left=7pt,
      "\scalebox{.6}{$\bot$}"
    ]
    \ar[
      ll,
      phantom,
      shift right=8pt,
      "\scalebox{.6}{$\bot$}"
    ]
    \ar[
      dd,
      "\sim"{sloped, description},
      "{
        \def\arraystretch{.7}
        \begin{array}{c}
          \mathcal{T} \acts \, \mathcal{X}
          \\
          \mapsdown
          \\
          \mathcal{X} \!\sslash\! \mathcal{T}
        \end{array}
      }"{left},
      "
        \rotatebox{-90}{
          \tiny
          \color{greenii}
          \bf
          homotopy quotient
        }
      "{right, xshift=2pt}
    ]
    \ar[
      rr,
      shift left=16pt,
      "{
        (-)_{\mathcal{T}}
      }"{description},
      "{
        \mbox{
        \tiny
        \color{greenii}
        \bf
        homotopy co-invariants
        }
      }"{above, yshift=6pt}
    ]
    \ar[
      rr,
      shift right=16pt,
      "{
        (-)^{\mathcal{T}}
      }"{description},
      "
        \mbox{
          \tiny
          \color{greenii}
          \bf
          homotopy invariants
        }
      "{below, yshift=-4pt}
    ]
    &{\phantom{AAAAAAAA}}&
    \Topos
    \ar[
      ll,
      "\mbox{
        \tiny
        \color{greenii}
        \bf
        trivial $\infty$-action
      }"{description}
    ]
    \ar[
      ll,
      phantom,
      shift left=7pt,
      "\scalebox{.6}{$\bot$}"
    ]
    \ar[
      ll,
      phantom,
      shift right=8pt,
      "\scalebox{.6}{$\bot$}"
    ]
    \ar[
     dd,-,
     shift left=1pt,
    ]
    \ar[
     dd,-,
     shift right=1pt,
    ]
    \\
    {\phantom{A}}
    \\
    \Topos
    %\ar[
    %  rr,
    %  shift left=18pt,
    %  "{
    %    (\mathrm{pt}_{\mathbf{B}\mathcal{T}})_!
    %  }"{description},
    %  "{
    %    \mbox{
    %      \tiny
    %      \color{greenii}
    %      \bf
    %      left base change
    %    }
    %  }"{above, yshift=+5pt}
    %]
    \ar[
      rr,
      shift right=18pt,
      "{
        \mathrm{Cyc}_{\mathcal{T}}
        \;\coloneqq\;
        (\mathrm{pt}_{\mathbf{B}\mathcal{T}})_\ast
      }"{description},
      "{
        \mbox{
          \tiny
          \color{greenii}
          \bf
          cyclification
          \;\;
          =
          \;\;
          right base change
        }
      }"{below, yshift=-5pt, xshift=5pt}
    ]
    &&
    \underset{
      \mathclap{
      \raisebox{-4pt}{
        \tiny
        \color{darkblue}
        \bf
        \begin{tabular}{c}
          $\infty$-stacks
          \\
          over
          \\
          $\mathcal{T}$-moduli $\infty$-stack
        \end{tabular}
      }
      }
    }{
      \Topos_{/\mathbf{B}\mathcal{T}}
    }
    \ar[
      ll,
      "{
        \mathrm{Ext}_{\mathcal{T}}
        \;\coloneqq\;
        (\mathrm{pt}_{\mathbf{B}\mathcal{T}})^\ast
      }"{description},
      "{
        \mbox{
          \tiny
          \color{greenii}
          \bf
          extension
          \;=\;
          base change
        }
      }"{above, yshift=6pt}
    ]
    \ar[
      ll,
      phantom,
      shift left=9pt,
      "\scalebox{.6}{$\bot$}"
    ]
    \ar[
      rr,
      shift left=18pt,
      "{
        (p_{{}_{\mathbf{B}\mathcal{T}}})_!
      }"{description},
      "{
        \mbox{
          \tiny
          \color{greenii}
          \bf
          left base change to absolute context
        }
      }"{above, yshift=+5pt}
    ]
    \ar[
      rr,
      shift right=18pt,
      "{
        (p_{{}_{\mathbf{B}\mathcal{T}}})_\ast
      }"{description},
      "{
        \mbox{
          \tiny
          \color{greenii}
          \bf
          right base change to absolute context
        }
      }"{below, yshift=-5pt, xshift=5pt}
    ]
    &&
    \Topos
    \ar[
      ll,
      "{
        (p_{{}_{\mathbf{B}\mathcal{T}}})^\ast
      }"{description}
    ]
    \ar[
      ll,
      phantom,
      shift left=9pt,
      "\scalebox{.6}{$\bot$}"
    ]
    \ar[
      ll,
      phantom,
      shift right=9pt,
      "\scalebox{.6}{$\bot$}"
    ]
    \\[-30pt]
    \ast
      \ar[
        rrrr,
        Rightarrow,-,
        rounded corners,
        to path={
           -- ([yshift=-10pt]\tikztostart.south)
           --node[below]{}
              ([yshift=-10pt]\tikztotarget.south)
           -- (\tikztotarget.south)}
        ]
    \ar[
      rr,
      " \mathrm{pt}_{\mathbf{B}\mathcal{T}} "{description},
      "
        \mbox{
          \tiny
          \color{greenii}
          \bf
          base point inclusion
        }
      "{below, yshift=-4.5pt}
    ]
    &&
    \mathbf{B}\mathcal{T}
    \ar[
      rr,
      " p_{{}_{\mathbf{B}\mathcal{T}}} "{description},
      "
        \mbox{
          \tiny
          \color{greenii}
          \bf
          projection  to base point
        }
      "{below, yshift=-4.5pt}
    ]
    &&
    \ast
  \end{tikzcd}
\end{equation}
\vspace{-2mm}

In the top left and in many following diagrams we abbreviate our notation for mapping stacks to
$$
  [-,-] \,:=\, \Maps{}{-}{-}
  \,.
$$

\noindent Notice how this $\infty$-group-theoretic adjunction
\eqref{CyclicicationAsRightDerivedBaseChange}
witnesses shaundamental aspects of loop stacks and their cyclification:

\vspace{2mm}
{\small
\begin{tabular}{|c||c|c|c|}
  \hline
  \bf Object in:
  &
  {\bf
  Mapping $\infty$-Stack Theory
  }
  &
  {\bf
  Transf. $\infty$-Group Theory
  }
  &
  {\bf
  Slice $\infty$-Topos Theory
  }
  \\
  \hline
  \hline
  $
    \Maps{}
      { \mathcal{T} }
      { \mathcal{X} }
  $
  &
  \begin{tabular}{c}
    $\mathcal{T}$-loop $\infty$-stack of $\mathcal{X}$
  \end{tabular}
  &
  \begin{tabular}{l}
    underlying $\infty$-stack of
    \\
    co-free $\mathcal{T}$-$\infty$-action
    \\
    co-induced by $\mathcal{X}$
  \end{tabular}
  &
  \begin{tabular}{c}
    comonadic descent
    \\
    of $\mathcal{X}$ along
    base point
    \\
    into $\mathcal{T}$-moduli $\infty$-stack
  \end{tabular}
  \\
  \hline
  $
    \Maps{}
      { \mathcal{T} }
      { \mathcal{X} }
    \!\sslash\!
    \mathcal{T}
  $
  &
  \begin{tabular}{l}
    $\mathcal{T}$-cyclic
    $\infty$-stack of $\mathcal{X}$
  \end{tabular}
  &
  \begin{tabular}{l}
    homotopy quotient by
    \\
    co-free $\mathcal{T}$-$\infty$-action
    \\
    co-induced by $\mathcal{X}$
  \end{tabular}
  &
  \begin{tabular}{c}
    right derived base change
    \\
    of $\mathcal{X}$ along
    base point
    \\
    into $\mathcal{T}$-moduli $\infty$-stack
  \end{tabular}
  \\
  \hline
\end{tabular}
}

\newpage

\noindent
{\bf Inertia $\infty$-Stacks.}
Moreover, if $\Topos$ is cohesive (Ntn. \ref{BasicNotation}) and thus equipped with a
shape modality \eqref{ShapeModality}, then
with $\mathcal{T}$ also its shape $\shape \mathcal{T}$
is canonically a group $\infty$-stack
(\cite[Prop. 3.4]{SS20OrbifoldCohomology})
and the shape unit
$\begin{tikzcd} \
  \mathcal{T} \ar[r,"\eta_{\mathcal{T}}^{\scalebox{.5}{$\shape$}}"{description}] & \shape \mathcal{T}  \end{tikzcd}$ is a homomorphism of
group $\infty$-stacks, so that we may consider
{\it $\mathcal{T}$-inertia $\infty$-stacks} and their
$\mathcal{T}$-cyclification in full generality:

\vspace{-3mm}
\begin{equation}
  \label{SInertiaInfinityStacks}
  \begin{tikzcd}[column sep=large]
      \mbox{
        \tiny
        \color{darkblue}
        \bf
        $\mathcal{T}$-inertia
        $\infty$-stack
      }
    &
    (\shape \mathcal{T})
    \acts \;
    \Maps{\big}
      { (\shape \mathcal{T}) }
      { \mathcal{X} }
    \ar[
      rr, |->,
      "{
        \mbox{
          \tiny
          \color{greenii}
          \bf
          restricted $\infty$-action
        }
      }"{above},
      "\mbox{
        \tiny
        \color{greenii}
        \bf
        along shape unit
      }"{below}
    ]
    \ar[
      dd,
      |->,
      "
        \rotatebox{-90}{
          \tiny
          \color{greenii}
          \bf
          homotopy
        }
      "{right, xshift=0pt},
      "
        \rotatebox{-90}{
          \tiny
          \color{greenii}
          \bf
          quotient
        }
      "{left, xshift=-0pt}
    ]
    &&
    \mathcal{T}
    \acts \;
    \Maps{\big}
      { (\shape \mathcal{T}) }
      { \mathcal{X} }
    \ar[
      dd,
      |->,
      "
        \rotatebox{-90}{
          \tiny
          \color{greenii}
          \bf
          homotopy
        }
      "{right, xshift=0pt},
      "
        \rotatebox{-90}{
          \tiny
          \color{greenii}
          \bf
          quotient
        }
      "{left, xshift=-0pt}
    ]
    \\
    \\
      \mbox{
        \tiny
        \color{darkblue}
        \bf
        \begin{tabular}{c}
          cyclic
          \\
          $\mathcal{T}$-inertia
          $\infty$-stacks
        \end{tabular}
      }
    &
    \Maps{\big}
      { (\shape \mathcal{T}) }
      { \mathcal{X} }
    \!\sslash\!
    (\shape \mathcal{T})
    \ar[
      rr, |->,
      "{
     \scalebox{0.7}{$   \left(
          \eta_{\mathbf{B}\mathcal{T}}^{\scalebox{.5}{$\shape$}}
        \right)^\ast
      $}
      }"
    ]
    &&
    \Maps{\big}
      { (\shape \mathcal{T}) }
      { \mathcal{X} }
    \!\sslash\!
    \mathcal{T}
  \end{tikzcd}
\end{equation}

Notice that we obtain a natural comparison morphism,  as clained in  \eqref{CircleShapeUnitInIntroduction},
from any cyclic $\mathcal{T}$-intertia stack
\eqref{SInertiaInfinityStacks}
to the full $\mathcal{T}$-cyclification \eqref{SCyclicInfinityStack}
by factoring the defining base change through the decomposition
$$
  \begin{tikzcd}[column sep=large]
    \ast
    \ar[
      r,
      "{
        \mathrm{pt}_{\scalebox{.7}{$
          \mathbf{B}\mathcal{T}
        $}}
      }"
    ]
    \ar[
      rr,
      rounded corners,
      to path={
           -- ([yshift=-02pt]\tikztostart.south)
           -- ([yshift=-12pt]\tikztostart.south)
           -- node[yshift=+7pt]{
             \scalebox{.7}{$
               \mathrm{pt}_{
                 \mathbf{B}\shape \mathcal{T}
               }
             $}
           }
              ([yshift=-10pt]\tikztotarget.south)
           -- ([yshift=-00pt]\tikztotarget.south)
        }
    ]
    &
    \mathbf{B}\mathcal{T}
    \ar[
      r,
      "{
        \mathbf{B}
        \eta^{\scalebox{.6}{$
          \shape
        $}}
        _{ \mathbf{B}\mathcal{T} }
      }"
    ]
    &
    \mathbf{B}\shape\mathcal{T}
  \end{tikzcd}
$$
and invoking the counit $\epsilon : f^\ast f_\ast \to \mathrm{id}$ of the right base change adjunction \eqref{BaseChange}:
\begin{equation}
  \label{TheComparisonMorphism}
  \hspace{-15pt}
  \begin{tikzcd}[column sep=30pt]
    \mathrm{Cyc}_{
      \scalebox{.7}{$\shape \mathcal{T}$}
    }
    \big(
      \mathcal{X}
    \big)
    \ar[d, Rightarrow,-]
    &[-48pt]
    &
    &[-20pt]
    &[-10pt]
    \mathrm{Cyc}_{
      \scalebox{.7}{$\mathcal{T}$}
    }
    \big(
      \mathcal{X}
    \big)
    \ar[d, Rightarrow,-]
    \\[-12pt]
    \HomotopyQuotient
    {
      \Maps{\big}
        { \shape \mathcal{T} }
        { \mathcal{X} }
    }
    { \shape \mathcal{T} }
    \ar[d, Rightarrow,-]
    \ar[d, Rightarrow,-]
    &&
    \HomotopyQuotient
    {
      \Maps{\big}
        { \shape \mathcal{T} }
        { \mathcal{X} }
    }
    { \mathcal{T} }
    \ar[d, Rightarrow,-]
    \ar[
      rr,
      dashed
    ]
    &[-20pt]
    &
    \HomotopyQuotient
    {
      \Maps{\big}
        { \mathcal{T} }
        { \mathcal{X} }
    }
    { \mathcal{T} }
    \ar[d, Rightarrow,-]
    \\[-12pt]
    \big(
      \mathrm{pt}_{
        \scalebox{.7}{$
          \mathbf{B}\shape\mathcal{T}
        $}
      }
    \big)_\ast (\mathcal{X})
    \ar[d]
    \ar[from=rr]
    \ar[
      drr,
      phantom,
      "{
        \scalebox{.7}{(pb)}
      }"{pos=.4}
    ]
    &&
    \big(
      \mathbf{B}
      \eta
        ^{ \scalebox{.6}{$\shape$} }
        _{ \mathcal{T} }
    \big)^\ast
    \big(
      \mathrm{pt}_{
        \scalebox{.7}{$
          \mathbf{B}\shape\mathcal{T}
        $}
      }
    \big)_\ast (\mathcal{X})
    \ar[d]
    \ar[
      r,
      phantom,
      "{\simeq}"
    ]
    &
    \big(
      \mathbf{B}
      \eta
        ^{ \scalebox{.6}{$\shape$} }
        _{ \mathcal{T} }
    \big)^\ast
    \big(
      \mathbf{B}
      \eta
        ^{ \scalebox{.6}{$\shape$} }
        _{ \mathcal{T} }
    \big)_\ast
    \big(
      \mathrm{pt}_{
        \scalebox{.7}{$
          \mathbf{B}\mathcal{T}
        $}
      }
    \big)_\ast
    (\mathcal{X})
    \ar[
      r,
      shorten=-1pt,
      "{
        \epsilon_{
          \scalebox{.6}{$
            (\mathrm{pt}_{
              \mathbf{B}\mathcal{T}
            })_\ast
            \mathcal{X}
          $}
        }
      }"
    ]
    &
    \big(
      \mathrm{pt}_{
        \scalebox{.7}{$
          \mathbf{B}\mathcal{T}
        $}
      }
    \big)_\ast
    (\mathcal{X})
    \ar[d]
    \\
    \mathbf{B} \shape \mathcal{T}
    \ar[
      from=rr,
      "{
        \mathbf{B}
        \eta
          ^{ \scalebox{.5}{$\shape$} }
          _{ \mathcal{T} }
      }"
    ]
    &&
    \mathbf{B}\mathcal{T}
    \ar[
       rr,
       Rightarrow,
       -
    ]
    &{}&
    \mathbf{B}\mathcal{T}
  \end{tikzcd}
\end{equation}

\begin{theorem}[GRH's extended inertia groupoid models the cyclified orbifold]
  \label{AbstractCharacterizationOfHuanInertiaOrbifold}
  For $\mathcal{X} \simeq X \!\sslash\! G \;\in\; \SmoothInfinityGroupoids$
  any good orbifold \eqref{GoodOrbifold},
  GRH's inertia orbifold \eqref{SkeletonOfHuanInertiaStackOfGoodOrbifold}
  is equivalently the $\mathcal{T} \coloneqq \CohesiveCircle$-cyclic inertia $\infty$-stack
  in the general sense of \eqref{SInertiaInfinityStacks}:
  \vspace{-5mm}
  $$
    \begin{tikzcd}
    \overset{
      \mathclap{
      \raisebox{4pt}{
        \tiny
        \color{darkblue}
        \bf
        \begin{tabular}{c}
          GRH's inertia
          \\
          orbifold
        \end{tabular}
      }
      }
    }{
      \Lambda_{\CohesiveCircle} \mathcal{X}
    }
    \;\simeq\;
    \overset{
      \raisebox{3pt}{
        \tiny
        \color{darkblue}
        \bf
      }
    }{
      \Maps{\big}
        { \shape \CohesiveCircle }
        { \mathcal{X} }
      \!\sslash\! \CohesiveCircle
    }
    \;\;
    =:
    \;\;
    \underset{
      \mathclap{
      \raisebox{-4pt}{
        \tiny
        \color{greenii}
        \bf
        \begin{tabular}{c}
          restrict action
          \\
          from shape of circle
          \\
          to full smooth circle
        \end{tabular}
      }
      }
    }{
      \big(
        \eta_{\mathbf{B}\CohesiveCircle}^{\scalebox{.5}{$\shape$}}
      \big)^\ast
    }
    \overset{
      \mathclap{
      \raisebox{4pt}{
        \tiny
        \color{darkblue}
        \bf
        \begin{tabular}{c}
          cyclification wrt
          \\
          shape of circle
        \end{tabular}
      }
      }
    }{
      \mathrm{Cyc}_{\scalebox{.6}{$\shape \CohesiveCircle$}}
      (
        \mathcal{X}
      )\,.
    }
   \end{tikzcd}
  $$
\end{theorem}
\begin{proof}
  In view of \eqref{TransformationGroupsAsSliceObjects}, this
  follows from Prop. \ref{ComparisonMorphismIsPullbackOfShapeUnit},
  discussed in detail in \cref{ReproducingTheGRHIntertiaOrbifold} below.
\end{proof}
\begin{remark}[Subtleties]
While the idea that Thm. \ref{AbstractCharacterizationOfHuanInertiaOrbifold}
should be true possibly motivated the definition \eqref{SkeletonOfHuanInertiaStackOfGoodOrbifold},
its proof requires some care
(see the proof of Lemma \ref{ComparisonMorphismFromHuanInertiaToCyclification} below).
Notice that the approach via a shape modality on higher stacks
is crucial in bringing out this result:
Our diagram \eqref{CyclicicationAsRightDerivedBaseChange} shows at once
that the
traditional
way of identifying the inertia stack inside the smooth loop stack
as the $\CohesiveCircle$-fixed locus in a suitable groupoid presentation
(following \cite[Thm. 3.6.4]{LupercioUribe02LoopGroupoids})
is {\it not homotopy-meaningful}, as the {\it homotopy}-fixed locus
(e.g. \cite[Ex. 3.2.78]{SS21EPB})
of every $\mathcal{T}$-loop $\infty$-stack is just the original $\infty$-stack:
\begin{equation}
  \label{HomotopyFixedLocusOfLoopStack}
  \big(
    \Maps{}
      { \mathcal{T} }
      { \mathcal{X} }
  \big)^{\mathcal{T}}
  \;\simeq\;
  (p_{\mathbf{B}\mathcal{T}})_\ast
  (\mathrm{pt}_{\mathbf{B}\mathcal{T}})_\ast
  (
    \mathcal{X}
  )
  \;\simeq\;
  (\mathrm{id})_\ast
  (
    \mathcal{X}
  )
  \;\simeq\;
  \mathcal{X}
  \,,
\end{equation}
since, by \eqref{CyclicicationAsRightDerivedBaseChange}, it ends up computing the right base change of $\mathcal{X}$ along the identity:
$$
  \begin{tikzcd}
    \ast
    \ar[
      r,
      "{
        \mathrm{pt}_{\scalebox{.7}{$
          \mathbf{B}\mathcal{T}
        $}}
      }"
    ]
    \ar[
      rr,
      rounded corners,
      to path={
           -- ([yshift=-02pt]\tikztostart.south)
           -- ([yshift=-12pt]\tikztostart.south)
           -- node[yshift=+7pt]{
             \scalebox{.7}{$
               \mathrm{id}
             $}
           }
              ([yshift=-12pt]\tikztotarget.south)
           -- ([yshift=-00pt]\tikztotarget.south)
        }
    ]
    &
    \mathbf{B}\mathcal{T}
    \ar[
      r,
      "{
        p_{{}_{\mathbf{B}\mathcal{T}}}
      }"
    ]
    &
    \ast
  \end{tikzcd}
$$
\end{remark}

%%%%%%%%%%%%%%%%%%%%%%%%%%%%%%%%%%%%%%%%%%%%%%%%%%%%%
\subsection{Reproducing GRH's extended inertia orbifold}
\label{ReproducingTheGRHIntertiaOrbifold}
%%%%%%%%%%%%%%%%%%%%%%%%%%%%%%%%%%%%%%%%%%%%%%%%%%%%%

Here we prove Theorem \ref{AbstractCharacterizationOfHuanInertiaOrbifold},
that GRH's inertia orbifold construction is a model
for an abstractly defined cyclic inertia orbifold:
$\Lambda_{\CohesiveCircle} \mathcal{X} \,\simeq\, \Maps{\big}{\shape \CohesiveCircle}{\mathcal{X}}  \!\sslash\! \CohesiveCircle$.

\begin{lemma}[Comparison morphism from GRH's inertia orbifold to cyclic orbifold]
  \label{ComparisonMorphismFromHuanInertiaToCyclification}
  For $\mathcal{X} \,\simeq\, X \!\sslash\! G$
  a good orbifold,
  there is a comparison morphism,
  as shown in \eqref{ComparisonMorphismFaceMaps}
  and \eqref{ComparisonMorphismDegeneracyMaps},
  \vspace{-2mm}
  \begin{equation}
    \label{FormOfComparisonMorphism}
    \Lambda_{\CohesiveCircle}(\mathcal{X})
    \xrightarrow{\;\;\mathrm{comp}_{\mathcal{X}}\;\;}
    \mathrm{Cyc}_{\scalebox{.6}{$\shape \CohesiveCircle$}}(\mathcal{X})
    \,=\,
    \Maps{\big}
      { \shape \CohesiveCircle }
      { \mathcal{X} }
    \,\sslash\,
    \shape \CohesiveCircle
  \end{equation}

  \vspace{-2mm}
  \noindent
  from GRH's extended inertia orbifold \eqref{SkeletonOfHuanInertiaStackOfGoodOrbifold} to the
  $\shape \CohesiveCircle$-cyclification \eqref{SInertiaInfinityStacks} of the inertia orbifold.
\end{lemma}
\begin{proof}
  Shown in
  \eqref{ComparisonMorphismFaceMaps}
  and \eqref{ComparisonMorphismDegeneracyMaps}
  (on the next two pages)
  is a morphism of simplicial presheaves. We need to see that:

{\bf  (i)} this is well-defined as a morphism in $\SimplicialPresheaves(\CartesianSpaces)$,

{\bf  (ii)} under localization it presents a morphism in $\SmoothInfinityGroupoids$
    of the claimed form \eqref{FormOfComparisonMorphism}.

  \medskip
  \noindent
  Regarding {\bf (i)}:
  On both sides of
  \eqref{ComparisonMorphismFaceMaps} and \eqref{ComparisonMorphismDegeneracyMaps}
  we show the diagonal quotient $\big((-) \times W\mathcal{G}\big)/\mathcal{G}$
  \eqref{ColimitQuillenAdjunctionOnLocalSimplicialPresheaves}
  of the Cartesian product with a universal simplicial classifying space
  $W \mathcal{G}$
  Def. \ref{StandardModelOfUniversalSimplicialPrincipalComplex} by the diagonal
  action of a simplial group $\mathcal{G}$, all extended to simplicial
  presheaves over $\CartesianSpaces$.

  On the left of \eqref{ComparisonMorphismFaceMaps} and \eqref{ComparisonMorphismDegeneracyMaps},
  the simplicial group
  $\mathcal{G} = C_g \times \mathbb{R} \times \mathbb{Z}^{\times^\bullet}$
  is the cofibrant resolution
  \eqref{CofibrantResolutionOfHuanCentralizerGroup} of GRH's centralizer group
  with its induced simplicial action on the fixed loci:
  \vspace{-2mm}
  $$
    \begin{tikzcd}[row sep=-5pt]
      X^g \times
      \big(
        C_g \times \mathbb{R} \times \mathbb{Z}^{\times \bullet}
      \big)
      \ar[
        rr,
        "(-)\cdot(-)"
      ]
      &&
      X^g
      \\
  \scalebox{0.7}{$       \big(
        x,\, (h, r, \vec n)
      \big)
      $}
      &\longmapsto&
\scalebox{0.7}{$      x \cdot h $}
    \end{tikzcd}
  $$

  \vspace{-2mm}
\noindent
  On the right of \eqref{ComparisonMorphismFaceMaps} and \eqref{ComparisonMorphismDegeneracyMaps},
  the simplicial group is $\mathcal{G} = \mathbb{Z}^{\times^\bullet}$,
  and its action is on the skeleton of
  the inertia hom-complex via \eqref{FormulaForCanonicalBZActionOnInertia}:
  \vspace{-2mm}
  \begin{equation}
    \label{FormulaForCanonicalBZActionOnInertiaOfXModG}
    \begin{tikzcd}[row sep=-5pt]
    \big[
      \mathbf{B}\mathbb{Z} \times \Delta[\bullet],
      \,
      X \!\sslash\! G
    \big]_{\mathrm{skel}}
    \times
    \mathbb{Z}^{\times^\bullet}
    \ar[rr]
    &&
    \big[
      \mathbf{B}\mathbb{Z} \times \Delta[\bullet],
      \,
      X \!\sslash\! G
    \big]_{\mathrm{skel}}
    \\
    \scalebox{0.8}{$ \Big(
      \big( x, (\cdots, h_1, h_0) \big),
      \,
      ( \cdots, n_1, n_0 )
    \Big)$}
    &\longmapsto&
     \scalebox{0.8}{$ \big( x, (\cdots, g^{n_1} \cdot h_1, g^{n_0} \cdot h_0) \big)$}
    \end{tikzcd}
  \end{equation}

\vspace{-2mm}
\noindent
  The bulk of the diagrams
  \eqref{ComparisonMorphismFaceMaps} and \eqref{ComparisonMorphismDegeneracyMaps}
  shows the face maps -- using \eqref{FaceMapsOfWG} --
  and the degeneracy maps -- using \eqref{DegeneracyMapsOfWG} --
  of the resulting quotients on both sides,
  to check that the comparison morphism indeed respects these.
  For illustration of how these maps are obtained, we spell out
  the computation of the one on the bottom left of \eqref{ComparisonMorphismFaceMaps}:
  \vspace{-2mm}
  $$
  \hspace{-1mm}
    \begin{tikzcd}[row sep=-1pt, column sep=35pt]
      X^g
        \times
      (C_g \times \mathbb{R})
      \ar[
        r,
        "\sim"{yshift=-1pt}
      ]
      \ar[
        rrrr,
        rounded corners,
        to path={
             -- ([yshift=+24pt]\tikztostart.north)
             --node[above]{
                 \scalebox{.7}{$
                   d_0
                 $}
               } ([yshift=+24pt]\tikztotarget.north)
             -- (\tikztotarget.north)}
        ]
      &
      \frac
      {
        \overset{(X^g)_1}{
        \overbrace{
          X^g
        }
        }
          \times
        \overset{
          W(C_g \times \mathbb{R} \times \mathbb{Z}^{\times^\bullet})_{1}
        }{
        \overbrace{
        (C_g \times \mathbb{R} \times \mathbb{Z})
          \times
        (C_g \times \mathbb{R})
        }
        }
      }
      {
        \mathclap{\phantom{\vert^{\vert^{\vert}}}}
        \underset{
          (C_g \times \mathbb{R} \times \mathbb{Z}^{\times^\bullet})_1
        }{
        \underbrace{
          C_g \times \mathbb{R} \times \mathbb{Z}
        }
        }
      }
      \ar[
        rr,
        "\scalebox{0.7}{${
          \frac
            { d_0^{X^g} \times d_0^{W(\mathcal{G})} }
            {
              \mathclap{\phantom{\vert^{\vert^{\vert}}}}
              d_0^{\mathcal{G}}
            }
        }$}"
      ]
      &&
      \frac
      {
        \overset{(X^g)_0}{
        \overbrace{
          X^g
        }
        }
          \times
        \overset{
          W(C_g \times \mathbb{R} \times \mathbb{Z}^{\times^\bullet})_{0}
        }{
        \overbrace{
          C_g \times \mathbb{R}
        }
        }
      }
      {
        \mathclap{\phantom{\vert^{\vert^{\vert}}}}
        \underset{
          (C_g \times \mathbb{R} \times \mathbb{Z}^{\times^\bullet})_0
        }{
        \underbrace{
          C_g \times \mathbb{R}
        }
        }
      }
      \ar[r,"\sim"{yshift=-1pt}]
      &
      X^g
      \\
    \scalebox{0.8}{$    (x, (h,r))$}
      \ar[r, |->]
      &
    \scalebox{0.8}{$    \big[
        x, (e,0,0), (h,r)
      \big]$}
      \ar[rr, |->]
      &&
      \scalebox{0.8}{$  \big[
        x, (e,0) \cdot (h,r)
      \big]
      =
      \big[
        x \cdot h, (e,0)
      \big]$}
      \ar[r, |->]
      &
     \scalebox{0.8}{$   x \cdot h $}
      \mathrlap{\,.}
    \end{tikzcd}
  $$

\vspace{-2mm}
\noindent  Proceeding this way, one checks (see the next two pages) that all parallel squares in
  \eqref{ComparisonMorphismFaceMaps} and \eqref{ComparisonMorphismDegeneracyMaps}
  indeed commute.

\vspace{4cm}

\newpage

\vspace{-5mm}
\begin{equation}
  \label{ComparisonMorphismFaceMaps}
  \hspace{-5mm}
  \begin{tikzcd}[row sep=60pt, column sep=10pt]
    \mbox{
      \tiny
      \color{orangeii}
      \bf
      \begin{tabular}{c}
        smooth
        \\
        $\infty$-stacks
      \end{tabular}
    }
    &
    \overset{
      \mathclap{
      \raisebox{4pt}{
        \tiny
        \color{darkblue}
        \bf
        \begin{tabular}{c}
          GRH's inertia
        \end{tabular}
      }
      }
    }{
      \Lambda_{\CohesiveCircle}
      \big(
        X \!\sslash\! G
      \big)
    }
    \,\coloneqq\,
    \underset{[g]}{\coprod}
    \Big(
      X^g
        \sslash
      \big(
        (C_g \times \mathbb{R})/\langle (g^{-1},1) \rangle
      \big)
    \Big)
    \ar[
      rr,
      "{
        \mbox{
          \tiny
          \color{greenii}
          \bf
          \begin{tabular}{c}
            comparison
            morphism
          \end{tabular}
        }
      }"{above,yshift=10pt},
      "{
        \mathrm{comp}_{X \sslash G}
      }"{above},
      "{ \simeq_{\scalebox{.6}{\shape}} }"{below}
    ]
    &&
    \big[
      \shape \CohesiveCircle,
      \,
      X \!\sslash\! G
    \big]
    \!\sslash\! \shape \CohesiveCircle
    =:
    \overset{
      \mathclap{
      \raisebox{4pt}{
        \tiny
          \color{darkblue}
        \bf
        FSS-cyclification
      }
      }
    }{
      \mathrm{Cyc}_{\scalebox{.5}{$\shape \CohesiveCircle$}}\big( X \!\sslash\! G \big)
    }
    \\[-50pt]
    \mbox{
      \tiny
          \color{orangeii}
      \bf
      \begin{tabular}{c}
        simplicial
        \\
        sheaves
      \end{tabular}
    }
    &
    \underset{[g]}{\coprod}
    \Big(
    \big(
      X^g \times W( C_g \times \mathbb{R} \times \mathbb{Z}^{\times^\bullet} )
    \big)
    /
    (C_g \times \mathbb{R} \times \mathbb{Z}^{\times^\bullet})
    \Big)
    \ar[
      rr
    ]
    &&
    \Big(
    \big[
      \mathbb{Z}^{\times^\bullet},
      \,
      X \times G^{\times^\bullet}
    \big]_{\mathrm{skel}}
    \times
    W( \mathbb{Z}^{\times^\bullet} )
    \Big)
    \big/
    \mathbb{Z}^{\times^\bullet}
    \\[-25pt]
    \mbox{
      \tiny
      \color{orangeii}
      \bf
      deg 4
    }
    \ar[
      ddd,
      phantom,
      "{
        \rotatebox{90}{
          \tiny
          \color{greenii}
          \bf
          face maps
        }
      }"
    ]
    &
    \underset{[g]}{\coprod}
    \left(
    X^g
    \begin{array}{l}
    \times
    (C_g \times \mathbb{R} \times \mathbb{Z}^3)
    \times (C_g \times \mathbb{R} \times \mathbb{Z}^2)
    \\
    \times
    (C_g \times \mathbb{R} \times \mathbb{Z})
    \times
    (C_g \times \mathbb{R})
    \end{array}
    \!\!\!
    \right)
    \ar[
      ddd,
      shift right=110pt,
      "\scalebox{0.75}{$
        \left(
          \!\!\!\!
          \def\arraystretch{.7}
          \begin{array}{l}
            x,
            \\
            (h_3, r_3, n_{3,2}, n_{3,1}, n_{3,0}),
            \\
            (h_2, r_2, n_{2,1}, n_{2,0}),
            \\
            (h_1, r_1, n_{1}),
            \\
            (h_0, r_0)
          \end{array}
          \!\!\!\!\!\!
        \right)
      $}"{description, pos=.15},
      "\mapsto"{description, sloped, pos=.5},
      "\scalebox{0.75}{$
        \left(
          \!\!\!\!\!
          \def\arraystretch{.7}
          \begin{array}{l}
            x \cdot h_3,
            \\
            (h_2, r_2, n_{2,1}, n_{2,0}),
            \\
            (h_1, r_1, n_{1}),
            (h_0, r_0)
          \end{array}
          \!\!\!\!\!\!
        \right)
      $}"{description, pos=.87}
    ]
    \ar[
      ddd,
      shift right=64pt,
      "\mapsto"{description, sloped, pos=.3},
      "\scalebox{0.75}{$
        \left(
          \def\arraystretch{.7}
          \begin{array}{l}
            x,
            \\
            \left(
            \!\!\!\!
            \def\arraystretch{.7}
            \begin{array}{l}
            g^{\scalebox{.57}{$- n_{3,2}$}} \cdot h_{3} \cdot h_2,
            \\
            r_3 + n_{3,2} + r_2,
            \\
            n_{3,1} + n_{2,1},
            \\
            n_{3,0} + n_{2,0}
            \end{array}
            \!\!\!\!\!\!
            \right),
            \\
            (h_2, r_3, n_{2,1}, n_{2,0}),
            \\
            (h_1, r_1, n_1),
            \\
            (h_0, r_0)
          \end{array}
          \!\!\!\!\!\!
        \right)
      $}"{description, pos=.55}
    ]
    \ar[
      ddd,
      shift right=17pt,
      "\scalebox{0.75}{$
        \left(
          \!\!\!\!\!
          \def\arraystretch{.7}
          \begin{array}{l}
            x,
            \\
            (h_3, r_3, n_{3,2} + n_{3,1}, n_{3,0}),
            \\
            \left(
              \!\!\!\!
              \def\arraystretch{.7}
              \begin{array}{l}
                g^{\scalebox{.57}{$-n_{2,1}$}} \cdot h_2 \cdot h_1,
                \\
                r_2 + n_{2,1} + r_{1},
                \\
                n_{2,0} + n_{1}
              \end{array}
              \!\!\!\!\!\!
            \right),
            \\
            (h_0, r_0)
          \end{array}
          \!\!\!\!\!\!
        \right)
      $}"{description, pos=.18},
      "\mapsto"{description, sloped, pos=.04}
    ]
    \ar[
      ddd,
      shift right=-31pt,
      "\scalebox{0.75}{$
        \left(
          \!\!\!\!\!
          \def\arraystretch{.7}
          \begin{array}{l}
            x,
            \\
            (h_3,   r_3, n_{3,2}, n_{3,1} + n_{3,0}),
            \\
            (h_2, r_2, n_{2,1} + n_{2,0}),
            \\
            g^{\scalebox{.57}{$-n_1$}} \cdot h_1 \cdot h_0, r_1 + n_1 + r_0
          \end{array}
          \!\!\!\!\!\!
        \right)
      $}"{description, pos=.45},
      "\mapsto"{description, sloped, pos=.2}
    ]
    \ar[
      ddd,
      shift right=-76pt,
      "\scalebox{0.75}{$
        \left(
          \!\!\!\!\!
          \def\arraystretch{.7}
          \begin{array}{l}
            x,
            (h_3, r_3, n_{3,2}, n_{3,1}),
            \\
            \phantom{x,}
            (h_2, r_2, n_{2,1})
            \\
            \phantom{x,}
            (h_1, r_1)
          \end{array}
          \!\!\!\!\!\!
        \right)
      $}"{description, pos=.68},
      "\mapsto"{description, sloped, pos=.3}
    ]
    \ar[
      rr,
      "\scalebox{0.75}{$
        \;\;\;\;\;\;\;\;\;\;\;
        \def\arraystretch{.7}
        \begin{array}{l}
          x_{\phantom{1}} \,\mapsto\, x
          \\
          h_3 \,\mapsto\, h_3
          \\
          h_2 \,\mapsto\, h_2 \cdot g^{\scalebox{.57}{$-n_{{}_{3,2}}$}},
          \\
          h_1 \,\mapsto\, h_1 \cdot g^{\scalebox{.57}{$- n_{3,1} - n_{2,1} $}}
          \\
          h_0 \,\mapsto\, h_1 \cdot g^{\scalebox{.57}{$- n_{3,0} - n_{2,0} - n_{1} $}}
          \\
          n_{(-)} \,\mapsto\, n_{(-)}
        \end{array}
      $}"{yshift=1pt}
    ]
    &&
    \underset{[g]}{\coprod}
    \big(
      X^g \times C_g \times C_g \times C_g \times C_g \times
      \mathbb{Z}^3 \times \mathbb{Z}^2 \times \mathbb{Z}
    \big)
    \ar[
      ddd,
      shift right=70pt,
      "\scalebox{0.75}{$
        \left(
        \!\!\!\!
        \def\arraystretch{.7}
        \begin{array}{l}
          (x, h_3, h_2, h_1, h_0),
          \\
          (n_{3,2}, n_{3,1}, n_{3,0}),
          \\
          (n_{2,1}, n_{2,0}),
          \\
          n_1
        \end{array}
        \!\!\!\!\!\!
        \right)
      $}"{description, pos=.2},
      "\mapsto"{description, sloped, pos=.5},
      "\scalebox{0.75}{$
        \left(
        \!\!\!\!
        \def\arraystretch{.7}
        \begin{array}{l}
          \left(
            \!\!\!\!\!
            \begin{array}{l}
            x \cdot h_3,
            g^{\scalebox{.57}{$n_{3,2}$}} \cdot h_2,
            \\
            g^{\scalebox{.57}{$n_{3,1}$}} \cdot h_1,
            g^{\scalebox{.57}{$n_{3,0}$}} \cdot h_0
            \end{array}
            \!\!\!\!\!\!
          \right),
          \\
          (n_{2,1}, n_{2,0}), (n_1)
        \end{array}
        \!\!\!\!\!\!
        \right)
      $}"{description, xshift=10pt, pos=.7}
    ]
    \ar[
      ddd,
      shift right=22pt,
      "\scalebox{0.75}{$
        \left(
        \!\!\!\!
        \def\arraystretch{.7}
        \begin{array}{l}
          x, h_3 \cdot h_2, h_1, h_0,
          \\
          (n_{3,1} + n_{2,1} , n_{3,0} + n_{2,0}),
          \\
          n_1
        \end{array}
        \!\!\!\!\!\!\!
        \right)
      $}"{description, pos=.46},
      "\mapsto"{description,sloped, pos=.25}
    ]
    \ar[
      ddd,
      shift right=-26pt,
      "\scalebox{0.75}{$
        \left(
        \!\!\!\!
        \def\arraystretch{.7}
        \begin{array}{l}
          x,
          \\
          h_3, h_2 \cdot h_1, h_0,
          \\
          (n_{3,2} + n_{3,1}, n_{3,0}),
          \\
          ( n_{2,0} + n_1 )
        \end{array}
        \!\!\!\!\!\!\!
        \right)
      $}"{description, pos=.2},
      "\mapsto"{description,sloped, pos=.08}
    ]
    \ar[
      ddd,
      shift right=-64pt,
      "\scalebox{0.75}{$
        \left(
        \!\!\!\!
        \def\arraystretch{.7}
        \begin{array}{l}
          x,
          \\
          h_3, h_2, h_1 \cdot h_0,
          \\
          ( n_{3,2},\, n_{3,1} + n_{3,0} ),
          \\
          ( n_{2,1} )
        \end{array}
        \!\!\!\!\!\!\!
        \right)
      $}"{description, pos=.5},
      "\mapsto"{description,sloped, pos=.3}
    ]
    \ar[
      ddd,
      shift right=-104pt,
      "\scalebox{0.75}{$
        \left(
        \!\!\!\!
        \def\arraystretch{.7}
        \begin{array}{l}
          x,
          \\
          h_3, h_2, h_1,
          \\
          ( n_{3,2}, n_{3,1}),
          \\
          ( n_{2,1} )
        \end{array}
        \!\!\!\!\!\!\!
        \right)
      $}"{description, pos=.75},
      "\mapsto"{description,sloped, pos=.5}
    ]
    \\
    \\
    \\
    \mbox{
      \tiny
      \color{orangeii}
      \bf
      deg 3
    }
    \ar[
      dd,
      phantom,
      "{
        \rotatebox{90}{
          \tiny
          \color{greenii}
          \bf
          face maps
        }
      }"
    ]
    &
    \underset{[g]}{\coprod}
    \left(
    X^g
    \begin{array}{l}
    \times (C_g \times \mathbb{R} \times \mathbb{Z}^2)
    \\
    \times
    (C_g \times \mathbb{R} \times \mathbb{Z})
    \times
    (C_g \times \mathbb{R})
    \end{array}
    \right)
    \ar[
      dd,
      shift right=76pt,
      "\scalebox{0.75}{$
          \left(
            \!\!\!\!
            \def\arraystretch{.7}
            \begin{array}{l}
            x,
            \\
            (h_2, r_2, n_{2,1}, n_{2,0}),
            \\
            (h_1,r_1,n_1),
            \\
            (h_0,r_0)
            \end{array}
            \!\!\!\!\!\!
          \right)
      $}"{description, pos=.2},
      "\mapsdown"{description, pos=.46},
      "\scalebox{0.75}{$
        \left(
          \!\!\!\!
          \def\arraystretch{.7}
          \begin{array}{l}
            x \cdot h_2,
            \\
            (h_1,r_1,n_1),
            \\
            (h_0,r_0)
          \end{array}
          \!\!\!\!\!\!
        \right)
      $}"{description, pos=.7}
    ]
    \ar[
      dd,
      shift right=13    pt,
      "\mapsdown"{description, pos=.4},
      "\scalebox{0.75}{$
        \left(
          \!\!\!\!
          \def\arraystretch{.7}
          \begin{array}{l}
            x,
            \\
            \left(
              \!\!\!\!
              \begin{array}{l}
                g^{\scalebox{.57}{$-n_{2,1}$}} \cdot h_2 \cdot h_1,
                \\
                r_2 + n_{2,1} + r_1 ,
                \\
                n_{2,0} + n_1
              \end{array}
              \!\!\!\!\!\!
            \right),
            \\
            (h_0, r_0)
          \end{array}
          \!\!\!\!\!\!\!\!
      \right)
      $}"{description, pos=.7}
    ]
    \ar[
      dd,
      shift right=-30pt,
      "\mapsdown"{description, pos=.07},
      "\scalebox{0.75}{$
        \left(
          \!\!\!\!
          \def\arraystretch{.7}
          \begin{array}{l}
            x,
            \\
            (h_2, r_2, n_{2,1}+ n_{2,0}),
            \\
            \left(
              \!\!\!\!
              \def\arraystretch{.7}
              \begin{array}{l}
                g^{\scalebox{.57}{$-n_1$}} \cdot h_1 \cdot h_0,
                \\
                r_1 + n_1 + r_0
              \end{array}
              \!\!\!\!\!\!
            \right)
          \end{array}
          \!\!\!\!\!\!
        \right)
      $}"{description, pos=.3}
    ]
    \ar[
      dd,
      shift right=-70pt,
      "\mapsdown"{description, pos=.4},
      "\scalebox{0.75}{$
        \left(
          \!\!\!\!
          \begin{array}{l}
            x,
            \\
            (h_2, r_2, n_{2,1}),
            \\
            (h_1, r_1)
          \end{array}
          \!\!\!\!\!\!
        \right)
      $}"{description, pos=.7}
    ]
    \ar[
      rr,
      "\scalebox{0.75}{$
        \def\arraystretch{.7}
        \begin{array}{l}
          x_{\phantom{1}} \mapsto x
          \\
          h_2 \mapsto h_2
          \\
          h_1 \mapsto h_1 \cdot g^{\scalebox{.57}{$-n_{2,1}$}}
          \\
          h_0 \mapsto h_0 \cdot g^{\scalebox{.57}{$-n_{2,0} - n_1$}}
          \\
          n_{(-)} \mapsto n_{(-)}
        \end{array}
      $}"{yshift=0pt}
    ]
    &&
    \underset{[g]}{\coprod}
    \big(
      X^g \times C_g \times C_g \times C_g \times \mathbb{Z}^2 \times \mathbb{Z}
    \big)
    \ar[
      dd,
      shift right=60pt,
      "\scalebox{0.75}{$
        \left(
          \!\!\!\!
          \def\arraystretch{.7}
          \begin{array}{l}
            (x, h_2, h_1, h_0),
            \\
            (n_{2,1}, n_{2,1}),
            \\
            (n_1)
          \end{array}
          \!\!\!\!\!\!
        \right)
      $}"{description, pos=.2},
      "\mapsto"{description, sloped, pos=.45},
      "\scalebox{0.75}{$
        \left(
        \!\!\!\!\
        \def\arraystretch{.7}
        \begin{array}{l}
          x \cdot h_2,
          \\
          g^{\scalebox{.57}{$n_{2,1}$}}
            \cdot
          h_1,
          g^{\scalebox{.57}{$n_{2,0}$}} \cdot h_0,
          \\
          n_1
        \end{array}
        \!\!\!\!\!\!
        \right)
      $}"{description, pos=.65}
    ]
    \ar[
      dd,
      shift right=5pt,
      "\scalebox{0.75}{$
        \left(
        \!\!\!\!
        \def\arraystretch{.7}
        \begin{array}{l}
          x,
          \\
          h_2 \cdot h_1, h_0,
          \\
          n_{2,0} + n_{1}
        \end{array}
        \!\!\!\!\!\!
        \right)
      $}"{description, pos=.65}
    ]
    \ar[
      dd,
      shift right=-22pt,
      "\mapsto"{sloped, description, pos=.08},
      "\scalebox{0.75}{$
        \left(
        \!\!\!\!
        \def\arraystretch{.7}
        \begin{array}{l}
          x,
          \\
          h_2, h_1 \cdot h_0,
          \\
          n_{2,1} + n_{2,0}
        \end{array}
        \!\!\!\!\!\!
        \right)
      $}"{description, pos=.3}
    ]
    \ar[
      dd,
      shift right=-50pt,
      "\mapsto"{sloped, description, pos=.4},
      "\scalebox{0.75}{$
        \left(
        \!\!\!\!
        \def\arraystretch{.7}
        \begin{array}{l}
          x,
          \\
          h_2,\, h_1,
          \\
          n_{2,1}
        \end{array}
        \!\!\!\!\!\!
        \right)
      $}"{description, pos=.65}
    ]
    \\
    \\
    \mbox{
      \tiny
      \color{orangeii}
      \bf
      deg 2
    }
    \ar[
      d,
      phantom,
      "{
        \rotatebox{90}{
          \tiny
          \color{greenii}
          \bf
          face maps
        }
      }"
    ]
    &
    \underset{[g]}{\coprod}
    \Big(
    X^g
    \times
    (C_g \times \mathbb{R} \times \mathbb{Z})
    \times
    (C_g \times \mathbb{R})
    \Big)
    \ar[
      rr,
      "\scalebox{0.75}{$
        \def\arraystretch{.7}
        \begin{array}{l}
          x_{\phantom{1}} \,\mapsto\, x
          \\
          h_1 \,\mapsto\, h_1
          \\
          h_0 \,\mapsto\, h_0 \cdot g^{-n}
          \\
          n_{\phantom{1}} \,\mapsto\, n
        \end{array}
      $}"{description}
    ]
    \ar[
      d,
      shift right=52pt+10pt,
      "\scalebox{0.75}{$
        \def\arraystretch{.7}
        \begin{array}{c}
          %\scalebox{1.2}{$($}
          %  x,\, (h_1,r_1,n),\, (h_0,r_0)
          %\scalebox{1.2}{$)$}
          {\phantom{A}}
          \\
          \mapsdown
          \\
          \scalebox{1.2}{$($}
            x \cdot h_1,\, (h_0,r_0)
          \scalebox{1.2}{$)$}
        \end{array}
      $}"{left, xshift=4pt}
    ]
    \ar[
      d,
      shift right=0pt+10pt,
      "\scalebox{0.75}{$
        \def\arraystretch{.7}
        \begin{array}{c}
          \scalebox{1.2}{$($}
            x,\, (h_1,r_1, n),\, (h_0,n_0)
          \scalebox{1.2}{$)$}
          \\
          \mapsdown
          \\
          \scalebox{1.2}{$($}
            x,\, (g^{-n} \cdot h_1 \cdot h_0,\, r_1 + r_0 + n)
          \scalebox{1.2}{$)$}
        \end{array}
      $}"{description}
    ]
    \ar[
      d,
      shift right=-52pt+10pt,
      "\scalebox{0.75}{$
        \def\arraystretch{.7}
        \begin{array}{c}
          %\scalebox{1.2}{$($}
          %  x,\, (h_1,r_1 n),\, (h_0, n_0)
          %\scalebox{1.2}{$)$}
          {\phantom{A}}
          \\
          \mapsdown
          \\
          \scalebox{1.2}{$($}
            x,\, (h_1,\, r_1)
          \scalebox{1.2}{$)$}
        \end{array}
      $}"{xshift=-4pt}
    ]
    &&
    \underset{[g]}{\coprod}
    \big(
      X^g \times C_g \times C_g \times \mathbb{Z}
    \big)
    \ar[
      d,
      shift right=30pt,
      "\scalebox{0.75}{$
        \def\arraystretch{.7}
        \begin{array}{c}
          (x, h_1, h_0, n)
          \\
          \mapsdown
          \\
          (x \cdot h_1 , g^{n} \cdot h_0)
        \end{array}
      $}"{left}
    ]
    \ar[
      d,
      "\scalebox{0.75}{$
        \def\arraystretch{.7}
        \begin{array}{c}
          (x, h_1, h_0, n)
          \\
          \mapsdown
          \\
          (x, h_1 \cdot h_0)
        \end{array}
      $}"{description}
    ]
    \ar[
      d,
      shift left=30pt,
      "\scalebox{0.75}{$
        \def\arraystretch{.7}
        \begin{array}{c}
          (x, h_1, h_0, n)
          \\
          \mapsdown
          \\
          (x, h_1)
        \end{array}
      $}"
    ]
    \\[-6pt]
    \mbox{
      \tiny
      \color{orangeii}
      \bf
      deg 1
    }
    \ar[
      d,
      phantom,
      "{
        \rotatebox{90}{
          \tiny
          \color{greenii}
          \bf
          face maps
        }
      }"
    ]
    &
    \underset{[g]}{\coprod}
    \big(
    X^g
    \times
    C_g
    \times
    \mathbb{R}
    \big)
    \ar[
      rr,
      "\scalebox{0.75}{$
        \def\arraystretch{.7}
        \begin{array}{l}
          x_{\phantom{1}} \mapsto x
          \\
          h_{\phantom{1}} \mapsto h
          \\
          n_{\phantom{1}} \mapsto n
        \end{array}
      $}"{description}
    ]
    \ar[
      d,
      shift right=5pt,
      "\scalebox{0.75}{$
        \def\arraystretch{.7}
        \begin{array}{c}
          (x,(h,r))
          \\
          \mapsdown
          \\
          x
        \end{array}
      $}"{right,xshift=+4pt}
    ]
    \ar[
      d,
      shift left=5pt,
      "\scalebox{0.75}{$
        \def\arraystretch{.7}
        \begin{array}{c}
          (x,(h,r))
          \\
          \mapsdown
          \\
          x \cdot h
        \end{array}
      $}"{left, xshift=-4pt}
    ]
    &&
    \underset{[g]}{\coprod}
    \big(
      X^g
      \times
      C_g
    \big)
    \ar[
      d,
      shift right=5pt,
      "\scalebox{0.75}{$
        \def\arraystretch{.7}
        \begin{array}{c}
          (x,h)
          \\
          \mapsdown
          \\
          x
        \end{array}
      $}"{right,xshift=+4pt}
    ]
    \ar[
      d,
      shift left=5pt,
      "\scalebox{0.75}{$
        \def\arraystretch{.7}
        \begin{array}{c}
          (x,h)
          \\
          \mapsdown
          \\
          x \cdot h
        \end{array}
      $}"{left, xshift=-4pt}
    ]
    \\[-18pt]
    \mbox{
      \tiny
      \color{orangeii}
      \bf
      deg 0
    }
    &
    \underset{[g]}{\coprod}
    \big(
      X^g
    \big)
    \ar[
      rr,
      "{
        x \,\mapsto\, x
      }"
    ]
    &&
    \underset{[g]}{\coprod}
    \big(
      X^g
    \big)
  \end{tikzcd}
\end{equation}

\begin{equation}
  \label{ComparisonMorphismDegeneracyMaps}
  \hspace{-8mm}
  \begin{tikzcd}[column sep=30pt]
    \mbox{
      \tiny
      \color{orangeii}
      \bf
      \begin{tabular}{c}
        smooth
        \\
        $\infty$-stacks
      \end{tabular}
    }
    &[-24pt]
    \overset{
      \mathclap{
      \raisebox{4pt}{
        \tiny
        \color{darkblue}
        \bf
        \begin{tabular}{c}
          GRH's inertia
        \end{tabular}
      }
      }
    }{
      \Lambda_{\CohesiveCircle}
      \big(
        X \!\sslash\! G
      \big)
    }
    \,\coloneqq\,
    \underset{[g]}{\coprod}
    \Big(
      X^g
        \sslash
      \big(
        (C_g \times \mathbb{R})/\langle (g^{-1},1) \rangle
      \big)
    \Big)
    \ar[
      rr,
      "{
        \mbox{
          \tiny
          \color{greenii}
          \bf
          \begin{tabular}{c}
            comparison
            morphism
          \end{tabular}
        }
      }"{above,yshift=10pt},
      "{
        \mathrm{comp}_{X \sslash G}
      }"{above},
      "{ \simeq_{\scalebox{.6}{\shape}} }"{below}
    ]
    &&
    \big[
      \shape \CohesiveCircle,
      \,
      X \!\sslash\! G
    \big]
    \!\sslash\! \shape \CohesiveCircle
    =:
    \overset{
      \mathclap{
      \raisebox{4pt}{
        \tiny
          \color{darkblue}
        \bf
        FSS-cyclification
      }
      }
    }{
      \mathrm{Cyc}_{\scalebox{.5}{$\shape \CohesiveCircle$}}\big( X \!\sslash\! G \big)
    }
    \\[-15pt]
    \mbox{
      \tiny
          \color{orangeii}
      \bf
      \begin{tabular}{c}
        simplicial
        \\
        sheaves
      \end{tabular}
    }
    &
    \underset{[g]}{\coprod}
    \Big(
    \big(
      X^g \times W( C_g \times \mathbb{R} \times \mathbb{Z}^{\times^\bullet} )
    \big)
    /
    (C_g \times \mathbb{R} \times \mathbb{Z}^{\times^\bullet})
    \Big)
    \ar[
      rr
    ]
    &&
    \Big(
    \big[
      \mathbb{Z}^{\times^\bullet},
      \,
      X \times G^{\times^\bullet}
    \big]_{\mathrm{skel}}
    \times
    W( \mathbb{Z}^{\times^\bullet} )
    \Big)
    \big/
    \mathbb{Z}^{\times^\bullet}
    \\[+10pt]
    \mbox{
      \tiny
      \color{orangeii}
      \bf
      deg 4
    }
    \ar[
      dd,
      phantom,
      "\rotatebox{90}{
        \tiny
        \color{greenii}
        \bf
        degeneracy maps
      }"
    ]
    &
    \underset{[g]}{\coprod}
    \left(
      X^g
      \begin{array}{l}
        \times
      (C_g \times \mathbb{R} \times \mathbb{Z}^3)
        \times
      (C_g \times \mathbb{R} \times \mathbb{Z}^2)
      \\
        \times
      (C_g \times \mathbb{R} \times \mathbb{Z})
        \times
      (C_g \times \mathbb{R})
      \end{array}
      \!\!\!\!
    \right)
    \ar[
      rr,
      "\scalebox{0.75}{$
        \;\;\;\;\;\;
        \def\arraystretch{.7}
        \begin{array}{l}
          x \,\mapsto\, x
          \\
          h_3 \,\mapsto\, h_3
          \\
          h_2 \,\mapsto\, h_2 \cdot g^{- n_{3,2}}
          \\
          h_1 \,\mapsto\, h_1 \cdot g^{- n_{3,1} - n_{2,1}}
          \\
          h_0 \,\mapsto\, h_0 \cdot g^{-n_{3,0} - n_{2,0} - n_1}
          \\
          n_{(-)} \,\mapsto\, n_{(-)}
        \end{array}
      $}"
    ]
    &&
    \underset{[g]}{\coprod}
    (X^g \times C_g \times C_g \times C_g \times C_g \times \mathbb{Z}^3  \times \mathbb{Z}^2 \times \mathbb{Z})
    \\[110pt]
    \\
    \mbox{
      \tiny
      \color{orangeii}
      \bf
      deg 3
    }
    \ar[
      dd,
      phantom,
      "\rotatebox{90}{
        \tiny
        \color{greenii}
        \bf
        degeneracy maps
      }"
    ]
    &
    \underset{[g]}{\coprod}
    \left(
      \!\!
      X^g
      \begin{array}{l}
        \times
      (C_g \times \mathbb{R} \times \mathbb{Z}^2)
      \\
        \times
      (C_g \times \mathbb{R} \times \mathbb{Z})
        \times
      (C_g \times \mathbb{R})
      \end{array}
      \!\!\!\!
    \right)
    \ar[
      uu,
      shift left=80pt,
      "\scalebox{0.75}{$
        \left(
          \!\!\!\!
          \def\arraystretch{.7}
          \begin{array}{l}
            x,
            \\
            (h_2, r_2, n_{2,1}, n_{2,0})
            \\
            (h_1,r_1,n_1),
            \\
            (h_0, r_0)
          \end{array}
          \!\!\!\!\!\!
        \right)
      $}"{description, pos=.16},
      "\mapsup"{description, pos=.44},
      "\scalebox{0.75}{$
        \left(
        \!\!\!\!
        \begin{array}{l}
          x,
          (e,0,0,0,0),
          \\
          (h_2,r_2, n_{2,1}, n_{2,0}),
          \\
          (h_1,r_1, n_1)
          \\
          (h_0, r_0)
        \end{array}
        \!\!\!\!\!\!
        \right)
      $}"{description, pos=.77}
    ]
    \ar[
      uu,
      shift left=30pt,
      "\mapsup"{description, pos=.2},
      "\scalebox{0.75}{$
        \left(
        \!\!\!\!
        \begin{array}{l}
          x,
          (h_2,r_2,0,n_{2,1}, n_{2,0}),
          \\
          (e,0,0,0),
          (h_1, r_1, n_1)
          \\
          (h_0,r_0)
        \end{array}
        \!\!\!\!\!\!
        \right)
      $}"{description, pos=.44}
    ]
    \ar[
      uu,
      shift right=16pt,
      "\mapsup"{description, pos=.4},
      "\scalebox{0.75}{$
        \left(
        \!\!\!\!
        \begin{array}{l}
          x,
          (h_2,r_2,n_{2,1}, 0, n_{2,0}),
          \\
          (h_1, r_1, 0, n_1),
          (e,0,0),
          \\
          (h_0, r_0)
        \end{array}
        \!\!\!\!\!\!
        \right)
      $}"{description, pos=.75}
    ]
    \ar[
      uu,
      shift right=64pt,
      "\mapsup"{description, pos=.2},
      "\scalebox{0.75}{$
        \left(
        \!\!\!\!
        \begin{array}{l}
          x,
          (h_2,r_2,n_{2,1}, n_{2,0}, 0),
          \\
          (h_1, r_1, n_1, 0),
          \\
          (h_0, r_0),
          (e,0,0)
        \end{array}
        \!\!\!\!\!\!
        \right)
      $}"{description, pos=.45}
    ]
    \ar[
      rr,
      "\scalebox{0.75}{$
        \begin{array}{l}
          x \,\mapsto\, x
          \\
          h_2 \,\mapsto\, h_2
          \\
          h_1 \,\mapsto\, h_1 \cdot g^{- n_{2,1}}
          \\
          h_0 \,\mapsto\, h_0 \cdot g^{- n_{2,0} - n_1}
          \\
          n_{(-)} \,\mapsto\, n_{(-)}
        \end{array}
      $}"
    ]
    &&
    \underset{[g]}{\coprod}
    (X^g \times C_g \times C_g \times C_g \times \mathbb{Z}^2 \times \mathbb{Z})
    \ar[
      uu,
      shift left=70pt,
      "\scalebox{0.75}{$
        \left(
          \!\!\!\!
          \def\arraystretch{.7}
          \begin{array}{l}
            (x, h_2, h_1, h_0),
            \\
            (n_{2,1}, n_{2,0}), n_1
          \end{array}
          \!\!\!\!\!\!
        \right)
      $}"{description, pos=.22},
      "\mapsup"{description, pos=.44},
      "\scalebox{0.75}{$
        \left(
        \!\!\!\!
        \begin{array}{l}
          (x,e,h_2, h_1, h_0),
          \\
          (0,0,0), (n_{2,1}, n_{2,0})
          \\
          n_1
        \end{array}
        \!\!\!\!\!\!
        \right)
      $}"{description, pos=.77}
    ]
    \ar[
      uu,
      shift left=20pt,
      "\mapsto"{description, sloped, pos=.2},
      "\scalebox{0.75}{$
        \left(
        \!\!\!\!
        \begin{array}{l}
          (x, h_2, e, h_1, h_0),
          \\
          (0, n_{2,1}, n_{2,0}),
          \\
          (0,0), n_1
        \end{array}
        \!\!\!\!\!\!
        \right)
      $}"{description, pos=.45}
    ]
    \ar[
      uu,
      shift right=28pt,
      "\mapsto"{description, sloped, pos=.4},
      "\scalebox{0.75}{$
        \left(
        \!\!\!\!
        \begin{array}{l}
          (x, h_2, h_1, e,  h_0),
          \\
          (n_{2,1}, 0, n_{2,0}),
          \\
          (0, n_1), 0
        \end{array}
        \!\!\!\!\!\!
        \right)
      $}"{description, pos=.76}
    ]
    \ar[
      uu,
      shift right=70pt,
      "\mapsto"{description, sloped, pos=.2},
      "\scalebox{0.75}{$
        \left(
        \!\!\!\!
        \begin{array}{l}
          (x, h_2, h_1, h_0, 0),
          \\
          (n_{2,1}, n_{2,0}, 0),
          \\
          (n_1, 0), 0
        \end{array}
        \!\!\!\!\!\!
        \right)
      $}"{description, pos=.46}
    ]
    \\[80pt]
    \\
    \mbox{
      \tiny
      \color{orangeii}
      \bf
      deg 2
    }
    \ar[
      dd,
      phantom,
      "\rotatebox{90}{
        \tiny
        \color{greenii}
        \bf
        degeneracy maps
      }"
    ]
    &
    \underset{[g]}{\coprod}
    \big(
      X^g
        \times
      (C_g \times \mathbb{R} \times \mathbb{Z})
        \times
      (C_g \times \mathbb{R})
    \big)
    \ar[
      uu,
      shift left=57pt,
      "\scalebox{0.75}{$
        \left(
          \!\!\!\!
          \def\arraystretch{.7}
          \begin{array}{l}
            x,
            \\
            (h_1,r_1,n),
            \\
            (h_0, r_0)
          \end{array}
          \!\!\!\!\!\!
        \right)
      $}"{description, pos=.2},
      "\mapsup"{description, pos=.44},
      "\scalebox{0.75}{$
        \left(
        \!\!\!\!
        \begin{array}{l}
          x,
          (e,0,0,0),
          \\
          (h_1,r_1, n),
          \\
          (h_0, r_0)
        \end{array}
        \!\!\!\!\!\!
        \right)
      $}"{description, pos=.75}
    ]
    \ar[
      uu,
      "\scalebox{0.75}{$
        \left(
        \!\!\!\!
        \def\arraystrethc{.7}
        \begin{array}{l}
          x,
          (h_1,r_1, 0,n),
          \\
          (e,0,0),
          \\
          (h_0, r_0)
        \end{array}
        \!\!\!\!\!\!
        \right)
      $}"{description, pos=.75},
      "\mapsup"{description, pos=.4},
    ]
    \ar[
      uu,
      shift right=60pt,
      "\scalebox{0.75}{$
        \left(
        \!\!\!\!
        \def\arraystrethc{.7}
        \begin{array}{l}
          x,
          (h_1,r_1, n, 0),
          \\
          (h_0, r_0, 0),
          \\
          (e,0)
        \end{array}
        \!\!\!\!\!\!
        \right)
      $}"{description, pos=.75},
      "\mapsup"{description, pos=.4},
    ]
    \ar[
      rr,
      "\scalebox{0.75}{$
        \def\arraystretch{.7}
        \begin{array}{l}
          x \,\mapsto\, x
          \\
          h_1 \,\mapsto\, h_1
          \\
          h_0 \,\mapsto\, h_0 \cdot g^{-n}
          \\
          n \,\mapsto\, n
        \end{array}
      $}"{description}
    ]
    &&
    \underset{[g]}{\coprod}
    (X^g \times C_g \times C_g \times \mathbb{Z})
    \ar[
      uu,
      shift left=55pt,
      "\scalebox{0.75}{$
        (x, h_1, h_0, n)
      $}"{description, pos=.2},
      "\mapsup"{description, pos=.44},
      "\scalebox{0.75}{$
        \left(
        \!\!\!\!
        \begin{array}{l}
          (x, e, h_1, h_0),
          \\
          (0,0), n
        \end{array}
        \!\!\!\!\!\!
        \right)
      $}"{description, pos=.75}
    ]
    \ar[
      uu,
      "\mapsup"{description, pos=.44},
      "\scalebox{0.75}{$
        \left(
        \!\!\!\!
        \begin{array}{l}
          (x, h_1, e, h_0),
          \\
          (0,n), 0
        \end{array}
        \!\!\!\!\!\!
        \right)
      $}"{description, pos=.75}
    ]
    \ar[
      uu,
      shift right=55pt,
      "\mapsup"{description, pos=.44},
      "\scalebox{0.75}{$
        \left(
        \!\!\!\!
        \begin{array}{l}
          (x, h_1, h_0, e),
          \\
          (n,0), 0
        \end{array}
        \!\!\!\!\!\!
        \right)
      $}"{description, pos=.75}
    ]
    \\
    {\phantom{A}}
    \\
    \mbox{
      \tiny
      \color{orangeii}
      \bf
      deg 1
    }
    \ar[
      dd,
      phantom,
      "\rotatebox{90}{
        \tiny
        \color{greenii}
        \bf
        degeneracy maps
      }"
    ]
    &
    \underset{[g]}{\coprod}
    (X^g \times C_g \times \mathbb{R})
    \ar[
      rr,
      "\scalebox{0.75}{$
                \def\arraystretch{.7}
          \begin{array}{l}
            x \,\mapsto\, x
            \\
            h \,\mapsto\, h
            \\
            n \,\mapsto\, n
          \end{array}
          $}"{description}
    ]
    \ar[
      uu,
      shift left=6pt,
      "\scalebox{0.75}{$
        \begin{array}{c}
          (x,(e,0,0),(h,r))
          \\
          \mapsup
          \\
          (x,(h,r))
        \end{array}
      $}"
    ]
    \ar[
      uu,
      shift right=6pt,
      "\scalebox{0.75}{$
        \begin{array}{c}
          (x,(h,r,0), (e,0))
          \\
          \mapsup
          \\
          (x,(h,r))
        \end{array}
      $}"{right}
    ]
    &&
    \underset{[g]}{\coprod}
    (X^g \times C_g)
    \ar[
      uu,
      shift left=6pt,
      "\scalebox{0.75}{$
        \def\arraystretch{.7}
        \begin{array}{c}
          ((x,e,h), 0)
          \\
          \mapsup
          \\
          (x,h)
        \end{array}
      $}"
    ]
    \ar[
      uu,
      shift right=6pt,
      "\scalebox{0.75}{$
        \def\arraystretch{.7}
        \begin{array}{c}
          ((x,h,e), 0)
          \\
          \mapsup
          \\
          (x,h)
        \end{array}
      $}"{right}
    ]
    \\
    {\phantom{A}}
    \\
    \mbox{
      \tiny
      \color{orangeii}
      \bf
      deg 0
    }
    &
    \underset{
      [g]
    }{\coprod}
    (X^g)
    \ar[
      rr,
      "{
        x \,\mapsto\, x
      }"
    ]
    \ar[
      uu,
      "\scalebox{0.75}{$
                \begin{array}{c}
            (x,(e,0))
            \\
            \mapsup
            \\
            x
          \end{array}
             $}"
    ]
    &&
    \underset{[g]}{\coprod}
    (X^g)
    \ar[
      uu,
      "{
        \scalebox{0.75}{$
          \begin{array}{c}
            (x,e)
            \\
            \mapsup
            \\
            x
          \end{array}
        $}
      }"
    ]
  \end{tikzcd}
\end{equation}

\newpage

  \noindent
  Regarding {\bf (ii)}:  It remains to show that the
  morphism of simplicial presheaves
  shown on the previous two pages, in
  \eqref{ComparisonMorphismFaceMaps} and \eqref{ComparisonMorphismDegeneracyMaps},
  indeed represents a morphism of $\SmoothInfinityGroupoids$ of the
  form  \eqref{FormOfComparisonMorphism}.
  %\eqref{ComparisonMorphismFromHuanInertiaToCyclification}.
  That the action on the right is the correct one is the content of
  Lemma \ref{CanonicalShapeS1ActionOnInertiaOrbifold}.
  The simplicial Borel constructions on both sides of
  \eqref{ComparisonMorphismFaceMaps}
  and \eqref{ComparisonMorphismDegeneracyMaps}
  present the respective homotopy quotients by \cite[\S 3.5]{NSS12a} (recalled as \cite[Lem. 3.2.73]{SS21EPB}).
\end{proof}

\begin{proposition}[Comparison morphism is pullback of shape unit]
  \label{ComparisonMorphismIsPullbackOfShapeUnit}
  The comparison morphism from Lemma \ref{ComparisonMorphismFromHuanInertiaToCyclification}
  sits in a homotopy-Cartesian square as follows
  (thus implying Thm. \ref{AbstractCharacterizationOfHuanInertiaOrbifold}):
  \vspace{-2mm}
\begin{equation}
  \label{ComparisonMorphimsInHomotopyPullback}
  \begin{tikzcd}[row sep=small]
    \Lambda_{\CohesiveCircle}(X \!\sslash\! G )
    \ar[
      rr,
      "\mathrm{comp}_{X \!\sslash\! G}"
    ]
    \ar[
      drr,
      phantom,
      "\mbox{\tiny\rm(pb)}"
    ]
    \ar[
      d
    ]
    &&
    \mathrm{Cyc}_{\scalebox{.6}{$\shape \CohesiveCircle$}}
    (X \!\sslash\! G)
    \ar[d]
    \\
    \mathbf{B}\CohesiveCircle
    \ar[
      rr,
      "\eta_{\CohesiveCircle}^{\scalebox{.6}{\shape}}"{below}
    ]
    &&
    \shape \mathbf{B}\CohesiveCircle
  \end{tikzcd}
  {\phantom{AAAA}}
    \Leftrightarrow
  {\phantom{AAAA}}
  \Lambda_{\CohesiveCircle}(X \!\sslash\! G)
  \;\simeq\;
  \big(
    \eta_{\CohesiveCircle}^{\scalebox{.6}{\shape}}
  \big)^\ast
  \mathrm{Cyc}_{\scalebox{.6}{$\shape \CohesiveCircle$}}
  (
    X \!\sslash\! G
  )\;.
\end{equation}
\end{proposition}
\begin{proof}
A glance at the component maps in
\eqref{ComparisonMorphismFaceMaps} and \eqref{ComparisonMorphismDegeneracyMaps}
reveals that the comparison morphism fails to be a degreewise isomorphism
of simplicial presheaves
only in that the elements $r \in \mathbb{R}$ on the left
hand side are forgotten.
Since pullbacks of simplicial diagrams are computed objectwise,
this means that the presentation of $\mathrm{comp}$
on simplicial presheaves factors through a dashed isomorphism
of simplicial presheaves in the following diagram:
\vspace{-3mm}
\begin{equation}
  \label{ComarisonMorphismFactors}
  \begin{tikzcd}[row sep=12pt]
    \underset{[g]}{\coprod}
    \Big(
    \big(
      X^g \times W( C_g \times \mathbb{R} \times \mathbb{Z}^{\times^\bullet} )
    \big)
    /
    (C_g \times \mathbb{R} \times \mathbb{Z}^{\times^\bullet})
    \Big)
    \ar[
      dr,
      shorten <=-14pt,
      "\mathrm{comp}_{X \!\sslash\! G}"
    ]
    \ar[
      d,
      dashed,
      shorten <=-8pt,
      "\sim"{sloped, pos=.1}
    ]
    \\
    \Big(
    \big[
      \mathbb{Z}^{\times^\bullet}\!\!,
      \,
      X \times G^{\times^\bullet}
    \big]_{\mathrm{skel}}
    \times
    W( \mathbb{R} \times \mathbb{Z}^{\times^\bullet} )
    \Big)
    \big/
    (\mathbb{R} \times \mathbb{Z}^{\times^\bullet})
    \ar[r]
    \ar[
      d,
      shorten <=-4pt
    ]
    \ar[
      dr,
      phantom,
      "\mbox{\tiny\rm (pb)}"
    ]
    &
    \Big(
    \big[
      \mathbb{Z}^{\times^\bullet},
      \,
      X \times G^{\times^\bullet}
    \big]_{\mathrm{skel}}
    \times
    W( \mathbb{Z}^{\times^\bullet} )
    \Big)
    \big/
    \mathbb{Z}^{\times^\bullet}
    \ar[
      d,
      "\in \mathrm{Fib}"{right}
    ]
    \\
    \overline{W}
    (\mathbb{R} \times \mathbb{Z}^{\times^\bullet})
    \ar[
      r,
      "{
        \overline{W}
        \left(
          (r, \vec n) \,\mapsto\, \vec n
        \right)
      }"{below}
    ]
    &
    \overline{W}
    ( \mathbb{Z}^{\times^\bullet} )
  \end{tikzcd}
\end{equation}

\vspace{-2mm}
\noindent
Here:

\begin{itemize}

\item[{(i)}]
the morphism on the right is a fibration, since
$\big[\mathbb{Z}^{\times^\bullet}, X \times G^{\times^\bullet} \big]_{\mathrm{skel}}$
is the nerve of a groupoid and hence a Kan complex (over each $U \in \CartesianSpaces$)
and because
$
  \big(
    (-) \times W (\mathbb{Z}^{\times^\bullet})
  \big)
  \big/
  (\mathbb{Z}^{\times^\bullet})$
is a right Quillen functor
(by Prop. \ref{QuillenEquivalenceBetweenBorelModelStructureAndSliceOverClassifyingComplex})
and hence preserves Kan fibrations;

\item[{(ii)}]
the object in the bottom left represents the delooping
$\mathbf{B}\CohesiveCircle$, by Lemma \ref{ShapeUnitOfTheSmoothCircle};

\item[{(iii)}]
the morphism on the bottom represents the
shape unit on $\mathbf{B}\CohesiveCircle$
by the same argument as in Prop. \ref{ShapeUnitOfTheSmoothCircle}.
\end{itemize}

\noindent
This implies that the diagram exhibits the
claimed homotopy pullback \eqref{ComparisonMorphimsInHomotopyPullback}.
\end{proof}

We conclude by proving the remaining Lemmas used in the above arguments.
\begin{lemma}[Cofibrant resolution of circle group]
\label{CofibrantOfCircleGroup}
For $g \in G$,
we have the following cofibrant replacements of the Lie group
$\Lambda_g$ \eqref{SkeletonOfHuanInertiaStackOfGoodOrbifold}
in $\SimplicialPresheaves(\CartesianSpaces)_{{\mathrm{proj}}\atop {\mathrm{loc}} }$:
\vspace{-.3cm}
\begin{equation}
  \label{CofibrantResolutionOfHuanCentralizerGroup}
  \hspace{-3cm}
  \begin{tikzcd}[row sep=-3pt]
    \varnothing
    \ar[
      r,
      "\in \mathrm{Cof}"{above}
    ]
    &
    \mathbb{R} \times \mathbb{Z}^{\times^\bullet}
    \ar[
      rr,
      "\sim"{above, yshift=-1pt}
    ]
    &&
    \qquad
    \mathbb{R}^{\times^\bullet_{\CohesiveCircle}}
    \qquad
    \ar[
      r,
      "\in \mathrm{W}"
    ]
    &
    \CohesiveCircle \;,
    \\
    &
    \scalebox{0.7}{$
    (r,\vec n)
    $}
    &\longmapsto&
    \hspace{-3cm}
    \mathrlap{
  \scalebox{0.7}{$    \big(
      r, (r + n_1), (r + n_1 + n_2), \cdots
    \big)
    $}
    }
    \\
    {\phantom{A}}
    \\
    \varnothing
    \ar[
      r,
      "\in \mathrm{Cof}"{above}
    ]
    &
    C_g \times \mathbb{R}
    \times
    \mathbb{Z}^{\times^\bullet}
    \ar[
      rr,
         "\sim"{above,yshift=-1pt}
    ]
    &{\phantom{AAA}}&
    \qquad \quad
    (C_g \times \mathbb{R})^{\times^{\bullet}_{\Lambda_g}}
    \qquad \qquad
    \ar[
      r,
      "\in \mathrm{W}"{above}
    ]
    &
    \Lambda_g\;.
    \\
    &
  \scalebox{0.7}{$    (h,r, -\vec n) $}
    &\longmapsto&
    \hspace{-5.5cm}
    \mathrlap{
    \scalebox{0.7}{$  \big(
      (h,r),
      (g^{n_1} \cdot h,\, r + n_1),
      (g^{n_1 + n_2} \cdot h,\, r + n_1 + n_2),
      \cdots
    \big)
    $}
    }
  \end{tikzcd}
\end{equation}
\end{lemma}
\begin{proof}
  The first statement is Example \ref{StandardCofibrantResolutionOfTheCircle}
  and the second follows analogously.
\end{proof}

\begin{lemma}[Shape unit of the smooth circle]
  \label{ShapeUnitOfTheSmoothCircle}
  A presentation for the shape unit of the smooth circle is given by
  \vspace{-2mm}
  \begin{equation}
    \label{PresentationOfShapeAdjunctionUnitOnSmoothCircle}
    \begin{tikzcd}[row sep=6pt]
      \mbox{
        \tiny
        \color{orangeii}
        \bf
        \begin{tabular}{c}
          smooth
          \\
          $\infty$-stacks
        \end{tabular}
      }
      &
      \CohesiveCircle
      \ar[
        rr,
        "\eta_{\CohesiveCircle}^{\scalebox{.6}{\shape}}"
      ]
      &&
      \shape \CohesiveCircle
      \,\simeq\,
      \mathbf{B}\mathbb{Z}
      &
      \in \SmoothInfinityGroupoids
      \\
      \mbox{
        \tiny
        \color{orangeii}
        \bf
        \begin{tabular}{l}
          simplicial
          \\
          presheaves
        \end{tabular}
      }
      &
      \mathbb{R} \times \mathbb{Z}^{\times^\bullet}
      \ar[
        rr,
        "{
          (r, \vec n) \;\mapsto\; \vec n
        }"{below}
      ]
      &&
      \mathbb{Z}^{\times^\bullet}
      &
      \in \SimplicialPresheaves(\CartesianSpaces)_{ {\mathrm{proj}} \atop {\mathrm{loc}} }
      \ar[
        u,
        "\, \Localization{\mathrm{W}}"{right}
      ]
    \end{tikzcd}
  \end{equation}
\end{lemma}
\begin{proof}
  By Prop. \ref{CofibrantOfCircleGroup}, the simplicial
  presheaf $\mathbb{R} \times \mathbb{Z}^{\times^\bullet}$
  is a cofibrant resolution of the circle,
  and by Prop. \ref{QuillenFunctorForShapeModalityOnSmoothInfinityGroupoids}
  its image under shape is given by
  \vspace{-2mm}
  $$
    \underset{\longrightarrow}{\mathrm{lim}}
    (\mathbb{R} \times \mathbb{Z}^{\times^\bullet})
    \;\simeq\;
    \mathbb{Z}^{\times^\bullet}
      \,\in\,
    (\SimplicialSets_{\mathrm{Qu}})_{\mathrm{fib}}
   \;   \xrightarrow{\;\;\mathrm{const}\;\;} \;
    \big(
      \SimplicialPresheaves(\CartesianSpaces)_{ {\mathrm{proj}} \atop {\mathrm{loc}} }
    \big)_{\mathrm{fib}}
    \,.
  $$

\vspace{-2mm}
\noindent
  This being fibrant means that the
  $(\underset{\longrightarrow}{\mathrm{lim}} \dashv \mathrm{const})$-adjunction unit,
  which is as shown in \eqref{PresentationOfShapeAdjunctionUnitOnSmoothCircle},
  is already the derived unit.
\end{proof}

\begin{lemma}[Canonical $\shape \CohesiveCircle$-action on inertia orbifold]
  \label{CanonicalShapeS1ActionOnInertiaOrbifold}
  The canonical
  action (by Prop. \ref{InducedSimplicialActions})
  of
  $\shape \CohesiveCircle \,\simeq\, \mathbf{B}\mathbb{Z}$ on
  any inertia stack \eqref{InertiaStackToSmoothLooStackAsShapeUnitPullback}
  --
  in particular on the inertia orbifold \eqref{SkeletonOfInertiaStackOfGoodOrbifold}
  of a good orbifold $\mathcal{X} \,\simeq\, X \!\sslash\! G$
  --
  is presented (under Prop. \ref{QuillenEquivalenceBetweenBorelModelStructureAndSliceOverClassifyingComplex})
  by the following simplicial action of
  $\mathbf{B}\mathbb{Z} \,\simeq\, \mathbb{Z}^{\times^\bullet}$:
  \vspace{-3mm}
  \begin{equation}
    \label{FormulaForCanonicalBZActionOnInertia}
    \hspace{-3mm}
    \begin{tikzcd}[row sep=-5pt]
    \Maps{\big}
    {
      \mathbf{B}\mathbb{Z} \times \Delta[\bullet]
    }
    {
      \mathcal{X}
    }
    \times
    \mathbb{Z}^{\times^\bullet}
    \ar[rr]
    &&
    \Maps{\big}
    {
      \mathbf{B}\mathbb{Z} \times \Delta[\bullet]
    }
    {
      \mathcal{X}
    }
    \\
  \scalebox{0.8}{$  \Big(
      \big( x, (h_k, \cdots, h_2, h_1) \big),
      \,
      ( n_k, \cdots, n_2, n_1 )
    \Big) $}
    &\longmapsto&
  \scalebox{0.8}{$    \big( x, (g^{n_k} \cdot h_k, \cdots, g^{n_2} \cdot h_2, g^{n_1} \cdot h_1) \big)
  $}.
    \end{tikzcd}
  \end{equation}
\end{lemma}
\begin{proof}
We show the case $k = 1$; the general case is directly analogous.
Here we unwind formula \eqref{CofreeSimplicialActionInComponents} as follows:
\vspace{-2mm}
$$
  \hspace{-2mm}
  \begin{tikzcd}[row sep=-12pt, column sep=32pt]
    \mathbf{B}\mathbb{Z}
    \times
    \Delta[1]
    \ar[
      r,
      "\mathrm{id} \times \mathrm{diag}"
    ]
    &
    \mathbf{B}\mathbb{Z}
    \times
    \Delta[1]
    \times
    \Delta[1]
    \ar[
      r,
      "\mathrm{id}
         \times
       {\color{purple}n}
         \times
       \mathrm{id}"
    ]
    &
    \mathbf{B}\mathbb{Z}
    \times
    \mathbf{B}\mathbb{Z}
    \times
    \Delta[1]
    \ar[
      r,
      "+ \times \mathrm{id}"
    ]
    &
    \mathbf{B}\mathbb{Z}
    \times
    \Delta[1]
    \ar[
      rrr,
      "{
        \raisebox{3pt}{$\overset{\scalebox{.7}{$$}}{\scalebox{.7}{$1$}}$}
        \overset{\scalebox{.6}{$[0,1]$}}{\scalebox{1.3}{$\Box$}}
        \;\;\raisebox{2pt}{\scalebox{.7}{$\mapsto$}}\;\;
        \raisebox{3pt}{$\overset{\scalebox{.7}{$x$}}{\scalebox{.7}{$g$}}$}
        \overset{\scalebox{.7}{$h$}}{\scalebox{1.3}{$\Box$}}
      }"
    ]
    &&&
    X \!\sslash\! G
    \\
    {\phantom{ A \atop A }}
    \\
 \scalebox{.7}{$
      \big(
        (0,1),
        \,
        [0,1,1]
      \big)
    $}
    \ar[
      r,
      phantom,
      "\mapsto"
    ]
    &
    \scalebox{.7}{$
      \big(
        (0,1),
        \,
        [0,1,1]
        \,
        [0,1,1]
      \big)
    $}
    \ar[
      r,
      phantom,
      "\mapsto"
    ]
    &
    \scalebox{.7}{$
      \big(
        (0,1),
        \,
        ({\color{purple}n},0),
        \,
        [0,1,1]
      \big)
    $}
    \ar[
      r,
      phantom,
      "\mapsto"
    ]
    &
    \scalebox{.7}{$
      \big(
        ({\color{purple}n},1),
        \,
        [0,1,1]
      \big)
    $}
    \ar[
      drrr,
      phantom,
      "\longmapsto"{sloped}
    ]
    &&&
    \\
    &&&
    &&&
    \raisebox{3pt}{$\overset{\scalebox{.7}{$x$}}{\scalebox{.7}{$g$}}$}
    \overset{\scalebox{.7}{$\;\,g^{\color{purple}n} \!\cdot\! h$}}{\scalebox{1.3}{$\Box$}}
    \\
    \scalebox{.7}{$ \big(
      (1,0),
      \,
      [0,0,1]
    \big)
    $}
    \ar[
      r,
      phantom,
      "\mapsto"
    ]
    &
    \scalebox{.7}{$
    \big(
      (1,0),
      \,
      [0,0,1]
      \,
      [0,0,1]
    \big)
    $}
    \ar[
      r,
      phantom,
      "\mapsto"
    ]
    &
    \scalebox{.7}{$
    \big(
      (1,0),
      \,
      (0,{\color{purple} n}),
      \,
      [0,0,1]
    \big)
    $}
    \ar[
      r,
      phantom,
      "\mapsto"
    ]
    &
    \scalebox{.7}{$
    \big(
      (1,{\color{purple}n}),
      \,
      [0,0,1]
    \big)
    $}
    \ar[
      urrr,
      phantom,
      "\longmapsto"{sloped}
    ]
  \end{tikzcd}
$$
$$
\begin{tikzcd}[column sep=60pt, row sep=60pt]
  \big(
    \bullet
    ,
    [0]
  \big)
  \ar[
    r,
    "{
      (
        0
        ,
        [0,1]
      )
    }"{above}
  ]
  \ar[
    d,
    "{
      (
        1
        ,
        [0,0]
      )
    }"{description}
  ]
  \ar[
    dr,
    "{
      (
        1
        ,\,
        [0,1]
      )
    }"{description, sloped}
  ]
  &
  \big(
    \bullet
    ,
    [1]
  \big)
  \ar[
    d,
    "{
      (
        1
        ,
        [1,1]
      )
    }"{description}
  ]
  \ar[
    dl,
    phantom,
    "{
      (
        (1,0)
        ,
        [0,0,1]
      )
    }"{near end, scale=.7, yshift=-7pt, xshift=-4pt}
  ]
  \ar[
    dl,
    phantom,
    "{
      (
        (0,1)
        ,
        [0,1,1]
      )
    }"{near start, scale=.7, yshift=+7pt, xshift=+4pt}
  ]
  \\
  \big(
    \bullet
    ,
    [0]
  \big)
  \ar[
    r,
    "{
      (
        0
        ,
        [0,1]
      )
    }"{below}
  ]
  &
  \big(
    \bullet
    ,
    [1]
  \big)
\end{tikzcd}
\qquad \quad \;\;\;\;\;\;\;\;\;\;\;\;\;
\longmapsto
\qquad \quad \;\;\;\;\;\;\;\;\;\;\;\;\;
\begin{tikzcd}[column sep=60pt, row sep=60pt]
  \big(
    \bullet
    ,
    [0]
  \big)
  \ar[
    r,
    "{
      (
        {\color{purple}n}
        ,
        [0,1]
      )
    }"{above}
  ]
  \ar[
    d,
    "{
      (
        1
        ,
        [0,0]
      )
    }"{description}
  ]
  \ar[
    dr,
    "{
      (
        1+{\color{purple}n}
        ,\,
        [0,1]
      )
    }"{description, sloped}
  ]
  &
  \big(
    \bullet
    ,
    [1]
  \big)
  \ar[
    d,
    "{
      (
        1
        ,
        [1,1]
      )
    }"{description}
  ]
  \ar[
    dl,
    phantom,
    "{
      (
        (1,{\color{purple}n})
        ,
        [0,0,1]
      )
    }"{near end, scale=.7, yshift=-7pt, xshift=-4pt}
  ]
  \ar[
    dl,
    phantom,
    "{
      (
        ({\color{purple}n},1)
        ,
        [0,1,1]
      )
    }"{near start, scale=.7, yshift=+7pt, xshift=+4pt}
  ]
  \\
  \big(
    \bullet
    ,
    [0]
  \big)
  \ar[
    r,
    "{
      (
        {\color{purple}n}
        ,
        [0,1]
      )
    }"{below}
  ]
  &
  \big(
    \bullet
    ,
    [1]
  \big)
\end{tikzcd}
\;
\longmapsto
\;
\begin{tikzcd}[column sep=20pt]
  x
  \ar[
    d,
    "g"{left}
  ]
  \ar[
    r,
    "g^{\color{purple}n} \cdot h"{above}
  ]
  &
  h \cdot x
  \ar[
    d,
    "{
      g'
     }"{right}
  ]
  \\
  x
  \ar[
    r,
    "{
      g^{\color{purple}n}
        \cdot
      h
    }"{swap}
  ]
  &
  h \cdot x
\end{tikzcd}
$$
On the left above we show the non-degenerate 2-cell in the Cartesian product (Ex. \ref{NondegenerateSimplicesInSquare})
whose image will pick up the generating data in the 1-cell of the hom-complex.
Chasing this 2-cell through the formula \eqref{CofreeSimplicialActionInComponents}
yields the result, as shown.
\end{proof}

%%%%%%%%%%%%%%%%%%%%%%%%%%%%%%%%%%%%%%%%%%%%%%%%%%%%%
\section{Transgression as cyclification}
\label{TransgressionViaCyclification}
%%%%%%%%%%%%%%%%%%%%%%%%%%%%%%%%%%%%%%%%%%%%%%%%%%%%%
\
A simple but archetypical special case of orbifold cohomology is {\it group cohomology}, which may be understood as given by
homotopy classes of maps of $\infty$-groupoids of the form $\mathbf{B}G \xrightarrow{\;} \mathbf{B}^n A$ (Ex. \ref{IncarnationsOfGroupCohomology}).
Here we discuss the technical detail of how cyclification of orbifolds leads to
transgression in group cohomology, by proving the equivalence expressed by the top horizontal arrow in the diagram \eqref{TransgressionViaCyclification}.

\medskip

The following component Definition \ref{TransgressionFormulaForGroupCocycles} is (up to an irrelevant global sign which we omit)
due to \cite[Def. 4.1]{AdemRuanZhang07}, and maybe independently due to \cite[\S 1.3.3]{Willerton08} where it is motivated by a geometric picture similar to that made precise by Thm. \ref{TransgressionInGroupCohomologyViaLooping} below. Notice that it is
tedious (albeit straightforward) to explicitly check that the component formula \eqref{TransgressionFormula} really makes sense,
in that it satisfies the cocycle condition; but this is implied by our more abstract characterization in Thm. \ref{TransgressionInGroupCohomologyViaLooping} below.

\begin{definition}[Transgression formula for discrete group cocycles]  \label{TransgressionFormulaForGroupCocycles}
  For $A \,\in\, \AbelianGroups$,
  $n \in \mathbb{N}$ and $c : G^{\times_{n+1}} \xrightarrow{\;} A$
  a group cocycle $[c] \,\in\, H^{n+1}_{\mathrm{grp}}(G;\, A)$,
  obtain an $n$-cocycle on the inertia groupoid $\Lambda \mathbf{B} G$
  (Ntn. \ref{EvaluationMapOnNerveOfInertiaGroupoid})
  \vspace{-2mm}
  $$
    [\mathrm{tr}(c)]
    \,\in\,
    H^n\big( \Lambda \mathbf{B}G;\, A\big)
  $$

  \vspace{-3mm}
  \noindent
  by setting:
  \begin{equation}
    \label{TransgressionFormula}
    \def\arraystretch{1.7}
    \begin{array}{l}
    \mathrm{tr}(c)
    \Big(
      \gamma
      \xrightarrow{g_{n-1}}
      \mathrm{Ad}_{n-1}(\gamma)
      \xrightarrow{g_{n-2}}
      \cdots
      \xrightarrow{g_{1}}
      \mathrm{Ad}_{1}(\gamma)
      \xrightarrow{g_0}
      \mathrm{Ad}_0(\gamma)
    \Big)
    \\
    \;\coloneqq\;
    \underset{
      0 \leq j \leq n
    }{
      \scalebox{1.3}{$\sum$}
    }
    (-1)^{j}
    \cdot
    c\Big(
      g_{n-1},
      \cdots,
      g_{n-j},
      \mathrm{Ad}_j(\gamma),
      g_{n-j-1},
      \cdots,
      g_0
    \Big)
    \,,
    \end{array}
  \end{equation}

  \vspace{-2mm}
  \noindent
  where we abbreviate
  \begin{equation}
     \label{AbbreviationForAdjointAction}
      \mathrm{Ad}_j(\gamma)
      \;\coloneqq\;
      \mathrm{Ad}_{(g_{n-1}\cdots g_j)}(\gamma)
      \;\coloneqq\;
      (g_{n-1}\cdots g_0)^{-1}
      \cdot \gamma \cdot
      (g_{n-1} \cdots g_0)
      \,.
  \end{equation}
\end{definition}
\begin{remark}[Group cohomology]
  Since the inertia groupoid of the delooping of a finite group decomposes as a disjoint union
  of delooped centralizer groups $C_g$
  indexed over the conjugacy classes $[g] \,\in\, G /_{\mathrm{ad}} G$
  \begin{equation}    \label{DecompositionOfInertiaGroupoid}
    \Lambda
    \mathbf{B}G
    \;\;\;
      \simeq
    \quad
    \underset{
      \mathclap{
        [g]
        \,\in\,
        G_{\mathrm{ad}}/G
      }
    }{\coprod}
    \quad
    \mathbf{B}C_g
    \;\;\;\;
    \simeq
    \;\;\;\;
    \mathbf{B}G
    \;\;\sqcup\;\;\;
    \underset{
      \mathclap{
        [g]
        \,\neq\,
        [\mathrm{e}]
      }
    }{\coprod}
    \;\;
    \mathbf{B}C_g
    \,,
  \end{equation}
  the formula \eqref{TransgressionFormula} restricts to transgression maps to the group cohomology of all these groups,
  in particular to $C_{\mathrm{e}} = G$ itself -- as highlighted on the right of \eqref{DecompositionOfInertiaGroupoid}.
  Often only this leading component of the full transgression map is considered (e.g., implicitly so in \cite[p. 14]{DijkgraafWitten90}).
\end{remark}

We now prove (Thm. \ref{TransgressionInGroupCohomologyViaLooping}) below, that the transgression formula \eqref{TransgressionFormula}
abstractly arises from looping/cyclification according to \eqref{TransgressionViaCyclification}.
(Something similar is alluded to in \cite[\S 1.3.3]{Willerton08}.)
First, we record the following Lemma \ref{EvaluationMapOnNerveOfInertiaGroupoid} in simplicial homotopy theory, which is elementary
but requires some care; in stating this we make heavy use of constructions and facts in simplicial homotopy theory that are
recalled/introduced in \cref{SomeSimplicialHomotopyTheory}:

\begin{lemma}[Evaluation map on nerve of inertia groupoid]
\label{EvaluationMapOnNerveOfInertiaGroupoid}
On non-degenerate cells (cf. Prop. \ref{NonDegenerateSimpliciesInProductOfSimplicies}), the evaluation map \eqref{EvaluationMap}
on the inertia function complex (Ntn. \ref{NerveOfInertiaGroupoids}) is given by
\vspace{-2mm}
\begin{equation}
  \begin{tikzcd}[
    row sep=-4pt,
    column sep=8pt
  ]
  \big(
    S^1_{\mathrm{min}}
  \big)_{n+1}
  \times
  \mathrm{Hom}
  \big(
    S^1_{\mathrm{min}}
      \times
    \Delta[n+1]
    ,\,
    \overline{W}G
  \big)
  \ar[
    rr,
    "{ \mathrm{ev} }"
  ]
  &&
  (
    \overline{W}G
  )_{n+1}
  \\
\scalebox{0.7}{$  \Big(
    s_{(0,\cdots, \widehat{k}, \cdots, n)}
    \ell
    ,\,
    s_k
  \big(
    \gamma
    ;\,
    g_{n-1}
    ,\,
    \cdots
    ,\,
    g_{0}
  \big)
  \Big)
  $}
  &\longmapsto&
\scalebox{0.7}{$  \big(
    g_{n-1}
    ,\,
    \cdots,
    g_{n-k}
    ,\,
    \mathrm{Ad}_k(\gamma)
    ,\,
    g_{n - (k+1)}
    ,\,
    \cdots
    ,\,
    g_0
  \big)
  $}
  \mathrlap{\,.}
  \end{tikzcd}
\end{equation}
\end{lemma}

\vspace{-2mm}
\noindent
Here, on the left, $\ell \,\in\,\big(S^1_{\mathrm{min}}\big)$ denotes the non-degenerate cell in the minimal simplicial
circle (Ntn. \ref{MinimalSimplicialCircle}) and $(\gamma; g_{n-1}, \cdots, g_0)$ (Ntn. \ref{NerveOfInertiaGroupoids}) denotes
a sequence of $n$ composable morphisms in the inertia groupoid.
\begin{proof}
Using formula \eqref{EvaluationMapOnSimplicialMappingComplexes} for the components of the evaluation map, and then unwinding
as in Ntn. \ref{NerveOfInertiaGroupoids}:

\vspace{.1cm}
\hspace{-.9cm}
\def\arraystretch{2}
\tabcolsep=-5pt
\begin{tabular}{cc}
\multirow{2}{*}{
$
  \begin{tikzcd}
    \\[-30pt]
    \big(
      s_{ (1, \cdots, \widehat{k}, \cdots, n+1) }\ell
      ,\
      \iota_{n+1}
    \big)
    \ar[
      d,
      phantom,
      "{ \in }"{rotate=-90}
    ]
    \\[-10pt]
    \big(
      S
        \,\times\,
      \Delta[n+1]
    \big)_{\mathrlap{n+1}}
    \ar[
      d,
      "{
        \def\arraystretch{1}
        \begin{array}{l}
        \scalebox{0.7}{$        s_k
        \big(
          \gamma
          ;\;
          g_{n-1}
          ,\,
          \cdots
          ,\,
          g_0
        \big)
        \,=\,
        $}
        \\
     \scalebox{0.7}{$   \big(
          \gamma
          ;\;
          g_{n-1}
          ,\,
          \cdots
          ,\,
          g_{n-k}
          ,\,
          \mathrm{e}
          ,\,
          g_{n-(k+1)}
          ,\,
          \cdots
          ,\,
          g_0
        \big)
        $}
        \end{array}
      }"{description}
    ]
    \\[+48pt]
    \big(
      \overline{W}G
    \big)_{\mathrlap{n+1}}
  \end{tikzcd}
$
}
&
\quad
$
  \begin{tikzcd}[
    column sep=35pt,
    row sep=25pt
  ]
    (
     \ast
     ,\,
      0
   )
    \ar[
      r,
      "{
        (
          \mathrm{id}_\ast
          ,\,
          [0,1]
        )
      }"{yshift=1pt}
    ]
    \ar[
      r,
      -,
      line width=5pt,
      blue,
      opacity=.2
    ]
    \ar[
      d,
      "{
        (
          \ell
          ,\,
          \mathrm{id}_0
        )
      }"{description}
    ]
    &
    (
      \ast
      ,\,
      1
   )
    \ar[
      d,
      "{
        (
          \mathrm{id}_\ast
          ,\,
          [0,1]
        )
      }"{description}
    ]
    \ar[
      r,
      "{
        (
          \mathrm{id}_\ast
          ,\,
          [1,2]
        )
      }"{yshift=1pt}
    ]
    \ar[
      r,
      -,
      line width=5pt,
      blue,
      opacity=.2
    ]
    \ar[
      d,
      "{
        (
          \mathrm{id}_\ast
          ,\,
          [0,1]
        )
      }"{description}
    ]
    \ar[r]
    &
    \cdots
    \ar[r]
    \ar[
      r,
      -,
      line width=5pt,
      blue,
      opacity=.2
    ]
    &[-10pt]
    (
     \ast
     ,
      k
   )
   \ar[
     d,
     "{
       (\ell,\, \mathrm{id}_k)
     }"{description}
   ]
   \ar[
     r,
     "{ \mathrm{id}_{(\ast,k)} }"
   ]
   \ar[
     dr
   ]
    \ar[
      dr,
      -,
      line width=5pt,
      blue,
      opacity=.2
    ]
   &
   (
     \ast
     ,
     k
   )
    \ar[
      d,
      "{
        (
          \ell
          ,\,
          \mathrm{id}_k
        )
      }"{description}
    ]
    \ar[r]
    &[-10pt]
    \cdots
    \ar[
      r,
      "{
        (
          \mathrm{id}_\ast
          ,\,
          [n-1,n]
        )
      }"{yshift=1pt}
    ]
    \ar[
      d,
      phantom,
      "{ \cdots }"
    ]
    &
    (
     \ast
     ,
      n
   )
    \ar[
      d,
      "{
        (
          \ell
          ,\,
          \mathrm{id}_n
        )
      }"{description}
    ]
    \\
    (
     \ast,
      0
    )
    \ar[
      r,
      "{
        (
          \mathrm{id}_\ast
          ,\,
          [0,1]
        )
      }"{swap, yshift=-1pt}
    ]
    &
    (
     \ast,
      1
   )
    \ar[
      r,
      "{
        (
          \mathrm{id}_\ast
          ,\,
          [1,2]
        )
      }"{swap, yshift=-1pt}
    ]
    &
    \cdots
    \ar[r]
    \ar[
      r
    ]
    &
    (
     \ast,
      k
    )
    \ar[
      r,
      "{ \mathrm{id}_{(\ast,k)} }"{swap}
    ]
    &
    (
      \ast
      ,
      k
    )
    \ar[r]
    \ar[
      r,
      -,
      line width=5pt,
      blue,
      opacity=.2
    ]
    &
    \cdots
    \ar[
      r,
      "{
        (
          \mathrm{id}_\ast
          ,\,
          [n-1,n]
        )
      }"{swap, yshift=-1pt}
    ]
    \ar[
      r,
      -,
      line width=5pt,
      blue,
      opacity=.2
    ]
    &
    (
     \ast,
     n
   )
  \end{tikzcd}
$
\\[-10pt]
& \rotatebox{-90}{$\mapsto$}
\\[+10pt]
&
\hspace{3pt}
$
  \begin{tikzcd}[
    column sep=50pt,
    row sep=35pt
  ]
    \bullet
    \ar[
      r,
      "{
        g_{n-1}
      }"{}
    ]
    \ar[
      r,
      -,
      line width=5pt,
      blue,
      opacity=.2
    ]
    \ar[
      d,
      "{ \gamma }"{description}
    ]
    &
    \bullet
    \ar[
      r,
      "{
        g_{n-2}
      }"
    ]
    \ar[
      r,
      -,
      line width=5pt,
      blue,
      opacity=.2
    ]
    \ar[
      d,
      "{ \mathrm{Ad}_1(\gamma) }"{description}
    ]
    &
    \cdots
    \ar[
      r,
      "{
       g_{n-k}
      }"
    ]
    \ar[
      r,
      -,
      line width=5pt,
      blue,
      opacity=.2
    ]
    &[-20pt]
    \bullet
    \ar[
      d,
      "{
        \mathrm{Ad}_k(\gamma)
      }"{description}
    ]
    \ar[
      dr,
      "{ \mathrm{Ad}_k(\gamma) }"{
        sloped, description}
    ]
    \ar[
      dr,
      -,
      line width=5pt,
      blue,
      opacity=.2
    ]
    \ar[
      r,
      "{ \mathrm{e} }"
    ]
    &[-10pt]
    \bullet
    \ar[
      r,
      "{ g_{n - (k+1)} }"
    ]
    \ar[
      d,
      "{
        \mathrm{Ad}_k(\gamma)
      }"{description}
    ]
    &[-20pt]
    \cdots
    \ar[r]
    \ar[
      d,
      phantom,
      "{ \cdots }"
    ]
    \ar[
      r,
      "{
        g_{0}
      }"
    ]
    &
    \bullet
    \ar[
      d,
      "{ \mathrm{Ad}_n(\gamma) }"{description}
    ]
    \\
    \bullet
    \ar[
      r,
      "{
        g_{n-1}
      }"{swap}
    ]
    &
    \bullet
    \ar[
      r,
      "{
        g_{n-2}
      }"{swap}
    ]
    &
    \cdots
    \ar[
      r,
      "{ g_{n-k} }"{swap}
    ]
    &
    \ar[
      r,
      "{ \mathrm{e} }"{swap}
    ]
    \bullet
    &
    \bullet
    \ar[
      r,
      "{ g_{n-(k+1)} }"{swap}
    ]
    \ar[
      r,
      -,
      line width=5pt,
      blue,
      opacity=.2
    ]
    &
    \cdots
    \ar[
      r,
      "{ g_0 }"{swap}
    ]
    \ar[
      r,
      -,
      line width=5pt,
      blue,
      opacity=.2
    ]
    &
    \bullet
  \end{tikzcd}
$
\end{tabular}
\vspace{-.5cm}

\end{proof}

\begin{theorem}[Transgression in discrete group cohomology via looping] \label{TransgressionInGroupCohomologyViaLooping}
  The transgression formula of Def. \ref{TransgressionFormulaForGroupCocycles}
  expresses equivalently the operation of applying the free homotopy loop functor $\Maps{}{\mathbf{B}\Integers}{-}$ to group cocycles
  understood as maps $\mathbf{B}G \to \mathbf{B}^{n+1} A$ (cf. Ex. \ref{IncarnationsOfGroupCohomology}), composed with projection
  onto the resulting shifted coefficients (via Ex. \ref{FreeLoopSpaceOfEilenbergMacLaneSpace}):
  \vspace{-6mm}
  $$
    \begin{tikzcd}[column sep=20pt, row sep=small]
      H^{n+1}_{\mathrm{grp}}
      (G;\, A)
      \ar[r, phantom, "\simeq"]
      &[-10pt]
      \pi_0
      \Maps{\big}
        { B G }
        {
          \mathbf{B}^{n+1} A
        }
      \ar[
        rr,
        "{
          \Maps{}
            { \mathbf{B} \Integers }
            { - }
        }"
      ]
      \ar[
        drr,
        shorten <=-10pt,
        shorten >=5pt,
        shift right=5pt,
        "{
          \scalebox{.7}{
            \color{greenii}
            \bf
            transgression
          }
        }"{sloped}
      ]
      &\phantom{AA}&
      \pi_0
      \Maps{\big}
       {
        \Lambda \mathbf{B}G
       }
       {
         \overset{
           \mathclap{
             \raisebox{3pt}{
               \scalebox{.7}{
               \rm
               by Ex. \ref{FreeLoopSpaceOfEilenbergMacLaneSpace}
               }
             }
           }
         }{
        \overbrace{
          \mathbf{B}^n A
          \,\times\,
          \mathbf{B}^{n+1} A
        }
        }
      }
      \ar[
        d,
        "{
          \scalebox{.7}{$
            \mathrm{pr}_1
          $}
        }"
      ]
      \\
      &
      &&
      \pi_0
      \Maps{}
        { \Lambda \mathbf{B}G }
        { \mathbf{B}^n A }
      \ar[
        r,
        phantom,
        "\simeq"
      ]
      \ar[
        d,
        ->>,
        "{
          \pi_0
          \Maps{}
            {
              \mathbf{B}G
              \hookrightarrow
              \Lambda \mathbf{B}G
            }
            { \mathbf{B}^n A }
        }"{left}
      ]
      &[-30pt]
      H^n
      (
        \Lambda \mathbf{B}G
        ;\,
        A
      )
      \ar[
        d,
        ->>
      ]
      \\[+15pt]
      & &&
      \pi_0
      \Maps{}
        { \mathbf{B}G }
        { \mathbf{B}^n A }
      \ar[
        r,
        phantom,
        "\simeq"
      ]
      &[-30pt]
      H^n_{\mathrm{grp}}
      (
        G
        ;\,
        A
      )      \;.
    \end{tikzcd}
  $$
\end{theorem}
\begin{proof}
  In outline, this follows by regarding
  \eqref{AdjunctOfMapsViaEvaluation}
  the functor $\Maps{}{\mathbf{B}\Integers}{-}$ as pre-composition with the evaluation map and further precomposing this, under the free abelian group functor,
  with the Eilenberg-Zilber map (Prop. \ref{EilenbergZilberAlexanderWhitneyRetraction}): the component formula for the evaluation map, from Lemma
  \ref{EvaluationMapOnNerveOfInertiaGroupoid}, fed through the Eilenberg-MacLane formula \eqref{EilenbergMaclaneFormula}
  produces the signed summands that appear in \eqref{TransgressionFormula}.
  In detail, we have the following sequence of natural transformations between homotopy classes\footnote{
    The collection of facts from
    model category theory needed for the routine verification that
    the underlying naive sequence of natural transformations passes to homotopy classes in each step (on p. \pageref{ProofSteps}) may  be found
    reviewed \cite[\S A]{FSS20CharacterMap}.
  } of morphisms:
\vspace{-4mm}
  $$
    \label{ProofSteps}
    \def\arraystretch{2}
    \begin{array}{lll}
      \hspace{-5pt}
      H^{n+1}_{\mathrm{grp}}
      (
        G;\, A
      )
           & \;=\;
      \pi_0
      \Maps{}
        { \mathbf{B}G }
        { \mathbf{B}^{n+1} A }
      &
      \proofstep{
        Ex. \ref{IncarnationsOfGroupCohomology}
      }
      \\[-8pt]
    &  \;\simeq\;
      \mathrm{Ho}(\SimplicialSets)
      \Big(
        \overline{W}G
        ;\,
        \mathrm{DK}
        \big(
          A[n+1]
        \big)
      \Big)
      &
      \proofstep{
        \hspace{-11pt}
        \def\arraystretch{1}
        \begin{tabular}{l}
        Ex. \ref{UniversalSimplicialPrincipalComplexForOrdoinaryGroupG}
        \\
        \&
        Ex. \ref{ShiftedAbelianGroups}
        \end{tabular}
      }
      \\[-8pt]
 &     \hspace{-10pt}
      \xrightarrow{
        \Maps{}{\mathbf{B}\Integers}{-}
      }\;
      \mathrm{Ho}(\SimplicialSets)
      \Big(
        \Maps{\big}
          { S^1_{\mathrm{min}}  }
          { \overline{W}G }
          ,\,
        \Maps{\big}
          { S^1_{\mathrm{min}}  }
          {
            \mathrm{DK}
            \big(
              A[n+1]
            \big)
          }
      \Big)
      &
      \proofstep{
        Ex. \ref{MinimalModelOfFreeLoopSpace}
      }
      \\[-5pt]
    &  \;\simeq\;
      \mathrm{Ho}(\SimplicialSets)
      \Big(
        S^1_{\mathrm{min}}
          \times
        \Maps{\big}
          { S^1_{\mathrm{min}}  }
          { \overline{W}G }
          ,\,
          \mathrm{DK}
          \big(
            A[n+1]
          \big)
      \Big)
      &
      \proofstep{
        \eqref{InternalHomAdjunction}
      }
      \\[-5pt]
   &   \;\simeq\;
      \mathrm{Ho}\big(
        \mathrm{Ch}^{\geq 0}(\AbelianGroups)
      \big)
      \Big(
        \NormalizedChains
        \circ
        \mathbb{Z}
        \big[
        S^1_{\mathrm{min}}
          \times
        \Maps{\big}
          { S^1_{\mathrm{min}}  }
          { \overline{W}G }
        \big]
          ,\,
          A[n+1]
      \Big)
      &
      \proofstep{
        \eqref{DoldKanCorrespondence}
      }
      \\
     & \;\simeq\;
      \mathrm{Ho}\big(
        \mathrm{Ch}^{\geq 0}(\AbelianGroups)
      \big)
      \bigg(
        \NormalizedChains
        \Big(
        \mathbb{Z}
        \big[
        S^1_{\mathrm{min}}
        \big]
          \otimes
        \mathbb{Z}
        \big[
        \Maps{\big}
          { S^1_{\mathrm{min}}  }
          { \overline{W}G }
        \big]
        \Big)
        ,\,
        A[n+1]
      \bigg)
      &
      \proofstep{
        \eqref{FreeAbIsStrongMonoidal}
      }
      \\
     & \xrightarrow
        [
          \nabla^\ast
        ]
        {
          \sim
        }
      \mathrm{Ho}\big(
        \mathrm{Ch}^{\geq 0}(\AbelianGroups)
      \big)
      \bigg(
        \NormalizedChains
        \Big(
        \mathbb{Z}
        \big[
        S^1_{\mathrm{min}}
        \big]
        \Big)
          \otimes
        \NormalizedChains
        \Big(
        \mathbb{Z}
        \big[
        \Maps{\big}
          { S^1_{\mathrm{min}}  }
          { \overline{W}G }
        \big]
        \Big)
        ,\,
        A[n+1]
      \bigg)
      &
      \proofstep{
        \eqref{EilenbergMaclaneFormula}
      }
      \\
    &  \;\simeq\;
      \mathrm{Ho}\big(
        \mathrm{Ch}^{\geq 0}(\AbelianGroups)
      \big)
      \bigg(
        \NormalizedChains
        \Big(
        \mathbb{Z}
        \big[
        \Maps{\big}
          { S^1_{\mathrm{min}}  }
          { \overline{W}G }
        \big]
        \Big)
        ,\,
        \Maps{\Big}
          {
            \NormalizedChains
            \big(
            \mathbb{Z}
            \big[
            S^1_{\mathrm{min}}
            \big]
            \big)
          }
          { A[n+1] }
      \bigg)
      &
      \proofstep{
        \eqref{InternalHomAdjunction}
      }
      \\
    &  \;\simeq\;
      \mathrm{Ho}\big(
        \mathrm{Ch}^{\geq 0}(\AbelianGroups)
      \big)
      \bigg(
        \NormalizedChains
        \Big(
        \mathbb{Z}
        \big[
        \Maps{\big}
          { S^1_{\mathrm{min}}  }
          { \overline{W}G }
        \big]
        \Big)
        ,\,
        A[n]
        \oplus
        A[n+1]
      \bigg)
      &
      \proofstep{
        \eqref{NormalizedChainsOnMinimalSimplicialCircle}
      }
      \\
&      \xrightarrow{ \mathrm{pr}_1 }
      \mathrm{Ho}\big(
        \mathrm{Ch}^{\geq 0}(\AbelianGroups)
      \big)
      \bigg(
        \NormalizedChains
        \Big(
        \mathbb{Z}
        \big[
        \Maps{\big}
          { S^1_{\mathrm{min}}  }
          { \overline{W}G }
        \big]
        \Big)
        ,\,
        A[n]
      \bigg)
      \\[-5pt]
&      \;=\;
      \pi_0
      \Maps{\big}
        { \Lambda\mathbf{B}G }
        { \mathbf{B}^n A }
      \;=\;
      H^m
      \big(
        \Lambda\mathbf{B}G
        ;\,
        A
      \big)
        &
      \proofstep{
        \eqref{DoldKanCorrespondence}}
       .
    \end{array}
  $$

  \vspace{-2mm}
\noindent   Noticing here
(see  \eqref{AdjunctOfMapsViaEvaluation})
  that the composite of the application of the functor $\Maps{}{\mathbf{B}\Integers}{-}$ with the adjointness relation is equivalent
  to pre-composition with the evaluation map, the net effect of this composite of transformations in degree $n+1$ is the
  pre-composition of the cocycle map in that degree, regarded as a homomorphism of abelian groups
  \vspace{-2mm}
  $$
    \begin{tikzcd}[
      row sep=3pt
    ]
      \widetilde c_{n+1}
      \;:\;
      \mathbb{Z}[
        G^{n+1}
      ]
      \simeq
      \mathbb{Z}[
        (\overline{W}G)_{n+1}
      ]
      \ar[r]
      &
      A
      \mathrlap{\,,}
    \end{tikzcd}
  $$

  \vspace{-2mm}
  \noindent
  with the composite of the   Eilenberg-Zilber-  and the
  evaluation map:
  \vspace{-2mm}
  \begin{equation}
    \label{EvaluationPrecomposedWithEilenbergZilber}
    \hspace{-3.5cm}
    \adjustbox{scale=.85}{
    \begin{tikzcd}[
      column sep=27pt,
      row sep=2pt
    ]
      \mathbb{Z}
      \big[
        \big(
          S^1_{
            \mathrm{min}
          }
        \big)_{1}
      \big]
      \otimes
      \mathbb{Z}
      \big[
      \Homs{\big}
        {
          S^1_{\mathrm{min} }
          \times
          \Delta[n]
        }
        { \overline{W}G }
      \big]
      {}
      \ar[
        r,
        shorten=-1pt,
        "{ \nabla }"
      ]
      &[-9pt]
      \mathbb{Z}
      \big[
      \big(
      S^1_{
      \mathrm{min}
      }
      \big)_{n+1}
      \times
      \Homs{}
        {
          S^1_{\mathrm{min} }
          \times
          \Delta[n+1]
        }
        { \overline{W}G }
      \big]
      \ar[
        r,
        shorten=-1pt,
        "{
        \mathbb{Z}[
          \mathrm{ev}
        ]
        }"
      ]
      &
      \mathbb{Z}
      [
        (\overline{W}G)_{n+1}
      ]
      \\
      \hspace{-1cm}
   \scalebox{0.7}{$  \ell
      \,\otimes\,
      \big(
        \gamma
        ;\,
        g_{n-1}
        ,\,
        \cdots
        ,\,
        g_0
      \big)
      $}
     \hspace{-4cm}
     \ar[
        r,
        phantom,
        "{
          \overset{
            \mathclap{
            \raisebox{2pt}{
            \scalebox{.7}{
               \eqref{EilenbergMaclaneFormula}
            }
            }
            }
          }{
            \longmapsto
          }
         }"
      ]
      &
      \hspace{-2cm}
        \scalebox{0.7}{$
        \hspace{-2cm}
      \underset{
        0 \leq j \leq n
      }{
        \scalebox{1.3}{$\sum$}
      }
      \,
      (-1)^{j}
      \cdot
      \Big(
      s_{(0,\cdots, \widehat{k}, \cdots,n)} \ell
      ,\,
      s_k
      \big(
        \gamma
        ;\,
        g_{n-1}
        ,\,
        \cdots
        ,\,
        g_0
      \big)
      \Big)
      $}
      \hspace{-1.8cm}
      \ar[
        r,
        |->,
        shorten=50pt,
        "{
          \scalebox{.7}{
          Lem. \ref{EvaluationMapOnNerveOfInertiaGroupoid}
          }
        }"{yshift=2pt, sloped}
      ]
           &
      \hspace{-50pt}
    \scalebox{0.7}{$     \underset{
        0 \leq j \leq n
      }{
        \scalebox{1.3}{$\sum$}
      }
      \,
      (-1)^{j}
      \cdot
      \big(
        g_{n-1}
        ,\,
        \cdots
        ,\,
        g_{n-k}
        ,\,
        \mathrm{Ad}_k(\gamma)
        ,\,
        g_{n-(k+1)}
        ,\
        \cdots
        ,\,
        g_0
      \big)
      $}
    \end{tikzcd}
    }
    \hspace{-3cm}
  \end{equation}

  \vspace{-2mm}
  \noindent
  followed by restriction to the coefficient of $\ell$
  \eqref{CircleCell}
  on the left. This manifestly yields the formula
  \eqref{TransgressionFormula}.
%  to be shown.
\end{proof}

\newpage

%%%%%%%%%%%%%%%%%%%%%%%%%%%%%%%%%%%%%%%%%%%%%%%%%%%%%%%%%%%%%%%
\section{Integral 4-classes of equivariant 4-Cohomotopy}
\label{CharacteristicFourClassesOfEquivariantFourCohomotopy}
%%%%%%%%%%%%%%%%%%%%%%%%%%%%%%%%%%%%%%%%%%%%%%%%%%%%%%%%%%%%%%%

Here we prove that the integral 4-class $\widetilde {\Gamma}_4$ which underlies any tangentially twisted 4-Cohomotopy cocycle {\it at an ADE-singularity} is that which classifies the ``Platonic'' 2-group extensions of \cite{EpaGanter16}. These 4-classes (equivalently in their incarnation as 3-classes after cyclification) are the twists of quasi-elliptic cohomology which are ``predicted'' by {\it Hypothesis H} in the sense explained at the end of \cref{IntroductionAndOverview}.

\medskip

\begin{notation}[Canonical representation 4-sphere of finite subgroup of $\mathrm{Sp}(1)$]
  \label{CanonicalRepresentation4SphereOfFiniteSubgroupOfSU2}
  Given a finite subgroup $G \xhookrightarrow{i} \mathrm{Sp}(1)$ (Prop. \ref{FiniteSubgroupsOfSU2}),
  we write
  $$
    G \acts \; \mathbb{H}
    \;\coloneqq\;
    \mathrm{res}_i
    \big(
      \mathrm{Sp}(1) \acts \, \mathbb{H}
    \big)
    \,\in\,
    \Actions{G}(\VectorSpaces)
  $$
  for the linear representation which is the restriction along $i$
  of the defining linear action of $\mathrm{Sp}(1)$ on the
  real vector space $\mathbb{H} \simeq_{\mathbb{R}} \mathbb{R}^4$
  underlying the algebra of quaternions. Accordingly, we write
  \vspace{-2mm}
  $$
    G \acts \; \shape S^\mathbb{H}
    \;\in\;
    \Actions{G}(\TopologicalSpaces)
    \xrightarrow{\mathrm{Sing}}
    \Actions{G}(\SimplicialSets)
  $$

  \vspace{-2mm}
\noindent
  for the shape of the corresponding representation sphere.
\end{notation}

\begin{theorem}[Equivariant integral characteristic classes of ADE-equivariant 4-Cohomotopy]
  For $G \xhookrightarrow{i} \mathrm{Sp}(1) \simeq \mathrm{SU}(2) \simeq \mathrm{Spin}(3)$
  a finite subgroup
  (Prop. \ref{FiniteSubgroupsOfSU2}),
  the group of
  equivariant integral characteristic 4-classes of
  ADE-equivariant 4-Cohomotopy is
  $$
    H^4_G
    \big(
      S^{\mathbb{H}},
      \,
      \mathbb{Z}
    \big)
    \;\;
      \simeq
    \;\;
    \mathbb{Z}
    \,\oplus\,
    \mathbb{Z}_{\left\vert G \right\vert}\;.
  $$
\end{theorem}
\begin{proof}
  By
  Lemma \ref{EquivariantIntegralCharacteristicClassesOfEquivariantCohomotopy}
  with
  Lemma \ref{TheIntegral4CohomologyOfHomotopyQuotientOf4SphereByADEGroup}.
\end{proof}

The following Lemma \ref{EquivariantIntegralCharacteristicClassesOfEquivariantCohomotopy} may be taken to be the definition of ordinary equivariant cohomology, or alse of a standard fact of Bredon cohomology with invariant coefficients. We include the following proof just to showcase how this fits into the singular-cohesive formalism of \cite{SS20OrbifoldCohomology}\cite{SS21EPB}:
\begin{lemma}[Equivariant integral characteristic classes of equivariant Cohomotopy]
  \label{EquivariantIntegralCharacteristicClassesOfEquivariantCohomotopy}
  The equivariant integral characteristic classes of
  equivariant Cohomotopy
  are naturally identified with the
  ordinary cohomology of
  the homotopy quotient of the sphere.
\end{lemma}
\begin{proof}
We have the following sequence of natural equivalences of
hom-$\infty$-groupoids:
$$
\hspace{-3mm}
  \def\arraystretch{1.2}
  \begin{array}{lll}
  \left\{\!\!\!\!\!\!\!\!\!\!\!\!\!
  \begin{tikzcd}[column sep=-4pt, row sep=4pt]
    \phantom{\smooth\;}
    \shape \orbisingular
    \big(
      S^{V}
        \!\sslash\!
      G
    \big)
    \ar[
      rr,
      dashed
    ]
    \ar[dr]
    &&
    \orbisingular
    \big(
      \mathbf{B}^4 \mathbb{Z}
      \!\sslash\!
      G
    \big)
    \ar[dl]
    \\
    &
    \orbisingular \mathbf{B}G
  \end{tikzcd}
  \!\!\!\!\! \right\}
   \;
   \simeq
   \!\!\!
  &
  \!\!
  \left\{\!\!\!\!\!\!\!\!\!\!\!
  \begin{tikzcd}[column sep=-4pt, row sep=4pt]
    \phantom{\smooth\;}
    \shape \orbisingular
    \big(
      S^{V}
        \!\sslash\!
      G
    \big)
    \ar[
      rr,
      dashed
    ]
    \ar[dr]
    &&
    \orbisingular
    \big(
      \mathbf{B}^4 \mathbb{Z}
      \times
      \mathbf{B}G
    \big)
    \ar[dl]
    \\
    &
    \orbisingular \mathbf{B}G
  \end{tikzcd}
  \!\!\!\!\! \right\}
  &
\proofstep{
  trivial action
  on coefficients
}
  \\
  &
  \!\!\!\!\!\!\!\!\!\!\!
  \simeq
  \left\{\!\!\!\!\!\!\!\!\!\!
  \begin{tikzcd}[column sep=-4pt, row sep=4pt]
    \phantom{\smooth\;}
    \shape \orbisingular
    \big(
      S^{V}
        \!\sslash\!
      G
    \big)
    \ar[
      rr,
      dashed
    ]
    \ar[dr]
    &&
    \big(
      \orbisingular
      \mathbf{B}^4 \mathbb{Z}
    \big)
      \times
    \big(
      \orbisingular
      \mathbf{B}G
    \big)
    \ar[dl]
    \\
    &
    \orbisingular \mathbf{B}G
  \end{tikzcd}
  \!\!\!\!\! \right\}
  &
  \proofstep{since $\orbisingular$ is right adjoint}
  \\
  &
  \!\!\!\!\!\!\!\!\!\!\!
  \simeq
  \left\{
  \begin{tikzcd}[column sep=14pt]
    \phantom{\smooth\;}
    \shape \orbisingular
    \big(
      S^{V}
       \!\sslash\!
      G
    \big)
    \ar[rr, dashed]
    &&
    \mathbf{B}^4 \mathbb{Z}
  \end{tikzcd}
\!\!\!  \right\}
  & \!\!
  \proofstep{by base change \eqref{BaseChange}}
  \\
&
\!\!\!\!\!\!\!\!\!\!\!
  \simeq
  \left\{
  \begin{tikzcd}[column sep=14pt]
    \smooth \, \shape \orbisingular
    \big(
      S^{V}
        \!\sslash\!
      G
    \big)
    \ar[rr, dashed]
    &&
    \mathbf{B}^4 \mathbb{Z}
  \end{tikzcd}
  \!\!\! \right\}
  & \!\!
  \proofstep{by singular cohesion $\smooth \dashv \orbisingular$}
  \\
  &
  \!\!\!\!\!\!\!\!\!\!\!
  \simeq
  \left\{
  \begin{tikzcd}[column sep=14pt]
    \shape \,\smooth \orbisingular
    \big(
      S^{V}
        \!\sslash\!
      G
    \big)
    \ar[rr, dashed]
    &&
    \mathbf{B}^4 \mathbb{Z}
  \end{tikzcd}
  \!\!\!\right\}
  & \!\!
  \proofstep{by \cite[Lem. 3.67]{SS20OrbifoldCohomology}}
  \\
&
\!\!\!\!\!\!\!\!\!\!\!
  \simeq
  \left\{
  \begin{tikzcd}[column sep=14pt]
    \phantom{\orbisingular\,}
    \shape \, \smooth
    \big(
      S^{H}
        \!\sslash\!
      G
    \big)
    \ar[rr, dashed]
    &&
    \mathbf{B}^4 \mathbb{Z}
  \end{tikzcd}
  \!\!\! \right\}
  &\!\!
  \proofstep{by singular cohesion}
  \\
&
\!\!\!\!\!\!\!\!\!\!\!
  \simeq
  \left\{
  \begin{tikzcd}[column sep=14pt]
    \phantom{\orbisingular \, \smooth\;}
    \shape
    \big(
      S^{V}
        \!\sslash\!
      G
    \big)
    \ar[rr, dashed]
    &&
    \mathbf{B}^4 \mathbb{Z}
  \end{tikzcd}
  \!\!\! \right\}
  & \!\!
  \proofstep{since $S^{V} \!\sslash\! G$ is smooth.}
  \end{array}
$$

\vspace{-7mm}
\end{proof}

\begin{lemma}[Integral 4-cohomology of homotopy quotient of 4-sphere by ADE-group]
  \label{TheIntegral4CohomologyOfHomotopyQuotientOf4SphereByADEGroup}
  Let $G \xhookrightarrow{i} \mathrm{Sp}(1)$ be a finite subgroup
  (Prop. \ref{FiniteSubgroupsOfSU2}).
  Then the integral 4-cohomology of
  the homotopy quotient of its canonical representation sphere
  (Nota. \ref{CanonicalRepresentation4SphereOfFiniteSubgroupOfSU2}) is
  \begin{equation}
    \label{Integral4CohomologyOfHomotopyQuotientOf4SphereByFiniteSubgroupOfSU2}
    H^4
    \big(
      (\shape S^\mathbb{H}) \!\sslash\! G,
      \,
      \mathbb{Z}
    \big)
    \;\simeq\;
    \mathbb{Z} \oplus (\mathbb{Z}/\left\vert G\right\vert)
    \,,
  \end{equation}
  where the first summand is the cohomology of the fiber, and the second
  is the group cohomology of $G$, in that we have a split short exact sequence
  in cohomology
  \vspace{-3mm}
  \begin{equation}
    \label{ShortExactSequenceOfCohomologyOfBorelConstructionOfADEGroupActingOn4Sphere}
    \begin{tikzcd}[row sep=2pt]
      0
      &
      \overset{
        \mathbb{Z}
      }{
      \overbrace{
      H^4
      \big(
        S^4;\, \mathbb{Z}
      \big)
      }
      }
      \ar[l]
      &&
      H^4
      \big(
        \shape S^{\mathbb{H}} \!\sslash\! G;
        \,
        \mathbb{Z}
      \big)
      \ar[
        ll,
        "q^\ast"{above}
      ]
      &&
      \overset{ \mathbb{Z}_{\left\vert G \right\vert} }{
        \overbrace{
          H^4(B G;\, \mathbb{Z})
        }
      }
      \ar[
        ll,
        "\rho^\ast"{above}
      ]
      &
      0
      \ar[l]
      \\
      &
      \shape S^4
      \ar[
        rr,
        "q"
      ]
      &&
      \shape
      \big(
        S^{\mathbb{H}} \!\sslash\! G
      \big)
      \ar[
        rr,
        "\rho"
      ]
      &&
      B G
    \end{tikzcd}
  \end{equation}

  \vspace{-2mm}
\noindent  induced by the homotopy fiber sequence of the Borel construction
  (Lem. \ref{EquivariantIntegralCharacteristicClassesOfEquivariantCohomotopy}), shown at
  the bottom in \eqref{ShortExactSequenceOfCohomologyOfBorelConstructionOfADEGroupActingOn4Sphere}.
\end{lemma}
\begin{proof}
  Observe that the fundamental group $\pi_1(B G) \,\simeq\, G$ of the
  base space of the fibration \eqref{ShortExactSequenceOfCohomologyOfBorelConstructionOfADEGroupActingOn4Sphere}
  acts trivially on the cohomology
  $$
    H^n\big(\shape S^4, \mathbb{Z}\big)
    \;\simeq\;
    \left\{
    \begin{array}{cl}
      \mathbb{Z} & \mbox{for $n \in \{0,4\}$}
      \\
      0 & \mbox{otherwise}
    \end{array}
    \right.
  $$
  of the fiber space. This is because:

  \begin{itemize}

    \item
    we have a $G$-equivariant isomorphism
    between the representation sphere of $\mathbb{H}$
    and the unit sphere in $\mathbb{R} \oplus \mathbb{H}$
    (by stereographic projection, e.g. \cite[p. 2]{MP04}):
    $
      S^{\mathbb{H}}
      \;\simeq_G\;
      S(\mathbb{R} \oplus \mathbb{H})
      \,;
    $

    \item
    the action of $\mathrm{Sp}(1)$ on $\mathbb{H} \simeq_{{}_{\mathbb{R}}} \mathbb{R}^4$
    is through the defining action of $\mathrm{SO}(4)$
    (since the quaternions are a normed division algebra, so that left multiplication
    by unit-norm quaternions $q \in \mathrm{Sp}(1) = S(\mathbb{H})$ preserves the norm),
    whence the action of $\mathrm{Sp}(1)$ on $\mathbb{R} \oplus \mathbb{H} \simeq_{{}_{\mathbb{R}}}$
    is through $\mathrm{SO}(5)$ (e.g. \cite[Rem. A.8]{HSS18});

    \item
    by the Hopf degree theorem and the de Rham theorem,
    the generator of $\mathbb{Z} \simeq H^4(S^4, \mathbb{Z})$
    is represented by the standard volume form on $S^4$, which is evidently
    preserved by the $\mathrm{SO}(5)$-action (e.g. \cite[p. 31]{BSS18}).

  \end{itemize}

  \noindent
  Therefore (e.g. \cite[Thm. 1.14]{HatcherSpectralSequences}),
  we have a cohomological Serre spectral sequence of the form
   \vspace{-2mm}
  \begin{equation}
    \label{SSSForIntegralCohomologyOfHomotopyQuotientOf4SphereByFiniteSubgroupOfSU2}
    E_2^{p,q}
    \;=\;
    H^p
    \big(
      B G;
      \,
      H^q(S^4, \mathbb{Z})
    \big)
    \;\;
      \Rightarrow
    \;\;
    H^{p + q}
    \big(
      S^\mathbb{H} \!\sslash\! G;
      \,
      \mathbb{Z}
    \big)
    \,.
  \end{equation}

  \vspace{-2mm}
  \noindent
  From
  Prop. \ref{IntegralGroupCohomologyOfFiniteSubgroupsOfSU2}
  with
  Ex. \ref{IncarnationsOfGroupCohomology},
  it follows that its second page looks as follows:

  \vspace{-5mm}
  $$
    E_2^{\bullet, \bullet}
    \;\;\;=\;\;\;
    \begin{tikzcd}[row sep={between origins, 21pt}, column sep={between origins, 21pt}]
      {}
      &
      {}
      \\[-12pt]
      {}
          &
      {}
      \\
      &
      \vdots
      &
      \vdots
      &
      \vdots
      &
      \vdots
      &
      \vdots
      &
      \vdots
      &
      \vdots
      &
      \vdots
      \\
      {}
      \ar[
        ddddddrrrrrr,
        -,
        line width=19pt,
        shift left=1pt,
        opacity=.15,
        color=blue
      ]
      &
      0
      &
      0
      &
      0
      &
      0
      &
      0
      &
      0
      &
      0
      &
      0
      &
      \cdots
      \\
      &
      \mathbb{Z}
      \ar[
        drr
      ]
      &
      0
      %\ar[
      %  drr
      %]
      &
      G^{\mathrm{ab}}
      \ar[
        drr
      ]
      &
      0
      &
      \mathbb{Z}_{\left\vert G\right\vert}
      \ar[
        drr
      ]
      &
      0
      &
      G^{\mathrm{ab}}
      \ar[
        drr
      ]
      &
      0
      &
      \cdots
      \\
      &
      0
      &
      0
      &
      0
      &
      0
      &
      0
      &
      0
      &
      0
      &
      0
      &
      \cdots
      \\
      &
      0
      &
      0
      &
      0
      &
      0
      &
      0
      &
      0
      &
      0
      &
      0
      &
      \cdots
      \\
      &
      0
      &
      0
      &
      0
      &
      0
      &
      0
      &
      0
      &
      0
      &
      0
      &
      \cdots
      \\
      \ar[
        rrrrrrrrrr,
        color=lightgray,
        shift right=9pt,
        "B G"{below, color=black, pos=.95}
      ]
      &
      \mathbb{Z}
      &
      0
      &
      G^{\mathrm{ab}}
      &
      0
      &
      \mathbb{Z}_{\left\vert G\right\vert}
      &
      0
      &
      G^{\mathrm{ab}}
      &
      0
      &
      \cdots
      &
      {}
      \\
      &
      \ar[
        uuuuuuuuu,
        color=lightgray,
        shift left=5pt,
        "S^4"{left, color=black, pos=.93}
      ]
      &{}&{}&{}&{}&{}
    \end{tikzcd}
  $$

  \vspace{-2mm}
\noindent  This shows that every differential on this and on every following page
  has zero domain or zero codomain, so that the spectral sequence
  collapses already on this page:
  $E^{\bullet, \bullet}_\infty \,\simeq\, E^{\bullet,\ \bullet}_2$.
  By its convergence
  \eqref{SSSForIntegralCohomologyOfHomotopyQuotientOf4SphereByFiniteSubgroupOfSU2},
  this means that we have a short exact sequence of the form
  \eqref{ShortExactSequenceOfCohomologyOfBorelConstructionOfADEGroupActingOn4Sphere}.
  Since $\mathrm{Ext}(\mathbb{Z},-) = 0$, the claim
  \eqref{Integral4CohomologyOfHomotopyQuotientOf4SphereByFiniteSubgroupOfSU2}
  follows.
\end{proof}

\begin{proposition}[Shifted integral characteristic 4-class for equivariant 4-Cohomotopy {(\cite[\S 3.4]{FSS19TwistedCohomotopy})}]
  \label{ShiftedIntegralCharacteristic4ClassForEquivariant4Cohomotopy}
  There exists an integral class
  \vspace{-1mm}
  \begin{equation}
    \label{ShiftedIntegralUniversal4Class}
    \widetilde \Gamma_4
    \;\coloneqq\;
    \tfrac{1}{2}\rchi_4 + \tfrac{1}{4}p_1
    \;\;
    \in
    \;\;
    H^4
    \big(
      \shape
        \HomotopyQuotient
          { S^{\mathbb{H}} }
          { \mathrm{Sp}(2) }
      ;\,
      \mathbb{Z}
    \big)
  \end{equation}
  whose rational image is the sum of half the Euler class with
  1/4th of the Pontrjagin class, hence whose restriction to the 4-sphere
  fiber is the volume class
  \vspace{-5mm}
  \begin{equation}
    \label{ShiftedIntegral4ClassRestrictsToVolumeClassOn4Sphere}
    \begin{tikzcd}[column sep=1pt, row sep=small]
      H^4
      \big(
        \shape S^4;\,
        \mathbb{Z}
      \big)
      &\ni&
      \scalebox{0.8}{$
        {[\mathrm{vol}_{S^4}]}
      $}
      \\
      H^4
      \Big(
        \shape
        \big(
          S^{\mathbb{H}} \!\sslash\! \mathrm{Sp}(2)
        \big);
        \,
        \mathbb{Z}
      \Big)
      \ar[u]
      &
      \ni
      &
      \scalebox{0.8}{$
         \widetilde{\Gamma}_4
      $}
      \ar[
        u,
        |->
      ]
    \end{tikzcd}
  \end{equation}
\end{proposition}

\begin{proposition}[Universal shifted integral 4-flux restricts to generator on ADE-singularity]
  \label{UniversalShiftedIntegral4FluxRestrictsToGeneratorOnADESingularity}
  For a finite subgroup of $\mathrm{Sp}(1)$ (Prop. \ref{FiniteSubgroupsOfSU2}),
  embedded as the {\it left diagonal entry} into $\mathrm{Sp}(2)$
  \vspace{-3mm}
  \begin{equation}
    \begin{tikzcd}[row sep=-3pt, column sep=2pt]
      G
      \ar[
        rr,
        hook,
        "i"
      ]
    \ar[
      rrrrrr,
      rounded corners,
      hook,
      to path={
           -- ([yshift=+9pt]\tikztostart.north)
           --node[above]{
               \scalebox{.7}{$j$}
             } ([yshift=+8pt]\tikztotarget.north)
           -- (\tikztotarget.north)}
    ]
      &&
      \mathrm{Sp}(1)
      \ar[
        rr,
        hook,
        "l"
      ]
      &&
      \mathrm{Sp}(1) \times \mathrm{Sp}(1)
      \ar[
        rr,
        hook
      ]
      &&
      \mathrm{Sp}(2)
      \\
     \scalebox{0.8}{$ q $}
      &\longmapsto&
    \scalebox{0.8}{$  q $}
      &\longmapsto&
    \scalebox{0.8}{$  (q,1) $}
      &\longmapsto&
     \scalebox{0.7}{$
      \begin{pmatrix}
          \raisebox{1.5pt}{q} \, \; 0
          \\
          0 \;  1
    \end{pmatrix}
    $}
    \end{tikzcd}
  \end{equation}

  \vspace{-3.5mm}
\noindent
  the pullback of $\widetilde \Gamma_4$ \eqref{ShiftedIntegralUniversal4Class}
  along $\shape (S^{\mathbb{H}} \sslash j)$
  is the element
  $(1,[1])$ according to Lem. \ref{TheIntegral4CohomologyOfHomotopyQuotientOf4SphereByADEGroup}:
  \vspace{-2mm}
  \begin{equation}
    \label{Unit4FluxOnADERepresentationSphere}
    \begin{tikzcd}[row sep=-5pt]
      \mathbb{Z}
      \oplus
      \mathbb{Z}_{\left\vert G \right\vert}
      \ar[
        r,
        phantom,
        "\simeq"
      ]
      &[-6pt]
      H^4
      \big(
        \shape S^{\mathbb{H}} \sslash G;
        \,
        \mathbb{Z}
      \big)
      &&
      H^4
      \Big(
        \big(
          \shape S^{\mathbb{H}} \sslash \mathrm{Sp}(2)
        \big);
        \,
        \mathbb{Z}
      \Big).
      \ar[
        ll,
        "B j^\ast"{above}
      ]
      \\
      \scalebox{0.8}{$
        (1,[1])
      $}
      & &&
      \scalebox{0.8}{$
        \widetilde{\Gamma}_4
      $}
      \ar[
        lll,
        |->
      ]
    \end{tikzcd}
  \end{equation}
\end{proposition}
\begin{proof}
  Consider the corresponding morphism of Borel constructions
  \eqref{BorelConstructionAdjunction}
  and observe the following pullbacks of cohomology generators
  through this diagram:
  \vspace{-3mm}
  \begin{equation}
    \begin{tikzcd}[row sep=14pt]
      \scalebox{0.8}{$
        (1,[0])
      $}
      &[-20pt]
      &&
      &[-20pt]
      \scalebox{0.8}{$
        {[\mathrm{vol}_{S^4}]}
      $}
      \ar[
        llll,
        |->,
        "\scalebox{.8}{
          Lem. \ref{TheIntegral4CohomologyOfHomotopyQuotientOf4SphereByADEGroup}
        }"{above}
      ]
      \\[-20pt]
      &
      S^4
        \ar[
          rr,-,
          shift left=1pt,
        ]
        \ar[
          rr,-,
          shift right=1pt,
        ]
        \ar[d]
        &&
      S^4
      \ar[d]
      \\
      \scalebox{0.8}{$
        (0,[1])
      $}
      &
      \shape
      \big(
        S^{\mathbb{H}}
          \!\sslash\!
        G
      \big)
      \ar[
        rr,
        "
          \scalebox{.8}{$
            \shape(S^{\mathbb{H}} \!\sslash\! j )
          $}
        "
      ]
      \ar[d]
      &&
      \shape
      \big(
        S^{\mathbb{H}}
          \!\sslash\!
        \mathrm{Sp}(1)
      \big)
      \ar[d]
      &
      \scalebox{0.8}{$
        \widetilde{\Gamma}_4
      $}
      \ar[
        uu,
        |->,
        "\scalebox{.8}{\eqref{ShiftedIntegral4ClassRestrictsToVolumeClassOn4Sphere}}"{right}
      ]
      \\
      &
      B G
      \ar[
        rr,
        "Bi"
      ]
      &&
      B\mathrm{Sp}(1)
      \\[-20pt]
      \scalebox{0.8}{$ [1] $}
      \ar[
        uu,
        |->,
        "\scalebox{.8}{
          Lem. \ref{TheIntegral4CohomologyOfHomotopyQuotientOf4SphereByADEGroup}
        }"{left}
      ]
      & && &
      \scalebox{0.8}{$
        \tfrac{1}{4}p_1
      $}
      \ar[
        llll,
        |->,
        "\scalebox{.8}{Prop. \ref{PullbackOfSecondChernClassToADESubgroup}}"
      ]
    \end{tikzcd}
  \end{equation}

\vspace{-8mm}
\end{proof}

\begin{remark}[Background integral 4-flux at flat ADE-singularity]
  \label{BackgroundIntegral4FluxAtFlatADESingularity}
  Noticing that $G \acts \, S^{\mathbb{H}}$ has  a fixed point
  (in fact two), hence that  $S^{\mathbb{H}} \!\sslash\! G \xrightarrow{\rho} B G$
  has a section, this means that in the vicinity of any $G$-orbifold-singularity
  $\mathcal{X} \simeq  \mathbb{R}^n \!\sslash\! G$,
  hence with $\shape \mathcal{X} \,\simeq\, \HomotopyQuotient{ \ast }{G} $,
  there is
  integral {\it background} C-flux of
  value $[1] \in H^4(B G; \mathbb{Z})$
  \vspace{-2mm}
  \begin{equation}
    \label{BackgroundOrbiCFlux}
    \begin{tikzcd}[row sep=small, column sep=large]
      \HomotopyQuotient
        { \ast }
        { G }
      \ar[
        rr,
        dashed,
        "
          \mbox{
            \tiny
            \color{greenii}
            \bf
            \begin{tabular}{c}
              unique cocycle in
              \\
              equivariant
              4-Cohomotopy
            \end{tabular}
          }
        "{below}
      ]
      \ar[
        d,
        "{\sim}"{sloped}
      ]
      \ar[
      rrrr,
      rounded corners,
      to path={
           -- ([yshift=+9pt]\tikztostart.north)
           --node[above]{
               \scalebox{.7}{$
                 (0,[1])
               $}
             } ([yshift=+10pt]\tikztotarget.north)
           -- (\tikztotarget.north)}
      ]
      &[+20pt]
      &
      S^{\mathbb{H}} \!\sslash\! \mathrm{Sp}(2)
      \ar[
        d,
        "\, \rho"
      ]
      \ar[
        rr,
        "\widetilde \Gamma_4"
      ]
      &&
      \mathbf{B}^4 \mathbb{Z}
      \\
      \mathbf{B}G
      \ar[rr]
      &&
      \mathbf{B}\mathrm{Sp}(2)
    \end{tikzcd}
  \end{equation}

  \vspace{-2mm}
  \noindent
  This is hence the background value of
  ``M-theoretic discrete torsion'',
  in the terminology of \cite{Sharpe00}\cite{Seki01}, see also \cite[\S 4.6]{dBDHKMMS02}.
  The 2-gerbe over $\HomotopyQuotient{\ast}{G}$ classified by this background 4-flux
  is (the delooping of) the universal {\it Platonic 2-group}-extension
  of $G$, in the terminology of \cite{EpaGanter16}.
 However, the proper definition of brane charge {\it localized at} a singularity (i.e. disregarding charges that are ``escaping to infinity'') is given by
 the cohomology/cohomotopy of the {\it one-point compactification} of the transverse space $\HomotopyQuotient{ \mathbb{R}^n }{G}$ to the brane
 \cite[(8)]{SS19Tad}\cite[\S 2.1]{SS19Conf}\cite[(12)]{SS21MF}. For  flat M5-branes at an ADE-singularity this is again the 4-sphere orbifold
 $\HomotopyQuotient{ S^{\mathbb{H}} }{G}$ \cite{SS19Tad}, and the unit integral brane charge of a single one of these is \eqref{Unit4FluxOnADERepresentationSphere}:
 \vspace{-2mm}
  \begin{equation}
    \label{LocalizedBackgroundOrbiCFlux}
    \begin{tikzcd}[row sep=small, column sep=large]
      \HomotopyQuotient
        { S^{\mathbb{H}} }
        { G }
      \ar[
        rr,
        dashed,
        "
          \mbox{
            \tiny
            \color{greenii}
            \bf
            \begin{tabular}{c}
              unit cocycle in
              \\
              equivariant
              4-Cohomotopy
            \end{tabular}
          }
        "{below}
      ]
      \ar[d]
      \ar[
      rrrr,
      rounded corners,
      to path={
           -- ([yshift=+9pt]\tikztostart.north)
           --node[above]{
               \scalebox{.7}{$
                 (1,[1])
               $}
             } ([yshift=+10pt]\tikztotarget.north)
           -- (\tikztotarget.north)}
      ]
      &[+20pt]
      &
      S^{\mathbb{H}} \!\sslash\! \mathrm{Sp}(2)
      \ar[
        d,
        "\, \rho"
      ]
      \ar[
        rr,
        "\widetilde \Gamma_4"
      ]
      &&
      \mathbf{B}^4 \mathbb{Z}
      \\
      \mathbf{B}G
      \ar[rr]
      &&
      \mathbf{B}\mathrm{Sp}(2)
    \end{tikzcd}
  \end{equation}
\end{remark}

\begin{remark}[{\bf Conclusion}.]
\label{ConclusionOnMBraneChargeTwist}
With this we have finished compiling the ingredients for the construction/application indicated at the end of \cref{IntroductionAndOverview}:
The unit integral background 4-flux in M-theory localized at an ADE-singularity, as  predicted by {\it Hypothesis H}, is \eqref{LocalizedBackgroundOrbiCFlux};
and its double dimensional reduction
\eqref{ReductionOfFourCohomotopyCharges}
to 3-flux according to \eqref{ReductionOfFourCohomotopyCharges} is, by \eqref{TransgressionMapViaCyclification}, the corresponding trangression class.
This is the twist for ADE-equivariant quasi-elliptic cohomology to be used in \eqref{QuasiEllitpicCohomologyOfFourSphere} for measuring M-brane charge in equivariant quasi-elliptic cohomology and thus potentially relating to the elliptic genus of the M5-brane.
\end{remark}

\medskip

\appendix

%%%%%%%%%%%%%%%%%%%%%%%%%%%%%%%%%%%
\section{Appendix: Technical background}
\label{TechnicalMaterial}
%%%%%%%%%%%%%%%%%%%%%%%%%%%%%%%%%%%%

For use in the main text, here we recall and reference some technical background:

\noindent
- on basic facts of simplicial homotopy theory (\cref{SomeSimplicialHomotopyTheory})

\noindent
- on basic notions in cohesive $\infty$-topos theory (\cref{NotionsFromCohesiveInfinityToposTheory}).

%%%%%%%%%%%%%%%%%%%%%%%%%%%%%%%%%%%%%%%%%%%%%%
\subsection{Some simplicial homotopy theory}
\label{SomeSimplicialHomotopyTheory}
%%%%%%%%%%%%%%%%%%%%%%%%%%%%%%%%%%%%%%%%%%%%%%

After recalling the combinatorics of products of simplicial sets in streamlined form,
here we compile some technical background on simplicial groups and their actions and prove some useful facts that do not seem to be easily citable from the literature. For more along these lines see \cite[\S 3.1.2]{SS21EPB}, whose notation we follow. For general background on simplicial homotopy theory see \cite{GoerssJardine99}\cite[\S 1]{Rezk22}.

\medskip

\noindent
{\bf Categories and simplicial sets.} Just to set up our notation:

\begin{notation}[Mapping objects]
  \label{MappingObjects}
 $\,$

  \noindent
  \begin{itemize}[leftmargin=.5cm]
  \item[--]
  Generally, $\Homs{}{-}{-}$ is to denote hom-sets in a given category, while $\Maps{}{-}{-}$ denotes internal hom-objects (for cartesian closed categories), defined to yield natural bijections:
  \begin{equation}
    \label{InternalHomAdjunction}
    \Homs{\big}
      { X \times Y }
      { Z }
    \xleftrightarrow
      [\sim]
      {\; \widetilde{(-)} \;}
    \Homs{}
      { X }
      {
        \Maps{}
          { Y }
          { Z }
      }
    \,.
  \end{equation}
  \item[--]
  In particular, for a pair of (small) categories\footnote{To be pedantic, in the present context by a ``category'' we mean, as is usual, a {\it strict} category, i.e. with a fixed set of objects (which is more information than retained in the equivalence class). It is (only) on these strict categories that the simplicial nerve (Ntn. \ref{NotationForSimplicialSets}) is defined.} $\mathcal{X}, \mathcal{Y} \,\in\, \Categories$, the notation $\Homs{}{\mathcal{C}}{\mathcal{D}}$ denotes the plain set of functors between them, while $\Maps{}{\mathcal{X}}{\mathcal{Y}}$ denotes the category of such functors with natural transformations between them, schematically:
  \vspace{-5mm}
  \begin{equation}
    \label{FunctorCategory}
    \Homs{}
      { \mathcal{X} }
      { \mathcal{Y} }
    \;=\;
    \big\{
      \mathcal{X}
      \xrightarrow{\;\;\;\; F \;\;\;\;}
      \mathcal{Y}
    \big\}
    \,,
    \hspace{1cm}
    \Maps{}
      { \mathcal{X} }
      { \mathcal{Y} }
    \;=\;
    \Bigg\{\!\!\!\!
    \begin{tikzcd}
     \mathcal{X}
     \ar[
       rr,
       bend left=25pt,
       "{ F }",
       "{  }"{swap, name=s}
     ]
     \ar[
       rr,
       bend right=25pt,
       "{ F' }"{swap},
       "{  }"{name=t}
     ]
     \ar[
       from=s,
       to=t,
       Rightarrow,
       shorten=-1pt,
       "{ \phi }"{pos=.4}
     ]
     &&
     \mathcal{Y}
    \end{tikzcd}
   \!\!\!\! \Bigg\}
    \,.
  \end{equation}

  \vspace{-2mm}
  \item[--]
  We write
  \begin{equation}
    \label{EvaluationMap}
    \mathrm{ev}^A_X
    \;:=\;
    \widetilde{
      \mathrm{id}_{
        \Maps{}{X}{A}
      }
    }
    \;:\;
      X \times \Maps{}{ X }{ A }
      \xrightarrow{\phantom{--}}
      A
  \end{equation}
  for the {\it evaluation map} on these mapping objects, i.e. the adjunct \eqref{InternalHomAdjunction} of the identity on $\Maps{}{X}{A}$.

  Notice that the naturality of \eqref{InternalHomAdjunction} implies that adjunct of the application of the mapping
  object functor to a morphism equals the pre-composition of that morphism with the evaluation map:
  \vspace{-2mm}
  \begin{equation}
    \label{AdjunctOfMapsViaEvaluation}
    f : A\to B
    \hspace{30pt}
    \vdash
    \hspace{30pt}
    \widetilde{\Maps{}{X}{f}}
    \;=\;
    f \circ \mathrm{ev}^A_X
    \;:\;
    X \times \Maps{}{X}{A}
    \xrightarrow{\phantom{--}}
    B\;.
  \end{equation}
  \end{itemize}
\end{notation}

\begin{notation}[Notation for simplicial sets]
\label{NotationForSimplicialSets}
We write, essentially as usual:

\begin{itemize}[leftmargin=.5cm]
\setlength\itemsep{6pt}
\item[--] $[n] := \{ 0 \to 1 \to \cdots \to n \}$ for the category free on a sequence of $n$ composable morphisms,
for $n \in \mathbb{N}$;

so that functors $[n] \xrightarrow{\;}\mathcal{C}$ may be identified with paths of $n$ composable morphisms in $\mathcal{C}$;

\smallskip
\item[--]
\begin{itemize}
%[leftmargin=.5cm]
\item[$\bullet$]
$d^i_n \;\; :\, [n] \xrightarrow{\;} [n+1]$
\hspace{10pt}
$
  \begin{tikzcd}[
    column sep=18pt
  ]
    0
    \ar[r]
    \ar[
      r,
      -,
      line width=5pt,
      blue,
      opacity=.2
    ]
    &
    1
    \ar[r]
    \ar[
      r,
      -,
      line width=5pt,
      blue,
      opacity=.2
    ]
    &
    \cdots
    \ar[r]
    \ar[
      r,
      -,
      line width=5pt,
      blue,
      opacity=.2
    ]
    &
    i-1
    \ar[r]
    \ar[
      rr,
      bend right=20
    ]
    \ar[
      rr,
      -,
      line width=5pt,
      blue,
      opacity=.2,
      bend right=20
    ]
    &
    i
    \ar[r]
    &
    i+1
    \ar[r]
    \ar[
      r,
      -,
      line width=5pt,
      blue,
      opacity=.2
    ]
    &
    \cdots
    \ar[r]
    \ar[
      r,
      -,
      line width=5pt,
      blue,
      opacity=.2
    ]
    &
    n
    \ar[r]
    \ar[
      r,
      -,
      line width=5pt,
      blue,
      opacity=.2
    ]
    &
    n+1 \mathrlap{\;,}
  \end{tikzcd}
$

\item[$\bullet$]
$s^i_{n+1} \!: [n+1] \xrightarrow{\;} [n]$
\hspace{8pt}
$
  \begin{tikzcd}[
    column sep=18pt,
    row sep= 8pt
  ]
    0
    \ar[r]
    \ar[
      r,
      line width=5pt,
      -,
      blue,
      opacity=.2
    ]
    &
    1
    \ar[r]
    \ar[
      r,
      line width=5pt,
      -,
      blue,
      opacity=.2
    ]
    &
    \cdots
    \ar[r]
    \ar[
      r,
      line width=5pt,
      -,
      blue,
      opacity=.2
    ]
    &
    i
    \ar[d, Rightarrow, -]
    \ar[
      d,
      line width=5pt,
      -,
      blue,
      opacity=.2
    ]
    \\
    &&&
    i
    \ar[r]
    \ar[
      r,
      line width=5pt,
      -,
      blue,
      opacity=.2
    ]
    &
    \cdots
    \ar[r]
    \ar[
      r,
      line width=5pt,
      -,
      blue,
      opacity=.2
    ]
    &
    n-1
    \ar[r]
    \ar[
      r,
      line width=5pt,
      -,
      blue,
      opacity=.2
    ]
    &
    n \mathrlap{\;,}
  \end{tikzcd}
$

\item[$\bullet$]
$\iota_n \;\;: \;\;\;[n] \xrightarrow{\sim} [n]$
\hspace{21pt}
$
  \begin{tikzcd}[
    column sep=18pt
  ]
    0
    \ar[r]
    \ar[
      r,
      -,
      line width=5pt,
      blue,
      opacity=.2
    ]
    &
    1
    \ar[r]
    \ar[
      r,
      -,
      line width=5pt,
      blue,
      opacity=.2
    ]
    &
    \cdots
    \ar[r]
    \ar[
      r,
      -,
      line width=5pt,
      blue,
      opacity=.2
    ]
    &
    n-1
    \ar[r]
    \ar[
      r,
      -,
      line width=5pt,
      blue,
      opacity=.2
    ]
    &
    n
    \mathrlap{\;,}
  \end{tikzcd}
$

\end{itemize}

\medskip
for the {\it co-face-}, the {\it co-degeneracy-}, and the {\it identity-}functors between these categories,
shown on the right as paths of composable edges in the respective codomain (as such used around Lem. \ref{EvaluationMapOnNerveOfInertiaGroupoid}, cf. \hyperlink{FigureNonDegeneratedSimplicesInProductOfSimplices}{\it Figure NDS});

\item[--] $\Delta \xhookrightarrow{\;} \mathrm{Cat}$ for the full subcategory of the 1-category of (strict) categories on those of the form $[n]$, $n \in \mathbb{N}$;

\item[--]
$\SimplicialSets \,=\, \Presheaves(\Delta) \,=\, \Homs{}{\Delta^{\mathrm{op}}}{\Sets}$ for the category of simplicial sets;

\item[--]
$
  N
  \,:\,
    \mathrm{Cat}^{\mathrm{smll}}
    \xrightarrow{
        \mathcal{C}
        \,\mapsto\,
        \left(
          [n]
          \,\mapsto\,
          \Homs{}
            { [n] }
            { \mathcal{C} }
        \right)
    }
    \SimplicialSets
\;\;\;$
for the simplicial nerve functor;

\item[--] $\Delta[n] \,:=\, N[n] \,\in\, \SimplicialSets$ for the standard simplicial simplices.

\end{itemize}

\end{notation}

\begin{example}[Product categories]
  \label{ProductCategories}
  For $\mathcal{C}$, $\mathcal{D}$ a pair of small categories, their {\it product category} $\mathcal{C} \times \mathcal{D}$ has morphisms forming commuting squares as follows, for $f$ any morphism in $\mathcal{C}$ and $g$ any morphism in $\mathcal{D}$:
  $$
    \begin{tikzcd}[row sep=small, column sep=large]
      (c_1, d_1)
      \ar[
        rr,
        "{
          (f,\, \mathrm{id}_{d_1}) }"{above}
      ]
      \ar[
        dd,
        "{
          (\mathrm{id}_{c_1},\, g)
        }"{left}
      ]
      \ar[
        ddrr,
        "{
          (f,g)
        }"{sloped, description}
      ]
      &&
      (c_2, d_1)
      \ar[
        dd,
        "{
          (\mathrm{id}_{c_2},\, g)
        }"{right}
      ]
      \\
      \\
      (c_1,\, d_2)
      \ar[
        rr,
        "{ (f, \mathrm{id}_{d_2}) }"{below}
      ]
      &&
      (c_2,\, d_2)
    \end{tikzcd}
  $$
  This implies that natural transformations between functors $\mathcal{C} \xrightarrow{\;} \mathcal{D}$ are in bijection to
  functors $\mathcal{C} \times [1] \xrightarrow{\;} \mathcal{D}$, hence that the simplicial nerve of functor categories \eqref{FunctorCategory}
  is given by:
  \begin{equation}
    \label{NerveOfFunctorCategory}
    N \Maps{}{\mathcal{C}}{\mathcal{D}}
    \;\;:\;\;
    [n]
    \;\longmapsto\;
    \Homs{}
      { \mathcal{C} \times [n] }
      { \mathcal{D} }
    \,.
  \end{equation}
\end{example}

\medskip

\noindent
{\bf Products and mapping complexes of simplicial sets.}
Famously, the cartesian product of simplicial sets, despite its evident component-wise construction,

\vspace{-.3cm}
\begin{equation}
  \label{ProductOfSimplicialSets}
  X, Y \,\in\, \SimplicialSets
  \hspace{10pt}
  \vdash
  \hspace{10pt}
  \left\{
  \def\arraystretch{1.3}
  \begin{array}{l}
    X \times Y \,\in\, \SimplicialSets
    \\
    \big(
      X \times Y
    \big)_n
    \;=\;
    X_n \times Y_n
    \\
    s_k^{X\times Y}
    \;=\;
    \big(
      s_k^X
      ,\,
      s_k^Y
    \big)
    ,\,
    \;\;\;\;
    d_k^{X\times Y}
    \;=\;
    \big(
      d_k^X
      ,\,
      d_k^Y
    \big)
    \mathrlap{\,,}
  \end{array}
  \right.
\end{equation}
has a remarkably rich combinatorics, which we briefly recall as Prop. \ref{NonDegenerateSimpliciesInProductOfSimplicies}.
The phenomena dually induced on the mapping complexes (these we recall as Prop. \ref{SimplicialMappingComplex} below and use as Lemma \ref{EvaluationMapOnNerveOfInertiaGroupoid} in the main text) may not have received comparable attention yet; in any case this is what drives the proof of Thm. \ref{TransgressionInGroupCohomologyViaLooping} in the main text.

\medskip

The content of the following Prop. \ref{NonDegenerateSimpliciesInProductOfSimplicies} is classical (it is at least implicit in  the ``Eilenberg-MacLane formula'' \cite[(5.3)]{EilenbergMacLane53}, as reviewed for instance in \cite[\S 5.2.1 \& \S B.6]{Friedman20}, with exposition in \cite[\S 5]{Friedman08}) but we will have need for the following concise formulation (cf. also \href{https://kerodon.net/tag/00RH}[Kerodon, Ntn. 2.5.7.2]):

\begin{proposition}[Non-degenerate simplices in products product simplicial sets]
  \label{NonDegenerateSimpliciesInProductOfSimplicies}
  For $p, q \,\in\, \mathbb{N}$, the non-degenerate $p+q$-simplices
  in the Cartesian product of the standard $p$-simplex
  with the standard $q$-simplex,

  \vspace{-.3cm}
  $$
    \begin{tikzcd}
    \Delta[p + q]
    \ar[
      rr,
      "{
        (
          \sigma^{\mathrm{hor}},
          \sigma^{\mathrm{ver}}
        )
      }"
    ]
    &&
    \Delta[p] \times \Delta[q]
    \\
    \end{tikzcd}
  $$
  \vspace{-1cm}

  \noindent
  are precisely (cf. \hyperlink{FigureNonDegeneratedSimplicesInProductOfSimplices}{\it Figure NDS}):

  \begin{itemize}[leftmargin=.5cm]

  \item[--] as pairs of paths of edges $\big( \sigma^{\mathrm{hor}}[k,k+1],\, \sigma^{\mathrm{ver}}[k,k+1] \big)$,
  those for which at each step precisely one of the two edges in the pair is degenerate and the other a generating edge;

  \item[--] equivalently, the sequences of step numbers at which one or the other is non-trivial, these being the
  ``$(p,q)$-shuffles'' $(\mu, \nu)$ of elements $(0,1,2, \cdots, p+q-1)$:
  \vspace{-1mm}
  \begin{equation}
    \label{Shuffles}
    \begin{array}{ccc}
    \multicolumn{3}{c}{\proofstep{ \rm
      steps $k$ at which}}
    \\
    \proofstep{\rm $\sigma^{\mathrm{hor}}$ is non-trivial: }
    &&
    \proofstep{\rm $\sigma^{\mathrm{ver}}$ is non-trivial: }
    \\
  \proofstep{  $\mu_0 < \cdots < \mu_{p-1}$}
    &\phantom{A}&
 \proofstep{   $\nu_0 < \cdots < \nu_{q-1}$}
    \,,
    \end{array}
  \end{equation}
  \item[--] from which the given non-degenerate cell $\big(
    \sigma^{\mathrm{hor}}
    ,\,
    \sigma^{\mathrm{ver}}
  \big)$
  is recovered as:
  \vspace{-1mm}
  \begin{equation}
    \label{ShuffleFormulaForNonDegenerateSimplices}
    \def\arraystretch{1.3}
    \begin{array}{l}
    \sigma^{\mathrm{hor}}
    \,=\,
    s^{\nu}_{p+q}
    \;:=\;
    [p + q]
    \xrightarrow{\;\;
      s^{ \nu_0  }_{p+q}
    \;\;}
    [p + q - 1]
    \xrightarrow{\;\;
      s^{ \nu_1  }_{p+q-1}
    \;\;}
    [p + q - 2]
    \xrightarrow{\;\;\;}
      \cdots
    \xrightarrow{\;\;\;}
    [p + 1]
    \xrightarrow{\;\; s^{\nu_{q-1}}_{p+1} \;\;}
    [p]
    \\
    \sigma^{\mathrm{ver}}
    \,=\,
    s^{\mu}
    \; \;\;\;\; :=\;
    [p + q]
    \xrightarrow{\;\;
      s^{ \mu_0  }_{ p + q }
    \;\;}
    [p - 1 + q ]
    \xrightarrow{\;\;
      s^{ \mu_1  }_{ p -1 + q  }
    \;\;}
    [p - 2 + q ]
    \xrightarrow{\;\;}
      \cdots
    \xrightarrow{\;\;}
    [1 + q ]
    \xrightarrow{\;\; s^{\mu_{p-1}}_{1 + q} \;\;}
    [q]\;.
  \end{array}
  \end{equation}
  \end{itemize}

  \noindent
  For $X, Y\,\in\, \SimplicialSets$,
  the non-degenerate $p+q$-simplices in
  $X \times Y$
  \eqref{ProductOfSimplicialSets}
  are of the form
  \vspace{-1mm}
  \begin{equation}
    \label{FormulaForNonDegProductsOfSimplices}
    \Big(
    s^X_{\nu}(x_p)
    ,\;
    s^Y_{\mu}(y_q)
    \Big)
    \;\;
    :
    \;\;
    \begin{tikzcd}[row sep=-8pt]
      &
      \Delta[p+q]
      \ar[
        rr,
        "{ s^{ \nu }_{p+q} }"
      ]
      &&
      \Delta[p]
      \ar[rr, "x_p" ]
      &&
      X
      \\
      \Delta[p+q]
      \ar[
        r,
        shorten >=12pt,
        "{\mathrm{diag}}"{pos=.3}
      ]
      &
      \times
      &&
      \times
      &&
      \times
      \\
      &
      \Delta[p+q]
      \ar[
        rr,
        "{ s^{ \mu }_{p+q} }"{swap}
      ]
      &&
      \Delta[q]
      \ar[rr, "y_q"{swap} ]
      &&
      Y
    \end{tikzcd}
  \end{equation}
  for $x_p \,\in\, X_p$ and $y_q \,\in\, Y_q$
  non-degenerate cells and $(\mu,\nu)$ a
  $(p,q)$-shuffle.
\end{proposition}
\begin{proof}
  Generally, since the degeneracy maps in a product $X \times Y$ of simplicial sets are the pairs of the separate degeneracy maps, $s_k^{X \times Y} = \big(s_k^X, s_k^Y \big)$ \eqref{ProductOfSimplicialSets}, a cell in the product is non-degenerate precisely if its two components are not both in the image of some $s_k$, for the same $k$.
  This implies generally that non-degenerate paths in products are those that do not have steps where both components are degenerate. On the other hand, since the particular paths in question have length $p+q$, a moment of reflection shows that if they had a step with both components non-trivial, then the previous constraint could not be satisfied in all other steps. This implies the first claim above. From this, the next two formulations follow by inspection.
\end{proof}

\noindent
\hspace{-.1cm}
\begin{tabular}{ll}
\hypertarget{FigureNonDegeneratedSimplicesInProductOfSimplices}{}
\begin{minipage}{7cm}
  \scriptsize
  {\bf
  Figure NDS.
  }
  The non-degenerate simplices in the product $\Delta[p] \times \Delta[q]$ (seen as in Ex. \ref{ProductCategories}) are (according to Prop. \ref{NonDegenerateSimpliciesInProductOfSimplicies}) given by those paths of $(p+q)$ steps for which each step is a unit step either horizontally or vertically. The lists of steps (counted starting at 0) going horizontally, $\mu = (\mu_0 < \mu_1 < \cdots, \mu_{p-1})$, or going vertically, $\nu = (\nu_0 < \nu_1 < \cdots, \nu_{p-1})$ form jointly a $(p,q)$-(un-){\it shuffle} permutation of $(0, 1, \cdots, p-q)$, which bijectively encodes the respective non-degenerate cell according to the formula \eqref{ShuffleFormulaForNonDegenerateSimplices}.

  \medskip

  (Beware that the dashed diagonal arrow shown on the right is a stand-in for any remaining {\it zig-zag} path and not for any actual diagonal steps.)
\end{minipage}
&
\qquad
$
  \adjustbox{scale=.65}{
  \begin{tikzcd}[
    row sep=20pt, column sep=40pt
  ]
    &[-12pt]
    \ar[
      r,
      phantom,
      "{
        \mathclap{
          \color{gray}
          \mu_0 = 0
        }
      }"
    ]
    &
    {}
    \ar[
      r,
      phantom,
      "{
        \mathclap{
          \color{gray}
          \mu_1 = 3
        }
      }"
    ]
    &
    {}
    &
    \color{gray}
    \cdots
    \\[-20pt]
    \ar[
      d,
      phantom,
      "{
         \color{gray}
         \nu_0 = 1
      }"
    ]
    &
    {(0,0)}
    \ar[r]
    \ar[d]
    \ar[
      r,-,
      color=blue,
      opacity=.25,
      line width=6pt
    ]
    &
    {(1,0)}
    \ar[r]
    \ar[d]
    \ar[
      d,-,
      color=blue,
      opacity=.25,
      line width=6pt
    ]
    &
    {(2,0)}
    \ar[rr,dotted]
    \ar[d]
    &
    &
    {(p,0)}
    \ar[d]
    \\
    {}
    \ar[
      d,
      phantom,
      "{
        \color{gray}
        \nu_1 = 2
      }"
    ]
    &
    {(0,1)}
    \ar[r]
    \ar[d]
    &
    {(1,1)}
    \ar[r]
    \ar[d]
    \ar[
      d,-,
      color=blue,
      opacity=.25,
      line width=6pt
    ]
    &
    {(2,1)}
    \ar[rr, dotted]
    \ar[d]
    &
    &
    {(p,1)}
    \ar[d]
    \\
    {}
    &
    {(0,2)}
    \ar[r]
    \ar[dd, dotted]
    &
    {(1,2)}
    \ar[r]
    \ar[dd, dotted]
    \ar[
      r,-,
      color=blue,
      opacity=.25,
      line width=6pt
    ]
    &
    {(2,2)}
    \ar[rr, dotted]
    \ar[dd, dotted]
    \ar[ddrr, dotted]
    \ar[
      ddrr,-,
      color=blue,
      opacity=.25,
      line width=6pt,
      dotted
    ]
    &
    &
    {(p,2)}
    \ar[dd, dotted]
    \\
    \color{gray}
    \vdots
    &
    &
    &
    &&
    \\
    &
    {(0,q)}
    \ar[r]
    &
    {(1,q)}
    \ar[r]
    &
    {(2,q)}
    \ar[rr, dotted]
    &
    &
    {(p,q)}
  \end{tikzcd}
  }
$
\end{tabular}

\begin{example}\label{NondegenerateSimplicesInSquare}
  The complete set of non-degenerate cells in
  the simplicial square
  $\Delta[1] \times \Delta[1]$ is the following, with the paths (according to \hyperlink{FigureNonDegeneratedSimplicesInProductOfSimplices}{\it Figure NDS}) highlighted
  which correspond to the two non-degenerate 2-simplices:
\vspace{-2mm}
  $$
  \begin{tikzcd}[column sep=huge]
    \scalebox{.8}{$
      \left( {[0]} \atop {[0]} \right)
    $}
    \ar[
      rr,
      "{
        \left(
        {[0,0]}
        \atop
        {[0,1]}
        \right)
      }"{above, scale=.8}
    ]
    \ar[
      rr,
      -,
      line width=5pt,
      blue,
      draw opacity=.2
    ]
    \ar[
      dd,
      "{
        \left(
        {[0,1]}
        \atop
        {[0,0]}
        \right)
      }"{left, scale=.8}
    ]
    \ar[
      dd,
      -,
      line width=5pt,
      green,
      draw opacity=.2
    ]
    \ar[
      ddrr,
      "{
        \left(
        { [0,1] }
        \atop
        { [0,1] }
        \right)
      }"{description, scale=.8}
    ]
    &&
    \scalebox{.8}{$
      \left( {[0]} \atop {[1]} \right)
    $}
    \ar[
      dd,
      "{
        \left(
        {[0,1]}
        \atop
        {[1,1]}
        \right)
      }"{right, scale=.8}
    ]
    \ar[
      dd,
      -,
      line width=5pt,
      blue,
      draw opacity=.2
    ]
    \ar[
      ddll,
      phantom,
      "{
        \left(
        { [0,0,1] }
        \atop
        { [0,1,1] }
        \right)
      }"{pos=.22, scale=.8}
    ]
    \ar[
      ddll,
      phantom,
      "{
        \left(
        { [0,1,1] }
        \atop
        { [0,0,1] }
        \right)
      }"{pos=.82, scale=.8}
    ]
    \\
    \\
    \scalebox{.78}{$
      \left( {[1]} \atop {[0]} \right)
    $}
    \ar[
      rr,
      "{
        \left(
        {[1,1]}
        \atop
        {[0,1]}
        \right)
      }"{below, scale=.8}
    ]
    \ar[
      rr,
      -,
      line width=5pt,
      green,
      draw opacity=.2
    ]
    &&
    \scalebox{.8}{$
      \left( {[1]} \atop {[1]} \right)
    $}
  \end{tikzcd}
$$
\end{example}

\noindent
The following
%Prop. \ref{SimplicialMappingComplex}
makes explicit the solution to the universal property \eqref{InternalHomAdjunction} of mapping objects in simplicial sets.
\begin{proposition}[Simplicial mapping complexes, e.g. {\cite[\S I.5]{GoerssJardine99}}]
  \label{SimplicialMappingComplex}
  For $X, A \,\in\, \SimplicialSets$:
  \begin{itemize}[leftmargin=.7cm]
  \item[\bf(i)]
  their {\it mapping complex}
  $\Maps{}{X}{A} \,\in\, \SimplicialSets$
  is
  \vspace{-2mm}
  \begin{equation}
    \label{ComponentsOfSimplicialMappingComplex}
    \Maps{}{X}{A}
    \;:\;
    [n]
    \;\longmapsto\;
    \Homs{}
      { X \times \Delta[n] }
      { A }
  \end{equation}

  \vspace{-2mm}
  \noindent
  in that with this formula we have,
  for  $S \,\in\, \SimplicialSets$,
  the required natural isomorphisms (cf. Ntn. \ref{MappingObjects})
  \vspace{-2mm}
  \begin{equation}
    \label{HomAdjunctionIsoForSimplicialSets}
    \Homs{}{S \times X}{A}
    \;\simeq\;
    \Homs{\big}{ S  }{
      \Maps{}
        { X }
        { A }
    }
    \,;
  \end{equation}

   \vspace{-2mm}
  \item[\bf(ii)] the corresponding evaluation map
  \eqref{EvaluationMap}
  is given by naive evaluation on given cells $\sigma_n$ but paired with the identity $n$-cell $\iota_n$:
  \vspace{-5mm}
  \begin{equation}
    \label{EvaluationMapOnSimplicialMappingComplexes}
    \begin{tikzcd}[
      row sep=-1pt
    ]
      \big(
        S \times \Maps{}{S}{A}
      \big)_n
      \ar[
        r,
        phantom,
        "{ \simeq }",
        "{
          \scalebox{.7}{
            \eqref{ComponentsOfSimplicialMappingComplex}
          }
        }"{yshift=-10pt}
      ]
      &[-15pt]
      S_n
        \times
      \Homs{\big}
        { S \times \Delta[n] }
        { A }
      \ar[
        rr,
        "{
          (
            \mathrm{ev}^{A}_S
          )_n
        }"
      ]
      &&
      A_n
      \\[-2pt]
      &
       \scalebox{0.75}{$
       (
        \sigma_n
        ,\,
        f
      )
      $}
      &\longmapsto&
      \scalebox{0.75}{$
        f(
          \sigma_n, \iota_n
        )
      $}
      \mathrlap{\,.}
    \end{tikzcd}
  \end{equation}
  \end{itemize}
\end{proposition}

\begin{remark}[Self-enrichment of simplicial sets]
  \label{SelfEnrichmentOfSimplicialSets}
  By \eqref{HomAdjunctionIsoForSimplicialSets}, the construction of simplicial mapping complexes is functorial in both arguments:
  \vspace{-2mm}
  $$
    \Maps{}
      {-}
      {-}
    \;:\;
    \begin{tikzcd}
      \SimplicialSets^{\mathrm{op}}
      \times
      \SimplicialSets
      \ar[r]
      &
      \SimplicialSets
    \end{tikzcd}
  $$
  and has a natural, associative and unital composition operation
  \vspace{-2mm}
  $$
    \begin{tikzcd}
    \Maps{}
      {X}{Y}
    \times
    \Maps{}
      {Y}{Z}
    \ar[r, "{ \circ }"]
    &
    \Maps{}
      {X}
      {Z}
      \,,
    \end{tikzcd}
  $$

  \vspace{-2mm}
  \noindent
  which on 0-cells reduces to the ordinary composition of morphisms.
  This gives $\SimplicialSets$ the structure of an $\SimplicialSets$-enriched category
  (e.g. \cite{Kelly82}\cite[\S 9.1.1]{Hirschhorn02}).
\end{remark}

\medskip
\begin{proposition}[Properties of the simplicial nerve]
  \label{PropertiesOfTheSimplicialNerve}
  The simplicial nerve
  $N : \mathrm{Cat}^{\mathrm{smll}} \xrightarrow{\;} \SimplicialSets$
  (Ntn. \ref{NotationForSimplicialSets})
  \begin{itemize}[leftmargin=.7cm]
    \item[\bf(i)] is fully faithful:
     \vspace{-1mm}
    \begin{equation}
    \label{FullyFaithfulnessOfSimplicialNerve}
    \begin{tikzcd}
      \Homs{}
        { \mathcal{C} }
        { \mathcal{D} }
      \ar[
        r,
        "{
           N_{\mathcal{C},\mathcal{D}}
        }",
        "{ \sim }"{swap}
      ]
      &
      \Homs{}
        { N\mathcal{C} }
        { N\mathcal{D} }
      \,;
    \end{tikzcd}
    \end{equation}

     \vspace{-1mm}
    \item[\bf(ii)] sends product categories (Ex. \ref{ProductCategories}) to products \eqref{ProductOfSimplicialSets} of their nerves
     \vspace{-2mm}
    \begin{equation}
      \label{SimplicialNerveStrongMonoidalness}
      N
      (
        \mathcal{C} \times \mathcal{D}
      )
      \;\;
      \simeq
      \;\;
      (N \mathcal{C})
        \times
      (N \mathcal{D})
      \,;
    \end{equation}

     \vspace{-2mm}
    \item[\bf(iii)] sends functor categories \eqref{FunctorCategory} to
    simplicial function complexes
    \eqref{ComponentsOfSimplicialMappingComplex}:
 \vspace{-2mm}
    \begin{equation}
      \label{NerveOfFunctorCategoriesIsFunctionComplexOfNerves}
      N
      \Maps{}
        { \mathcal{C} }
        { \mathcal{D} }
      \;\simeq\;
      \Maps{}
        { N \mathcal{C} }
        { N \mathcal{D} }
    \end{equation}
  \end{itemize}

   \vspace{-2mm}
\noindent
  (all these being natural isomorphisms).
\end{proposition}
\begin{proof}
  The first statement is classical, the (elementary) proof is spelled out e.g. in \cite[Prop. 4.10]{Rezk22} [\href{https://kerodon.net/tag/002Z}{Kerodon, Prop. 1.2.2.1}].
  The second statement is immediate from the two definitions, as both are given by pairs of all structure morphisms. From this the third statement is obtained as the following sequence of natural isomorphisms, for $n \in \mathbb{N}$:
  $$
    \def\arraystretch{1.3}
    \begin{array}{lll}
      \big(
      N \Maps{}
        { \mathcal{C} }
        { \mathcal{D} }
      \big)_n
       &   \;\simeq\;
     \Homs{}
       { \mathcal{C}\times[n] }
       { \mathcal{D} }
     &
     \proofstep{by \eqref{NerveOfFunctorCategory} }
     \\
  &  \;\simeq\;
     \Homs{\big}
       { N( \mathcal{C} \times [n] ) }
       { N \mathcal{D} }
     &
     \proofstep{by \eqref{FullyFaithfulnessOfSimplicialNerve} }
     \\
 &    \;\simeq\;
     \Homs{\big}
       { N \mathcal{C} \times \Delta[n]  }
       { N \mathcal{D} }
     &
     \proofstep{by \eqref{SimplicialNerveStrongMonoidalness} }
     \\
&     \;\simeq\;
     \big(
     \Homs{\big}
       { N \mathcal{C}  }
       { N \mathcal{D} }
     \big)_n
     &
     \proofstep{by \eqref{ComponentsOfSimplicialMappingComplex} }
     \!\rlap{.}
    \end{array}
  $$

  \vspace{-6mm}
\end{proof}

\medskip

\noindent
{\bf Homotopy theory of simplicial Sets.} We make use of basic (simplicial) model category theory. A standard account is \cite{Hirschhorn02},
a concise overview of all the ingredients that we need is in \cite[\S A]{FSS20CharacterMap}.

\begin{notation}[Homotopy theory of simplical sets, e.g. {\cite[\S I.11]{GoerssJardine99}}]
We write
\vspace{-2mm}
\begin{equation}
  \label{ClassicalModelStuctureOnSimplicialSets}
  \SimplicialSets_{\mathrm{Qu}}
  \;\in\;
  \ModelCategories
\end{equation}

\vspace{-2mm}
\noindent for the classical Kan-Quillen model category structure on simplicial sets.
\end{notation}

\begin{notation}[Homotopy theory of reduced simplicial sets, e.g. {\cite[\S V Prop. 6.12]{GoerssJardine99}}]
  \label{HomotopyTheoryOfReducedSimplicialSets}
  We write
  \vspace{-2mm}
  \begin{equation}
    \label{ModelStructureOnReducedSimplicialSets}
    \SimplicialSets_{\geq 1, \mathrm{inj}}
    \;\;
    \in
    \;\;
    \ModelCategories
  \end{equation}

  \vspace{-2mm}
  \noindent
  for the model category of reduced simplicial sets
  (those $S \in \SimplicialSets$ with a single vertex, $S_0 = \ast$)
  whose weak equivalences and cofibrations are those of the underlying
  $\SimplicialSets_{\mathrm{Qu}}$ \eqref{ClassicalModelStuctureOnSimplicialSets}.
\end{notation}

\begin{lemma}[Fibrant reduced simplicial sets are Kan complexes {\cite[\S V, Lem. 6.6]{GoerssJardine99}}]
  \label{FibrantReducedSimplicialSetsAreKanComplexes}
  While the forgetful functor
  from \eqref{ModelStructureOnReducedSimplicialSets} to \eqref{ClassicalModelStuctureOnSimplicialSets}
  \vspace{-2mm}
  $$
    \begin{tikzcd}
      \SimplicialSets_{\geq 1, \mathrm{inj}}
      \ar[
        rr,
        "\underlying"
      ]
      &&
      \SimplicialSets_{\mathrm{Qu}}
    \end{tikzcd}
  $$

 \vspace{-2mm}
 \noindent
  does not preserve all fibrations, it does preserve fibrant objects.
\end{lemma}

\noindent
{\bf Simplicial Groups.}
\begin{notation}[Homotopy theory of simplicial groups, e.g. {\cite[\S II 3.7]{Quillen67}\cite[\S V]{GoerssJardine99}}]
  \label{HomotopyTheoryOfSimplicialGroups}
  We write
   \vspace{-2mm}
  $$
    \Groups(\SimplicialSets)_{\mathrm{proj}}
    \;\in\;
    \ModelCategories
  $$

 \vspace{-2mm}
  \noindent
  for the model category of simplicial groups
  (\cite[\S 17]{May67}\cite[\S 3]{Curtis71}),
  whose
  weak equivalences and fibrations are those of
  the underlying  $\SimplicialSets_{\mathrm{Qu}}$ \eqref{ClassicalModelStuctureOnSimplicialSets}.
\end{notation}
\begin{lemma}[Simplicial groups are Kan complexes, e.g. {\cite[Thm. 3]{Moore54}\cite[Thm. 17.1]{May67}\cite[Lem. 3.1]{Curtis71}}]
  \label{UnderlyingSimplicialSetOfASimplicialGroupIsKanComplex}
  The underling simplicial set of any simplicial group is a Kan complex.
\end{lemma}

\begin{proposition}[Quillen equivalence between simplicial groups and reduced simplicial sets
 {\cite[\S V, Prop. 6.3]{GoerssJardine99}}]
  \label{QuillenEquivalenceBetweenSimplicialGroupsAndReducedSimplicialSets}
  The simplicial classifying space construction (Def. \ref{StandardModelOfUniversalSimplicialPrincipalComplex})
  is the right adjoint of a Quillen equivalence between
  the
  projective model structure on simplicial groups (Nota. \ref{HomotopyTheoryOfSimplicialGroups})
  and the injective model structure on reduced simplicial sets
  (Nota \ref{HomotopyTheoryOfReducedSimplicialSets}):
   \vspace{-2mm}
  $$
    \begin{tikzcd}
      \Groups(\SimplicialSets)_{\mathrm{proj}}
      \ar[
        rr,
        shift right=6pt,
        "\overline{W}(-)"{below}
      ]
      \ar[
        rr,
        phantom,
        "\scalebox{.7}{$\simeq_{\mathrlap{\rm Qu}}$}"
      ]
      &&
      (\SimplicialSets_{\geq 1})_{\mathrm{inj}}\;.
      \ar[
        ll,
        shift right=6pt
      ]
    \end{tikzcd}
  $$
\end{proposition}

\medskip

\noindent
{\bf Abelian simplicial groups and chain complexes.}

\begin{notation}[Dold-Kan correspondence -- e.g. {\cite[\S III.2]{GoerssJardine99}\cite[\S 2]{SchwedeShipley03}}]
We write:
\begin{itemize}[leftmargin=.5cm]

\item[--] $\big( \AbelianGroups(\SimplicialSets), \otimes  \big)$ for the category of simplicial abelian groups equipped with the degree-wise tensor product of abelian groups;

\item[--] $\big( \ConnectiveChainComplexes(\AbelianGroups), \otimes \big)$ for the category of connective chain complexes of abelian groups, equipped with its natural tensor product;
\end{itemize}

\begin{itemize}[leftmargin=.5cm]
\item[--] the Dold-Kan correspondence (reviewed in our context in \cite[A.63]{FSS20CharacterMap}):
 \vspace{-2mm}
\begin{equation}
  \label{DoldKanCorrespondence}
    \begin{tikzcd}
    \mathrm{Ch}^{\geq 0}
    (\AbelianGroups)
    \ar[
      from=rr,
      shift right=8pt,
      "{
        \NormalizedChains
      }"{swap}
    ]
    \ar[
      rr,
      shift right=8pt,
      "{
        \Gamma
      }"{swap}
    ]
    \ar[
      rr,
      phantom,
      "{
        \bot
      }"{yshift=+2pt, scale=.7},
      "{
        \sim
      }"{yshift=-4pt, scale=.7},
    ]
    \ar[
      rrrr,
      rounded corners,
      to path={
        -- ([yshift=-00pt]\tikztostart.south)
        -- ([yshift=-14.5pt]\tikztostart.south)
        -- node[yshift=-8pt]{
             \scalebox{.7}{
               $\mathrm{DK}$
             }
           }
            ([yshift=-15pt]\tikztotarget.south)
        -- ([yshift=-00pt]\tikztotarget.south)
      }
    ]
    &&
    \AbelianGroups(\SimplicialSets)
    \ar[
      from=rr,
      shift right=8pt,
      "{
        \mathbb{Z}[-]
      }"{swap}
    ]
    \ar[
      rr,
      shift right=8pt,
      "{
        \underlying
      }"{swap}
    ]
    \ar[
      rr,
      phantom,
      "{
        \bot
      }"{scale=.7}
    ]
    &&
    \SimplicialSets
    \mathrlap{\,,}
    \end{tikzcd}
\end{equation}

 \vspace{-2mm}
\noindent
where $\NormalizedChains(-)$ forms (``normalized'') chain complexes of non-degenerate cells, and

where $\mathbb{Z}[-]$ denotes the (degreewise) free abelian group functor;

\item[--]
notice here that $\NormalizedChains(-)$ respects the tensor product only up to homotopy (this is the content of Prop. \ref{EilenbergZilberAlexanderWhitneyRetraction} below), while $\mathbb{Z}[-]$ is strong monoidal:
 \vspace{-2mm}
\begin{equation}
  \label{FreeAbIsStrongMonoidal}
  X,Y \,:\, \SimplicialSets
  \;\;\;\;\;\;\;\;\;\;
  \vdash
  \;\;\;\;\;\;\;\;\;
  \mathbb{Z}
  [
    X \times Y
  ]
  \;\simeq\;
  \mathbb{Z}[X] \otimes \mathbb{Z}[Y]\;;
\end{equation}

 \vspace{-1mm}
\item[--] the Eilenberg-MacLane formula \cite[(5.3)]{EilenbergMacLane53}
for the Eilenberg-Zilber map \cite{EilenbergZilber53}
(review in \cite[\S 1.6]{Loday92}):
 \vspace{-2mm}
\begin{equation}
  \label{EilenbergMaclaneFormula}
  A,B
    \,\in\,
  \AbelianGroups(\SimplicialSets)
  \hspace{.8cm}
  \vdash
  \hspace{.8cm}
  \begin{tikzcd}[
    row sep=-2pt,
    column sep=8pt
  ]
    \NormalizedChains(A)
    \otimes
    \NormalizedChains(B)
    \ar[
      rr,
      "{
        \nabla_{A.B}
      }"
    ]
    &&
    \NormalizedChains
    (
      A \otimes B
    )
    \\
  \scalebox{0.75}{$ a_p \otimes b_q$}
    &\longmapsto&
  \scalebox{0.75}{$    \underset{
      \mathclap{
        (\mu,\nu)
      }
    }{\scalebox{1.3}{$\sum$}}
    \;
    \mathrm{sgn}(\mu,\nu)
    \cdot
    \big(
      s_{ \nu }(a_p)
    \big)
    \otimes
    \big(
      s_{ \mu }(b_q)
    \big)
    $}
    \,,
  \end{tikzcd}
\end{equation}

 \vspace{-3mm}
\noindent
where the sum is over all $(p,q)$-shuffles $(\mu, \nu)$
from \eqref{Shuffles}
and $\left(s_\nu(-), s_\mu(-)\right)$ is according to \eqref{FormulaForNonDegProductsOfSimplices}.
\end{itemize}
\end{notation}

\begin{example}[Shifted abelian groups]
  \label{ShiftedAbelianGroups}
  For $A \in \AbelianGroups(\Sets)$ and $n \in \mathbb{N}$, the image under the Dold-Kan construction
  \eqref{DoldKanCorrespondence} of the chain complex that is concentrated on $A$ in degree $n$ is a model for the $n$-fold delooping of $A$:
  $$
    \mathbf{B}^n A
    \;\;\simeq\;\;
    \mathrm{DK}\big(A[n]\big)
    \;\;
    \in
    \;\;
    \mathrm{Ho}\big(
      \SimplicialSets_{\mathrm{Qu}}
    \big).
  $$
\end{example}

\begin{example}[Incarnations of group cohomology]
  \label{IncarnationsOfGroupCohomology}
  For
  $G, A
    \in
    \Groups(\Sets)
      \xhookrightarrow{\;}
    \Groups(\SimplicialSets)
      \xrightarrow{\Localization{\mathrm{W}}}
    \Groups(\InfinityGroupoids)$
  with $A$ abelian,
the following abelian groups are all naturally
isomorphic, for all $n \in \mathbb{N}$:
\vspace{-2mm}
  $$
    \overset{
      \mathclap{
      \raisebox{3pt}{
        \tiny
        \color{darkblue}
        \bf
        group cohomology
      }
      }
    }{
    H^{n+1}_{\mathrm{grp}}
    \big(
      G,
      \,
      A
    \big)
    }
    \;\simeq\;
    \pi_0
    \overset{
      \mathclap{
      \raisebox{5pt}{
        \tiny
        \color{darkblue}
        \bf
        \begin{tabular}{c}
          homotopy classes of homs
          \\
          of delooped simplicial groups
        \end{tabular}
      }
      }
    }{
    \SimplicialSets
    \Big(
      \overline{W}G,
      \,
      \mathrm{DK}(A[n])
    \Big)
    }
    \;\simeq\;
    \overset{
      \mathclap{
      \raisebox{6pt}{
        \tiny
        \color{darkblue}
        \bf
        \begin{tabular}{c}
          cohomology of
          \\
          classifying space
        \end{tabular}
      }
      }
    }{
    H^{n+1}
    (
      B G,
      A
    )
    }
    \;\simeq\;
    \pi_0
    \overset{
      \mathclap{
      \raisebox{5pt}{
        \tiny
        \color{darkblue}
        \bf
        \begin{tabular}{c}
          homotopy classes of homs
          \\
          of delooped $\infty$-groups
        \end{tabular}
      }
      }
    }{
    \InfinityGroupoids
    \big(
      B G,
      \,
      B^{n+1} A
    \big).
    }
  $$
\end{example}

\begin{example}[Normalized chains on minimal simplicial circle]
  The normalized chains complex
  of the free simplicial abelian group
  \eqref{DoldKanCorrespondence} on the minimal simplicial circle (Ntn. \ref{MinimalSimplicialCircle}) is
  \begin{equation}
    \label{NormalizedChainsOnMinimalSimplicialCircle}
    \NormalizedChains
    \big(
      \mathbb{Z}
      \big[
        S^1_{\mathrm{min}}
      \big]
    \big)
    \;\;
    =
    \;\;
    \NormalizedChains
    \left(\!\!\!\!\!\!
      \adjustbox{raise=2pt}{
      \begin{tikzcd}[
        row sep=16pt
      ]
        \mathbb{Z}[
          \{\ast, \ell\}
        ]
        \ar[d, shorten=-1pt, shift left=12pt]
        \ar[d, shorten=-1pt, shift right=12pt]
        \ar[from=d, shorten=-2pt, " s_0 "{description, pos=.45}]
        \\
        \mathbb{Z}[
          \{\ast\}
        ]
      \end{tikzcd}
      }
   \!\!\!\!\!\! \right)
    \;\;
    =
    \;\;
      \adjustbox{raise=2pt}{
      \begin{tikzcd}[
        row sep=16pt
      ]
      \mathbb{Z}[\{\ell\}]
      \ar[d, "{ \partial =  0}"]
      \\
      \mathbb{Z}[\{\ast\}]
    \end{tikzcd}
    }
    \;\;
    =
    \;\;
    \mathbb{Z} \,\oplus\, \mathbb{Z}[1]\;.
  \end{equation}
\end{example}

\begin{proposition}[Eilenberg-Zilber/Alexander-Whitney deformation retraction]
\label{EilenbergZilberAlexanderWhitneyRetraction}
The Eilenberg-Zilber map $\nabla_{A,B}$
\eqref{EilenbergMaclaneFormula}
on normalized chain complexes is a homotopy equivalence, in fact, it has a deformation retraction (given by the Alexander-Whitney map  $\Delta_{A,B}$):
\vspace{-4mm}
$$
  \begin{tikzcd}
    \NormalizedChains(A)
    \otimes
    \NormalizedChains(B)
    \ar[
      rrrr,
      phantom,
      "{ }"{name=t}
    ]
    \ar[
      rr,
      "{
        \nabla_{A,B}
      }"{description}
    ]
    \ar[
      rrrr,
      bend left=15,
      "{ \mathrm{id} }",
      "{\ }"{swap, name=s}
    ]
    &&
    \NormalizedChains
    (
      A \otimes B
    )
    \ar[
      rr,
      "{ \Delta_{A,B} }"{description}
    ]
    \ar[
      rrrr,
      phantom,
      "{\ }"{yshift=-4pt, name=s2}
    ]
    \ar[
      rrrr,
      bend right=15,
      "{ \mathrm{id} }"{swap},
      "{\ }"{name=t2}
    ]
    &&
    \NormalizedChains(A)
    \otimes
    \NormalizedChains(B)
    \ar[
      rr,
      "{
        \nabla_{A,B}
      }"{description}
    ]
    &&
    \NormalizedChains
    (
      A \otimes B
    )
    \ar[
      from=s,
      to=t,
      shorten >=6pt,
      Rightarrow,
      -
    ]\;.
    \ar[
      from=s2,
      to=t2,
      shorten <=2pt,
      shorten >=0pt,
      Rightarrow
    ]
  \end{tikzcd}
$$
\end{proposition}
For the case of unnormalized chain complexes (where we have just  an homotopy equivalence) this is due to \cite{EilenbergZilber53}, the case for normalized chain complexes is due to \cite[Thm. 2.1a]{EilenbergMacLane54}, both are reviewed in \cite[Cor. 29.10]{May67}. Explicit description of the homotopy operator is in \cite{GDR99}.

\medskip
\medskip

\noindent
{\bf Universal principal simplicial complex.}

\begin{definition}[Universal principal simplicial complex
{\cite[Def. 10.3]{Kan58}\cite[p. 269]{GoerssJardine99}}]
 \label{StandardModelOfUniversalSimplicialPrincipalComplex}
  Let $\mathcal{G} \in \Groups(\SimplicialSets)$.

  \noindent
  {\bf (i)} Its {\it standard universal principal complex} is the simplicial set
  \vspace{-2mm}
  $$
    W \mathcal{G}
    \;\;
    \in
    \;
    \SimplicialSets
  $$

  \vspace{-2mm}
  \noindent
  whose

  \begin{itemize}
  \item
  component sets are
  $$
    (W \mathcal{G})_n
    \;\coloneqq\;
    \mathcal{G}_{n}
      \times
    \mathcal{G}_{n-1}
      \times
    \cdots
      \times
    \mathcal{G}_0
    \,,
  $$

  \item
  face maps are given by
  \vspace{-2mm}
  \begin{equation}
    \label{FaceMapsOfWG}
    \hspace{-5mm}
    d_i
    \big(
      \gamma_n, \, \gamma_{n-1},\, \cdots ,\, \gamma_0
    \big)
    \;\coloneqq\;
    \left\{\!\!\!\!
    \begin{array}{lcl}
      \big(
        d_i(\gamma_n),
        \,
        d_{i-1}(\gamma_{n-1}),
        \,
        \cdots,
        \,
        d_0(\gamma_{n-i}) \cdot \gamma_{n-i-1},
        \,
        \gamma_{n-i-2},
        \,
        \cdots,
        \,
        \gamma_0
      \big)
      &\mbox{for}&
      0 < i < n
      \\[5pt]
      \big(
        d_n(\gamma_n),
        \,
        d_{n-1}(\gamma_{n-1}),
        \,
        \cdots,
        \,
        d_1(\gamma_1)
      \big)
      &\mbox{for}&
      i = n
      \,,
    \end{array}
    \right.
  \end{equation}

  \item
  degeneracy maps are given by
  \vspace{-2mm}
  \begin{equation}
    \label{DegeneracyMapsOfWG}
    s_i
    \big(
      \gamma_n,
      \,
      \gamma_{n-1},
      \,
      \cdots,
      \,
      \gamma_0
    \big)
    \;:=\;
    \big(
      s_i(\gamma_n),
      \,
      s_{i-1}(\gamma_{n - 1}),
      \,
      \cdots,
      \,
      s_0(\gamma_{n - i}),
      \,
      e,
      \,
      \gamma_{n - i - 1},
      \,
      \cdots,
      \,
      \gamma_0
    \big)
    \,,
  \end{equation}

  \item
  and equipped with the left $\mathcal{G}$-action
  (Ex. \ref{UniversalPrincipalSimplicialComplexInGActions})
  given by
  \vspace{-2mm}
  \begin{equation}
    \label{LeftActionOnUniversalSimplicialPrincipalSpace}
    \begin{tikzcd}[row sep=-3pt]
      \mathcal{G} \times W\mathcal{G}
      \ar[rr]
      &&
      W \mathcal{G}
      \\
  \scalebox{0.7}{$    \big(
        h_n,
        \,
        (\gamma_n, \gamma_{n-1}, \cdots, \gamma_0)
      \big)
      $}
      &\longmapsto&
      \scalebox{0.7}{$   \big(
        h_n \cdot \gamma_n,
        \,
        \gamma_{n-1},
        \,
        \cdots,
        \,
        \gamma_0
      \big)
      $}
      \,.
    \end{tikzcd}
  \end{equation}

  \end{itemize}

  \noindent
  {\bf (ii)}  Its standard {\it simplicial delooping} or
  {\it simplicial classifying complex}
  $\overline{W}\mathcal{G}$ is the quotient
  by that action \eqref{LeftActionOnUniversalSimplicialPrincipalSpace}:
  \vspace{-2mm}
  \begin{equation}
    \label{StandardSimplicialDeloopingAsQuotient}
    W\mathcal{G}
    \xrightarrow{ \;\;q_{{}_{W\mathcal{G}}}\;\; }
    \overline{W}\mathcal{G}
    \;:=\;
    \big(
      W \mathcal{G}
    \big)/\mathcal{G}
    \,.
  \end{equation}
\end{definition}

\begin{example}[Low-dimensional cells of universal simplicial principal complex]
Unwinding the definition \eqref{FaceMapsOfWG} of the face maps of $W \mathcal{G}$
(Def. \ref{StandardModelOfUniversalSimplicialPrincipalComplex}) shows that its
1-simplices are of the form
$$
  (W \mathcal{G})_1
  \;\;
  =
  \;\;
  \left\{\!\!\!\!
  \left.
  \begin{tikzcd}
    d_1(g_1)
    \ar[
      rr,
      "{(g_1, g_0)}"
    ]
    &&
    d_0(g_1)\cdot g_0
  \end{tikzcd}
  \;\right\vert\;
  \begin{array}{l}
    g_0 \,\in\, \mathcal{G}_0
    \\
    g_1 \,\in\, \mathcal{G}_1
  \end{array}
\!\!  \right\}
$$
and its 2-simplices are of this form:
$$
  (W\mathcal{G})_2
  \;\;
  =
  \;\;
  \left\{
  \left.
  \begin{tikzcd}[column sep=-10pt, row sep=20pt]
    &&
    \scalebox{.8}{$
      {d_0 d_2(g_2) \cdot d_1(g_1)}
      \,=\,
      {d_1 d_0(g_2) \cdot d_1(g_1)}
    $}
    {}
    \ar[
      ddrr,
      "{
        \scalebox{.8}{$\left(
          d_0(g_2) \cdot g_1,
          \,
          g_0
       \right)$ }
      }"{above, sloped, pos=.6},
      "\ "{below, name=s, pos=.02}
    ]
    \\
    \\
    {
      {\phantom{=\,} d_1 d_2(g_2)}
      \atop
      {=\, d_1 d_1(g_2)}
    }
    \ar[
      rrrr,
      "{
        \scalebox{.8}{$\left(
        d_1(g_2),
        \,
        d_0(g_1)\cdot g_0
        \right)$}
      }"{below},
      "{\ }"{above, name=t}
    ]
    \ar[
      uurr,
      "{
        \scalebox{.8}{$\left(
          d_2(g_2),
          \,
          d_1(g_1)
        \right)$}
      }"{above, sloped, pos=.4}
    ]
    &&&&
    {
      {\phantom{=,} d_0 d_0(g_2) \cdot d_0(g_1) \cdot g_0}
      \atop
      {=\,  d_0 d_1(g_2)\cdot d_0(g_1) \cdot g_0}
    }
    \ar[
      from=s,
      to=t,
      Rightarrow,
      "{
        (g_2, g_1, g_0)
      }"{description}
    ]
  \end{tikzcd}
  \;
  \right\vert
  \;
  \begin{array}{l}
    g_0 \,\in\, \mathcal{G}_0,
    \\
    g_1 \,\in\, \mathcal{G}_1,
    \\
    g_2 \,\in\, \mathcal{G}_2
  \end{array}
  \right\}.
$$

\end{example}

\begin{example}[Universal principal simplicial complex for ordinary group $G$]
  \label{UniversalSimplicialPrincipalComplexForOrdoinaryGroupG}
  If
  \vspace{-2mm}
  $$
    G
    \;\in\;
    \Groups
    \longhookrightarrow
    \Groups
    \big(
      \SimplicialSets
    \big)
  $$

  \vspace{-2mm}
\noindent  is an ordinary discrete group, regarded as a simplicial group
  (hence the functor constant on $G$ on the opposite simplex category),
  then the  standard model of its universal principal complex (Def. \ref{StandardModelOfUniversalSimplicialPrincipalComplex})
  is isomorphic to the nerve of the
  action groupoid
  of the right multiplication action of $G$ on itself:
  \vspace{-2mm}
  \begin{equation}
    \label{WGForOrdinaryGroupIsNerveOfActionGroupoid}
    W G
    \;=\;
    N
    \big(
      G \times G
        \rightrightarrows
      G
    \big)
    \,.
  \end{equation}
\vspace{-4mm}
$$
(W G)_2
\;\;
=
\;\;
\left\{
\left.
\begin{tikzcd}
  & g_2 g_1
  \ar[
    ddr,
    "{(g_2 g_1, g_0)}"{sloped},
    ""{name=s, below, pos=0.01}
  ]
  \\
  \\
  g_2
  \ar[
    uur,
    "{(g_2, g_1)}"{sloped}
  ]
  \ar[
    rr,
    "{(g_2, g_1 g_0)}"{below},
    ""{name=t, above}
  ]
  &
  {}
  &
  g_2 g_1 g_0
  \ar[
    from=s, to=t,
    Rightarrow,
    "{(g_2, g_1, g_0)}"{description}
  ]
\end{tikzcd}
\;\right\vert\;
 g_0, g_1, g_2 \,\in\, G
\right\}
$$
Accordingly, the standard simplicial delooping \eqref{StandardSimplicialDeloopingAsQuotient}
of an ordinary group is isomorphic to the simplicial nerve of its delooping groupoid:
\begin{equation}
  \label{WGAsNeverOfDeloopinggroupoid}
  \overline{W}G
  \;\simeq\;
  N\big(\,
    \underset{
      =: \, \mathbf{B}G
    }{
    \underbrace{
      G \rightrightarrows \ast
    }
    }
 \, \big)
  \;\;
  \in
  \;
  \SimplicialSets
  \,.
\end{equation}
\end{example}
\begin{proposition}[Basic properties of standard simplicial principal complex {\cite[\S V, Lem. 4.1, 4.6, Cor. 6.8]{GoerssJardine99}}]
  \label{BasicPropertiesOfStandardSimplicialPrincipalComplex}
  For $\mathcal{G} \in \Groups(\mathrm{SimplSets})$,
  its standard universal principal complex (Def. \ref{StandardModelOfUniversalSimplicialPrincipalComplex})
  has the following properties:

  {\bf (i)} $W \mathcal{G}$ is contractible;

  {\bf (ii)}
    $W \mathcal{G}$ and $\overline{W} \mathcal{G}$ are Kan complexes;
\end{proposition}
\begin{proof}
  That $\overline{W}\mathcal{G}$ is Kan fibrant
  follows as the combination of
  Lem.
  \ref{UnderlyingSimplicialSetOfASimplicialGroupIsKanComplex},
  Prop.   \ref{QuillenEquivalenceBetweenSimplicialGroupsAndReducedSimplicialSets},
  and
  Lem. \ref{FibrantReducedSimplicialSetsAreKanComplexes}.
  This implies that $W \mathcal{G}$ is Kan fibrant since
  $W \mathcal{G} \xrightarrow{q} \overline{W}\mathcal{G}$
  is a Kan fibration \eqref{TheUniversalPrincipalSimplicialBundle}
  (by Prop. \ref{QuillenEquivalenceBetweenSimplicialGroupsAndReducedSimplicialSets},
  see Ex. \ref{CoprojectionsOutOfBorelConstructionAreKanFibrations}).
\end{proof}

\medskip

\noindent
{\bf Inertia groupoids.}

\begin{notation}[Minimal simplicial circle]
\label{MinimalSimplicialCircle}
We write
\vspace{-1mm}
$$
  S^1_{\mathrm{min}}
  \;:=\;
  \Delta[1]/\partial\Delta[1]
  \;\;
  \in
  \;
  \SimplicialSets \;,
$$
and denote its unique non-degenerate 1-cell by
\vspace{-2mm}
\begin{equation}
  \label{CircleCell}
  \ell
  \;:=\;
  [0,1]
  \in
  \Delta[1]
  \; \twoheadrightarrow \;
  S^1_{\mathrm{min}} \;.
\end{equation}
\end{notation}

\begin{example}[Minimal model of free loop space]
  \label{MinimalModelOfFreeLoopSpace}
  The classifying space $\overline{W}\mathbb{Z}$
  \eqref{WGAsNeverOfDeloopinggroupoid}
  is a Kan fibrant replacement of the minimal simplicial circle (Ntn. \ref{MinimalSimplicialCircle})
  and both are models for the homotopy cofiber product of $\ast \sqcup \ast \xrightarrow{\;} \ast$ with itself:

  \vspace{-4mm}
  $$
    \adjustbox{raise=-10pt}{
    \begin{tikzcd}[row sep=small, column sep=large]
      \ast \coprod \ast
      \ar[
        r,
        hook,
        "{ \in \Cofibrations }"{}
      ]
      \ar[d]
      \ar[
        dr,
        phantom,
        "\mbox{\tiny\rm (po)}"
      ]
      &
      \Delta[1]
      \ar[d]
      \ar[r, " \in \WeakEquivalences"{}]
      &
      \ast
      \\
      \ast
      \ar[r]
      &
      S^1_{\mathrm{min}}
      \ar[
        r,
        "\in \WeakEquivalences"
      ]
      &
      \overline{W}\mathbb{Z}
      \ar[
        r,
        "\in \Fibrations"
      ]
      &
      \ast
      \\[-12pt]
      &
 \scalebox{.7}{$      \ell $}
      \ar[r, phantom, "{\longmapsto}"]
      &
   \scalebox{.7}{$    1 $}
    \end{tikzcd}
    }
    \in
    \SimplicialSets_{\mathrm{Qu}}
    \phantom{AAAAA}
    \Rightarrow
    \phantom{AAAAA}
    \begin{tikzcd}[row sep=small]
      \ast \coprod \ast
      \ar[r]
      \ar[d]
      \ar[
        dr,
        phantom,
        "\mbox{\tiny\rm (hpo)}"
      ]
      &
      \ast
      \ar[d]
      \\
      \ast
      \ar[r]
      &
      \mathbf{B}\mathbb{Z}
    \end{tikzcd}
    \in \InfinityGroupoids\;.
  $$

  \vspace{-2mm}
  \noindent
  This implies that for any
  $\mathcal{X} \,\in\, \InfinityGroupoids$,
  the free loop $\infty$-groupoid
  is equivalently the homotopy fiber of
  its diagonal map with itself:
  \vspace{-4mm}
  $$
    \Maps{\big}
      { \mathbf{B}\mathbb{Z} }
      { \mathcal{X} }
    \;\;
    \simeq
    \;\;
    \Maps{\Big}
      {
        \ast
        \underset{\ast \sqcup \ast}{\coprod} \ast
      }
      {
        \mathcal{X}
      }
    \;\;
    \simeq
    \;\;
      \mathcal{X}
        \underset
          {\mathcal{X} \times \mathcal{X}}
          {\prod}
      \mathcal{X}
    \,,
    \;\;\;\;
    \mbox{where}
    \;\;\;\;
    \begin{tikzcd}[row sep=small]
      \mathcal{X}
        \underset
          {
            \mathcal{X}
              \times
            \mathcal{X}
          }
          {\prod}
      \mathcal{X}
      \ar[d]
      \ar[r]
      \ar[
        dr,
        phantom,
        "\mbox{\tiny\rm(hpb)}"{pos=.4}
      ]
      &
      \mathcal{X}
      \ar[
        d,
        "\mathrm{diag}_{\mathcal{X}}"
      ]
      \\
      \mathcal{X}
      \ar[
        r,
        "\mathrm{diag}_{\mathcal{X}}"{below}
      ]
      &
      \mathcal{X}
        \times
      \mathcal{X}
    \end{tikzcd}
    \;\;\;
    \in
    \;
    \InfinityGroupoids
    \;.
  $$
  It also implies that for fibrant $X \,\in\, \SimplicialSets$
  the following comparison map is a simplicial weak equivalence; but direct inspection reveals that it is even an isomorphism (regarding $\mathbb{Z}$ as the free group generated by one element $\ell$):
  \begin{equation}
    \label{MapsOutOfMinimalCircle}
    \begin{tikzcd}[
      row sep=4pt,
      column sep=30pt
    ]
      \Maps{\big}
        { N \mathbf{B}\mathbb{Z} }
        { X }
      \ar[
        rr,
        "{
          \Maps{}
            { \ell \,\mapsto\, 1 }
            { X }
        }",
        "{
          \sim
        }"{swap}
      ]
      &&
      \Maps{\big}
        { S^1_{\mathrm{min}} }
        { X }
    \end{tikzcd}
  \end{equation}
\end{example}

\begin{notation}[Nerves of inertia groupoids]
\label{NerveOfInertiaGroupoids}
For $G \,\in\, \Groups(\Sets)$, its {\it inertia groupoid} is the functor groupoid
\eqref{FunctorCategory}
$$
  \Lambda \mathbf{B}G
  \;:=\;
  \Maps{}
    { \mathbf{B}\mathbb{Z} }
    { \mathbf{B}G }
  \,,
$$
hence under the simplicial nerve -- and using Prop. \ref{PropertiesOfTheSimplicialNerve} with Ex. \ref{UniversalSimplicialPrincipalComplexForOrdoinaryGroupG} -- is the simplicial mapping complex
$$
  N \Lambda \mathbf{B}G
  \;=\;
  \Maps{}
    { N \mathbf{B}\mathbb{Z} }
    { N \mathbf{B}G }
  \;\simeq\;
  \Maps{}
    { \overline{W} \mathbb{Z} }
    { \overline{W} G }
$$
In view of \eqref{ComponentsOfSimplicialMappingComplex} and \eqref{MapsOutOfMinimalCircle} we denote by
\begin{equation}
  \label{SequenceOfMorphismsInInertiaGroupoid}
 \big(
  \gamma
  ;\,
  g_{n-1}
  ,\,
  g_{n-2}
  ,\,
  \cdots
  ,\
  g_{0}
 \big)
 \;\; \in \;\;
 \mathrm{Hom}
 \big(
   S \times \Delta[n]
   ,\,
   \overline{W}G
 \big)
 \;
 \underset{
   \scalebox{.7}{
     \eqref{MapsOutOfMinimalCircle}
   }
 }{
   \simeq
 }
 \;
 \big(
 N
 \Maps{}
   { \mathbf{B}\mathbb{Z} }
   { \mathbf{B}G }
 \big)_{n}
\end{equation}
the $n$-cell in the nerve of the inertia groupoid which corresponds to the sequence of natural transformation
that start at the functor
$$
  \gamma \,\in\, G \,\simeq\, \mathrm{Hom}_{\mathrm{Grp}}(\mathbb{Z}, G) \,\simeq\, \mathrm{Hom}(\mathbf{B}\mathbb{Z}, \mathbf{B}G)
$$
and successively have components $g_{n-\bullet} \in G$.
\vspace{.3cm}
Using Prop. \ref{PropertiesOfTheSimplicialNerve},
one sees that this is characterized as mapping non-degenerate $(n+1)$-cells in $S \times \Delta[n]$ (according to Prop. \ref{NonDegenerateSimpliciesInProductOfSimplicies})
as follows (where ``$\mathrm{Ad}_j(-)$'' is as in \eqref{AbbreviationForAdjointAction}):

\vspace{1cm}

\hspace{-1cm}
\def\arraystretch{2}
\begin{tabular}{cc}
\multirow{2}{*}{
$
  \begin{tikzcd}
    \\[-65pt]
    \Big(
      \overset{
        \mathclap{
          \ell
          \,\circ\,
          s
            ^{(0, \cdots, \widehat{k}, \cdots, n)}
            _{n+1}
        }
      }{
      \overbrace{
      s^{S^1_{\mathrm{min}}}_{(0, \cdots, \widehat{k}, \cdots, n)}
      \ell
      }
      }
      ,\,
      s^k_{n+1}
    \Big)
    \ar[
      d,
      phantom,
      "{ \in }"{rotate=-90}
    ]
    \\[-10pt]
    \big(
      S
        \,\times\,
      \Delta[n]
    \big)_{\mathrlap{n+1}}
    \ar[
      d,
      "{
   \scalebox{0.7}{$     \big(
          \gamma
          ;\;
          g_{n-1}
          ,\,
          g_{n-2}
          ,\,
          \cdots
          ,\,
          g_0
        \big)
        $}
      }"{description}
    ]
    \\[+48pt]
    \big(
      \overline{W}G
    \big)_{\mathrlap{n+1}}
  \end{tikzcd}
$
}
&
$
  \begin{tikzcd}[
    column sep=45pt,
    row sep=25pt
  ]
    (
     \ast
     ,\,
      0
   )
    \ar[
      r,
      "{
        (
          \mathrm{id}_\ast
          ,\,
          [0,1]
        )
      }"{yshift=1pt}
    ]
    \ar[
      r,
      -,
      line width=5pt,
      blue,
      opacity=.2
    ]
    \ar[
      d,
      "{
        (
          \ell
          ,\,
          \mathrm{id}_0
        )
      }"{description}
    ]
    &
    (
      \ast
      ,\,
      1
   )
    \ar[
      d,
      "{
        (
          \mathrm{id}_\ast
          ,\,
          [0,1]
        )
      }"{description}
    ]
    \ar[
      r,
      "{
        (
          \mathrm{id}_\ast
          ,\,
          [1,2]
        )
      }"{yshift=1pt}
    ]
    \ar[
      r,
      -,
      line width=5pt,
      blue,
      opacity=.2
    ]
    \ar[
      d,
      "{
        (
          \mathrm{id}_\ast
          ,\,
          [0,1]
        )
      }"{description}
    ]
    \ar[r]
    &
    \cdots
    \ar[r]
    \ar[
      r,
      -,
      line width=5pt,
      blue,
      opacity=.2
    ]
    &[-10pt]
    (
     \ast
     ,
      k
   )
    \ar[
      r
    ]
    \ar[
      d,
      "{
        (
          \ell
          ,\,
          \mathrm{id}_k
        )
      }"{description}
    ]
    \ar[
      d,
      -,
      line width=5pt,
      blue,
      opacity=.2
    ]
    &[-10pt]
    \cdots
    \ar[
      r,
      "{
        (
          \mathrm{id}_\ast
          ,\,
          [n-1,n]
        )
      }"{yshift=1pt}
    ]
    \ar[
      d,
      phantom,
      "{ \cdots }"
    ]
    &
    (
     \ast
     ,
      n
   )
    \ar[
      d,
      "{
        (
          \ell
          ,\,
          \mathrm{id}_n
        )
      }"{description}
    ]
    \\
    (
     \ast,
      0
   )
    \ar[
      r,
      "{
        (
          \mathrm{id}_\ast
          ,\,
          [0,1]
        )
      }"{swap, yshift=-1pt}
    ]
    &
    (
     \ast,
      1
   )
    \ar[
      r,
      "{
        (
          \mathrm{id}_\ast
          ,\,
          [1,2]
        )
      }"{swap, yshift=-1pt}
    ]
    &
    \cdots
    \ar[r]
    \ar[
      r
    ]
    &
    (
     \ast,
      k
    )
    \ar[
      r
    ]
    \ar[
      r,
      -,
      line width=5pt,
      blue,
      opacity=.2
    ]
    &
    \cdots
    \ar[
      r,
      "{
        (
          \mathrm{id}_\ast
          ,\,
          [n-1,n]
        )
      }"{swap, yshift=-1pt}
    ]
    \ar[
      r,
      -,
      line width=5pt,
      blue,
      opacity=.2
    ]
    &
    (
     \ast,
     n
   )
  \end{tikzcd}
$
\\[-10pt]
& \rotatebox{-90}{$\mapsto$}
\\[+10pt]
&
\hspace{3pt}
$
  \begin{tikzcd}[
    column sep=60pt,
    row sep=35pt
  ]
    \bullet
    \ar[
      r,
      "{
        g_{n-1}
      }"{}
    ]
    \ar[
      r,
      -,
      line width=5pt,
      blue,
      opacity=.2
    ]
    \ar[
      d,
      "{ \gamma }"{description}
    ]
    &
    \bullet
    \ar[
      r,
      "{
        g_{n-2}
      }"
    ]
    \ar[
      r,
      -,
      line width=5pt,
      blue,
      opacity=.2
    ]
    \ar[
      d,
      "{ \mathrm{Ad}_1(\gamma) }"{description}
    ]
    &
    \cdots
    \ar[
      r,
      "{ g_{n-k} }"
    ]
    \ar[
      r,
      -,
      line width=5pt,
      blue,
      opacity=.2
    ]
    &[-20pt]
    \bullet
    \ar[
      d,
      "{ \mathrm{Ad}_k(\gamma) }"{description}
    ]
    \ar[
      d,
      -,
      line width=5pt,
      blue,
      opacity=.2
    ]
    \ar[
      r,
      "{ g_{n-(k+1)} }"
    ]
    &[-20pt]
    \cdots
    \ar[
      d,
      phantom,
      "{ \cdots }"
    ]
    \ar[
      r,
      "{
        g_{0}
      }"
    ]
    &
    \bullet
    \ar[
      d,
      "{ \mathrm{Ad}_n(\gamma) }"{description}
    ]
    \\
    \bullet
    \ar[
      r,
      "{
        g_{n-1}
      }"{swap}
    ]
    &
    \bullet
    \ar[
      r,
      "{
        g_{n-2}
      }"{swap}
    ]
    &
    \cdots
    \ar[
      r,
      "{ g_{n-k} }"{swap}
    ]
    &
    \bullet
    \ar[
      r,
      "{ g_{n-(k+1)} }"{swap}
    ]
    \ar[
      r,
      -,
      line width=5pt,
      blue,
      opacity=.2
    ]
    &
    \cdots
    \ar[
      r,
      "{ g_0 }"{swap}
    ]
    \ar[
      r,
      -,
      line width=5pt,
      blue,
      opacity=.2
    ]
    &
    \bullet
  \end{tikzcd}
$
\end{tabular}
\end{notation}

\medskip

\noindent
{\bf Simplicial group actions.}

\begin{notation}[Simplicial group actions]
  \label{SimplicialGroupActions}
  For $\mathcal{G} \in \Groups\big(\SimplicialSets\big)$,
  we denote

  \noindent
  {\bf (i)}
  by

  \vspace{-.5cm}
  \begin{equation}
    \label{SimplicialDeloopingOfSimplicialGroup}
    \mathbf{B}\mathcal{G}
    \;\in\;
    \EnrichedCategories{\SimplicialSets}
    \,,
    \;\;\;\;\;
    \mathbf{B}\mathcal{G}(\ast,\ast) \;:=\; \mathcal{G}
    \,,
  \end{equation}
  \vspace{-.4cm}

  \noindent
  the simplicial groupoid with a single object $\ast$, with
  $\mathcal{G}$ as its unique hom-object
  and with composition ``$\circ$''
  given by the {\it reverse} of the group product ``$\cdot$''

  \vspace{-.5cm}
  \begin{equation}
    \label{CompositionInDeloopedSimplicialGroup}
    \begin{tikzcd}[row sep =-4pt]
      \mathcal{G} \times \mathcal{G}
      \ar[
        rr,
        "\circ"
      ]
      &&
      \mathcal{G}
      \\
   \scalebox{0.7}{$      (g_n , h_n) $}
      &\longmapsto&
  \scalebox{0.7}{$       h_n \cdot g_n $}
      \mathrlap{\,;}
    \end{tikzcd}
  \end{equation}
  \vspace{-.4cm}

  \noindent
  {\bf (ii)}
  the category of $\mathcal{G}$-actions on simplicial sets by:
  \vspace{-2mm}
  \begin{equation}
    \label{CategoryOfSimplicialGroupActions}
    \Actions{\mathcal{G}}
    (
      \SimplicialSets
    )
    \;:=\;
    \SimplicialFunctors
    (
      \mathbf{B} \mathcal{G},
      \,
      \SimplicialSets
    )
    \,,
  \end{equation}

  \vspace{-2mm}
\noindent
  identified
  with the category of $\SimplicialSets$-enriched functors
  from the delooping \eqref{SimplicialDeloopingOfSimplicialGroup}
  of $\mathcal{G}$ to $\SimplicialSets$ (via Rem. \ref{SelfEnrichmentOfSimplicialSets}).
\end{notation}

\bigskip
\begin{remark}[Simplicial group actions are from the left]
  \label{SimplicialGroupActionsAreFromTheLeft}
  The convention \eqref{CompositionInDeloopedSimplicialGroup}
  for the delooping $\mathbf{B}\mathcal{G}$ \eqref{SimplicialDeloopingOfSimplicialGroup}
  implies
  that the simplicial $\mathcal{G}$-actions
  \eqref{CategoryOfSimplicialGroupActions} are {\it left} actions:
  \vspace{-2mm}
  $$
    \begin{tikzcd}[row sep=0pt]
      \mathbf{B}\mathcal{G}
      \ar[
        rr,
        "\mathcal{G} \acts \, X"
      ]
      &&
      \SimplicialSets
      \\[2pt]
      \bullet
      \ar[
        d,
        "g_1"{left}
      ]
      \ar[
        dd,
        rounded corners,
        to path={
           -- ([xshift=-8pt]\tikztostart.west)
           --node[below, sloped]{
               \scalebox{.7}{$
                 g_1 \circ g_2
                 \,\coloneqq\,
                 g_2 \cdot g_1
               $}
             } ([xshift=-8pt]\tikztotarget.west)
           -- (\tikztotarget.west)}
      ]
      &&
      X
      \ar[
        d,
        "g_1 \cdot (-)"
      ]
      \ar[
        dd,
        rounded corners,
        to path={
           -- ([xshift=+21pt]\tikztostart.east)
           --node[above, sloped]{
               \scalebox{.7}{$
                 (g_2 \cdot g_1) \cdot (-)
               $}
             } ([xshift=+21pt]\tikztotarget.east)
           -- (\tikztotarget.east)}
      ]
      \\[17pt]
      \bullet
      \ar[
        d,
        "g_2"{left}
      ]
      &&
      X
      \ar[
        d,
        "g_2 \cdot (-)"
      ]
      \\[17pt]
      \bullet
      &&
      X
    \end{tikzcd}
    {\phantom{AAAAAA}}
    \begin{tikzcd}[row sep=-2pt]
      \mathcal{G} \times X
      \ar[
        rr,
        "(-) \cdot (-)"
      ]
      &&
      X
      \\
    \scalebox{0.7}{$     (g_n, x_n) $}
      &\longmapsto&
     \scalebox{0.7}{$    g_n \cdot x_n $}
      \mathrlap{\,.}
    \end{tikzcd}
  $$
  \vspace{-.4cm}
\end{remark}
\begin{example}[Universal principal simplicial complex in $\mathcal{G}$-actions]
  \label{UniversalPrincipalSimplicialComplexInGActions}
  For $\mathcal{G} \in \SimplicialGroups$,
  the universal principal simplicial complex
  $W \mathcal{G}$ (Def. \ref{StandardModelOfUniversalSimplicialPrincipalComplex})
  becomes an object of \eqref{CategoryOfSimplicialGroupActions}
  by the formula \eqref{LeftActionOnUniversalSimplicialPrincipalSpace}.
  \vspace{-2mm}
  \begin{equation}
    \label{UniversalPrincipalSimplicialComplexAsLeftSimplicialGroupAction}
    \mathcal{G} \acts \; W \mathcal{G}
    \,\in\,
    \Actions{\mathcal{G}}(\SimplicialSets)\;.
  \end{equation}
    \end{example}

Making explicit the following elementary Ex. \ref{SimplicialGroupCanonicallyActingOnItself}
serves to straighten out a web of conventions about (simplicial) group actions.
\begin{example}[Simplicial group canonically acting on itself]
  \label{SimplicialGroupCanonicallyActingOnItself}
  Any $\mathcal{G} \in \Groups(\SimplicialSets)$ becomes
  an object of the category of simplicial $\mathcal{G}$-actions
  \eqref{CategoryOfSimplicialGroupActions}
  in three canonical ways:

  \vspace{-6mm}
  \begin{equation}
  \hspace{-3mm}
    \begin{tikzcd}[column sep=18pt, row sep=-4pt]
      \mathcal{G} \times \mathcal{G}
      \ar[
        rr,
        "\mathclap{\mbox{
          \tiny
          \color{greenii}
          \bf
          \def\arraystretch{.9}
          \begin{tabular}{c}
            left
            \\
            multiplication
            action
          \end{tabular}
        }}"
      ]
      &&
      \mathcal{G}\;,
      \\
    \scalebox{0.7}{$     (g_n, h_n) $}
      &\longmapsto&
     \scalebox{0.7}{$    g_n \cdot h_n $}
    \end{tikzcd}
    \qquad  \quad
    \begin{tikzcd}[column sep=18pt, row sep=-4pt]
      \mathcal{G} \times \mathcal{G}
      \ar[
        rr,
        "\mathclap{\mbox{
          \tiny
          \color{greenii}
          \bf
          \def\arraystretch{.9}
          \begin{tabular}{c}
            right inverse
            \\
            multiplication
            action
          \end{tabular}
        }}"
      ]
      &&
      \mathcal{G}\;,
      \\
     \scalebox{0.7}{$    (g_n, h_n) $}
      &\longmapsto&
     \scalebox{0.7}{$    h_n \cdot g_n^{-1} $}
    \end{tikzcd}
    \qquad  \quad
    \begin{tikzcd}[column sep=15pt, row sep=-4pt]
      \mathcal{G} \times \mathcal{G}
      \ar[
        rr,
        "\mathclap{\mbox{
          \tiny
          \color{greenii}
          \bf
          \def\arraystretch{.9}
          \begin{tabular}{c}
            adjoint/conjugation
            \\
            action
          \end{tabular}
        }}"
      ]
      &&
      \mathcal{G}\;.
      \\
      \scalebox{0.7}{$   (g_n, h_n) $}
      &\longmapsto&
      \scalebox{0.7}{$   g_n \cdot h_n \cdot g_n^{-1} $}
    \end{tikzcd}
  \end{equation}

  \vspace{-3mm}
\noindent
  The first two are isomorphic in $\Actions{\mathcal{G}}(\SimplicialSets)$
  via the inversion operation:
  \vspace{-.3cm}
  \begin{equation}
    \begin{tikzcd}[row sep=0pt, column sep=14pt]
      \scalebox{0.7}{$   (g_n, h_n) $}
      \ar[
       rrrr,
       |->
      ]
      \ar[
        ddd,
        |->
      ]
      &&[40pt] &&
   \scalebox{0.7}{$      g_n \cdot h_n $}
      \ar[
        ddd,
        |->
      ]
      \\
      &
      \mathcal{G} \times \mathcal{G}
      \ar[
        rr,
        "{
          \mbox{
            \tiny
            \color{greenii}
            \bf
            \def\arraystretch{.9}
            \begin{tabular}{c}
              left
              \\
              multiplication
            \end{tabular}
          }
        }"{above}
      ]
      \ar[
        d,
        "{
          \mathrm{id} \times (-)^{-1}
        }"{left},
        "\sim"{above, sloped}
      ]
      &&
      \mathcal{G}
      \ar[
        d,
        "{
          (-)^{-1}
        }"{right},
        "\sim"{below, sloped}
      ]
      \\[20pt]
      &
      \mathcal{G} \times \mathcal{G}
      \ar[
        rr,
        "{
          \mbox{
            \tiny
            \color{greenii}
            \bf
            \def\arraystretch{.9}
            \begin{tabular}{c}
              right inverse
              \\
              multiplication
            \end{tabular}
          }
        }"{below}
      ]
      &&
      \mathcal{G}
      \\
       \scalebox{0.7}{$  (g_n, h_n^{-1}) $}
      \ar[
        rrrr,
        |->
      ]
      &&&&
   \scalebox{0.7}{$      h_n^{-1} \cdot g_n^{-1} $}
    \end{tikzcd}
  \end{equation}
  \vspace{-.4cm}

\end{example}

\newpage

\noindent
    {\bf Homotopy theory of simplicial group actions.}

\begin{notation}[Model category of simplicial group actions {(\cite[\S 2]{DDK80}\cite[\S 5]{Guillou06}\cite[\S V Thm. 2.3]{GoerssJardine99})}]
   \label{ProjectiveModelStructureOnActionsOfSimplicialGroups}
  For $\mathcal{G} \in \Groups\big( \SimplicialSets\big)$,
  we have on
  the category of $\mathcal{G}$-actions \eqref{CategoryOfSimplicialGroupActions}
  the projective model structure
  (the {\it coarse-} or {\it Borel- equivariant model structure})
  whose fibrations and weak equivalences are those
  of the underlying $\SimplicialSets_{\mathrm{Qu}}$
  (Ntn. \ref{ModelCategoriesOfSimplicialPresheaves}), which we denote as:

  \vspace{-.5cm}
  \begin{equation}
    \label{BorelModelStructure}
    \Actions{\mathcal{G}}
    (
      \SimplicialSets
    )_{\mathrm{proj}}
    \;:=\;
    \SimplicialFunctors
    (
      \mathbf{B} \mathcal{G},
      \,
      \SimplicialSets
    )_{\mathrm{proj}}
    \,.
  \end{equation}
  \vspace{-.4cm}
\end{notation}

\begin{lemma}[Cofibrations of simplicial group actions {\cite[Prop. 2.2 (ii)]{DDK80}\cite[Prop. 5.3]{Guillou06}\cite[\S V Lem. 2.4]{GoerssJardine99}}]
  \label{CofibrationsOfSimplicialGroupActions}
  The cofibrations of
  $
    \Actions{\mathcal{G}}
    \big(
      \SimplicialSets
    \big)_{\mathrm{proj}}
  $ \eqref{BorelModelStructure} are the monomorphisms
  such that the $\mathcal{G}$-action on the simplices not in their image is free.
\end{lemma}

\begin{lemma}[Equivariant equivalence of simplicial universal principal complexes]
  \label{EquivariantEquivalenceOfSimplicialUniversalPrincipalComplexes}
  For $\mathcal{H} \xhookrightarrow{\;i\;} \mathcal{G}$
  a simplicial subgroup inclusion, the induced inclusion

  \vspace{-.4cm}
  $$
    \begin{tikzcd}
      W \mathcal{H}
      \ar[
        rr,
        "W(i)"{above},
        "\in \mathrm{W}"{below}
      ]
      &&
      W \mathcal{G}
    \end{tikzcd}
    \;\;\;
    \in
    \;
    \Actions{\mathcal{H}} (\SimplicialSets)_{\mathrm{proj}}
  $$
  \vspace{-.5cm}

  \noindent
  of their standard simplicial principal complexes
  (Def. \ref{StandardModelOfUniversalSimplicialPrincipalComplex})
  equipped with their canonical $\mathcal{H}$-action
  \eqref{LeftActionOnUniversalSimplicialPrincipalSpace}
  is a weak equivalence
  in the Borel-equivariant model structure \eqref{BorelModelStructure}
\end{lemma}
\begin{proof}
The underlying simplicial sets of both are contractible,
  by Prop. \ref{BasicPropertiesOfStandardSimplicialPrincipalComplex},
so that underlying any equivariant morphism between them is an
simplicial weak homotopy equivalence.
\end{proof}

The following Prop. \ref{QuillenEquivalenceBetweenBorelModelStructureAndSliceOverClassifyingComplex} is the model category theoretic avatar of the slice characterization of $\infty$-group actions in \eqref{CyclicicationAsRightDerivedBaseChange}:
\begin{proposition}[Quillen equivalence between Borel model structure and the slice over
classifying complex]
\label{QuillenEquivalenceBetweenBorelModelStructureAndSliceOverClassifyingComplex}
For
any $\mathcal{G} \in \Groups\big( \SimplicialSets\big)$
there is a simplicial adjunction

\vspace{-3mm}
\begin{equation}
  \label{BorelConstructionAdjunction}
  \overset{
    \mathclap{
    \raisebox{3pt}{
      \tiny
      \color{darkblue}
      \bf
      homotopy fiber
    }
    }
  }{
    (-) \times_{\overline{W}\mathcal{G}} W \mathcal{G}
  }
  \;\;
  \dashv
  \;\;
  \overset{
    \mathclap{
    \raisebox{3pt}{
      \tiny
      \color{darkblue}
      \bf
      Borel construction
    }
  }
  }{
    \big(
      (-) \times W \mathcal{G}
    \big)/\mathcal{G}
  }
\end{equation}
\vspace{-5mm}

\noindent
between the Borel model structure
\eqref{BorelModelStructure}
and the slice model structure
of $\SimplicialSets_{\mathrm{Qu}}$ \eqref{ClassicalModelStuctureOnSimplicialSets}
over the simplicial classifying complex $\overline{W}\mathcal{G}$ \eqref{StandardSimplicialDeloopingAsQuotient},
hence a natural isomorphism of hom-complexes
\vspace{-2mm}
\begin{equation}
  \label{BorelSliceAdjunctionAsNaturalIsoOfSimplicialHomComplexes}
  \hspace{-3mm}
  \begin{aligned}
 % &
  \Actions{\mathcal{G}}
  (\SimplicialSets)
  \big(
    (-) \times_{\overline{W}\mathcal{G}} W \mathcal{G},
    \,
    (-)
  \big)
%  \\
  \;\simeq\;
  &
  \SimplicialSets_{/\overline{W}\mathcal{G}}
  \Big(
    (-),
    \,
    \big(
      (-) \times W \mathcal{G}
    \big)/\mathcal{G}
  \Big)
  \end{aligned}
  \!\!\!\!\!\!\!
  \in
  \;
  \SimplicialSets
  \,,
\end{equation}

  \vspace{-2mm}
\noindent
which is a Quillen equivalence:
\vspace{-4mm}
\begin{equation}
  \label{QuillenEquivalenceBetweenBorelModelStructureAndSliceOverSimplicialClassifyingComplex}
  \begin{tikzcd}[column sep=large]
    \Actions{\mathcal{G}}
    \big(
      \SimplicialSets
    \big)_{\mathrm{proj}}
    %\;:=\;
    %\SimplicialFunctors
    %\big(
    %  \mathbf{B} \mathcal{G},
    %  \,
    %  \SimplicialSets
    %\big)_{\mathrm{proj}}
    \ar[
      rr,
      shift right=7pt,
      "{
        (
          (-) \times W \mathcal{G}
        )/ \mathcal{G}
      }"{below}
    ]
    \ar[
      rr,
      phantom,
      "{\simeq_{\mathrlap{\mathrm{Qu}}}}"
    ]
    &&
    \big(
      \SimplicialSets_{\mathrm{Qu}}
    \big)_{/\overline{W}\mathcal{G}}
    \mathrlap{\,.}
    \ar[
      ll,
      shift right=7pt,
      "{
        (-) \times_{\overline{W}\mathcal{G}} W\mathcal{G}
      }"{above}
    ]
  \end{tikzcd}
\end{equation}
\end{proposition}
\begin{proof}
  The plain adjunction constituting a Quillen equivalence \eqref{QuillenEquivalenceBetweenBorelModelStructureAndSliceOverSimplicialClassifyingComplex}
  is \cite[Prop. 2.3]{DDK80}. The simplicial enrichment \eqref{BorelSliceAdjunctionAsNaturalIsoOfSimplicialHomComplexes},
  hence the natural bijections
  \vspace{-2mm}
  $$
    \mathrm{Hom}
    \Big(
      \big(
        (-) \times_{\overline{W}\mathcal{G}}
      \big)
      \times \Delta[k]
      ,\,
      (-)
    \Big)
    \;\simeq\;
    \mathrm{Hom}
    \Big(
      (-) \times \Delta[k]
      ,\,
      \big(
        (-) \times W \mathcal{G}
      \big)/\mathcal{G}
    \Big)
    \;\;\;
    \in
    \;
    \Sets
    \,,
  $$

  \vspace{-2mm}
  \noindent
  are left somewhat implicit in \cite[Prop. 2.4]{DDK80}, but it follows
  readily from the plain adjunction via the natural isomorphism
  $$
    \big(
      (-) \times_{\overline{W}\mathcal{G}} W G
    \big) \times \Delta[k]
    \;\simeq\;
    \big(
      (-) \times \Delta[k]
    \big) \times_{\overline{W}\mathcal{G}} W \mathcal{G}
    \,,
  $$
  which, in turn, follows from the pasting law (Fact \ref{PastingLaw}):
  \vspace{-2mm}
  $$
  \hspace{-1cm}
  \begin{tikzcd}[column sep=large, row sep=small]
    { (X \times_{\overline{W}G} W G) \times \Delta[k] }
%    \atop
\hspace{-1.2cm}
&
    { \mathllap{\simeq \;} (X \times \Delta[k]) \times_{\overline{W}G} W G  }
    \ar[r]
  \ar[d]
  \ar[dr,phantom,"\mbox{\tiny\rm(pb)}"]
  &
  X \times \Delta[k]
  \ar[
    d,
    "\;\mathrm{pr}_1"
  ]
  \\
 & X \times_{\overline{W}G} W G
  \ar[r]
  \ar[d]
  \ar[dr,phantom,"\mbox{\tiny\rm(pb)}"]
  &
  X
  \ar[
    d
  ]
  \\
 & W G
  \ar[r]
  &
  \overline{W}G
  \mathrlap{\,.}
\end{tikzcd}
  $$

  \vspace{-7mm}
\end{proof}

\begin{example}[Coprojections out of Borel construction are Kan fibrations]
  \label{CoprojectionsOutOfBorelConstructionAreKanFibrations}
  For $\mathcal{G} \acts \, X \,\in\, \Actions{\mathcal{G}}(\SimplicialSets)$
  such that the underlying simplicial set $X$ being a Kan complex, hence such that
  \vspace{-2mm}
  $$
    \mathcal{G} \acts \, X
    \xrightarrow{ \;\in\; \mathrm{Fib} }
    \ast
    \;\;\;
    \in
    \;
    \Actions{\mathcal{G}}(\SimplicialSets)_{\mathrm{proj}}
    \,,
  $$

  \vspace{-2mm}
  \noindent
  the projection from the Borel construction
  \eqref{BorelConstructionAdjunction}
  to the simplicial classifying space
  (Def. \ref{StandardModelOfUniversalSimplicialPrincipalComplex})
  is a Kan fibration, due to the right Quillen functor property \eqref{QuillenEquivalenceBetweenBorelModelStructureAndSliceOverSimplicialClassifyingComplex}:
   \vspace{-2mm}
  $$
    \begin{tikzcd}
      X
      \ar[
        r,
        "{\mathrm{fib}(q)}"
      ]
      &
      (W\mathcal{G} \times X)/\mathcal{G}
      \ar[
        d,
        "{\in \mathrm{Fib}}"{right},
        "q"{left}
      ]
      \\
      &
      \overline{W}\mathcal{G}
    \end{tikzcd}
    \;\;\;
    =
    \;\;\;
    \left(
    W\mathcal{G}
    \times
    \left(\!\!\!
      \def\arraystretch{.6}
      \begin{array}{c}
        X
        \\
        \downarrow
        \\
        \ast
      \end{array}
\!\!\!    \right)
   \!\! \right)\big/ \mathcal{G}
    \,.
  $$

   \vspace{-2mm}
\noindent  The fiber of this fibration, hence the {\it homotopy fiber}
  is clearly $X$.

  In the special case where $\mathcal{G} \acts \, X = \mathcal{G} \acts \, \mathcal{G}_L$
  is the multiplication action
  of the simplicial group on itself, this construction reduces
  to the {\it universal principal simplicial bundle} \eqref{StandardSimplicialDeloopingAsQuotient}
   \vspace{-2mm}
  \begin{equation}
    \label{TheUniversalPrincipalSimplicialBundle}
    \begin{tikzcd}
      \mathcal{G}
      \ar[r]
      &
      W\mathcal{G}
      \ar[
        d,
        "q"{left},
        "\in \mathrm{Fib}"{right}
      ]
      \\
      &
      \overline{W}\mathcal{G}
    \end{tikzcd}
    {\phantom{AAA}}
    =
    {\phantom{AAA}}
    \left(
    W\mathcal{G}
    \times
    \left(\!\!\!
      \def\arraystretch{.6}
      \begin{array}{c}
        \mathcal{G}_L
        \\
        \downarrow
        \\
        \ast
      \end{array}
   \!\!\! \right)
    \!\! \right) \big/ \mathcal{G}
  \end{equation}
\end{example}

\begin{notation}[Homotopy quotient of simplicial group actions]
  \label{HomotopyQuotientOfSimplicialGroupActions}
  For $\mathcal{G} \in \Groups(\SimplicialSets)$,
  we denote the right derived functor of the
  {\it Borel construction} right Quillen functor
  \eqref{QuillenEquivalenceBetweenBorelModelStructureAndSliceOverSimplicialClassifyingComplex}
  by
   \vspace{-2mm}
  $$
    (-) \!\sslash\! \mathcal{G}
    \;\coloneqq\;
    \mathbb{R}
    \big(\!
      \left(
        (-) \times W\mathcal{G}
      \right)\big/\mathcal{G}
    \big)
    \;\;\;
      :
    \;\;\;
    \mathrm{Ho}
    \big(
      \Actions{\mathcal{G}}(\SimplicialSets)_{\mathrm{proj}}
    \big)
    \xrightarrow{\qquad}
    \mathrm{Ho}
    \big(
      \SimplicialSets_{\mathrm{Qu}}
    \big)
    \,.
  $$

   \vspace{-2mm}
\noindent  Applied to the point $\mathcal{G} \acts \, \ast$ we also write
  \begin{equation}
    \label{DeloopingAsHomotopyQuotientOfThePoint}
    \mathbf{B}\mathcal{G}
    \;\;
    \coloneqq
    \;\;
    \ast \!\sslash\! \mathcal{G}
    \,.
  \end{equation}
\end{notation}

As an example:
\begin{proposition}[$G$-Sets in the homotopy theory over $B G$]
  \label{GSetsInTheHomotopyTheoryOverBG}
  For $G \in \Groups(\Sets) \xhookrightarrow{\;} \Groups(\SimplicialSets)$,
  we have an equivalence between $G$-actions on sets and 0-truncated
  objects in the homotopy theory slice over $B G$:

  \vspace{-.4cm}
  $$
    \begin{tikzcd}[row sep=-4pt]
      \Actions{G}(\Sets)
      \ar[r, phantom, "\simeq"]
      &
      \left(
        (\InfinityGroupoids \right)_{/B G}
      )_{\leq 0}
      \ar[r, hook]
      &
      \big(\InfinityGroupoids\big)_{/B G}
      \\
    \scalebox{0.7}{$     G \acts \, X $}
      &\longmapsto&
     \scalebox{0.7}{$    \big( X \!\sslash\! G \xrightarrow{\;} B G  \big) $}
    \end{tikzcd}
  $$
  \vspace{-.4cm}

\end{proposition}
\begin{proof}
  Since $\big( X \times G \rightrightarrows X  \big) \xrightarrow{\;} (G \rightrightarrows \ast)$
  is clearly a Kan fibration with fiber $X$, the latter is the
  homotopy fiber of $X \!\sslash\! G \xrightarrow B G$.
  With this, the statement follows by Prop. \ref{QuillenEquivalenceBetweenBorelModelStructureAndSliceOverClassifyingComplex}.
\end{proof}

\medskip

\noindent
{\bf Free loop space of classifying spaces.}

\begin{proposition}[Free loop space of simplicial classifying space]
  \label{FreeLoopSpaceOfSimplicialClassifyingSpace}
  For $\mathcal{G} \,\in\, \Groups(\SimplicialSets)$, the
  function complex from the simplicial classifying space
  (Def. \ref{StandardModelOfUniversalSimplicialPrincipalComplex})
  of the additive group of integers to that of $\mathcal{G}$
  is equivalent to the homotopy quotient (Nota. \ref{HomotopyQuotientOfSimplicialGroupActions})
  of the conjugation action of $\mathcal{G}$ on itself
  (Ex. \ref{SimplicialGroupCanonicallyActingOnItself})
  \vspace{-2mm}
  $$
    \Maps{\big}
      { \overline{W} \mathbb{Z} }
      { \overline{W}\mathcal{G} }
    \;\simeq\;
    \mathcal{G}
      \sslash_{\!\!\mathrm{ad}}
    \mathcal{G}
    \;\;\;
    \in
    \;
    \mathrm{Ho}
    \big(
      \SimplicialSets_{\mathrm{Qu}}
    \big)
    \,.
  $$
\end{proposition}
For topological groups this statement is folklore but rarely argued in detail.
A point-set topology argument is spelled out in \cite[\S A]{Gruher07},
and the idea of the following abstract argument is in \cite[Lem. 9.1]{KleinSchochetSmith09},
which we adapt to simplicial groups:
\begin{proof}
Consider the following commuting diagram in $\SimplicialSets$:
\vspace{-5mm}
\begin{equation}
  \label{HomotopyPullbackComputingFreeLoopSpaceOfSimplicialClassifyingSpace}
  \begin{tikzcd}[column sep=70pt]
    &&
    &[+15pt]
    \frac{W \mathcal{G}}{\mathcal{G}}
    \times W\mathcal{G}
    \ar[
      dr,
      shorten <=-3pt,
      "\;\;\;\mathrm{pr}_1 \in \mathrm{W}"{right}
    ]
    &[-30pt]
    \\
    \frac{
      W \mathcal{G} \times \mathcal{G}_{\!\scalebox{.4}{ad}}
    }{
      \mathcal{G}
    }
    \ar[dd]
    \ar[rr]
    \ar[
      ddrr,
      phantom,
      "\mbox{\tiny\rm (pb)}"{pos=.14}
    ]
    &&
    \frac{
      W \mathcal{G} \times W\mathcal{G} \times \mathcal{G}_{\!\scalebox{.4}{ad}}
    }
    {
      \mathcal{G} \times \mathcal{G}
    }
    \ar[
      ur,
      shorten <=-3pt,
      shorten >=-4pt,
      "\sim"{sloped, below},
      "{
        \scalebox{.8}{$
        \!\!\!\!
        [
          \vec g, (\vec g',h)
        ]
        \mapsto
        \left(
          [\vec g],\,
          \left(
            \vec g^{-1}
            \cdot
            (h \cdot \vec g')
          \right)
        \right)
        $}
      }"{sloped, yshift=2pt}
    ]
    \ar[
      dd,
      "\in \mathrm{Fib}"{right},
      "
       \frac{
         W(\mathcal{G} \times \mathcal{G})
         \times
         \left(
           \!\!\!\!\!\!
           \def\arraystretch{.45}
           \begin{array}{c}
             \mathcal{G}_{\!\scalebox{.4}{ad}}
             \\
             \downarrow
             \\
             \ast
           \end{array}
           \!\!\!\!\!\!
         \right)
       }{
         \mathclap{\phantom{\vert^{a}}}
         \mathcal{G} \times \mathcal{G}
       }
      "{left}
    ]
    &
    &
    \frac{
      W\mathcal{G}
    }
    {
     \mathcal{G}
    }
    \ar[
      ddll,
      "\mathrm{diag}"{right, xshift=4pt, yshift=-4pt}
    ]
    \ar[
      ll,
      "\in \mathrm{W}"{below},
      "{
        [
          \vec g,
          (\vec g, \mathrm{e})
        ]
        \,\mapsfrom\,
        [
          \vec g
        ]
      }"{above}
    ]
    \\
    \\
    \frac{
      W\mathcal{G}
    }{
      \mathcal{G}
    }
    \ar[
      rr,
      "\mathrm{diag}"{below}
    ]
    &&
    \frac{
      W \mathcal{G} \times W\mathcal{G}
    }{
      \mathcal{G} \times \mathcal{G}
    }
%    \mathrlap{\,,}
  \end{tikzcd}
\end{equation}

 \vspace{-2mm}
\noindent
where
$$
  \mathcal{G}_{\mathrm{ad}}
  \;\in\;
  \Actions{
    (\mathcal{G} \times \mathcal{G})
  }(\SimplicialSets)
  \xrightarrow{ \;\mathrm{res}_{\mathrm{diag}_{\mathcal{G}}} \;}
  \Actions{\mathcal{G}}(\SimplicialSets)
$$
denotes the conjugation action on $\mathcal{G} = \underlying(\mathcal{G}_{\mathrm{ad}})$
\vspace{-3mm}
$$
  \begin{tikzcd}[row sep=-5pt]
    \mathcal{G}_{\mathrm{ad}}
    \times
    \big(
      \mathcal{G} \times \mathcal{G}
    \big)
    \ar[
      rr
    ]
    &&
    \mathcal{G}_{\mathrm{ad}}
    \\
  \scalebox{0.7}{$     \big(
      g_k, (h'_k,h_k)
    \big)
    $}
    &\longmapsto&
  \scalebox{0.7}{$     h'_k \cdot g_k \cdot h_k^{-1} $}
  \end{tikzcd}
$$

\vspace{-3mm}
\noindent and its restriction along the diagonal to $h' = h$.
Here
\begin{itemize}

  \item
  the right vertical morphism is a Kan fibration by
  Lemma \ref{UnderlyingSimplicialSetOfASimplicialGroupIsKanComplex} and
  Prop. \ref{QuillenEquivalenceBetweenBorelModelStructureAndSliceOverClassifyingComplex}
  (see Ex. \ref{CoprojectionsOutOfBorelConstructionAreKanFibrations});

  \item
  all objects are fibrant (are Kan complexes),
  by Prop. \ref{BasicPropertiesOfStandardSimplicialPrincipalComplex}
  and by the previous item;

  \item
  the horizontal morphism on the top right is a weak equivalence,
  by 2-out-of-3, as it is the left inverse to the triangular composite
  of morphisms, which is a weak equivalence by
  Prop. \ref{BasicPropertiesOfStandardSimplicialPrincipalComplex}.
\end{itemize}
Therefore (as recalled in \cite[Def. A.24]{FSS20CharacterMap}),
the pullback diagram in \eqref{HomotopyPullbackComputingFreeLoopSpaceOfSimplicialClassifyingSpace}
represents the homotopy pullback
that characterizes
$\big[ \overline{W}\mathbb{Z}, \overline{W}\mathcal{G}\big]$,
according to Example \ref{MinimalModelOfFreeLoopSpace}.
\end{proof}
As an immediate consequence:

\begin{proposition}[Free loop space of classifying space of simplicial abelian group]
  \label{FreeLoopSpaceOfClassifyingSpaceOfSimplicialAbelianGroup}
  $\,$

\noindent  Let $\mathcal{A} \in \AbelianGroups(\SimplicialSets) \xhookrightarrow{\Discrete}
  \Groups(\SmoothInfinityGroupoids)$
  be a simplicial abelian group.
  Then
   \vspace{-2mm}
  $$
    \Maps{}
      { \mathbf{B}\mathbb{Z} }
      { \mathbf{B}\mathcal{A} }
    \;\simeq\;
    \mathcal{A} \times \mathbf{B}\mathcal{A}
    \qquad
    \mbox{and}
    \qquad
    \mathrm{Cyc}_{\scalebox{.7}{$\shape \CohesiveCircle$}}
    \big(
      \mathbf{B}\mathcal{A}
    \big)
    \;\simeq\;
    \mathbf{B}^2 \mathbb{Z}
    \times
    \mathcal{A}
    \times
    \mathbf{B}\mathcal{A}
    \,.
  $$
\end{proposition}
\begin{proof}
This is given by the following sequence of equivalences:
 \vspace{-2mm}
$$
  \def\arraystretch{1.2}
  \begin{array}{lll}
    \Maps{\big}
      { \mathbf{B}\mathbb{Z} }
      { \mathbf{B}\mathcal{A} }
    &
    \;\simeq\;
    \frac{
      \mathcal{A}_{\mathrm{ad}}
        \times
      W\mathcal{A}
    }{ \mathcal{A} }
    &
    \proofstep{by Prop. \ref{FreeLoopSpaceOfSimplicialClassifyingSpace}}
    \\
    & \;\simeq\;
    \frac{
      \mathcal{A}_{\mathrm{triv}}
      \times
      W\mathcal{A}
    }{ \mathcal{A} }
    &
    \proofstep{since each $\mathcal{A}_n$ is abelian}
    \\
    & \;\simeq\;
    \mathcal{A}
      \times
    \frac{
      W\mathcal{A}
    }{ \mathcal{A} }
    \\
    & \;\simeq\;
    \mathcal{A}
    \times
    \overline{W}\mathcal{A}
    &
    \proofstep{by \eqref{StandardSimplicialDeloopingAsQuotient}}
    \\
    & \;\simeq\;
    \mathcal{A} \times \mathbf{B}\mathcal{A}
    &
    \proofstep{by \eqref{DeloopingAsHomotopyQuotientOfThePoint}}.
  \end{array}
$$

\vspace{-7mm}
\end{proof}

\begin{example}[Free loop space of Eilenberg-MacLane space]
  \label{FreeLoopSpaceOfEilenbergMacLaneSpace}
  For $A \in \AbelianGroups(\Sets)$
  with $B^n A \,\simeq\, K(A,n)$ for $n \in \mathbb{N}$
  (Ex. \ref{ShiftedAbelianGroups}),
  Prop. \ref{FreeLoopSpaceOfClassifyingSpaceOfSimplicialAbelianGroup}
  yields:
   \vspace{-1mm}
  $$
    \Maps{}
      { \mathbf{B}\mathbb{Z} }
      { \mathbf{B}^{n+1}A }
    \;\simeq\;
    \mathbf{B}^n A
      \,\times\,
    \mathbf{B}^{n+1}A
    \,.
  $$
\end{example}

\noindent
{\bf Free and co-free simplicial actions.}

\begin{proposition}[Free and co-free simplicial actions]
  \label{InducedSimplicialActions}
  For $\mathcal{G} \,\in\, \Groups(\SimplicialSets)$,
  the forgetful functor that sends $\mathcal{G}$-actions \eqref{CategoryOfSimplicialGroupActions}
  to their underlying simplicial set has both a left adjoint
  (``free action'') and a right adjoint (``co-free action''), both of
  which are Quillen adjunctions with respect to the classical model structure
  on $\SimplicialSets$ and the projective model structure on $\Actions{\mathcal{G}}(\SimplicialSets)$
  \eqref{BorelModelStructure}:
  \vspace{-2mm}
  $$
    \begin{tikzcd}
      \SimplicialSets_{\mathrm{Qu}}
      \;\;
      \ar[
        rr,
        shift left=15pt,
        "{\mathcal{G} \times (-)}"{above}
      ]
      \ar[
        rr,
        shift right=15pt,
        "{[\mathcal{G},-]}"{below}
      ]
      &&
    \;\;  \Actions{\mathcal{G}}(\SimplicialSets_{\mathrm{Qu}})_{\mathrm{proj}} \;.
      \ar[
        ll,
        "\underlying"{description}
      ]
      \ar[
        ll,
        phantom,
        shift left=9pt,
        "\scalebox{.7}{$\bot_{\mathrlap{\mathrm{Qu}}}$}"
      ]
      \ar[
        ll,
        phantom,
        shift right=9pt,
        "\scalebox{.7}{$\bot_{\mathrlap{\mathrm{Qu}}}$}"
      ]
    \end{tikzcd}
  $$

  \vspace{-2mm}
  \noindent
  Here the underlying object of the co-free action on
  $X \,\in\, \SimplicialSets$ is
  the simplicial function complex \eqref{ComponentsOfSimplicialMappingComplex}
  \vspace{-2mm}
  $$
    \Maps{}{\mathcal{G}}{X}
    \;=\;
    \Homs{}
     {
       \mathcal{G} \times \Delta[\bullet]
    }
    {
      X
    }
    \big)
    \;\;\;
    \in
    \;
    \SimplicialSets
  $$

  \vspace{-2mm}
  \noindent
  and its $\mathcal{G}$-action
  \vspace{-2mm}
  $$
    \mathcal{G}
      \times
    \Maps{}{\mathcal{G}}{X}
    \xrightarrow{
      \;\; (-) \cdot (-) \;\;
    }
    \mathcal{G}
  $$

  \vspace{-1mm}
\noindent
   in degree $m \in \mathbb{N}$ is the function
  $
    \mathrm{Hom}
    \big(
      \Delta[n],
      \,
      \mathcal{G}
    \big)
    \times
    \mathrm{Hom}
    \big(
      \mathcal{G} \times \Delta[n],
      \,
      X
    \big)
    \xrightarrow{\;\;\;}
    \mathrm{Hom}
    \big(
      \mathcal{G} \times \Delta[n],
      \,
      X
    \big)
  $
  given by
  \vspace{-2mm}
  \begin{equation}
   \label{CofreeSimplicialActionInComponents}
   \hspace{-4mm}
  \small
   %  \label{CofreeSimplicialActionInComponents}
    \begin{aligned}
     % &
      \Big(
        \Delta[n]
        \xrightarrow{ g_n }
        \mathcal{G},
        \,
        \mathcal{G}
        \times
        \Delta[n]
        \xrightarrow{\phi}
        X
      \Big)
  %    \\
      \mapsto
      %&
      \Big(
        \mathcal{G} \times \Delta[n]
        \xrightarrow{ \mathrm{id} \times \mathrm{diag} }
        \mathcal{G} \times \Delta[n] \times \Delta[n]
        \xrightarrow{ \mathrm{id} \times g_n \times \mathrm{id} }
        \mathcal{G} \times \mathcal{G} \times \Delta[n]
        \xrightarrow{ (-) \cdot (-) \times \mathrm{id} }
        \mathcal{G} \times \Delta[n]
        \xrightarrow{ \phi }
        X
      \Big)
      \,.
    \end{aligned}
  \end{equation}
\end{proposition}
\begin{proof}
  The $\underlying$-functor preserves fibrations and weak equivalences
  by definition of the projective model structure \eqref{BorelModelStructure},
  and it preserves
  cofibrations by Lemma \ref{CofibrationsOfSimplicialGroupActions}.
  Therefore it is both a left and a right Quillen adjoint as soon
  as it is a left and a right adjoint functor at all.

  The left adjoint free action is straightforward.
  Also the right adjoint cofree action follows the
  same idea as that of the cofree action on topological $G$-spaces, only that
  for simplicial sets one cannot argue point-wise as in point-set topology,
  but needs the more abstract formula \eqref{CofreeSimplicialActionInComponents}.
  The structure of this formula manifestly gives a simplicial homomorphism,
  and by focusing on its
  images of the unique non-degenerate top-degree cell,
  $
    \iota_n \,\in\, (\Delta[n])_n
  $ (Ntn. \ref{NotationForSimplicialSets}),
  right-adjointness is seen essentially by the standard argument:

  It is sufficient to check that we have a hom-isomorphism of the form
    \vspace{-2mm}
  \begin{equation}
    \label{HomIsomorphismForCofreeSimplicialActionAdjunction}
    \Big\{
      P
      \xrightarrow{\;\; \phi_{(-)}  \;\;}
      [\mathcal{G},X]
    \Big\}
    \;\;
      \xleftrightarrow{\;\; \widetilde{(-)}  \;\;}
    \;\;
    \Big\{
      \underlying(P)
      \xrightarrow{\;\; {\widetilde \phi}_{(-)} \;\;}
      X
    \Big\}
  \end{equation}

    \vspace{-2mm}
\noindent
  as a bijection natural in $P \,\in\, \Actions{\mathcal{G}}(\SimplicialSets)$
  and $X \,\in\, \SimplicialSets$.
  So, for
  \vspace{-2mm}
  $$
    \phi_{(-)}
    \;\colon\;
    p_n \,\longmapsto\,
    \big(
      \phi_{p_n}
      \,\colon\,
      \mathcal{G} \times \Delta[n]
      \xrightarrow{\;}
      X
    \big)
  $$

  \vspace{-2mm}
  \noindent
  on the left in \eqref{HomIsomorphismForCofreeSimplicialActionAdjunction}, define its
  candidate adjunct to be
  \begin{equation}
    \label{UnderlyingAdjunctOfMorphismToCofreeSimplicialAction}
    {\widetilde \phi}_{(-)}
    \;\colon\;
    p_n
    \,\longmapsto\,
    \phi_{p_n}(\mathrm{e}_n, p_n)
    \;\;
    \in
    \;
    X_n
    \,,
  \end{equation}
  where $\mathrm{e}_n \in \mathcal{G}_n$ denotes the neutral element in
  degree $n \in \mathbb{N}$.
  It is clear that this assignment is a natural transformation in
  $P$ and $X$, hence it remains to show that ${\widetilde \phi}$ uniquely
  determines $\phi_{(-)}$.

  To that end, observe for any $g_n \in \mathcal{G}_n$ the following
  sequence of identifications:
  \vspace{-2mm}
  $$
   \def\arraystretch{1.3}
    \begin{array}{lll}
        \phi_{p_n}(g_n, \iota_n)
      &
      \;=\;
      \phi_{p_n}(\mathrm{e}_n \cdot g_n, \iota_n)
      &
      \proofstep{by the unit law in $\mathcal{G}_n$ }
      \\
      & \;=\;
      (g_n \cdot \phi_{p_n})
        (\mathrm{e}_n, \iota_n)
      &
      \proofstep{by \eqref{CofreeSimplicialActionInComponents}}
      \\
      & \;=\;
      \phi_{g_n \cdot p_n}
        (\mathrm{e}_n, \iota_n)
      &
      \proofstep{by equivariance of $\phi$}
      \\
      & \;=\;
      {\widetilde \phi}_{g_n \cdot p_n}
      &
      \proofstep{by \eqref{UnderlyingAdjunctOfMorphismToCofreeSimplicialAction}}
      \,.
    \end{array}
  $$

    \vspace{-2mm}
\noindent
  This shows that the morphsims
  $\phi_{(-)}$ and ${\widetilde \phi}_{(-)}$ uniquely determine
  each other, establishing the bijection \eqref{HomIsomorphismForCofreeSimplicialActionAdjunction},
  and hence the claimed adjunction.
\end{proof}

%%%%%%%%%%%%%%%%%%%%%%%%%%%%%%%%%%
\subsubsection{Group cohomology}
%%%%%%%%%%%%%%%%%%%%%%%%%%%%%%%%%%

\begin{proposition}[Finite subgroups of $\mathrm{Sp}(1)$ {\cite{Klein1884}}]
  \label{FiniteSubgroupsOfSU2}
  The finite subgroups
    $
    G
    \xhookrightarrow{\;i\;}
    \mathrm{Spin}(3) \simeq \mathrm{SU}(2) \simeq \mathrm{Sp}(1)
  $
 are given, up to conjugacy, by the following classification (where $n \in \mathbb{N}$):
  \begin{center}
  \footnotesize
  \begin{tabular}{|c|c|c|}
    \hline
    \begin{tabular}{c}
     \small \bf Label
    \end{tabular}
    &
    \begin{tabular}{c}
     \small  \bf Finite
   subgroup
      \\
      \bf of $\mathrm{SU}(2)$
    \end{tabular}
    &
     \begin{tabular}{c}
     \small  \bf Name of
      \\
     \small  \bf group
    \end{tabular}
    \\
    \hline
    \hline
    $\mathbb{A}_{\mathrlap{n}}$ &
    $\phantom{2}\mathbb{Z}_{\mathrlap{n+1}}$ &   Cyclic
    \\
    \hline
    $\mathbb{D}_{\mathrlap{n+4}}$ &
     $2\mathrm{D}_{\mathrlap{n+2}}$ & Binary dihedral
    \\
    \hline
    $\mathbb{E}_{\mathrlap{6}}$ & $2\mathrm{T}$ & Binary tetrahedral
    \\
    \hline
    $\mathbb{E}_{\mathrlap{7}}$ & $2\mathrm{O}$ & Binary octahedral
    \\
    \hline
    $\mathbb{E}_{\mathrlap{8}}$ & $2\mathrm{I}$ & Binary icosahedral
    \\
    \hline
  \end{tabular}
  \end{center}
\end{proposition}
For pointers to modern proofs, see \cite[Rem. A.9]{HSS18}.

\begin{proposition}[Integral group cohomology of finite subgroups of $\mathrm{SU}(2)$]
  \label{IntegralGroupCohomologyOfFiniteSubgroupsOfSU2}
  The integral group cohomology (Ex. \ref{IncarnationsOfGroupCohomology})
  of the finite subgroups $G \xhookrightarrow{i} \mathrm{SU}(2)$
  (Prop. \ref{FiniteSubgroupsOfSU2}) is as follows:
  \vspace{-2mm}
  $$
    H^n_{\mathrm{grp}}
    (
      G,
      \,
      \mathbb{Z}
    )
    \;\;
    \simeq
    \;\;
    \left\{
    \begin{array}{cl}
      \mathbb{Z}
      &
      \mbox{for $n = 0$}
      \\
      G/[G,G]
      &
      \mbox{for $n = 2 \,\mathrm{mod}\, 4$}
      \\
      \mathbb{Z}/\left\vert G\right\vert
      &
      \mbox{for $n = 0 \,\mathrm{mod}\, 4, \; n \geq 1$};
      \\
      0 &
      \mbox{otherwise}
      \,.
    \end{array}
    \right.
  $$
\end{proposition}
\noindent
This is summarized in \cite[p. 12]{EpaGanter16}.
Detailed computation for the three exceptional cases is given in
\cite[\S 4]{TomodaZvengrowski08}. The vanishing of
$H^3_{\mathrm{grp}}(G,\mathbb{Z}) \simeq H^2_{\mathrm{grp}}(G,\mathrm{U}(1))$
is also made explicit in \cite[Cor. 3.1]{FHHP00}.

\begin{proposition}[Pullback of 2nd Chern class to ADE-subgroup]
  \label{PullbackOfSecondChernClassToADESubgroup}
  For $G \xhookrightarrow{i} \mathrm{SU}(2) \,\simeq\, \mathrm{Spin}(3)$
  a finite subgroup (Prop. \ref{FiniteSubgroupsOfSU2})
  the induced pullback in
  degree-4 integral group cohomology (Ex. \ref{IncarnationsOfGroupCohomology})
  takes the generator (e.g. \cite{MilnorStasheff74}\cite[Lem. 2.1]{CadekVanzura98})
  \vspace{-2mm}
  $$
    c_2 \,=\, 1 \,\in\, \mathbb{Z} \;\simeq\;  H^4\big(B \mathrm{SU}(2);\, \mathbb{Z}\big)
    \mbox{\;\;\;equivalently\;\;\;}
    \tfrac{1}{4}p_1 \,=\, 1 \,\in\, \mathbb{Z} \;\simeq\; H^4\big(B \mathrm{Spin}(3);\, \mathbb{Z}\big)
  $$

  \vspace{-2mm}
\noindent  to a generator (via Prop. \ref{IntegralGroupCohomologyOfFiniteSubgroupsOfSU2})
  $$
    [1] \,\in\, \mathbb{Z}/\left\vert G \right\vert \;\simeq\;
    H^4\big( G;\, \mathbb{Z} \big)
  $$
  in that the pullback is the quotient projection
  $\mathbb{Z} \xrightarrow{q} \mathbb{Z}/\left\vert G \right\vert$:
    \vspace{-2mm}
  \begin{equation}
    \label{PullbackIn4CohomologyFromBSpin3ToADESubgroup}
    \begin{tikzcd}[column sep=small, row sep=small]
      H^4_{\mathrm{grp}}
      \big(
        \mathrm{Sp}(1);
        \,
        \mathbb{Z}
      \big)
      \ar[
        r,
        phantom,
        "\simeq"
      ]
      \ar[
        d,
        "\, i^\ast"
      ]
      &
      H^4
      \big(
        B \mathrm{Sp}(1)l;
        \,
        \mathbb{Z}
      \big)
      \ar[
        r,
        phantom,
        "\simeq"
      ]
      \ar[
        d,
        "\, B i^\ast"
      ]
      &
      \mathbb{Z}
      \ar[
        r,
        phantom,
        "\ni"
      ]
      \ar[
        d,
        "q"
      ]
      &
    \scalebox{0.7}{$     \tfrac{1}{2}p_1 $}
      \ar[
        d,
        |->
      ]
      \\
      H^4_{\mathrm{grp}}
      \big(
        G;
        \,
        \mathbb{Z}
      \big)
      \ar[
        r,
        phantom,
        "\simeq"
      ]
      &
      H^4
      \big(
        B G;
        \,
        \mathbb{Z}
      \big)
      \ar[
        r,
        phantom,
        "\simeq"
      ]
      &
      \mathbb{Z}/\left\vert G \right\vert
      \ar[
        r,
        phantom,
        "\ni"
      ]
      &
     \scalebox{0.7}{$    [1] $}
    \end{tikzcd}
  \end{equation}
\end{proposition}
\begin{proof}
  This is essentially the statement of
  \cite[Prop. 4.1]{EpaGanter16}, whose proof is analogous to
  Lemma 3.1 there, where the analogue of \eqref{PullbackIn4CohomologyFromBSpin3ToADESubgroup}
  is the top left square of the commuting diagram on p. 11.
\end{proof}

%%%%%%%%%%%%%%%%%%%%%%%%%%%%%%%%%%%%%%%%
\subsection{Notions of cohesive $\infty$-topos theory}
\label{NotionsFromCohesiveInfinityToposTheory}
%%%%%%%%%%%%%%%%%%%%%%%%%%%%%%%%%%%%%%%

In the main text we make free use of basic notions from $\infty$-topos theory (\cite{ToenVezzosi05}\cite{Lurie09HTT}\cite{Rezk10}),
and of  cohesive $\infty$-topos theory \cite[\S 3.1]{SSS09}\cite{dcct} as laid out and developed in \cite{SS20OrbifoldCohomology}\cite{SS21EPB}.
For reference, here we list some key notation and facts.

\medskip

We first list abstract notions in $\infty$-toposes and then recall, in some cases, their presentation by model categories of simplicial presheaves.

\begin{notation}[$\infty$-Toposes]
\label{NotationInfinityToposes}
Given any $\infty$-topos $\mathbf{H}$, we write:

\begin{itemize}[leftmargin=.5cm]

\item $\HomotopyQuotient{(-)}{\mathcal{G}} \;:\;\Actions{\mathcal{G}}\big(\mathbf{H}\big)
\xrightarrow{\;\;} \mathbf{H}$

for the homotopy quotient construction by $\infty$-actions of smooth $\infty$-groups $\mathcal{G}  \,\in\, \Groups\big(\mathbf{H}\big)$

\item $\mathbf{B}\mathcal{G} \simeq \HomotopyQuotient{\ast}{\mathcal{G}}\;\in\; \mathbf{H}$

for the delooping (``moduli stack'') of any group $\infty$-stack $\mathcal{G} \in \Groups(\SmoothInfinityGroupoids)$;

\item $\mathbf{B}^{n+1} \mathcal{A} \,=\, \mathbf{B}\big( \mathbf{B}^n \mathcal{A} \big)$

for the iterative delooping of $(n+1)$-fold commutative $\infty$-groups

\end{itemize}

\end{notation}

We denote homotopy Cartesian squares by putting the label ``\scalebox{.7}{(pb)}'' at their center.

\begin{fact}[Homotopy pasting law, {\cite[Pro. 13.3.15]{Hirschhorn02}\cite[Lem. 4.4.2.1]{Lurie09HTT}}]\label{PastingLaw}
  Given a pasting diagram of homotopy commutative squares
  \vspace{-2mm}
  $$
    \begin{tikzcd}
      {}
      \ar[r]
      \ar[d]
      &
      \ar[r]
      \ar[
        dr,
        phantom,
        "{ \scalebox{.7}{\rm(pb)} }"
      ]
      {}
      \ar[d]
      &
      {}
      \ar[d]
      \\
      {}
      \ar[r]
      &
      \ar[r]
      {}
      &
      {}
    \end{tikzcd}
  $$

  \vspace{-2mm}
\noindent  where the right square is homotopy cartesian (is a homotopy pullback), then the left square is
so if and only if the total rectangle is. {\rm (NB: In 1-categories this restricts to the pasting law for plain Cartesian/pullback squares.)}
\end{fact}

\begin{fact}[Fundamental theorem of $\infty$-topos theory, {\cite[\S 6.3.5]{Lurie09HTT}}]
Given an $\infty$-topos $\mathbf{H}$, then:
\begin{itemize}

\item[\bf (i)] for every object $\mathcal{X} \,\in\, \mathbf{H}$ the slice $\infty$-category $\mathbf{H}_{/\mathcal{X}}$ is again an $\infty$-topos;

\item[\bf (ii)] for every morphism $\mathcal{X} \xrightarrow{f} \mathcal{Y}$ in $\mathbf{H}$ there is an induced  {\it base change} adjoint triple:
\vspace{-2mm}
\begin{equation}
  \label{BaseChange}
  \begin{tikzcd}
    \mathbf{H}_{/\mathcal{X}}
    \ar[
      rr,
      shift left=13pt,
      "{ f_! }"{}
    ]
    \ar[
      from=rr,
      "{ f^\ast }"{description}
    ]
    \ar[
      rr,
      shift right=12pt,
      "{ f_\ast }"{swap}
    ]
    \ar[
      rr,
      phantom,
      shift left=8.5pt,
      "{ \scalebox{.65}{$\bot$} }"
    ]
    \ar[
      rr,
      phantom,
      shift right=7pt,
      "{ \scalebox{.65}{$\bot$} }"
    ]
    &&
    \mathbf{H}_{/\mathcal{Y}}
    \mathrlap{\,,}
  \end{tikzcd}
\end{equation}

\vspace{-2mm}
\noindent
where, in terms of $\mathbf{H}$, $f^\ast$ is given by pullback along $f$ and $f_!$ by postcomposition with $f$.
\end{itemize}
\end{fact}

\begin{notation}[Cohesive/smooth $\infty$-toposes]
\label{BasicNotation}

\noindent
We write:

\begin{itemize}[leftmargin=.5cm]
\setlength\itemsep{-5pt}
\item $\SmoothInfinityGroupoids \;:=\; \InfinitySheaves(\CartesianSpaces)$

for the $\infty$-topos over the site of Cartesian spaces (equivalently over all smooth manifolds) with smooth functions between them and with respect to differentiably good open covers,
presented by the projective local model structure on simplicial presheaves over this site (cf. Def. \ref{SmoothInfinityGroupoids}).
\\

\item
$ \SmoothManifolds \xhookrightarrow{\;\;} \mathrm{SmthOrbfld} \xhookrightarrow{\;\;} \SmoothInfinityGroupoids$

for the full inclusion of the $(2,1)$-category of orbifolds regarded as differentiable stacks,

among these the {\it good orbifolds} are the
global homotopy quotients $\HomotopyQuotient{X}{G} \in \SmoothInfinityGroupoids$,

for $G\acts \, X$ a smooth action of a discrete group on a smooth manifold $X$. \\

\item $\Maps{\big}{\mathcal{X}}{\mathcal{A}} \,\in\, \SmoothInfinityGroupoids$

for the stack of maps between $\mathcal{X}, \mathcal{A} \,\in\, \SmoothInfinityGroupoids$ (their internal hom).

\item
\begin{equation}
  \label{ShapeModality}
  \begin{tikzcd}[row sep=5pt, column sep=20pt]
  \SmoothInfinityGroupoids
    \ar[
      rr,
      "\shape"{description},
      "
        \mbox{
          \tiny
          \color{darkblue}
          \bf
          shape modality
        }
      "{above, yshift=5pt}
    ]
  \ar[
    dr,
    "{\left\Vert - \right\Vert\;\;\;\;\;\;}"{below}
  ]
    &&
  \SmoothInfinityGroupoids
  \\
  &
  \InfinityGroupoids
  \ar[
    ur,
    hook,
    "\Discrete"{below, sloped}
  ]
  \end{tikzcd}
  \;\;\;\;
  \;\;\;\;
  \begin{tikzcd}[column sep=40pt]
    {(-)}
    \ar[
      r,
      "{
        \eta^{\scalebox{.5}{\shape}}_{(-)}
      }"{below},
      "{
        \mbox{
          \tiny
          \color{greenii}
          \bf
          shape unit
        }
      }"{above, yshift=7pt}
    ]
    &
    \shape (-)
  \end{tikzcd}
  ,
  \;\;\;\;
  \;\;\;\;
  \overset{
    \raisebox{3pt}{
      \tiny
      \color{greenii}
      \bf
      shape idempotency
    }
  }{
    \shape \circ \shape \xrightarrow{\;\;\sim\;\;} \shape
  }
\end{equation}
\vspace{-.4cm}

for the {\it shape monad}
sending any stack to the homotopy type of its
fat geometric realization \cite[Ex. 3.18]{SS20OrbifoldCohomology},
re-embedded as a geometrically discrete $\infty$-stack.
\\

\item
$
  \CohesiveCircle
  \;\in\;
  \mathrm{DTopSp}
  \xhookrightarrow{\;}
  \mathrm{DfflgSpc}
  \xhookrightarrow{\;}
  \SmoothInfinityStacks
  \,.
$

 for the {\it cohesive circle},
 i.e. the circle with its usual structure of a topologogical space\footnote{We may alternatively regard $\CohesiveCircle \,\in\, \SmoothManifolds \hookrightarrow \mathrm{DfflgSpc} \hookrightarrow \SmoothInfinityGroupoids$ with its standard structure of a smooth manifold. All constructions in the main text remain valid, only that in this case they pass through differentiable stacks instead of topological stacks.}, so that

 $\shape \,\CohesiveCircle \;\simeq\; \mathbf{B}\mathbb{Z} \;\in\; \mathrm{Grpd}_1 \hookrightarrow \InfinityGroupoids \xhookrightarrow{ \mathrm{Disc} } \SmoothInfinityGroupoids$

 denotes the homotopy type underlying the circle (cf. Lemma \ref{ShapeUnitOfTheSmoothCircle}).

\end{itemize}

\end{notation}

\noindent
{\bf Simplicial presheaves.} The possibly earliest reference on the homotopy theory of simplicial presheaves is \cite{BrownAHT}, which is still highly recommended reading. A comprehensive modern monograph is \cite{Jardine15}. The cohesive example over the site $\CartesianSpaces$ originates with \cite[App.]{FStS12}. All facts that we need here are concisely reviewed and referenced in \cite[\S A]{FSS20CharacterMap} and \cite[\S 3.2]{SS21EPB}.

\begin{notation}[Model categories of simplicial presheaves]
  \label{ModelCategoriesOfSimplicialPresheaves}
  Given a simplicial site $(\mathcal{C}, J)$,
  i.e. a small simplicially enriched category (Rem. \ref{SelfEnrichmentOfSimplicialSets})
  equipped with a Grothendieck pre-topology (coverage) on its homotopy category, we write:

  \noindent
  {\bf (i)}
  $
    \SimplicialPresheaves(\mathcal{C})
    \;\in\;
    \SimplicialCategories
  $
  for the simplicial category of simplicial presheaves on $\mathcal{C}$;

  \noindent
  {\bf (ii)}
  $
    \SimplicialPresheaves(\mathcal{C})_{ {\mathrm{inj}/} \atop \mathrm{proj} }
    \;\in\;
    \ProperCombinatorialSimplicialModelCategories
  $
  for the global injective or projective model category structure, respectively;

  \noindent
  {\bf (iii)}
  $
    \SimplicialPresheaves(\mathcal{C},J)_{ {\mathrm{inj}/} \atop \mathrm{proj}, \mathrm{loc} }
    \;\in\;
    \ProperCombinatorialSimplicialModelCategories
  $
  for the $J$-local injective or projective model category structure,
  whose weak equivalences are the $J$-stalk-wise weak equivalences in
  $\SimplicialSets_{\mathrm{Qu}}$ (the ``hypercomplete'' local model structure).
\end{notation}

\begin{definition}[Smooth $\infty$-groupoids]
  \label{SmoothInfinityGroupoids}
  We write
  $$
    \SmoothInfinityGroupoids
    \;\coloneqq\;
    \Localization{\mathrm{W}}
    \big(
      \SimplicialPresheaves(\mathrm{CartSp})_{ {\mathrm{proj}} \atop {\mathrm{loc}} }
    \big)
    \;\;
    \simeq
    \;\;
    \Localization{\mathrm{W}_{\mathrm{loc}}}
    \big(
      \SimplicialPresheaves(\mathrm{\SmoothManifolds})_{ {\mathrm{proj}} \atop  {\mathrm{loc}} }
    \big)
    \;\;
    \in
    \;
    2\mathrm{Ho}
    \big(
      \InfinityToposes
    \big)
    \,.
  $$
  for the $\infty$-topos presented by the local model structure
  (Nota. \ref{ModelCategoriesOfSimplicialPresheaves}) with respect to
  differentiably good open covers of Cartesian spaces.
\end{definition}

The point of working over the site $\CartesianSpaces$ instead of over the (hypercompletely) equivalent site of all smooth manifolds is that it allows for more efficient computations (ultimately related to the fact that smooth manifolds are themselves already glued from Cartesian spaces, hence that the inclusion $\CartesianSpaces \hookrightarrow \SmoothManifolds$ exhibits a dense subsite). Namely, over $\CartesianSpaces$

\begin{itemize}[leftmargin=.9cm]
\item[\bf (i)] the simplicial delooping $\overline{W}$ (Def. \ref{StandardModelOfUniversalSimplicialPrincipalComplex}) of every homotopy-sheaf of simplicial groups is itself a homotopy-sheaf, in that it is fibrant in the local model structure (Ntn. \ref{ProjectiveModelStructureOnActionsOfSimplicialGroups});

(\cite[Prop. 3.3.30]{SS21EPB}\cite[Prop. 4.13]{Pavlov22})

\item[\bf (ii)] cofibrant replacement is still nicely tractable, namely given by (differentiably) {\it good open covers}.
\end{itemize}

\noindent
In the main text we mainly appeal to this second fact, and so in the remainer we spell this out further:

\begin{proposition}[Dugger's cofibrancy recognition {\cite[Cor . 9.4]{Dugger01Universal}}]
  \label{DuggerCofibrancyRecognition}
  Let $\mathcal{S}$ be a 1-site.
  A sufficient condition for
  $X \in \mathrm{SimplPSh}(\mathcal{S})_{ {\mathrm{proj}} \atop {\mathrm{loc}} }$
  (Nota. \ref{ModelCategoriesOfSimplicialPresheaves})
  to be
  projectively cofibrant
  is that in each simplicial degree $k$,
  the component presheaf $X_k \in \mathrm{PSh}(\mathcal{S})$ is

  {\bf (i)} a coproduct $X_k \simeq \underset{i_k}{\coprod} U_{i_k}$
    of representables $U_{i_k} \in \mathcal{S} \xhookrightarrow{y} \mathrm{PSh}(\mathcal{S})$;

  {\bf (ii)} whose degenerate cells split off as a disjoint summand:
    $X_k \simeq N_k \coprod \mathrm{im}(\sigma)$
    for some $N_k$.
\end{proposition}

\begin{example}[Basic examples of projectively cofibrant simplicial presheaves]
  \label{ExamplesOfProjectiveleCofibrantSimplicialPresheaves}
  Let $\mathcal{S}$ be a 1-site.
  Examples of projectively cofibrant simplicial presheaves
  over $\mathcal{S}$ include:

  \begin{itemize}
  \item[\bf (i)]
  every representable $U \in \mathcal{Y} \xhookrightarrow{\;y\;} \mathrm{SimplPSh}(\mathcal{S})$;

  \item[\bf (ii)] every constant
    simplicial presheaf
    $S \in \mathrm{SimpSets} \xhookrightarrow{\rm const} \mathrm{SimplPSh}(\mathcal{S})$;

  \hspace{-1cm}and in joint generalization of these two cases:

  \item[\bf (iv)]
   every product $U \times S$ of a representable with a
  simplicial set.
\end{itemize}

  \noindent
  In all cases, the defining lifting property is readily checked.
  Alternatively, these follow with
  Prop. \ref{DuggerCofibrancyRecognition}.
\end{example}

\begin{proposition}[Quillen functor for shape modality on smooth $\infty$-groupoids]
  \label{QuillenFunctorForShapeModalityOnSmoothInfinityGroupoids}
  The shape-monad \eqref{ShapeModality}
  on smooth $\infty$-groupoids (Def. \ref{SmoothInfinityGroupoids})
  \vspace{-2mm}
  $$
    \begin{tikzcd}
      \SmoothInfinityGroupoids
      \ar[
        rr,
        shift left=4pt,
        "{\Shape}"
      ]
      \ar[
        rr,
        phantom,
        "\scalebox{.6}{$\bot$}"
      ]
      &&
      \InfinityGroupoids
      \ar[
        ll,
        shift left=4pt,
        "{\Discrete}"
      ]
    \end{tikzcd}
  $$

  \vspace{-2mm}
  \noindent
  is equivalently the left derived functor of
  the colimit operation on simplicial presheaves over Cartesian spaces,
  regarded as functors $\CartesianSpaces^{\mathrm{op}} \xrightarrow{\;} S\SimplicialSets$,
  in that the following is a Quillen adjunction:
    \vspace{-2mm}
  \begin{equation}
    \label{ColimitQuillenAdjunctionOnLocalSimplicialPresheaves}
    \begin{tikzcd}
      \SimplicialPresheaves(\CartesianSpaces)_{ {\mathrm{proj}} \atop {\mathrm{loc}} }
      \ar[
        rr,
        shift left=5pt,
        "{\underset{\longrightarrow}{\mathrm{lim}}}"{above}
      ]
      \ar[
        rr,
        phantom,
        "{\scalebox{.65}{$\bot_{\mathrlap{\mathrm{Qu}}}$}}"
      ]
      &&
      \SimplicialSets_{\mathrm{Qu}}
      \ar[
        ll,
        shift left=5pt,
        "{\mathrm{const}}"{below}
      ]
    \end{tikzcd}
  \end{equation}

  \vspace{-2mm}
  \noindent
  Moreover, on a simplicial presheaf satisfying Dugger's cofibrancy condition
    \vspace{-2mm}
  (Prop. \ref{DuggerCofibrancyRecognition})
  \begin{equation}
    \label{DuggerCofibrantResolution}
    \begin{tikzcd}
      \varnothing
      \ar[
        r,
        "\in \mathrm{Cof}"
      ]
      &
      \underset{
        i_\bullet \in I_\bullet
      }{\coprod}
      \mathbb{R}^{n_{i_\bullet}}
      \ar[
        r,
        "\in \mathrm{W}"
      ]
      &
      X
    \end{tikzcd}
    \;\;
    \in
    \SimplicialPresheaves(\CartesianSpaces)_{ {\mathrm{proh}} \atop {\mathrm{loc}} }
  \end{equation}

    \vspace{-3mm}
\noindent
  the shape is given by the simplicial set obtained by contracting all copies of
  Cartesian spaces to the point:
    \vspace{-2mm}
  $$
    \shape\, X
    \;\simeq\;
      \underset{
        i_\bullet \in I_\bullet
      }{\coprod}
      \ast
      \;\;\;
      \in
      \;
      \SimplicialSets_{\mathrm{Qu}}
      \,.
  $$
\end{proposition}
\begin{proof}
  First observe that the colimit over a representable functor is the point
  (e.g. \cite[Lem. 2.40]{SS20OrbifoldCohomology})
    \vspace{-2mm}
  \begin{equation}
    \label{ColimitOfRepresentableIsPoint}
    \underset{\longrightarrow}{\mathrm{lim}}
    \,
    \mathbb{R}^n
    \;\coloneqq\;
    \underset{\longrightarrow}{\mathrm{lim}}
    \,
    y(\mathbb{R}^n)
    \;\simeq\;
    \ast
    \;\;\;
    \in
    \;
    \Sets
    \xhookrightarrow{\;\;}
    \SimplicialSets
    \,,
  \end{equation}

    \vspace{-2mm}
\noindent
  so that the colimit of a simplicial presheaf of the form \eqref{DuggerCofibrantResolution}
  is the simplicial set obtained by replacing all copies of Cartesian spaces
  by a point:
    \vspace{-2mm}
  \begin{equation}
    \label{ColimitOfDuggerCofibrantSimplicialPresheaf}
    \begin{array}{lll}
      \underset{\longrightarrow}{\mathrm{lim}}
      \Big(\;
        \underset{ i_\bullet \in I_\bullet }{\coprod}
        \mathbb{R}^{n_{i_\bullet}}
     \!\! \Big)
      & \;\simeq\;
        \big(
          \underset{\longrightarrow}{\mathrm{lim}}
          \,
          \mathbb{R}^{n_{i_\bullet}}
        \big)
      &
      \proofstep{since colimits commute with coproducts
      }
      \\
      & \;\simeq\;
      \underset{ i_\bullet \in I_\bullet }{\coprod}
      \ast
      &
      \proofstep{by \eqref{ColimitOfRepresentableIsPoint}}
      \!.
    \end{array}
  \end{equation}

    \vspace{-2mm}
\noindent
  Next, it is clear that \eqref{ColimitQuillenAdjunctionOnLocalSimplicialPresheaves}
  is a simplicial Quillen adjunction for the {\it global} projective model structure.
  To show that it is also Quillen for the local model structure it is hence sufficient,
  by \cite[Cor. A.3.7.2]{Lurie09HTT}, to see that the right adjoint preserves
  fibrant objects. By adjunction this is equivalent to the statement that
  for $\big\{ U_i \hookrightarrow X \big\}$ a differentiably good open cover,
  with $U \,\coloneqq\, \underset{i}{\coprod} U_i$,
  we
  have a simplicial weak homotopy equivalence
  \vspace{-4mm}
  \begin{equation}
    \label{WeakHomotopyEquivalenceFromColimitOfCechNerveOfGoddOpenCoverOfCartesianSpace}
    \underset{\longrightarrow}{\mathrm{lim}}\, y(U^{\times^\bullet_X}  )
    \xrightarrow{\;\in \mathrm{W}\;}
    \ast
    \,.
  \end{equation}

  \vspace{-2mm}
 \noindent But, by \eqref{ColimitOfDuggerCofibrantSimplicialPresheaf},
  the left hand side of \eqref{WeakHomotopyEquivalenceFromColimitOfCechNerveOfGoddOpenCoverOfCartesianSpace} is
  the simplicial set obtained by contracting summands of the Cech nerve of the good cover to the point.
  Therefore, since any Cartesian space is contractible,
  the {\it nerve theorem}
  (\cite[Thm. 2]{McCord67}, review in \cite[Prop. 4G.3]{HatcherAlgebraicTopology})
  implies \eqref{WeakHomotopyEquivalenceFromColimitOfCechNerveOfGoddOpenCoverOfCartesianSpace}.
    With this, the last statement follows from
  the fact that
  left derived functors may be computed on any cofibrant resolution:
    \vspace{-2mm}
  $$
    \begin{array}{lll}
      \shape X
      &
      \;\simeq\;
      (\mathbb{L} \underset{\longrightarrow}{\mathrm{lim}})(X)
      &
      \proofstep{by \eqref{ColimitQuillenAdjunctionOnLocalSimplicialPresheaves}}
      \\
      &
      \;\simeq\;
      \underset{\longrightarrow}{\mathrm{lim}}
      \Big(\;
        \underset{i_\bullet \in I_\bullet}{\coprod}
        \mathbb{R}^n
      \!\Big)
      &
      \proofstep{by Prop. \ref{DuggerCofibrancyRecognition}}
      \\
      & \;\simeq\;
        \underset{i_\bullet \in I_\bullet}{\coprod}
        \ast
      &
      \proofstep{by \eqref{ColimitOfDuggerCofibrantSimplicialPresheaf}}
      \!.
    \end{array}
  $$

  \vspace{-8mm}
\end{proof}

\begin{example}[Good open covers are projectively cofibrant resolutions of smooth manifolds]
  \label{GoodOpenCoversAreProjectivelyCofibrantResolutionsOfSmoothManiolds}
  $\,$

  \noindent
  Any $X \in \SmoothManifolds \xrightarrow{\; y \;} \SimplicialPresheaves(\CartesianSpaces)$
  admits a
  {\it differentiably good open cover}
  (\cite[Prop. A.1]{FStS12}),
  namely an open cover
  such that all non-empty finite intersections of patches are
  {\it diffeomorphic} to an open ball, and hence to $\mathbb{R}^{\mathrm{dim}(X)}$:
  \vspace{-3mm}
  \begin{equation}
    \label{GoodOpenCover}
    \big\{
      U_i \,\simeq\, \mathbb{R}^{\mathrm{dim}(X)} \xhookrightarrow{\;} X
    \big\}_{i \in I}
    \,,
    \;\;\;
    \mbox{s.t.}
    \quad
    \underset
      {
        { k \in \mathbb{N} }
        \atop
        { i_0, i_1, \cdots, i_k \in I }
      }
      {\forall}
      \;\;\;
      U_{i_0} \cap U_{i_1} \cap  \cdots \cap U_{i_k}
      \;
      \underset{
        \mathclap{
          \mathrm{diff}
        }
      }{
        \simeq
      }
      \;
      \mathbb{R}^{\mathrm{dim}(X)}
      \;\;\;
      \proofstep{if non-empty}
      \,.
  \end{equation}

     \vspace{-3mm}
\noindent  By Dugger's recognition (Prop. \ref{DuggerCofibrancyRecognition})
  this means that the corresponding Cech nerve is projectively cofibrant;
  moreover, its canonical morphism to $X$ is
  clearly a stalkwise weak equivalence, so that it provides a
  cofibrant resolution of $X$ in the local model structure (Def. \ref{SmoothInfinityGroupoids}):
   \vspace{-3mm}
  $$
    \begin{tikzcd}
      \varnothing
      \ar[
        r,
        "\in \mathrm{Cof}"
      ]
      &
      U^{\times^\bullet_X}
      \ar[
        r,
        "\in \mathrm{W}_{\mathrm{loc}}"
      ]
      &
      X
    \end{tikzcd}
    \;\in\;
    \SimplicialPresheaves(\CartesianSpaces)_{ {\mathrm{proj}} \atop {\mathrm{loc}} }
    \,,
    \;\;\;\;
    \mbox{for}
    \;
    U \;\coloneqq \underset{i}{\coprod} U_i\;.
  $$
\end{example}
\begin{proposition}[Shape of smooth manifolds is their homotopy type]
  For $X \in \SmoothManifolds \,\xhookrightarrow{\;y\;}\, \SmoothInfinityGroupoids$
  (Def. \ref{SmoothInfinityGroupoids}) their shape is their standard homotopy type.
\end{proposition}

\begin{example}[Standard cofibrant resolution of the smooth circle]
  \label{StandardCofibrantResolutionOfTheCircle}
  Considering the smooth circle as the quotient of the real numbers by the
  integers
    \vspace{-2mm}
  $$
    \begin{tikzcd}[row sep=5pt]
      \mathbb{Z}
      \ar[
        r,
        hook,
        "i"
      ]
      &
      \mathbb{R}
      \ar[
        r,
        ->>,
        "p"
      ]
      &
      S^1
    \end{tikzcd}
    \in
    \;\;
    \SmoothManifolds
    \xhookrightarrow{\;y\;}
    \SimplicialPresheaves(\CartesianSpaces)_{ {\mathrm{proj}} \atop {\mathrm{loc}} }
    \,,
  $$

    \vspace{-2mm}
\noindent  Dugger's recognition (Prop. \ref{DuggerCofibrancyRecognition})
  shows that
  the Cech nerve of $p$ constitutes a cofibrant resolution of the circle
  in the projective local model structure (Def. \ref{SmoothInfinityGroupoids})
  \vspace{-2mm}
  $$
  \begin{tikzcd}[row sep=-2pt]
    \varnothing
    \ar[
      r,
      "\in \mathrm{Cof}"{above}
    ]
    &
    \mathbb{R} \times \mathbb{Z}^{\times^\bullet}
    \ar[
      rr,
      "\sim"{above, yshift=-1pt}
    ]
    &&
    \qquad
    \mathbb{R}^{\times^\bullet_{S^1}}
    \qquad
    \ar[
      r,
      "\in \mathrm{W}"
    ]
    &
    S^1\;.
    \\
    &
      \scalebox{0.7}{$  (r,\vec n) $}
    &\longmapsto&
    \hspace{-3cm}
    \mathrlap{
     \scalebox{0.7}{$  \big(
      r, (r + n_1), (r + n_1 + n_2), \cdots
    \big)
    $}
    }
  \end{tikzcd}
  $$
\end{example}

\begin{proposition}[Presentation of $\infty$-topos by simplicial presheaves]
  \label{PresentationOfInfinityToposesBySimplicialPresheaves}
  Let $\mathcal{C}$ be a 1-site. Then the
  {\v C}ech/stalk-local
  injective or projective model category structure on
  simplicial presheaves over $\mathcal{C}$ presents
  the topological/hypercomplete $\infty$-topos over $\mathcal{C}$
  in that there is an equivalence of homotopy categories
  \vspace{-3mm}
  $$
    \begin{tikzcd}
    \HomotopyCategory
    \Big(
      \SimplicialPresheaves(\mathcal{C})
        _{
          { {\mathrm{inj}/} \atop { \mathrm{proj} } },
          \mathrm{loc}
        }
    \Big)
    \ar[
      rr,
      "\sim"{above, yshift=-2pt},
      " \HomotopyCategory(L) "{below}
    ]
    &&
    \HomotopyCategory
    \big(
      \infinitySheaves(\mathcal{C})
    \big)
    \end{tikzcd}
  $$

  \vspace{-3mm}
  \noindent
  and for all cofibrant $X$ and fibrant $A$
  in $\SimplicialPresheaves(\mathcal{C})$ an equivalence of
  hom-$\infty$-groupoids
  \vspace{-1mm}
  \begin{equation}
    \SimplicialPresheaves(\mathcal{C})
    (
      X, A
    )
    \;\simeq\;
    \InfinitySheaves
    \big(
      L(X), L(A)
    \big)
    \,.
  \end{equation}
\end{proposition}

\begin{proposition}[$\infty$-Yoneda lemma]
  \label{InfinityYonedaLemma}
  For $\mathcal{S} \,\in\, \mathrm{Categories}_\infty$
  we have a fully faithful embedding
  \vspace{-2mm}
  \begin{equation}
    \label{YonedaEmbedding}
    \begin{tikzcd}[row sep=-3pt]
      \mathcal{S}
      \ar[
        rr,
        hook,
        "y"
      ]
      &&
      \mathrm{PSh}_\infty(\mathcal{S})
      \\
        \scalebox{0.7}{$  U$}
         &\longmapsto&
           \scalebox{0.7}{$ \mathcal{S}(-,U) $}
    \end{tikzcd}
  \end{equation}

  \vspace{-3mm}
  \noindent
  and a
  natural equivalence
  for $U \in \mathcal{C}$ and $X \in \mathrm{PSh}_\infty(\mathcal{S})$:

  \vspace{-2mm}
  \begin{equation}
    \label{YonedaEquivalence}
    \mathrm{PSh}_\infty
    \big(
      y(U),
      \,
      X
    \big)
    \;\simeq\;
    X(U)
    \,.
  \end{equation}
\end{proposition}

%\newpage

%%%%%%%%%%%%%%%%%%%%%%%%%%%%

\end{document}